\RequirePackage{rotating}
\documentclass[a4paper,twopage,reqno,12pt]{amsart}
\usepackage[top=20mm,right=30mm,bottom=20mm,left=30mm]{geometry}
\usepackage{rotating}
\usepackage{tabularx}
\usepackage{graphicx}
\usepackage{amsmath}
\usepackage{amssymb}
\usepackage{amsfonts}
\usepackage{amsthm}
\usepackage{amstext}
\usepackage{amsbsy}
\usepackage{amsopn}
\usepackage{amscd}
\usepackage{enumitem}
\usepackage{color}
\usepackage{comment}
\usepackage[colorlinks]{hyperref}
\usepackage[hyperpageref]{backref}
\usepackage{multirow}
\usepackage{rotating}
\usepackage{longtable}
\usepackage{float}
\usepackage{lscape}
\usepackage{pdflscape}
\usepackage{url}

\newtheorem{theorem}{Theorem}[section]
\newtheorem{corollary}[theorem]{Corollary}
\newtheorem{lemma}[theorem]{Lemma}
\newtheorem{proposition}[theorem]{Proposition}
\newtheorem{claim}[]{Claim}
\usepackage{etoolbox}
\AtEndEnvironment{proof}{\setcounter{claim}{0}}
\theoremstyle{definition}

\newtheorem{remark}[theorem]{Remark}

\newtheorem{hypothesis}[theorem]{Hypothesis}
\numberwithin{equation}{section}

\newcommand{\C}{\mathrm{C}}

\newcommand{\Out}{\mathrm{Out}}

\newcommand{\F}{\mathrm{F}}

\renewcommand{\leq}{\leqslant}
\renewcommand{\geq}{\geqslant}

\begin{document}
\title[The Higman-M\lowercase{c}Laughlin Theorem]{The Higman-M\lowercase{c}Laughlin Theorem for the flag-transitive $2$-designs with $\lambda$ prime}

\author[]{Alessandro Montinaro}
\address{Alessandro Montinaro, Dipartimento di Matematica e Fisica “E. De Giorgi”, University of Salento, Lecce 73100, Italy}
\email{alessandro.montinaro@unisalento.it}
\subjclass[MSC 2020:]{05B05; 05B25; 20B25}
\date{\today }

\subjclass[MSC 2020:]{05B05; 05B25; 20B25}%
\date{\today}%

\begin{abstract}
A famous result of Higman and McLaughlin \cite{HM} in 1961 asserts that any flag-transitive automorphism group $G$ of a $2$-design $\mathcal{D}$ with $\lambda=1$ acts point-primitively on $\mathcal{D}$. In this paper, we show that the Higman and McLaughlin theorem is still true when $\lambda$ is a prime and $\mathcal{D}$ is not isomorphic to one of the two $2$-$(16,6,2)$ designs as in \cite[Section 1.2]{ORR}, or the $2$-$(45,12,3)$ design as in \cite[Construction 4.2]{P}, or, when $2^{2^{j}}+1$ is a Fermat prime, a possible $2$-$(2^{2^{j+1}}(2^{2^{j}}+2),2^{2^{j}}(2^{2^{j}}+1),2^{2^{j}}+1)$ design having very specific features.
\end{abstract}

\maketitle

\section{Introduction and Main Result}

A $2$-$(v,k,\lambda )$ \emph{design} $\mathcal{D}$ is a pair $(\mathcal{P},%
\mathcal{B})$ with a set $\mathcal{P}$ of $v$ points and a set $\mathcal{B}$
of $b$ blocks such that each block is a $k$-subset of $\mathcal{P}$ and each two distinct points are contained in $\lambda $ blocks. The \emph{replication number} $r$ of $\mathcal{D}$ is the number of blocks containing a given point.
We say $\mathcal{D}$ is \emph{non-trivial} if $2<k<v$, and \emph{symmetric} if $v=b$. An automorphism of $%
\mathcal{D}$ is a permutation of the point set which preserves the block
set. The set of all automorphisms of $\mathcal{D}$ with the composition of
permutations forms a group, denoted by $\mathrm{Aut(\mathcal{D})}$. For a
subgroup $G$ of $\mathrm{Aut(\mathcal{D})}$, $G$ is said to be \emph{%
point-primitive} if $G$ acts primitively on $\mathcal{P}$, and said to be 
\emph{point-imprimitive} otherwise. In this setting, we also say that $%
\mathcal{D}$ is either \emph{point-primitive} or \emph{point-imprimitive}, respectively. A \emph{flag} of $\mathcal{D}$ is a pair $(x,B)$ where $x$ is
a point and $B$ is a block containing $x$. If $G\leq \mathrm{Aut(\mathcal{D})%
}$ acts transitively on the set of flags of $\mathcal{D}$, then we say that $%
G$ is \emph{flag-transitive} and that $\mathcal{D}$ is a \emph{%
flag-transitive design}.

A famous result of Higman and McLaughlin \cite{HM} in 1961 asserts that, if $\lambda=1$, that is $\mathcal{D}$ is a linear space, then $G$ acts point-primitively on $\mathcal{D}$. Since then, several authors investigated under which conditions on $\mathcal{D}$ or on $G$ this implication is still true that a flag-transitive automorphism group $G$ of $2$-$(v,k,\lambda )$ design a $\mathcal{D}$ acts point-primitively on $\mathcal{D}$. In his famous book \emph{Finite Geometries}, Dembowski \cite{Demb} provided several conditions on the parameters of $\mathcal{D}$ such that the flag-transitivity implies the point-primitivity. More recently, Zhong and Zhou \cite{CZ} proved that above implication is true when $(v-1,k-1)\leq (v-1)^{1/2}$. However, as shown by Davies in \cite{Da}, in general it is not true that point-primitivity arises from flag-transitivity for $\lambda>1$. Nevertheless, Devillers and Praeger \cite{DP} in 2022 showed that the Higman-McLaughling theorem can be extended to the case $\lambda \leq 4$ except for $v \in \{15,16,36,45,96\}$ and $\mathcal{D}$ is isomorphic to eleven specific $2$-designs. In the present paper, we move further by showing that the conclusion of the Higman-McLaughling theorem still holds when $\lambda$ is a prime number except when $\mathcal{D}$ is isomorphic to three specific symmetric $2$-designs, $\lambda$ is a Fermat prime greater than 3 and $\mathcal{D}$ is a $2$-design with very specific features and whose existence is open. More precisely, we obtain the following result:

\begin{theorem}\label{main}
 Let $\mathcal{D}$ be a $2$-$(v,k,\lambda )$ design with $\lambda$ prime admitting a flag-transitive point-imprimitive automorphism group. Then one of the
following holds:
\begin{enumerate}
\item $\mathcal{D}$ is one of the two $2$-$(16,6,2)$ designs as in \cite[Section 1.2]{ORR};
\item $\mathcal{D}$ is the $2$-$(45,12,3)$ design as in \cite[Construction 4.2]{P}.
\item $\mathcal{D}$ is a $2$-$(2^{2^{j+1}}(2^{2^{j}}+2),2^{2^{j}}(2^{2^{j}}+1),2^{2^{j}}+1)$ design when $2^{2^{j}}+1>3$ is a Fermat prime.
\end{enumerate}
\end{theorem}

\bigskip
More details on $(\mathcal{D},G)$ when $\mathcal{D}$ is a $2$-design as in (3) of Theorem \ref{main} can be found in Theorem \ref{Teo6}. The existence of such $2$-designs is investigated by the author in a forthcoming paper. Combining Theorem \ref{main} with \cite[Theorem]{Zie} and \cite[Theorem 1]{ZC}, we also obtain the following result.
\bigskip
\begin{corollary} \label{mainCor}
 Let $\mathcal{D}$ be a $2$-$(v,k,\lambda )$ design with $\lambda$ prime admitting a flag-transitive automorphism group $G$. If $\mathcal{D}$ not isomorphic to the two $2$-$(16,6,2)$ designs as in \cite[Section 1.2]{ORR}, or to $2$-$(45,12,3)$ design as in \cite[Construction 4.2]{P}, or, when $2^{2^{j}}+1>3$ is a Fermat prime, to possible $2$-$(2^{2^{j+1}}(2^{2^{j}}+2),2^{2^{j}}(2^{2^{j}}+1),2^{2^{j}}+1)$ design with very specific features, then $G$ acts point-primitively on $\mathcal{D}$. Moreover, either $G$ is of affine type, or $G$ is an almost simple group. 
\end{corollary}

\bigskip

\subsection{Structure of the paper and outline of the proof.} 
The paper consists of nine sections briefly described below. In Section 1, we introduce the problem, present the state of the art, and state our main results: Theorem \ref{main} and Corollary \ref{mainCor}. In section 2, we provide some useful reduction results in both design theory and group theory. In particular, we rely on the following three results:
    \begin{enumerate}
        \item Theorem of Dembowski \cite[2.3.7(a)]{Demb}: any flag-transitive automorphism group $G$ of a $2$-$(v,k,\lambda)$ design $\mathcal{D}$ with $(r,\lambda)=1$ acts point-primitively.
        \item Theorem of Devillers-Praeger \cite[Theorem 1]{DP}: if $\mathcal{D}$ is any $2$-$(v,k, \lambda)$ design admitting a flag-transitive point-imprimitive automorphism group $G$, then $v \leq 2\lambda^{2}(\lambda-1)$.
        \item Theorem of Camina-Zieschang \cite[Propositions 2.1 and 2.3]{CZ}: let $\mathcal{D}=(\mathcal{P}, \mathcal{B})$ be any $2-(v,k,\lambda)$ design admitting a
flag-transitive, point-imprimitive automorphism group $G$ preserving a nontrivial 
partition $\Sigma $ of $\mathcal{P}$ with $v_1$ classes of size $v_0$, where $v=v_0v_1$. Then there is a constant $k_0 \geq 2$ such that $\left\vert B\cap \Delta \right\vert =0$ or $%
k_0 $ for each $B\in \mathcal{B}$ and $\Delta \in \Sigma $, the parameter $k_0 $ divides $k$. Moreover, the following hold:
\begin{enumerate}
\item $\mathcal{D}_{\Delta}=(\Delta,\mathcal{B}_{\Delta })$ with $\Delta \in \Sigma $ and $\mathcal{B}_{\Delta }=\left\{ B\cap \Delta
\neq \varnothing :B\in \mathcal{B}\right\} $, is either a $2$-$(v_{0},k_{0},\lambda_{0})$ design or a symmetric $1$-design with $v_{0}$ points and $k_{0}=v_{0}-1$. In both cases, $G_{\Delta}^{\Delta}$ acts flag-transitively on $\mathcal{D}_{0}$.  As the designs corresponding to distinct classes $\Delta,\Delta'\in\Sigma$ are isomorphic (under elements of $G$ mapping $\Delta$ to $\Delta'$) we refer to $\mathcal{D}_{\Delta}$ as $\mathcal{D}_{0}$. 
\item For each block $B$ of $\mathcal{D}$ the set $B(\Sigma)=\{\Delta \in \Sigma: B \cap \Delta \neq \varnothing\}$ has a constant size $k_{1}=\frac{k}{k_{0}}$, and the incidence structure $\mathcal{D}^{\Sigma}=\left(\Sigma, \mathcal{B}^{\Sigma}, \mathcal{I} \right)$, where $\mathcal{B}^{\Sigma}$ is the quotient set defined by the equivalence relation $\mathcal{R}=\{(C,C^{\prime}) \in \mathcal{B} \times \mathcal{B}:C(\Sigma)=C^{\prime}(\Sigma)\}$ on $\mathcal{B}$ and $\mathcal{I}=\{(\Delta,C^{\Sigma}) \in \Sigma \times \mathcal{B}^{\Sigma}: \Delta \in C(\Sigma) \}$, is either a $2$-$(v_{1},k_{1},\lambda_{1})$ design or a symmetric $1$-design with $v_{1}$ points and $k_{1}=v_{1}-1$. In both cases, $G^{\Sigma}$ acts flag-transitively on $\mathcal{D}_{1}$.
\end{enumerate}
\end{enumerate}
Section 3 focuses on the interactions between the parameters of $\mathcal{D}$, $\mathcal{D}_{0}$ and $\mathcal{D}_{1}$, and the arguments used are mostly combinatorial. In particular, in Corollary \ref{C1} we prove that the conclusions (1) and (2) of Theorem \ref{main} hold when $\mathcal{D}$ is a symmetric $2$-design. Furthermore, we provide a classification of $\mathcal{D}_{1}$ in five possible families when $\mathcal{D}_{2}$ is a $2$-design. 
    
    In Section 4, we firstly exclude three out of the aforementioned five families for $\mathcal{D}_{1}$. Then, denoted by $r_{1}$ the replication number of $\mathcal{D}_{1}$, we prove Theorem \ref{Teo2}, which states that either $\frac{r_{1}}{(r_{1}, \lambda_{1})} <\lambda $, and hence $G_{\Delta}^{\Delta}$ contains large prime order automorphisms (namely automorphisms of order $\lambda$), or $\frac{r_{1}}{(r_{1}, \lambda_{1})} \geq \lambda $, $(v_{1}-1,k_{1}-1) \leq (v_{1}-1)^{2}$ and $(v_{1}-1)^{2} \neq 2(v_{1}-1,k_{1})-1$, and hence the socle of $G^\Sigma$ is either elementary abelian, or non-abelian simple by a result of Zhong and Zhou \cite[Theorem 1.1]{ZZ}.
    
    Section 5 is dedicated to the case $\frac{r_{1}}{(r_{1}, \lambda_{1})} <\lambda $ and $G_{\Delta}^{\Delta}$ containing large prime order automorphisms. We completely classifies the admissible $(\mathcal{D}_{0},G_{\Delta}^{\Delta})$ by using \cite{LS} together with an argument that combines a result of Camina \cite[Lemma7]{Ca} and the classification of the groups admitting transitive permutation representation of degree a prime and that links the rank of any block stabilizer in $G_{\Delta}^{\Delta}$ with the rank of $G_{\Delta}^{\Delta}$.
    
    Section 6 provides additional meaningful restrictions on $\mathcal{D}_{0}$ and $\mathcal{D}_{1}$ when the action of $G$ on $\Sigma$ in not faithful.
    
    In Sections 7 and 8, we adapt the group-theoretic arguments for flag-transitive linear spaces, developed by Liebeck in \cite{LiebF} and by Saxl in \cite{Saxl}, to severely restrict the possibilities for $(\mathcal{D}_{1}, G^{\Sigma})$. Then we use these information with those on $(\mathcal{D}_{0},G_{\Delta}^{\Delta})$ provided in Section 5 and on the kernel of the action of $G$ on $\Sigma$ provided in Section 6 to force $\mathcal{D}_{1}$ to be a symmetric $1$-design with $k_{1}=v_{1}-1$ 
    
    Section 9 is the final section and is devoted to the completion of the proof of Theorem \ref{main} and Corollary \ref{mainCor}. We firstly show that $\mathcal{D}_{0}$ is a translation plane by matching the admissible parameters of $\mathcal{D}_{0}$ with the group theoretic information provided \cite{LS}. Then, we use the classification of the flag-transitive translation plane given in \cite{LiebF} with the fact that $G$ induces a $2$-transitive  group on $\mathcal{D}_{1}$ to achieve the conclusion of Theorem \ref{main}. Finally, Corollary \ref{mainCor} follows from Theorem \ref{main} and the results of Zieschang \cite{Zie} and Zhang, Chen and Zhou \cite{ZCZ} according as $\lambda$ does not divide or does divide $r$, respectively.

\bigskip

\subsection{Terminology and notation.} The terminology and notation are standard and will follow that of \cite{Demb} for the Design Theory, that of \cite{Asch2,BHRD, DM,Go, Hup,KL,PS} for the Group Theory, and that of \cite{Demb, Hir,HT,HP,Lu} for the geometric structures used throughout the paper.

\section{Preliminaries}\label{prel}

In this section, we provide some useful results in both design theory and group theory. First, we give the following theorems which allow us to reduce the analysis of the flag-transitive automorphism group $G$ of a $2$-$(v,k,\lambda)$ design $\mathcal{D}$.

\begin{lemma}\label{L0}
 The parameters $v$, $b$, $k$, $r$, $\lambda$ of $\mathcal{D}$ satisfy $vr=bk$, $\lambda (v-1)=r(k-1)$, and $k \leq r$.   
\end{lemma}

\bigskip

For a proof see \cite[1.3.8 and 2.1.5]{Demb}.

\bigskip

\begin{theorem}[Dembowski]\label{PetereDemb}
Let $\mathcal{D}=(\mathcal{P}, \mathcal{B})$ be a $2$-$(v,k,\lambda)$ design admitting a
flag-transitive automorphism group $G$. If $G$ acts point-imprimitively on $\mathcal{D}$, then $(r,\lambda)>1$.
\end{theorem}

\bigskip

For a proof see \cite[2.3.7(a)]{Demb}.

\bigskip

\begin{theorem}[Devillers-Praeger]\label{DP}
Let $\mathcal{D}=(\mathcal{P}, \mathcal{B})$ be a $2$-$(v,k,\lambda)$ design admitting a
flag-transitive automorphism group $G$. If $G$ acts point-imprimitively on $\mathcal{D}$, then $v \le 2 \lambda^{2}(\lambda -1)$.
\end{theorem}

\bigskip

For a proof see \cite[Theorem 1]{DP}.

\bigskip

\begin{theorem}[Camina-Zieschang]\label{CamZie}
Let $\mathcal{D}=(\mathcal{P}, \mathcal{B})$ be a $2$-$(v,k,\lambda)$ design admitting a
flag-transitive, point-imprimitive automorphism group $G$ preserving a nontrivial 
partition $\Sigma $ of $\mathcal{P}$ with $v_1$ classes of size $v_0$. Then $v=v_{0}v_{1}$ and the following hold:
\begin{enumerate}
    \item[(1)] There is a constant $k_0 \geq 2$ such that $\left\vert B\cap \Delta \right\vert =0$ or $%
k_0 $ for each $B\in \mathcal{B}$ and $\Delta \in \Sigma $. The parameter $k_0 $ divides $k$, 
\begin{equation}\label{rel1}
\frac{v-1}{k-1}=\frac{v_{0}-1}{k_{0}-1}\text{,}    
\end{equation}
and the following hold:
\begin{enumerate}
\item[(i)] For each $\Delta \in \Sigma $, let $\mathcal{B}_{\Delta }=\left\{ B\cap \Delta
\neq \varnothing :B\in \mathcal{B}\right\} $. Then the set $$\mathcal{R}_{\Delta}=\{(B,C) \in \mathcal{B}_{\Delta } \times \mathcal{B}_{\Delta }:  B\cap \Delta =\C\cap \Delta \}$$ is an equivalence relation on $\mathcal{B}_{\Delta }$ with each equivalence class of size $\mu$.
\item[(ii)] The incidence structure $\mathcal{D}_{\Delta }=(\Delta ,%
\mathcal{B}_{\Delta })$ is either a symmetric $1$-design with $k_{0}=v_{0}-1$, or a $2$-$(v_0,k_0 ,\lambda
_{0})$ design with $\lambda_{0}=\frac{\lambda}{\mu}$.
\item[(iii)] The group $G_{\Delta}^{\Delta}$ acts flag-transitively on $\mathcal{D}_{\Delta }$.
\end{enumerate}
\item[(2)] For each block $B$ of $\mathcal{D}$ the set $B(\Sigma)=\{\Delta \in \Sigma: B \cap \Delta \neq \varnothing\}$ has a constant size $k_{1}=\frac{k}{k_{0}}$. Further, 
\begin{equation}\label{rel2}
\frac{v_{1}-1}{k_{1}-1}=\frac{k_{0}(v_{0}-1)}{v_{0}(k_{0}-1)}    
\end{equation}
and the following hold:
\begin{enumerate}
\item[(i)] The set $$\mathcal{R}=\{(C,C^{\prime}) \in \mathcal{B} \times \mathcal{B}:C(\Sigma)=C^{\prime}(\Sigma)\}$$ is an equivalence relation on $\mathcal{B}$ with each class of size $\eta$;
\item[(ii)] Let $\mathcal{B}^{\Sigma}$ be the quotient set defined by $\mathcal{R}$, and for any block $C$ of $\mathcal{D}$ denote by $C^{\Sigma}$ the $\mathcal{R}$-equivalence class containing $C$. Then the incidence structure $\mathcal{D}^{\Sigma}=\left(\Sigma, \mathcal{B}^{\Sigma}, \mathcal{I} \right)$ with $\mathcal{I}=\{(\Delta,C^{\Sigma}) \in \Sigma \times \mathcal{B}^{\Sigma}: \Delta \in C(\Sigma) \}$ is either a symmetric $1$-design with $k_{1}=v_{1}-1$, or a $2$-$(v_{1},k_{1} ,\lambda_{1})$ design with $\lambda_{1}=\frac{v_{0}^{2}\lambda}{k_{0}^{2}\eta}$.
\item[(iii)] The group $G^{\Sigma}$ acts flag-transitively on $\mathcal{D}^{\Sigma}$.
\end{enumerate}
\end{enumerate}
\end{theorem}

\bigskip
For a proof see \cite[Propositions 2.1 and 2.3]{CZ}.

\bigskip

We refer to $\mathcal{D}^{\Sigma}$ simply as $\mathcal{D}_{1}$. Moreover, as mentioned in the introduction, the designs corresponding to distinct classes $\Delta,\Delta'\in\Sigma$ are isomorphic under elements of $G$ mapping $\Delta$ to $\Delta'$, we refer to $\mathcal{D}_{\Delta }$ as $\mathcal{D}_{0}$. The parameters of $\mathcal{D}_{0}$ and of $\mathcal{D}_{1}$ will be indexed by $0$ and $1$, respectively. Hence, the conclusions of Lemma \ref{L0} hold for $\mathcal{D}_{i}$ when this one is a $2$-$(v_{i},k_{i},\lambda_{i})$ design. That is, $v_{i}r_{i}=b_{i}k_{i}$, $\lambda (v_{i}-1)=r_{i}(k-1)$, and $k_{i} \leq r_{i}$, where $r_{i}$ and $b_{i}$ are the replication number and the number of blocks of $\mathcal{D}_{i}$, respectively. 

\bigskip

\begin{lemma}\label{PP}If $\mathcal{D}_{i}$ is a $2$-$(v_{i},k_{i},\lambda_{i})$ design, $i=0,1$, then the following hold:
\begin{enumerate}
\item $r_{i}>\lambda_{i}^{1/2}v_{i}^{1/2}$;
\item $\frac{r_ {i}}{(r_{i},\lambda_{i})}\mid (e_{i1},...,e_{im_{i}},v_{i}-1)$, where the $e_{0j}$'s are the lengths of the point-$G_{x}^{\Delta}$-orbits on $\mathcal{D}_{0}$ distinct from $\{x\}$, and the $e_{1j}$'s are the lengths of the point-$G_{\Delta}^{\Sigma}$-orbits on $\mathcal{D}_{1}$ distinct from $\{\Delta\}$.
\end{enumerate}
\end{lemma}

\begin{proof}
Since $r_{i} \geq k_{i}$, it follows that $r_{i}^{2}>r_{i}k_{i}-r_{i}+\lambda_{i}=\lambda_{i}v_{i}$ and hence (1) holds true. Finally, for any $x,y\in \Delta$ with $x \neq y$ and any $B$ block of $\mathcal{D}_{0}$ containing $x$, the incidence structure $(y^{G_{x}^{\Delta}},B^{G_{x}^{\Delta}})$ is a $1$-design by \cite[1.2.6]%
{Demb}, (2) follows. The same conclusions hold for the incidence structure $((\Delta^{\prime})^{G_{\Delta}^{\Sigma}},(B^\Sigma)^{G_{\Delta}^{\Sigma}}, \mathcal{I}^{\prime})$ with $\Delta, \Delta^{\prime} \in \Sigma$ and $\Delta\neq \Delta^{\prime}$, $\mathcal{I}^{\prime}$ the incidence relation inherited by $\mathcal{I}$, and $B^{\Sigma}$ any block of $\mathcal{D}_{1}$ incident to $\Delta$ with respect to $\mathcal{I}$.  
\end{proof}

\bigskip

As we will see, an important parameter in the analysis of $\mathcal{D}$ is the greatest common divisor $\left(
k_{1}-1,v_{1}-1\right)$, which will simply denoted by $A$ throughout the paper. 

\bigskip

\subsection{Minimal partition}\label{min} Assume that $G$ preserves a further non-trivial partition $\Sigma^{\prime}$ of the point-set of $\mathcal{D}$ in $v_{1}^{\prime}$ classes of size $v_{0}^{\prime}$. We may apply Theorem \ref{CamZie} to the triple $(\mathcal{D},G,\Sigma^{\prime})$ deriving information on the incidence structures $\mathcal{D}_{0}^{\prime}$ and $\mathcal{D}_{1}^{\prime}$ which have the same meaning as $\mathcal{D}_{0}$ and $\mathcal{D}_{1}$, respectively, but they are related to the partition $\Sigma^{\prime}$. In this setting, we say that $\Sigma^{\prime}$ \emph{refines} $\Sigma$, or that $\Sigma^{\prime}$ is \emph{finer} than $\Sigma$, and write $\Sigma^{\prime} \preceq \Sigma$, if any element of $\Sigma^{\prime}$ is contained in a (unique) element of $\Sigma^{\prime}$. The binary relation '$\preceq$' on the set of the $G$-invariant non-trivial partitions of the point-set of $\mathcal{D}$ is a partial ordering. Therefore, it makes sense to call any such partition \emph{minimal} if it is minimal with respect to ordering '$\preceq$'. Clearly, for any given non-trivial $G$-invariant partition $\Sigma$ of the point-set of $\mathcal{D}$, there always is a minimal non-trivial $G$-invariant partition $\Sigma^{\prime}$ of the point-set of $\mathcal{D}$ refining $\Sigma$.   

\bigskip

Based on Theorems \ref{PetereDemb}, \ref{DP} and \ref{CamZie}, and on the previous remark, we make the following hypothesis.

\bigskip

\begin{hypothesis}\label{hyp1}$\mathcal{D}=(\mathcal{P}, \mathcal{B})$ is a $2$-$(v,k,\lambda)$ design with $\lambda$ a prime divisor of $r$ such that $v\leq 2\lambda^{2}(\lambda-1)$, and $G$ is a flag-transitive, point-imprimitive automorphism group of $\mathcal{D}$ preserving a minimal nontrivial 
partition $\Sigma $ of $\mathcal{P}$ with $v_1$ classes of size $v_0$.     
\end{hypothesis}

\bigskip

\begin{lemma}\label{base}
Assume that Hypothesis \ref{hyp1} holds. Then the following hold
\begin{enumerate}
    \item $k-1 \mid v-1$ and $k_{0}-1 \mid v_{0}-1$;
    \item If $\mathcal{D}_{0}$ is a $2$-design, then either $\mu=\lambda$, $\lambda_{0}=1$ and $\mathcal{D}_{0}$ is a linear space, or $\mu=1$, $\lambda_{0}=\lambda$, $\mathcal{D}_{0}$ is a $2$-$(v_{0},k_{0}, \lambda)$ design and $r=r_{0}=\frac{v_{0}-1}{k_{0}-1}\lambda$;
    \item $G_{\Delta}^{\Delta}$ acts point-primitively on $\mathcal{D}_{0}$
    \item If $\mathcal{D}_{1}$ is a $2$-design, then 
    \begin{equation}\label{Lecce}
    r_{1}=\frac{(v_{1}-1)\lambda_{1}}{k_{1}-1}=\frac{v_{0}(v_{0}-1)\lambda }{k_{0}(k_{0}-1)\eta} \quad \text{ and }  \quad b_{1}=\frac{v(v_{0}-1)\lambda }{k(k_{0}-1)\eta}\text{,}
     \end{equation}
\end{enumerate}
\end{lemma}

\begin{proof}
 Assertion (1) immediately follows from (\ref{rel1}) in Theorem \ref{CamZie}(1) since $\lambda \mid r$ by Hypothesis \ref{hyp1}. Assertion (2) follows from (1) and Theorem \ref{CamZie}(1.ii) since $\mu \mid \lambda$ and $\lambda$ is prime.
 
 Assume to the contrary that $G_{\Delta}^{\Delta}$ does not act point-primitively on $\mathcal{D}_{0}$. Then there is a subgroup a subgroup $M$ of $G_{\Delta}$ containing $G_{x}$, where $x \in\Delta$ Let $\Delta \in \Sigma$. Then $\Sigma^{\prime}=\{x^{Mg}:g \in G\}$ is a non-trivial $G$-invariant of the point-set of $\mathcal{D}$ refining $\Sigma$ by \cite[Theorem 1.5A]{DM}, which is contrary to the minimality of $\Sigma$. Thus, assertion (3) holds true. 

Finally, assertion (4) follows from  Theorem \ref{CamZie}(2).   
\end{proof}

\bigskip
It is worthwhile noting that, the minimality of $\Sigma$ in the Hypothesis \ref{hyp1} is only needed when $k_{0} \geq 3$ and $\mu=1$. Indeed, the minimality of $\Sigma$ is equivalent to the fact that $G_{\Delta}^{\Delta}$ acts point-primitively on $\mathcal{D}_{0}$ by \cite[Theorem 1.5A]{DM}. Now, $G_{\Delta}^{\Delta}$ acts point-$2$-transitively on $\mathcal{D}_{0}$, an hence $G_{\Delta}^{\Delta}$ acts point-primitively on $\mathcal{D}_{0}$, when $k_{0}=2$; the same conclusion holds when $k_{0} \geq 3$ and $\mu=\lambda$ by \cite[Theorem 8]{HM}. 

\bigskip

It results that, $\eta \mid \frac{v_{0}}{k_{0}}\lambda$ by Lemma \ref{base}(4) since $\lambda_{1}=\frac{v_{0}^{2}\lambda}{k_{0}^{2}\eta}$ by Theorem \ref{CamZie}(2.ii). Set $\eta _{0}=\left( \eta ,\frac{%
v_{0}}{k_{0}}\right) $, $\eta _{1}=\frac{\eta }{\eta _{0}}$ and $r_{1}^{\prime }=\frac{v_{0}}{k_{0}\eta _{0}}\cdot \frac{v_{0}-1}{k_{0}-1}$. Then 
\begin{equation}\label{Salento}
r_{1}=r_{1}^{\prime }\cdot \frac{%
\lambda }{\eta _{1}}
\end{equation}
by (\ref{Lecce}).
The symbols $\eta_{0}$, $\eta_{1}$ and $r_{1}^{\prime }$ will have these fixed meanings throughout the paper.

\bigskip

\section{Combinatorial reductions}
This section focuses on the interactions between the parameters of $\mathcal{D}$, $\mathcal{D}_{0}$ and $\mathcal{D}_{1}$. The argument are mostly combinatorial  and fully exploit the fact that $\lambda$ is a prime number. In particular, we prove that the conclusions of Theorem \ref{main} hold when $\mathcal{D}$ is a symmetric $2$-design. Furthermore, when $\mathcal{D}_{2}$ is a $2$-design we provide a classification of $\mathcal{D}_{1}$ in five possible families. These families are excluded in the next sections. 

\bigskip

\begin{lemma}\label{L1}
Assume that Hypothesis \ref{hyp1} holds. Then $2 \leq k_{0}<v_{0}$ and $v_{0}\neq 3$.  
\end{lemma}

\begin{proof}
Suppose that $k_{0}=v_{0}$. Then $v=k$ by (\ref{rel1}) in Theorem \ref{CamZie}, a contradiction. Thus $2 \leq k_{0}<v_{0}$, hence $v_{0} \geq 3$. 

Assume that $v_{0}=3$. Then $k_{0}=2$, and hence $v=2(k-1)+1=2k-1$ and $r=2\lambda $ by (\ref{rel1}) in Theorem \ref{CamZie}(1). Thus $%
k\mid r$ since $k\mid vr$, and hence $k\mid 2\lambda $. If $\mathcal{D}_{1}$
is a $2$-design, then $\lambda _{1}=\frac{\left( v_{0}\right) ^{2}\lambda }{%
\left( k_{0}\right) ^{2}\eta }=\frac{9\lambda }{4\eta }$ is an integer by Theorem \ref{CamZie}(2.ii), and
hence $4\mid \lambda $, which is not the case since $\lambda $ is a prime.
Thus $\mathcal{D}_{1}$ is a symmetric $1$-design with $k_{1}=v_{1}-1$ again by Theorem \ref{CamZie}(2.ii). Then $%
k/2+1=v/3$ since $k_{1}=k/2$ and $v_{1}=v/3$, and so $v=3\frac{k+2}{2}$. Therefore $3\frac{k+2}{2}=2k-1$ since $v=2k-1$, 
and hence $k=8$. Then $4 \mid \lambda$ since $k\mid 2\lambda $, a contradiction. Thus, $v_{0} \neq 3$.
\end{proof}

\bigskip

\begin{lemma}\label{L2}
Assume that Hypothesis \ref{hyp1} holds. Then $\mathcal{D}_{0}$ is a $2$-design.
\end{lemma}

\begin{proof}
Suppose the contrary. Then $\mathcal{D}_{0}$ is a symmetric $1$-design and $%
k_{0}=v_{0}-1$ by Theorem \ref{CamZie}(1.ii). Since $\frac{r}{\lambda }$ is an integer and 
\[
\frac{r}{\lambda }=\frac{v-1}{k-1}=\frac{v_{0}-1}{k_{0}-1}=\frac{v_{0}-1}{%
v_{0}-2} 
\]%
by (\ref{rel1}) in Theorem \ref{CamZie}(1), it follows that $v_{0}-2=1$. So $v_{0}=3$, contrary to Lemma \ref{L1}.
\end{proof}

\begin{lemma}\label{L2bis}
Assume that Hypothesis \ref{hyp1} holds. If $\mathcal{D}_{1}$ is a symmetric $1$-design with $k_{1}=v_{1}-1$, then the parameters of $\mathcal{D}_{0}$, $\mathcal{D}_{1}$, $\mathcal{D}$ are as in Table \ref{D1sym}. In particular, $\lambda \geq k_{0}$, and $\lambda = k_{0}$ if and only if $\mathcal{D}$ is a symmetric $2$-design.
\begin{table}[h!]
\tiny
\caption{Admissible parameters for $\mathcal{D}_{0}$, $\mathcal{D}_{1}$, $\mathcal{D}$ for $k_{1}=v_{1}-1$}\label{D1sym}
\begin{tabular}{ccccc|cc|cccccc}
\hline
$v_{0}$ & $k_{0}$ & $\lambda_{0}$ & $r_{0}$ & $b_{0}$ & $v_{1}$ & $k_{1}$ & $v$ & $k$& $\lambda$ & $r$ & $b$ \\ 
\hline
$k_{0}^{2}$ & $k_{0}$ & $\frac{\lambda}{\mu}$& $(k_{0}+1)\frac{\lambda}{\mu}$ & $k_{0}(k_{0}+1)\frac{\lambda}{\mu}$ & $k_{0}+2$ & $k_{0}+1$ & $k_{0}^2(k_{0}+2)$& $k_{0}(k_{0}+1)$ & $\lambda$ & $(k_{0}+1)\lambda$ & $k_{0}(k_{0}+2)\lambda$  \\
\hline
\end{tabular}
\end{table}
\end{lemma}
\normalsize
\begin{proof}
Assume that $\mathcal{D}_{1}$ is a symmetric $1$-design with $%
k_{1}=v_{1}-1$. Then $k=k_{0}(v_{1}-1)$. Moreover,
\begin{equation}
v_{0}\left( v_{1}-1\right) =\frac{v_{0}-1}{k_{0}-1}\left( v_{1}-2\right)
k_{0} \label{jedan}
\end{equation}%
by (\ref{rel2}) in Theorem \ref{CamZie}(2.ii), which implies $v_{1}-2\mid v_{0}$ and $v_{1}-1\mid \frac{v_{0}-1}{k_{0}-1}k_{0}$.
Then $v_{0}=\theta(v_{1}-2)$ and $\frac{v_{0}-1}{k_{0}-1}k_{0}=\rho(v_{1}-1)$ for
some positive integers $\theta$ and $\rho$, respectively. These equalities, substituted in (\ref{jedan}), yield $%
\theta=\rho$. Therefore, 
\begin{equation}\label{upupa}
v_{0}=\theta(v_{1}-2) \text{ \quad  and  \quad } \frac{v_{0}-1}{k_{0}-1}%
k_{0}=\theta(v_{1}-1)\text{.}
\end{equation}
Now substituting the value of $v_{0}$ in the second equality of (\ref{upupa}), one obtains
\begin{equation*}
\left( \theta(v_{1}-2)-1\right) k_{0} =\theta(v_{1}-1)\left( k_{0}-1\right)=\theta(v_{1}-1)k_{0}-\theta(v_{1}-1)
\end{equation*}%
Easy computations lead to $k_{0}=\theta\frac{v_{1}-1}{\theta+1}$, and hence $\theta=\left( v_{0},k_{0}\right) $
since $v_{0}=\theta(v_{1}-2)$. Therefore, one has
\begin{equation}\label{erre}
  r=\frac{\theta (v_{1}-2)-1}{\theta \frac{v_{1}-1}{\theta +1}-1}\lambda
=\lambda \left( \theta +1\right)\text{.}  
\end{equation}
Now, $k\mid vr$ implies $\theta \frac{v_{1}-1}{\theta +1}\left(
v_{1}-1\right) \mid \theta (v_{1}-2)v_{1}\lambda \left( \theta +1\right)$, from which we derive
\begin{equation}
  \left( \frac{v_{1}-1}{\theta +1}\right) ^{2}\mid v_{1}(v_{1}-2)\lambda
\text{\quad and hence \quad} \left( \frac{v_{1}-1}{\theta +1}\right) ^{2}\mid \lambda   
\end{equation}
Then $\frac{v_{1}-1}{\theta +1}=1$ since $\lambda$ is a prime number. 
Then $\theta =v_{1}-2$. Therefore $v_{0}=(v_{1}-2)^{2}$, $k_{0}=v_{1}-2$, and hence Table \ref{D1sym} holds. Moreover, $\lambda \geq k_{0}$ since $k \leq r$, and $\lambda = k_{0}$ if and only if $\mathcal{D}$ is a symmetric $2$-design. 
\end{proof}

\bigskip 
From now on, unless stated otherwise, we assume $\mathcal{D}_{1}$ is not a symmetric $1$-design with $k_{1}=v_{1}-1$. Hence, $\mathcal{D}_{1}$ is a $2$-design by Theorem \ref{CamZie}. The remaining case, namely the case where $\mathcal{D}_{1}$ is a symmetric $1$-design with $k_{1}=v_{1}-1$, is settled in the last section of the paper. Hence, from now on we make the following hypothesis.
\bigskip

\begin{hypothesis}\label{hyp3/2}
Hypothesis \ref{hyp1} holds, $\mathcal{D}_{1}$ is not a symmetric $1$-design with $k_{1}=v_{1}-1$.     
\end{hypothesis}

\bigskip 
For a given positive integer $n$ and a prime divisor $p$ of $n$, we denote the \emph{$p$-part} of $n$ by $n_{p}$, that is to say, $n_{p}=p^{t}$ with $p^{t}\mid n$ but $p^{t+1}\nmid n$. In this case, $t$ is called the \emph{$p$-adic value} of $n$ and is denoted by $v_{p}(n)$.

\bigskip

\begin{proposition}\label{P2}
Assume that Hypothesis \ref{hyp3/2} holds. Then then the following hold:

\begin{enumerate}
\item $v_{0}=\left( z(k_{0}-1)+1\right) k_{0}$ with $z\geq 1$;

\item $\frac{v_{0}-1}{k_{0}-1}=zk_{0}+1$;

\item $k_{1}=A\left( z(k_{0}-1)+1\right) +1$ and $v_{1}=A\left(
zk_{0}+1\right) +1$.

\item $\lambda \geq \left( k_{0}-1\right) A+1$. In particular, either $%
\lambda >k_{0}$, or $\lambda =k_{0}$ and $A=1$.
\end{enumerate}
\end{proposition}

\begin{proof}
Suppose that $k_{0}$ does not divide $v_{0}$. Then $\frac{k_{0}^{2}}{%
(k_{0}^{2},v_{0}^{2})}>1$, and hence $\frac{k_{0}^{2}}{(k_{0}^{2},v_{0}^{2})}%
=\lambda $ since $\mathcal{D}_{1}$ is a $2$-design with $\lambda _{1}=\frac{v_{0}^{2}\lambda }{k_{0}^{2}\eta }$. Thus, $v_{\lambda}\left( \frac{k_{0}^{2}}{(k_{0}^{2},v_{0}^{2})}\right)=1$. 

Let $i=v_{\lambda}(k_{0})$, then $i \geq 1$, $v_{\lambda}(k_{0}^{2})=2i$ and $v_{\lambda}(k_{0}^{2}/\lambda)=2i-1$. Further, $v_{\lambda}(v_{0 }^{2}) \geq 2i-1$ since $k_{0}^{2}/\lambda =(k_{0}^{2},v_{0}^{2})$ divides $v_{0}^{2}$. Actually, $v_{\lambda}(v_{0}^{2})\geq 2i$ since $v_{\lambda}(v_{0}^{2})=2\cdot v_{\lambda}(v_{0})$. Therefore $v_{\lambda}\left(
(k_{0}^{2},v_{0}^{2})\right) =2i$ since $v_{\lambda}(k_{0}^{2})=2i$, and hence
\[
1=v_{\lambda }\left( \frac{k_{0}^{2}}{(k_{0}^{2},v_{0}^{2})}\right)
=v_{\lambda }(k_{0}^{2})-v_{\lambda }\left( (k_{0}^{2},v_{0}^{2})\right)
=2i-2i=0\text{,} 
\]%
a contradiction. Thus $k_{0}$ divides $v_{0}$, and hence $%
v_{0}=\theta k_{0}$ for some $\theta \geq 1$.

We know that $k_{0}-1 \mid v_{0}-1$ by Lemma \ref{base}(1), hence 
\[
\frac{v_{0}-1}{k_{0}-1}=\frac{\theta k_{0}-1}{k_{0}-1}=\frac{\theta k_{0}-\theta +\theta -1}{k_{0}-1}=\theta +\frac{\theta -1}{%
k_{0}-1}\text{.} 
\]%
Therefore, $\theta =z(k_{0}-1)+1$ for some positive integer $z$. Thus $$%
v_{0}=\left( z(k_{0}-1)+1\right) k_{0}$$ with $z\geq 1$, which is (1), and%
\[
\frac{v_{0}-1}{k_{0}-1}=\left( z(k_{0}-1)+1\right) +z=zk_{0}+1\text{,} 
\]%
which is (2).

Now (1), Lemma \ref{base}(1), (\ref{rel2}) in Theorem \ref{CamZie}(2) imply
\[
\frac{v_{0}}{k_{0}}\left( v_{1}-1\right) =\frac{v_{0}-1}{k_{0}-1}\left(
k_{1}-1\right) 
\]%
and hence $\frac{v_{0}}{k_{0}}\mid k_{1}-1$ and $\frac{v_{0}-1}{k_{0}-1}\mid
v_{1}-1$. Thus, there are positive integers $x$ and $y$ such that%
\[
x\frac{v_{0}}{k_{0}}=k_{1}-1\text{ and }y\frac{v_{0}-1}{k_{0}-1}=v_{1}-1 
\]%
Therefore%
\[
\frac{v_{0}}{k_{0}}y\frac{v_{0}-1}{k_{0}-1}=\frac{v_{0}-1}{k_{0}-1}x\frac{%
v_{0}}{k_{0}}\text{,} 
\]%
and hence $y=x$. So, $k_{1}-1=A\frac{v_{0}}{k_{0}}$ and $v_{1}-1=A\frac{%
v_{0}-1}{k_{0}-1}$ since $A=\left( k_{1}-1,v_{1}-1\right) $ and $\left( 
\frac{v_{0}}{k_{0}},\frac{v_{0}-1}{k_{0}-1}\right) =1$. Thus 
\[
k_{1}=A\left( z(k_{0}-1)+1\right) +1\text{ and }v_{1}=A\left(
zk_{0}+1\right) +1\text{,} 
\]%
by (1) and (2), and we obtain (3).

Now $r^{2}>\lambda v$ with $r=\frac{v_{0}-1}{k_{0}-1}\lambda $ by Lemma \ref{base}(2) since $\mathcal{D}_{0}$ is a $2$-design by Lemma \ref{L2}. Moreover, $v=v_{0}v_{1}$ and (2) holds. All these facts imply 
\[
\left( \frac{v_{0}-1}{k_{0}-1}\right) ^{2}\lambda ^{2}>\lambda
v_{0}v_{1}=\lambda v_{0}\left( A\frac{v_{0}-1}{k_{0}-1} +1\right) \text{,}
\]%
and hence 
\begin{equation}\label{pardon}
\left( \frac{v_{0}-1}{k_{0}-1}\right) ^{2}\lambda >\left( v_{0}-1\right) \left( A\frac{v_{0}-1}{%
k_{0}-1}+1\right) +A\frac{v_{0}-1}{k_{0}-1}+1\text{.}
\end{equation}
Dividing (\ref{pardon}) by $\frac{v_{0}-1}{k_{0}-1}$, one obtains
\begin{equation}\label{mercy}
\left( \frac{v_{0}-1}{k_{0}-1}\right)\lambda > \left( k_{0}-1\right) \left( A\frac{v_{0}-1}{%
k_{0}-1}+1\right) +A\text{.}
\end{equation}
Finally, dividing (\ref{mercy}) once again by $\frac{v_{0}-1}{k_{0}-1}$, (4) follows.
\end{proof}

\bigskip
\begin{remark}\label{oss}
The following equalities can be deduced from the proof of Proposition \ref{P2}(3): 
\begin{equation}\label{quasi}
v_{1}=A\frac{v_{0}-1}{k_{0}-1}+1 \text{\quad and \quad} k_{1}=A\frac{v_{0}}{k_{0}}+1\text{.}   
\end{equation}
In the sequel, we will use the equalities in (\ref{quasi}) without recalling them all the times.
\end{remark}
\bigskip

\begin{lemma}\label{L3}
Assume that Hypothesis \ref{hyp3/2} holds. Then the following hold:

\begin{enumerate}
\item $\left( k_{1},v_{1}\right) =(z,A+1)$;

\item $(k_{1},r_{1}^{\prime })=\left( Az-1,zk_{0}+1\right)=\left( Az-1,A+k_{0}\right) $;

\item $\left( k_{1},v_{1},r_{1}^{\prime }\right) =1$ and $\left(
k_{1},v_{1}\right) \cdot (k_{1},r_{1}^{\prime })\mid k_{1}$;

\item One of the following holds:
\begin{enumerate}
    \item[(i)] $k_{1}=\left( k_{1},v_{1}\right) \cdot \left( k_{1},\frac{r}{\lambda }%
\right) \cdot \lambda $, $\eta _{1}=1$, $\eta _{0}=\eta $, $\eta \mid \frac{v_{0}}{k_{0}}$ and $r_{1}^{\prime }=\frac{r_{1}}{\lambda }$;
 \item[(ii)] $k_{1}=\left( k_{1},v_{1}\right) \cdot \left( k_{1},r_{1}\right)$, $v_{0}$ is even and
\begin{equation}\label{surprise}
\left(k_{0},k_{1},v_{1},r_{1} \right)=\left(2,\frac{1}{2}(v_{0}-2)(v_{0}-1),(v_{0}-2)^{2},\frac{v_{0}}{2 \eta_{0}} (v_{0}-1) \frac{\lambda}{\eta_{1}} \right)\text{.}
\end{equation}
Further, if $\eta_{1}=\lambda$ then $\eta_{0}=1$.
\end{enumerate}
\end{enumerate}
\end{lemma}
\begin{proof}
By using Proposition \ref{P2}(2)--(3), it is easy to see that 
\begin{eqnarray*}
\left( k_{1},v_{1}\right) &=&\left( A\left( zk_{0}+1-z\right) +1,A\left(
zk_{0}+1\right) +1\right)\\ 
&=&\left( A\left( zk_{0}+1\right) +1-Az,A\left(
zk_{0}+1\right) +1\right) \\
&=&\left( Az,A\left( zk_{0}+1\right) +1\right) =\left( Az,A+1\right) =(z,A+1)
\end{eqnarray*}
and
\begin{eqnarray*}
(k_{1},r_{1}^{\prime }) &=&\left( A\frac{v_{0}}{k_{0}}+1,\frac{v_{0}}{%
k_{0}\eta _{0}}\cdot \frac{v_{0}-1}{k_{0}-1}\right) =\left( A\frac{v_{0}}{%
k_{0}}+1,\frac{v_{0}-1}{k_{0}-1}\right)\\ 
&=&\left( A\left( zk_{0}+1\right)
+1-Az,zk_{0}+1\right) =\left( Az-1,zk_{0}+1\right) \\
&=&\left( Az-1,Az+zk_{0}\right) =\left( Az-1,A+k_{0}\right)\text{,}
\end{eqnarray*}%
which proves (1) and (2). Further, from (1) and since $(k_{1},r_{1}^{\prime })=\left( Az-1,zk_{0}+1\right)$ in (2), it follows that $\left(
k_{1},v_{1},r_{1}^{\prime }\right) =1$. Hence $\left(
k_{1},v_{1}\right) \cdot (k_{1},r_{1}^{\prime })\mid k_{1}$, which is (3).

Now, $k_{1}\mid v_{1}\cdot r_{1}$ implies $k_{1}\mid \left(
k_{1},v_{1}\right) \cdot \left( k_{1},r_{1}^{\prime}\right) \cdot \left( k_{1},\frac{%
\lambda }{\eta _{1}}\right) $. On the other hand, $\left(
k_{1},v_{1}\right) \cdot (k_{1},r_{1}^{\prime })\mid k_{1}$ by (3). Thus, either
\begin{equation}\label{smeagol}
k_{1}=\left(k_{1},v_{1}\right) \cdot \left( k_{1},r_{1}^{\prime}\right) \cdot \lambda \text{ and } \eta_{1}=1 \text{, \quad or \quad } k_{1}=\left(
k_{1},v_{1}\right) \cdot \left( k_{1},r_{1}^{\prime}\right)
\end{equation}
since $\lambda$ is a prime. The former leads to (4.i). Hence, assume that $k_{1}=\left(
k_{1},v_{1}\right) \cdot \left( k_{1},r_{1}^{\prime}\right)$. Then (1), (2) and Proposition \ref{P2}(3) lead to
\begin{equation}\label{krasì}
A\left( z\left( k_{0}-1\right) +1\right) +1=(z,A+1) \cdot \left(
Az-1,A+k_{0}\right)\text{.} 
\end{equation}
If either $(z,A+1)\leq \frac{z}{2}$ or $\left( Az-1,A+k_{0}\right) \leq \frac{A+k_{0}%
}{2}$, then (\ref{krasì}) implies 
\begin{equation*}
A\left( z\left( k_{0}-1\right) +1\right) +1\leq \frac{z(A+k_{0})}{2}
\end{equation*}
from which we derive
\begin{equation}\label{aceto}
A \leq \frac{zk_{0}-2}{\left( 2zk_{0}-3z+2\right)}<\frac{zk_{0}}{\left( 2zk_{0}-3z\right)}=\frac{k_{0}}{2k_{0}-3}
\end{equation}
since $k_{0}\geq 2$. Since $A<1$ for $k_{0} \geq 3$, it follows that $k_{0}=2$ and $A \leq \frac{2(z-1)}{ z+2}<2$. Thus, $A=1$ and $z \geq 4$. On the other hand, (\ref{krasì}) becomes $z+2=(z,2) \cdot (z-1,3)$ and $z\leq 4$. Therefore $z=4$, and hence $(v_{0},k_{0},v_{1},k_{1})=(10,2,10,6)$ by Proposition \ref{P2}(1)(3). Moreover, $G^{\Sigma}$ acts point-$2$-transitively on $\mathcal{D}_{1}$ by \cite[Corollary 4.6]{Ka0}, and hence either $PSL_{2}(9)\unlhd G^{\Sigma} \leq P \Gamma L_{2}(9)$ or $A_{10}\unlhd G^{\Sigma} \leq S_{10}$ by \cite[Table B.4]{DM}. The latter implies $b_{1}=\binom{10}{6}$ since $G^{\Sigma}$ acts point-$6$-transitively on $\mathcal{D}_{1}$. Then $r_{1}=126$ since $b_{1}k_{1}=v_{1}r_{1}$ and hence $\lambda_{1}=70$ since $r_{1}=\frac{(v_{1}-1)\lambda_{1}}{k_{1}-1}$. Therefore $70=5\cdot \frac{5\lambda}{\eta}$ by Theorem \ref{CamZie}(2.ii), and hence $14 \mid \lambda$, whereas $\lambda$ is a prime. Thus $PSL_{2}(9)\unlhd G^{\Sigma} \leq P \Gamma L_{2}(9)$, and hence $G_{\Delta}^{\Sigma}$, with $\Delta \in \Sigma$, is a $\{2,3\}$-group. Now, $r_{1}=\frac{9\lambda_{1}}{5}$ is an integer dividing the order of $G_{\Delta}^{\Sigma}$ since $G^{\Sigma}$ acts flag-transitively on $\mathcal{D}_{1}$. If $r_{1}$ is even then $\lambda_{1}$ is even, and hence $\lambda=2$ since $\lambda_{1}=5\cdot \frac{5\lambda}{\eta}$ by Theorem \ref{CamZie}(2.ii), but this is contrary to Theorem \ref{DP}. Therefore, $r_{1}$ is odd, and hence $r_{1}=9$ and $\lambda_{1}=5$ since $r_{1}=\frac{9\lambda_{1}}{5}$ and $PSL_{2}(9)\unlhd G^{\Sigma} \leq P \Gamma L_{2}(9)$. So $\mathcal{D}_{1}$ is a $2$-$(10,6,5)$ design admitting $A_{6}\unlhd G^{\Sigma}$ since $PSL_{2}(9)\cong A_{6}$, but this is contrary to \cite[Theorem 1]{ZCZ}. Thus $(z,A+1)=z$ and $\left( Az-1,A+k_{0}\right) =A+k_{0}$, and hence (\ref{krasì}) becomes
\begin{equation}\label{mieru}
A\left( z\left( k_{0}-1\right) +1\right) +1=z \cdot \left(A+k_{0}\right)\text{,}  
\end{equation}
and so 
\begin{equation}\label{mulsum}
k_{0}(A-1)z=(2z-1)A-1\text{.}
\end{equation}
Assume that $A>1$. Then (\ref{mulsum}) implies
\begin{equation}\label{vina}
k_{0}=\frac{(2z-1)A-1}{(A-1)z}<\frac{2zA}{(A-1)z}\leq 4\text{,}
\end{equation}
forcing $k_{0}=2$ or $3$.

Suppose that $k_{0}=3$. Then $z(3-A)=A+1$ by (\ref{mulsum}), and hence $A=2$ and $z=3$ since $z\geq 0$ and $A>1$. Thus $v_{0}=v_{1}=21$, $k_{1}=15$ and hence $r_{1}=70 \cdot \frac{7\lambda}{\eta}$ with $\eta \mid 7\lambda$. Moreover, $G^{\Sigma}$ acts either point-$2$-transitively or as a point-primitive group with subdegrees $1,10,10$ on $\mathcal{D}_{1}$ by \cite[Corollary 4.6]{Ka0}. Actually, $A_{21}\unlhd G^{\Sigma} \leq S_{21}$ by \cite[Table B.4]{DM} and \cite{At} since $70$ divides $r_{1}$ and hence the order of $G_{\Delta}^{\Sigma}$ with $\Delta \in \Sigma$. Then $b_{1}=\binom{21}{15}$ since $G^{\Sigma}$ acts point-$15$-transitively on $\mathcal{D}_{1}$, and so $r_{1}= 54264$ since $b_{1}k_{1}=v_{1}r_{1}$. Then $r_{1}$ is not divisible by $7$, whereas $r_{1}=70 \cdot \frac{7\lambda}{\eta}$ with $\eta \mid 7\lambda$, a contradiction.

Suppose that that $k_{0}=2$. Then $A=2z-1=v_{0}-3$ by (\ref{mulsum}), and hence
\begin{equation}
v_{1}=(v_{0}-2)^{2} \text{ \quad and \quad } k_{1}=\frac{1}{2}(v_{0}-2)(v_{0}-1) \text{ \quad and \quad } r_{1}=\frac{v_{0}}{2\eta_{0}}(v_{0}-1) \frac{\lambda}{\eta_{1}}\text{,}    
\end{equation}
If $\eta_{1}=1$, then $\eta_{0}=1$ since $k_{1}\leq r_{1}$, $v_{0}>3$ by Lemma \ref{L1} and $v_{0}$ even, and we obtain the the last part of (4.ii).

Finally, assume that $A=1$. Then $z=1$ by (\ref{mulsum}), $v_{0}$ is a square, $k_{0}=\sqrt{v_{0}}$, $v_{1}=\sqrt{v_{0}}+2$ and $k_{1}=\sqrt{v_{0}}+1$ by Proposition \ref{P2}(2)--(3). Then $\mathcal{D}_{1}$ is a symmetric $2$-designs, and hence $\lambda_{1}=\sqrt{v_{0}}$. In particular,  $\mathcal{D}_{1}$ is a symmetric $1$-designs with $k_{1}=v_{1}-1$, but this contradicts our assumption.
\end{proof}

\bigskip

\begin{proposition}\label{P4}
Assume that Hypothesis \ref{hyp3/2} holds. If $\mathcal{D}_{1}$ is a $2$-design with $k_{1}=\left( k_{1},v_{1}\right) \cdot \left( k_{1},\frac{r}{\lambda }%
\right) \cdot \lambda $, then one of the following holds:

\begin{enumerate}
\item $\lambda >k_{0}$, $\lambda$ divides both $v_{0}-1$ and $v_{1}-1$, $\lambda \equiv 1 \pmod{4}$ and
\begin{equation}\label{equa1}
\left( k_{0},v_{0},k_{1},v_{1}\right) =\left( 2 ,\frac{1}{2} \left(
\lambda ^{2}+\lambda +2\right) ,\frac{1}{4}\lambda^{2} \left( \lambda -1\right)
,\frac{1}{2}\left(\lambda-1\right)\left(\lambda ^{2}-2\right)\right)\text{;}
\end{equation}
\item $\lambda$ divides $v_{0}$ and
\begin{equation}\label{equa2}
\left( k_{0},v_{0},k_{1},v_{1}\right) =\left(
\lambda ,\lambda \left( \lambda ^{2}+\lambda -1\right),\lambda
\left( \lambda +1\right) , \lambda ^{2}+2\lambda +2\right)\text{;}
\end{equation}
\item $\lambda >k_{0}$ and $\lambda $ does not divide any of the integers $v_{0}(v_{0}-1)$, $v_{1}(v_{1}-1)$ or $v(v-1)$.
\end{enumerate}
\end{proposition}

\begin{proof}
Assume that $\lambda \mid v_{1}(v_{1}-1)$. If $\lambda \mid v_{1}$, then $\lambda \mid (k_{1},v_{1})$ since $%
\lambda \mid k_{1}$, and so $\lambda \mid (z,A+1)$ by Lemma \ref{L3}(1). On the other hand, $\lambda \geq (k_{0}-1)A+1$ by Proposition \ref{P2}(4). Hence, $k_{0}=2$ and $\lambda=A+1$ divides $z$. Then there is a positive integer $w$ such $z=w\lambda$, and hence 
$$k_{1}=A\left(z+1\right) +1=(\lambda-1)(w\lambda+1)+1=w\lambda^{2}-\lambda(w-1)\text{.}$$
Since $k_{1}=\left( k_{1},v_{1}\right) \cdot \left( k_{1},\frac{r}{\lambda }%
\right) \cdot \lambda $ with $\lambda \mid (k_{1},v_{1})$, it follows that $\lambda^{2} \mid k_{1}$ and hence $w=a\lambda+1$ for some positive integer $a$. So $v \geq k=2\lambda^{2}(a\lambda-a+1) \geq 2\lambda^{3}$, but this contradicts the Hypothesis \ref{hyp3/2}. Thus $\lambda \mid v_{1}-1$, and hence $\lambda \mid v_{0}-1$ by (\ref{rel2}) in Theorem \ref{CamZie}(2) since $\lambda \nmid k_{1}-1$. Then $%
\lambda \mid A\left( zk_{0}+1\right) $ since $v_{1}=A\left( zk_{0}+1\right)
+1$ by Proposition \ref{P2}(3), and hence $\lambda \mid
zk_{0}+1$ since $\lambda \geq A+1$. Then $\lambda \mid r_{1}^{\prime }$ since $r_{1}^{\prime }=%
\frac{v_{0}}{k_{0}\eta _{0}}\cdot \frac{v_{0}-1}{k_{0}-1}$ and $\frac{v_{0}-1%
}{k_{0}-1}=zk_{0}+1$ by Proposition \ref{P2}(2), and hence $\lambda \mid (Az-1,A+k_{0})$ by Lemma \ref{L3}(2). In particular, $\lambda^{2} \mid k_{1}$ since $k_{1}=\left( k_{1},v_{1}\right) \cdot \left( k_{1},\frac{r}{\lambda }%
\right) \cdot \lambda $ and $r^{\prime}=r/\lambda$. Combining $\lambda \mid A +k_{0}$ with Proposition \ref{P2}(4), it follows that 
\begin{equation}\label{dugino}
(k_{0}-1)A+1 \leq \lambda \leq A +k_{0}\text{.}    
\end{equation}
Suppose that $k_{0} \geq 3$, then $A \leq \frac{k_{0}-1}{k_{0}-2}$. Hence, either $k_{0}=3$ and $A\leq 2$, or $k_{0} \geq 4$ and $A=1$. Assume that the latter occurs. Then $k_{1}=z(k_{0}-1)+2$ and $k_{0}\leq \lambda \leq k_{0}+1$. Actually, $\lambda=k_{0}+1$ since $\lambda$ divides $\left( k_{1},\frac{r}{\lambda }\right) =(z-1,k_{0}+1)$, and hence $k_{1}=(e\lambda+1)(\lambda-2)+2$ for some positive integer $e$ such that $\lambda \mid 2e-1$ since $\lambda^{2} \mid k_{1}$. Then
\begin{equation*}
\left[\left(\frac{\lambda+1}{2}\lambda+1\right)(\lambda-2)+2\right](\lambda-1)\leq k \leq v \leq 2\lambda^{2}(\lambda-1)     
\end{equation*}
by Hypothesis \ref{hyp3/2}, and hence $(\lambda,e)=(5,3)$ since $\lambda=k_{0}+1$ and $k_{0} \geq 4$. So $(k_{0},v_{0},k_{1},v_{1})=(4,196,50,66)$ and $A=1$, hence $A_{196} \unlhd G_{\Delta}^{\Delta}\leq S_{196}$ by \cite[Corollary 4.2]{Ka0} and \cite[Table B.4]{DM}, contrary to \cite[Proposition 2.4(ii)]{DP1}.
Thus, $k_{0}=3$ and $A\leq 2$. Then either $A=1$ and $\lambda =2$, or $A=2$ and $\lambda=5$ since $\lambda \mid (Az-1,A+k_{0})$ and $\lambda$ is a prime number. In the former case, $k_{1}=3z+2 \geq 5$, hence $k \geq 10$, but this contradicts by \cite[Theorem 1]{DP}. Then  $A=2$ and $\lambda=5$ and hence $k\leq 200$ again by \cite[Theorem 1]{DP}. On the other hand, either $k=75$ or $150$ since $k_{0}=3$ and $25 \mid k_{1}$. So $k_{1}=25$ or $50$, and both are excluded since they contradict $k_{1}=4z+3$.

Suppose that $k_{0}=2$. Then $\lambda=A+2$ by (\ref{dugino}) since $\lambda \mid A+2$ and $A \geq 1$. Then $2z+1=e\lambda$ for some positive integer $e$ since $\lambda \mid Az-1$. Therefore,
\begin{equation}\label{poselo}
k_{1}=A(z+1)+1=(\lambda-2)\left(\frac{e\lambda-1}{2}+1\right)+1=\frac{1}{2}\lambda(e\lambda-2e+1)\text{.}
\end{equation}
Further, $\lambda \mid 2e-1$ since $\lambda^{2} \mid k_{1}$. Then $e=\frac{f\lambda+1}{2}$ for some positive integer $f$, which substituted in (\ref{poselo}), leads to
\begin{equation}\label{ravna}
k=2k_{1}=\frac{1}{2}\lambda^{2}((\lambda-2)f+1)\text{.}
\end{equation}
If $\lambda=2$ then $k=2$, contrary to $\mathcal{D}$ non-trivial. Thus $\lambda$ is odd since $\lambda$ is a prime number, and hence $f$ is odd since $e=\frac{f\lambda+1}{2}$ is an integer. As $k \leq 2\lambda^{2}(\lambda-1)$ by \cite[Theorem 1]{DP}, then $(\lambda-2)f+1 \leq 4(\lambda-1)$. Hence, $f\leq7$ since $\lambda>2$. Actually, either $f=1$ or $3$, or $f=5$ and $\lambda=3$ or $5$, or $f=7$ and $\lambda=3$ since both $f$ and $\lambda$ are odd. 

Since $$k_{1}=\left( k_{1},v_{1}\right) \cdot \left( k_{1},\frac{r}{\lambda }%
\right) \cdot \lambda =(z,A+1) \cdot (Az-1,A+2) \cdot \lambda$$ by Lemma \ref{L3}(1)--(2), (5) and $(Az-1,A+2)=\lambda$, it follows from (\ref{ravna}) that $$(z,A+1)=\left(\frac{f\lambda^2+\lambda-2}{4},\lambda-1 \right)=\frac{1}{4}((\lambda-2)f+1)\text{.}$$
Thus $\frac{1}{4}((\lambda-2)f+1)\mid \lambda-1$, forcing $f=1$ and $\lambda \equiv 1 \pmod{4}$, or $f=\lambda=3$ or $5$. The latter yields the tuples $\left( k_{0},v_{0},k_{1},v_{1},k,v,\lambda\right)=(2,16,9,16,18,256,3)$ or $(2,66,100,196,200, 12936,5)$. The former is ruled out by \cite[Lemma 3.5]{DP1}, whereas $A=3$ in the latter. Then $A_{196} \unlhd G_{\Delta}^{\Delta}\leq S_{196}$ by \cite[Theorem 1.2]{ZZ} and \cite[Table B.4]{DM}, but this contradicts \cite[Proposition 2.4(ii)]{DP1}. Thus $f=1$ and $\lambda \equiv 1 \pmod{4}$, and hence
\begin{equation*}
\left( k_{0},v_{0},k_{1},v_{1}\right) =\left( 2 ,\frac{1}{2} \left(
\lambda ^{2}+\lambda +2\right) ,\frac{1}{4}\lambda^{2} \left( \lambda -1\right)
,\frac{1}{2}\left(\lambda-1\right)\left(\lambda ^{2}-2\right)\right) 
\end{equation*}
with $\lambda \equiv 1 \pmod{4}$, which is (1).

Assume that $\lambda $ does not divide $v_{1}(v_{1}-1)$ but divides $v_{0}(v_{0}-1)$. If $\lambda \mid v_{0}-1$ then $\lambda \mid \frac{v_{0}-1}{k_{0}-1%
}$ since $\lambda \geq k_{0}$, and hence $\lambda \mid v_{1}-1$ since $%
v_{1}=A\left( zk_{0}+1\right) +1$ and $\frac{v_{0}-1}{k_{0}-1}=zk_{0}+1$ by Proposition \ref{P2}(2)(3), a contradiction. Therefore, $\lambda \mid v_{0}$. If $\lambda \mid \frac{%
v_{0}}{k_{0}}$, then $\lambda \mid k_{1}-1$ since $k_{1}=A\frac{v_{0}}{k_{0}}%
+1$ by Proposition \ref{P2}(3), whereas $\lambda \mid k_{1}$. Thus $\lambda \mid k_{0}$, and hence $%
\lambda =k_{0}$ and $A=1$ by Proposition \ref{P2}(4). Therefore $k_{1}=z\lambda-z+2$ by Proposition \ref{P2}(3), and
hence $\lambda \mid z-2$. Then there is a positive integer $h$ such that $z=h\lambda
+2$. Then $k_{1}=\left(  h\lambda +2\right) \left( \lambda
-1\right) +2$, and hence $k=\lambda \left( h\lambda +2\right)
\left( \lambda -1\right) +2\lambda$ since $k_{0}=\lambda$. Then $h=1$ since $k\leq 2\lambda ^{2}(\lambda -1)$ by \cite[Theorem 1%
]{DP}. Therefore $k_{1}=\lambda \left( \lambda
+1\right) $, and hence 
\begin{equation*}
\left( k_{0},v_{0},k_{1},v_{1}\right) =\left( \lambda ,\lambda \left(
\lambda ^{2}+\lambda -1\right) ,\lambda \left( \lambda +1\right)
,\lambda ^{2}+2\lambda +2\right)\text{,} 
\end{equation*}
which is (2). 

Finally, assume that  $\lambda $ divides neither $v_{1}(v_{1}-1)$ nor $v_{0}(v_{0}-1)$. Hence, $\lambda$ does not divide $v$ since $v=v_{0}v_{1}$. Further, $\lambda>k_{0}$ by Proposition \ref{P2}(1)(4). Assume that $\lambda \mid v-1$. Since $\lambda \mid k_{1}$ and $%
k=k_{0}k_{1}$, it follows that $\lambda \mid k$. Then $\lambda \mid \frac{v-1}{k-1}$ since $\lambda \nmid k-1$, and hence $\lambda \mid \frac{v_{0}-1}{k_{0}-1}$ by (\ref{rel1}) in Theorem \ref{CamZie}(1). So, $\lambda
\mid v_{1}-1$ since $v_{1}=A\left( zk_{0}+1\right) +1$ and $\frac{v_{0}-1}{%
k_{0}-1}=zk_{0}+1$ by Proposition \ref{P2}(2)(3), a contradiction. Thus $\lambda \nmid v(v-1)$, and we obtain (3).
\end{proof}

\bigskip
\subsection{Types of $\mathcal{D}_{1}$}\label{types} Based on Lemma \ref{L3}(5), we may define a $2$-design $\mathcal{D}_{1}$ to be of \emph{type I or II} according as $k_{1}=\left( k_{1},v_{1}\right) \cdot \left( k_{1},\frac{r}{\lambda }\right) \cdot \lambda $ or $k_{1}=\left( k_{1},v_{1}\right) \cdot \left( k_{1},\frac{r}{\lambda }%
\right)$, respectively. 
When $\mathcal{D}_{1}$ is of type I, we say that $\mathcal{D}_{1}$ is of \emph{type Ia, Ib or Ic} according as (3), (2) or (1) of Proposition \ref{P4} holds, respectively.

When $\mathcal{D}_{1}$ is of type II, $v_{0}$ is even and it is shown in the proof of Lemma \ref{L3}(4) that $A=(k_{1}-1,v_{1}-1)=v_{0}-3$. Then $\lambda \geq v_{0}-2$ by Proposition \ref{P2}(4) since $k_{0}=2$. If $\lambda = v_{0}-2$, then $v_{0}=4$ since $\lambda$ is a prime number. Then $v_{1}=4$ and $k_{1}=3$, hence $\mathcal{D}_{1}$ is a symmetric $1$-design with $k_{1}=v_{1}$, contrary to our assumptions. Also, $\lambda \neq v_{0}$, otherwise $\lambda \neq v_{0}=2$, contrary to Lemma \ref{L1}. Thus, either $\lambda=v_{0}-1$, or $\lambda>v_{0}$. When this occurs, we say that $\mathcal{D}_{1}$ is of \emph{type IIa or IIb}, respectively.

\bigskip

\begin{theorem}\label{Teo1}
Assume that Hypothesis \ref{hyp3/2} holds. Then $\mathcal{D}_{1}$ is a $2$-design, its type and the corresponding parameters are as in Table \ref{D0-D1}. 
\begin{table}[h!]
\tiny
\caption{Admissible parameters for $2$-design $\mathcal{D}_{1}$}\label{D0-D1}
\begin{tabular}{c|ccc|ccc|l}
\hline
Type & $v_{0}$ & $k_{0}$ & $\lambda_{0}$ & $v_{1}$ & $k_{1}$ & $\lambda_{1}$ & Constraints \\ 
\hline
Ia & $\left( z(k_{0}-1)+1\right) k_{0}$ & $k_{0}$  & $\frac{\lambda}{\mu}$  & $A\left(
zk_{0}+1\right) +1$ & $A\left( z(k_{0}-1)+1\right) +1$ & $\frac{\left( z(k_{0}-1)+1\right)^2}{\eta}\lambda$ & $z\geq 1$, (\ref{fundamental}) holds, \\
&&&&&&& $\lambda \geq \left( k_{0}-1\right) A+1>k_{0}$ \\
&&&&&&& $\lambda \nmid v_{0}\left( v_{0}-1\right)$, $\lambda \nmid v_{1}\left( v_{1}-1\right)$ \\
&&&&&&& $\lambda \nmid v\left( v-1\right)$. \\
\hline
Ib & $\lambda \left( \lambda ^{2}+\lambda -1\right)$ & $\lambda$  & $\frac{\lambda}{\mu}$ & $\lambda ^{2}+2\lambda +2$ & $\lambda \left( \lambda +1\right)$ & $\frac{\left( \lambda ^{2}+\lambda -1\right)^{2}}{\eta}\lambda$ & $\lambda\mid v_{0}$, $\lambda=k_{0}>2$  \\
\hline
Ic & $\frac{1}{2} \left(\lambda ^{2}+\lambda +2\right) $ & $2$ & $1$ & $\frac{1}{2}\left(\lambda-1\right)\left(\lambda ^{2}-2\right)$ & $\frac{1}{4} \lambda ^{2} \left(\lambda -1\right)$ & $\frac{\left(\lambda ^{2}+\lambda +2\right)^{2}}{16 \eta}\lambda$ & $\lambda \mid v_{0}-1$, $ \lambda \mid v_{1}-1$ \\
   &                                                    &      &     &                                                               &     & & $\lambda>k_{0}$, $\lambda \equiv 1 \pmod{4}$ \\
\hline
IIa & $\lambda+1$ & $2$ & $1$ & $(\lambda-1)^{2}$ & $\frac{1}{2}\lambda(\lambda-1)$ & $\frac{(\lambda+1)^{2}}{4\eta}$ & $\lambda =v_{0}-1=\sqrt{v_{1}}+1$ \\
    &                                                    &      &     &                                                               &     & & $\lambda>k_{0}$ \\
\hline
IIb & $v_{0}$ & $2$ & $1$ & $(v_{0}-2)^{2}$ & $\frac{1}{2}(v_{0}-2)(v_{0}-1)$ & $\frac{v_{0}^{2}\lambda}{4\eta}$ & $\lambda >v_{0}$, $v_{0}$ even \\
\hline
\end{tabular}
\end{table}

\normalsize
Moreover, if $\mathcal{D}_{1}$ is of type Ia, it results that
\begin{equation}\label{fundamental}
A\left( z(k_{0}-1)+1\right) +1=(z,A+1)\cdot(Az-1,A+k_{0}) \cdot \lambda
\end{equation}
\end{theorem}

\begin{proof}
Let $\mathcal{D}$ be a flag-transitive point-imprimitive $2$-$(v,k,\lambda)$ design with $\lambda$ is a prime number. Then  $\mathcal{D}_{1}$ is a $2$-design by Theorem \ref{CamZie} since $\mathcal{D}_{1}$ is not a symmetric $1$-design with $k{1}=v_{1}-1$ by Hypothesis \ref{hyp3/2}. If $\mathcal{D}_{1}$ is of type IIa or IIb, then the parameters $v_{0},k_{0},\lambda_{0},v_{1},k_{1}$ and $\lambda_{1}$ are as in the last two lines of Table \ref{D0-D1} by Lemma \ref{L3}(4) according as $\lambda=v_{0}-1$ or $\lambda >v_{0}$, respectively. Finally, if $\mathcal{D}_{1}$ is of type Ia, Ib or Ic, then $v_{0},k_{0},\lambda_{0},v_{1},k_{1}$ and $\lambda_{1}$ are as in the first three lines of Table \ref{D0-D1} by Proposition \ref{P4}. In particular, $\lambda>2$ when $\mathcal{D}_{1}$ is of type Ib by Theorem \cite[Theorem 1]{DP}, and (\ref{fundamental}) holds when $\mathcal{D}_{1}$ is of type Ia since $k_{1}=A\left( z(k_{0}-1)+1\right) +1$ by Proposition \ref{P2}(3) and $k_{1}=(z,A+1)\cdot(Az-1,A+k_{0}) \cdot \lambda$ by Lemma \ref{L3}(1)--(2) and (4).  
\end{proof}

\bigskip
\begin{corollary}\label{C1}
Assume that Hypothesis \ref{hyp1} holds. If $\mathcal{D}$ is a symmetric $2$-design, then one of the
following holds:
\begin{enumerate}
\item $\mathcal{D}$ is one of the two $2$-$(16,6,2)$ designs as in \cite[Section 1.2]{ORR};
\item $\mathcal{D}$ is the $2$-$(45,12,3)$ design as in \cite[Construction 4.2]{P}.
\end{enumerate}
\end{corollary}

\begin{proof}
Assume that $\mathcal{D}$ is a flag-transitive point-imprimitive symmetric $2$-$(v,k,\lambda)$ design with $\lambda$ is a prime number. Assume that $\mathcal{D}_{1}$ is a $2$%
-design. Then
\begin{equation}\label{umoran}
k_{0}k_{1}=\frac{v_{0}-1}{k_{0}-1}\lambda
\end{equation}
since $k=r$, $k=k_{0}k_{1}$, $r=\frac{v-1}{k-1}\lambda$, and $\frac{v-1}{k-1}=\frac{v_{0}-1}{k_{0}-1}$
by Theorem \ref{CamZie}(1).Then (\ref{umoran}) becomes
\begin{equation}
k_{0}\left[ A(zk_{0}-z+1)+1\right] =\left( zk_{0}+1\right) \lambda 
\label{dobro}
\end{equation}%
since $k_{1}=A(zk_{0}-z+1)+1$ and $\frac{v_{0}-1}{k_{0}-1}=zk_{0}+1$ by
Proposition \ref{P2}(2)--(3), and hence $k_{0}\mid \lambda $. Actually, $%
k_{0}=\lambda $ since $\lambda $ is a prime number and $k_{0}\geq 2$. Then $k_{1}=\lambda (\lambda +1)$ and $v_{0}=\lambda (\lambda
^{2}+\lambda -1)$ by Theorem \ref{Teo1}(2). Then $k=\lambda ^{2}\left(
\lambda +1\right) $ and 
\[
r=\frac{v-1}{k-1}\lambda =\frac{v_{0}-1}{k_{0}-1}\lambda =\frac{\lambda (\lambda ^{2}+\lambda -1)-1}{%
\lambda -1}\lambda =\lambda \left( \lambda +1\right) ^{2}\text{,}
\]%
contrary to $r=k$. Therefore, $\mathcal{D}_{1}$ is a symmetric $1$-design with $k_{1}=v_{1}-1$. Then $\lambda=k_{0}$ and $k >\lambda(\lambda-3)/2$ by Lemma \ref{L2bis} since $r=k$, and hence the assertion follows from \cite[Theorem 1.1]{Mo}. This completes the proof.
\end{proof}

\bigskip
\begin{corollary}\label{C2}
Assume that Hypothesis \ref{hyp1} holds. If $\lambda \leq 3$, then one of the
following holds:
\begin{enumerate}
\item $\mathcal{D}$ is one of the two $2$-$(16,6,2)$ designs as in \cite[Section 1.2]{ORR};
\item $\mathcal{D}$ is the $2$-$(45,12,3)$ design as in \cite[Construction 4.2]{P}.
\end{enumerate}
\end{corollary}

\begin{proof}

Assume that that Hypothesis \ref{hyp1} holds. Assume $\mathcal{D}_{1}$ is a $2$-design. Therefore, the parameters $v_{0},k_{0},\lambda_{0},v_{1},k_{1}$ and $\lambda_{1}$ are as Table \ref{D0-D1} by Theorem \ref{Teo1}(2). Then $\mathcal{D}_{1}$ is not of type Ic or IIb since $\lambda \leq 3$, whereas $\lambda \equiv 1 \pmod{4}$ in the former case, and $\lambda>v_{0}>3$ by Lemma \ref{L1} in the latter. Thus $\mathcal{D}_{1}$ is of type Ia--Ib or IIa, and hence $\lambda >k_{0}\geq 2$, forcing $\lambda=3$. Then $\mathcal{D}_{1}$ cannot be of type Ib, otherwise $(v_{0},k_{0},\lambda_{0},v_{1},k_{1},\lambda_{1})=(33,3,3/\mu,17,12,11\cdot 11/\eta)$ by Theorem \ref{Teo1}(2), and this contradicts \cite[Lemma 3.6]{DP1}. Furthermore, $\mathcal{D}_{1}$ cannot be of type IIa since there are no $2$-$(4,3,4)$ designs. 

Finally, assume that $\mathcal{D}_{1}$ is of type Ia, then $k_{0}=2$ and $A=1$. Then $v_{0}=v_{1}=2(z+1)$, $\lambda_{0}=1$, $k_{1}=z+2$ and $\lambda_{1}=3\frac{(z+1)^{2}}{\eta}$. Thus $r_{1}=3(2z+1)\frac{(z+1)}{\eta}$. Then $k_{1}\mid r_{1}$ since $k_{1}\mid v_{1} r_{1}$ and $(k_{1},v_{1})=1$. So $z+2 \mid 9$, and hence $z=1$ or $7$ since $z \geq 1$. If $z=1$ then $v_{1}=4$, $k_{1}=3$, and hence $\mathcal{D}_{1}$ is a symmetric $1$-design, which is contrary to our assumption. Therefore $z=7$, and hence $(v_{1},k_{1})=(16,9)$. So $(v,k,\lambda,v_{0},v_{1},k_{0})=(256,18,3,16,16,2)$, which cannot occur by \cite[Lemma 3.5]{DP1}. Therefore, $\mathcal{D}_{1}$ is a symmetric $1$-design with $k_{1}=v_{1}-1$. Now, $2 \leq k_{0} \leq \lambda \leq 3$ by Lemma \ref{L2bis}. Therefore, $(k_{0},\lambda)=(2,2),(2,3),(3,3)$. Actually, $(k_{0},\lambda)=(2,3)$ by \cite[Theorem 2]{DP} (obtained with the aid of the computer). Therefore, $(k_{0},\lambda)=(2,2),(3,3)$ and hence $\mathcal{D}$ is symmetric and so the assertion follows from Corollary \ref{C2}.
\end{proof}

\bigskip
An alternative proof of Corollary \ref{C2} is provided in \cite[Theorem 1.1]{DP1}. Our proof, as well as the proof of \cite[Theorem 1.1]{DP1}, relies on \cite[Theorem 2]{DP} and \cite[Lemmas 3.5 and 3.6]{DP1} to exclude three numerical cases.
\bigskip  



Based on Theorem \ref{Teo1} and Corollaries \ref{C1} and \ref{C2}, from now on we make the following hypothesis.

\bigskip
\begin{hypothesis}\label{hyp2}
The Hypothesis \ref{hyp3/2} holds, $r>k$ and $\lambda >3$.   
\end{hypothesis}
\bigskip

Recall that $\lambda_{1}=\frac{v_{0}}{k_{0}}\cdot \frac{v_{0}}{k_{0}\eta_{0}}\cdot \frac{\lambda}{\eta_{1}}$ and $r_{1}=\frac{v_{0}-1}{k_{0}-1}\cdot \frac{v_{0}}{k_{0}\eta_{0}}\cdot \frac{\lambda}{\eta_{1}}$ by Theorem \ref{CamZie}(2.ii), Lemma \ref{base}(4) and (\ref{Salento}). Hence,  
\begin{equation}\label{double}
\frac{r_{1}}{(r_{1},\lambda_{1})}=\frac{v_{0}-1}{k_{0}-1}\text{.}
\end{equation}
As we will see, we provide two different proof strategies based on the discriminant $\frac{r_{1}}{(r_{1},\lambda_{1})} \geq\lambda$ or $\frac{r_{1}}{(r_{1},\lambda_{1})} < \lambda$; the following proposition is in this direction. 

\bigskip
\begin{proposition}\label{C3}
Assume that Hypothesis \ref{hyp2} holds. Then one of the
following holds:
\begin{enumerate}
\item $\frac{r_{1}}{(r_{1},\lambda_{1})} \geq \lambda$, $A^{2}\leq v_{1}-1$ and $v_{1} \neq \left(2A-1\right)^{2}$;
\item $\frac{r_{1}}{(r_{1},\lambda_{1})}<\lambda$ and $\mathcal{D}_{1}$ is of type Ia or IIb.
\end{enumerate}
\end{proposition}

\begin{proof}
Assume that $\frac{r_{1}}{(r_{1},\lambda_{1})} \geq \lambda$. Then $\frac{v_{0}-1}{k_{0}-1}\geq \lambda$ by (\ref{double}), and hence $\frac{v_{0}-1}{k_{0}-1} \geq A$ since $\lambda \geq (k_{0}-1)A+1$ by Proposition \ref{P2}(4). Then
$$v_{1}-1=A\frac{v_{0}-1}{k_{0}-1} \geq A^{2}$$
since $v_{1}=A\frac{v_{0}-1}{k_{0}-1}+1$ by Proposition \ref{P2}(3), which proves the first part of (1).

If $v_{1} = \left(2A-1\right)^{2}$ then $A\frac{v_{0}-1}{k_{0}-1}+1=(2A-1)^2$ again by Proposition \ref{P2}(3), and hence
\begin{equation}\label{ChouChouChou}
(k_{0}-1)A+1 \leq\frac{v_{0}-1}{k_{0}-1}=4(A-1) 
\end{equation}
since $\lambda \geq (k_{0}-1)A+1$. Thus, $k_{0}=2,3$ or $4$. Actually, $k_{0}=3$ since $zk_{0}+1=4(A-1)$ by Proposition \ref{P2}(2). Then $\mathcal{D}_{1}$ is of type Ia by Theorem \ref{Teo1} since $\lambda>3$ by our assumption. Further $z=\frac{4A-5}{3}$ with $A \equiv 2 \pmod{3}$, and
hence (\ref{fundamental}) becomes   
\begin{equation}\label{saturday}
A\left( \frac{8A-7}{3}\right) +1=\left( \frac{4A-5}{3},A+1\right) \cdot
\left( \frac{1}{3}\left( 4A^{2}-5A-3\right) ,A+3\right) \lambda 
\end{equation}
Now $\left( \frac{4A-5}{3},A+1\right) $ divides $9$, and $\left( \frac{1%
}{3}\left( 4A^{2}-5A-3\right) ,A+3\right) $ divides $16$ since $A \equiv 2 \pmod{3}$. Therefore,
\begin{equation}\label{corba}
8A^{2}-7A+3 =3\cdot \frac{9}{\alpha }\cdot \frac{16}{\beta }%
\cdot \lambda \leq 1728(A-1)
\end{equation}
for some positive integer divisors $\alpha$ and $\beta$ of $9$ and $23$, respectively. Thus $A\leq 215$, and easy computations show that no cases fulfill  (\ref{corba}) since $\lambda$ is an odd prime with $\lambda \geq 2A+1$ and $\lambda>3$. This completes the proof of (1).

Assume that $\frac{r_{1}}{(r_{1},\lambda_{1})}<\lambda$. Then $\frac{v_{0}-1}{k_{0}-1}< \lambda$ by (\ref{double}), and hence $\mathcal{D}_{1}$ is not of type Ib--IIa since $\frac{v_{0}-1}{k_{0}-1}$ is equal to $(\lambda+1)^{2}$, $\frac{\lambda (\lambda +1)}{2}$ or  $\lambda$, respectively. Thus, $\mathcal{D}_{1}$ is of type Ia or IIb by Theorem \ref{Teo1}, which is (2).
\end{proof}

\begin{remark}\label{notuse}
If $\mathcal{D}_{1}$ is of type IIb, then $A^{2}\leq v_{1}-1$ and $v_{1} \neq \left(2A-1\right)^{2}$. Indeed, as shown in the proof of Lemma \ref{L3}(4), $A=v_{0}-3$. Hence $A^{2} \leq (v_{0}-2)(v_{0}-3)=v_{1}-1$. Moreover, if $v_{1}=(2A-1)^{2}$ then $v_{0}-2= 2v_{0}-7$ since $v_{1}=(v_{0}-2)^{2}$, and hence $v_{0}=5$, whereas $v_{0}$ is even (see Theorem \ref{Teo1}). Therefore $v_{1}\neq (2A-1)^{2}$.
\end{remark}

\bigskip

\section{Group-theoretic Reductions}

In this section, we prove the following reduction theorem for the $2$-designs $\mathcal{D}_{0}$ and $\mathcal{D}_{1}$. In particular, we show that $\mathcal{D}_{1}$ can never be of type Ib or IIb.   

\begin{theorem}\label{Teo2}
Assume that Hypothesis \ref{hyp2} holds. Then $\lambda \mid k_{1}$, $\lambda >k_{0}$ and one of the
following holds:
\begin{enumerate}
\item $\lambda > \frac{r_{1}}{(r_{1},\lambda_{1})}$ and the following hold:
\begin{enumerate}
    \item $\mathcal{D}_{0}$ is $2$-$(v_{0},k_{0},\lambda)$ design admitting $G_{\Delta}^{\Delta}$ as a flag-transitive automorphism group;
    \item $\mathcal{D}_{1}$ is $2$-$(v_{1},k_{1},\lambda_{1})$ design of type Ia.
\end{enumerate}
\item $\lambda \leq \frac{r_{1}}{(r_{1},\lambda_{1})}$ and the following hold:
\begin{enumerate}
\item $\mathcal{D}_{1}$ is $2$-$(v_{1},k_{1},\lambda_{1})$ design of type Ia or Ic admitting $G^{\Sigma}$ as a flag-transitive point-primitive automorphism group of affine or almost simple type; 
    \item $\left\vert G^{\Sigma }\right\vert < \left\vert G_{\Delta }^{\Sigma }\right\vert ^{2}$ for any $\Delta \in \Sigma$;
    \item $\frac{r_{1}}{(r_{1},\lambda_{1})}>v_{1}^{1/2}$.
\end{enumerate}
\end{enumerate}    
\end{theorem}

\bigskip

As we will see in the next sections, the usefulness of Theorem \ref{Teo2} is the following: either (1) holds, hence $G_{\Delta}^{\Delta}$ contains large prime order automorphisms (namely automorphisms of order $\lambda$), and the pair $(\mathcal{D}_{0},G_{\Delta}^{\Delta})$ is determined by using \cite{LS}; or (2) holds and much of the group-theoretic arguments for flag-transitive linear spaces, developed by Liebeck in \cite{LiebF} and by Saxl in \cite{Saxl}, can be transferred to $(\mathcal{D}_{1},G^{\Sigma})$. These arguments severely restrict the possibilities for $(\mathcal{D}_{1},G^{\Sigma})$.

\bigskip

\subsection{Notation}Throughout the remainder of this section, for simplicity, we denote the point set and the block set of $\mathcal{D}_{0}$ by $\mathcal{P}_{0}$ and $\mathcal{B}_{0}$, respectively, and the points and the blocks of $\mathcal{D}_{0}$ by small and capital Latin letters, respectively. The parameters of $\mathcal{D}_{0}$ will be denoted as usual. Finally, the group $G_{\Delta}^{\Delta}$ is denoted by $\Gamma$. Hence, $\mathcal{D}_{0}=(\mathcal{P}_{0},\mathcal{B}_{0})$ is a $2$-$(v_{0},k_{0},\lambda_{0} )$ design and $r_{0}=\frac{v_{0}-1}{k_{0}-1}\lambda_{0}$ admitting $\Gamma$ as a flag-transitive automorphism group.

If $x,y$ are any two distinct points of $\mathcal{D}_{0}$, then $\mathcal{B}_{0}(x)$
denotes the set of blocks of $\mathcal{D}_{0}$ containing $x$, and $\mathcal{B}_{0}%
(x,y)$ denotes the set of blocks of $\mathcal{D}_{0}$ containing the points $x$
and $y$. Clearly, $\left\vert \mathcal{B}_{0}(x)\right\vert =r_{0}=\frac{v_{0}-1}{%
k_{0}-1}\lambda $ and $\left\vert \mathcal{B}_{0}(x,y)\right\vert =\lambda $. Moreover, if $X$ is any subgroup of $\Gamma$ and $Y$ is any $X$-invariant subset of points, or blocks, of $\mathcal{D}_{0}$, we denote by $rank(X,Y)$ the rank of $X$ on $Y$, that is the number of $X$-orbits on $Y$.

\bigskip

\begin{lemma}\label{SameRank}
Assume that Hypothesis \ref{hyp1} holds and $\mathcal{D}_{0}$ is a $2$-$(v_{0},k_{0},\lambda )$ design. If $\lambda>k_{0}$, then the following hold:
\begin{enumerate}
\item if $B$ is any block of $\mathcal{D}_{0}$, then $rank(\Gamma_{B},B)=rank(\Gamma,\mathcal{P}_{0})$;
 \item if $x,y$ are any two distinct points of $\mathcal{D}_{0}$, then the group $\Gamma_{x,y}$ acts primitively on $\mathcal{B}_{0}(x,y)$. Moreover, denoted by $K$ the kernel of the action of $\Gamma_{x,y}$ on $\mathcal{B}_{0}(x,y)$, one of the following holds:
\begin{enumerate}
\item $Z_{\lambda }\trianglelefteq \Gamma_{x,y}/K\leq AGL_{1}(\lambda )$;

\item $A_{\lambda }\trianglelefteq \Gamma_{x,y}/K\leq S_{\lambda }$;

\item $PSL_{j}(w)\trianglelefteq \Gamma_{x,y}/K\leq P\Gamma L_{j}(w)$, with $j$
prime and $\frac{w^{j}-1}{w-1}=\lambda $.

\item $\Gamma_{x,y}/K\cong PSL_{2}(11)$ and $\lambda =11$;

\item $\Gamma_{x,y}/K$ is one of the groups $M_{11}$ or $M_{23}$ and $\lambda $
is $11$ or $23$, respectively.
\end{enumerate}
\end{enumerate}
\end{lemma}

\begin{proof}
The group $\Gamma_{x,y}$ permutes the elements of $\mathcal{B}_{0}(x,y)$. Let $U$ be any
Sylow $u$-subgroup of $\Gamma_{x,y}$. If $u\neq \lambda $, then there is $B_{U}$
in $\mathcal{B}_{0}(x,y)$ fixed by $U$ since $\left\vert \mathcal{B}_{0}%
(x,y)\right\vert =\lambda $. Thus $\left\vert U\right\vert \mid \left\vert
\Gamma_{x,y,B_{U}}\right\vert $, and hence $\left\vert U\right\vert \mid
\left\vert \Gamma_{x,B_{U}}\right\vert $. Let $B$ be any fixed block of $\mathcal{%
B}_{0}(x)$. Since $\Gamma_{x}$ is transitive on the set of blocks incident with $x$,
it follows that there is $\psi \in \Gamma_{x}$ such that $B_{U}^{\psi }=B$. Thus $%
\Gamma_{x,B_{U}}^{\psi }=\Gamma_{x,B}$, and hence $\left\vert U\right\vert \mid
\left\vert \Gamma_{x,B}\right\vert $. Therefore, $\frac{\left\vert \Gamma_{x,y}\right\vert}
{\left\vert \Gamma_{x,y}\right\vert_{\lambda} }\mid \left\vert \Gamma_{x,B}\right\vert $.

If $\Gamma_{x,y}$ is intransitive on $\mathcal{B}_{0}(x,y)$, then any Sylow $\lambda $%
-subgroup of $\Gamma_{x,y}$ fixes each element $\mathcal{B}_{0}(x,y)$ since $\left\vert \mathcal{B}_{0}%
(x,y)\right\vert =\lambda $. Therefore $%
\left\vert \Gamma_{x,y}\right\vert _{\lambda }\mid \left\vert \Gamma_{x,B}\right\vert $%
, and hence $\left\vert \Gamma_{x,y}\right\vert \mid \left\vert
\Gamma_{x,B}\right\vert $ since $\frac{\left\vert \Gamma_{x,y}\right\vert}
{\left\vert \Gamma_{x,y}\right\vert_{\lambda} }\mid \left\vert \Gamma_{x,B}\right\vert $. Then

\begin{equation}
\left\vert y^{\Gamma_{x}}\right\vert =\frac{\left\vert \Gamma_{x}\right\vert }{%
\left\vert \Gamma_{x,y}\right\vert }=\frac{\left\vert \Gamma_{x}\right\vert }{%
\left\vert \Gamma_{x,B}\right\vert }\cdot \frac{\left\vert \Gamma_{x,B}\right\vert }{%
\left\vert \Gamma_{x,y}\right\vert }=r_{0}\frac{\left\vert \Gamma_{x,B}\right\vert }{%
\left\vert \Gamma_{x,y}\right\vert }\mathit{,}
\end{equation}%
and hence $r\mid \left\vert y^{\Gamma_{x}}\right\vert $ since $\frac{\left\vert
\Gamma_{x,B}\right\vert }{\left\vert \Gamma_{x,y}\right\vert }$ is an integer. Thus, for any pair of distinct points $x,y$ of $\mathcal{D}_{0}$ either $\Gamma_{x,y}$ is intransitive on $\mathcal{B}_{0}(x,y)$ and $r_{0}\mid \left\vert y^{\Gamma_{x}}\right\vert $ by the previous argument, or $\Gamma_{x,y}$ is transitive on $\mathcal{B}_{0}(x,y)$ and $\frac{r_{0}}{\lambda}\mid \left\vert y^{\Gamma_{x}}\right\vert $ by Lemma \ref{PP}(2). 

Let $y_{1},...,y_{s}$ be the representatives of the $\Gamma_{x}$-orbits on $\mathcal{P}_{0}$ distinct from $\{x\}$. Let $d_{i}=\left\vert y_{i}^{G_{x}}\right\vert$ with $i=1,...,s$, and $\Theta \subseteq \{1,...,s\}$ be such that the group $\Gamma_{x,y_{i}}$ acts intransitively on $\mathcal{B}_{0}(x,y_{i})$ if and only if $i \in \Theta$. Suppose that $\Theta \neq \varnothing$. If $\theta$ denotes the size of $\Theta$, without loss we may assume that $\Theta=\{ y_{1},...,y_{\theta} \}$. Then $r_{0} \mid d_{i}$ for $i=1,...,\theta$ since $\Gamma_{x,y_{i}}$ acts intransitively on $\mathcal{B}_{0}(x,y_{i})$ for each $i \in \Theta$. On the other hand, $\frac{r_{0}}{\lambda} \mid d_{i}$ for $i>\theta$. Hence, $d_{1}+\cdots +d_{\theta}=e\cdot r_{0}$ and  $d_{\theta +1}+\cdots +d_{s}=f\cdot \frac{r_{0}}{\lambda }$ for some non-negative integer $e$ and $f$. Actually, $e \geq 1$ since $\Theta \neq \varnothing$ by our assumption. Therefore,%
\begin{equation*}
v_{0}-1=\left( d_{1}+\cdots +d_{\theta }\right) +\left( d_{\theta +1}+\cdots
+d_{s}\right) =e\cdot r_{0}+f\cdot \frac{r_{0}}{\lambda }=\frac{r_{0}}{\lambda }(\lambda
e+f)\text{.} 
\end{equation*}
Thus $v_{0}-1=\frac{v_{0}-1}{k_{0}-1}(\lambda e+f)$ since $r_{0}=\frac{v_{0}-1}{k_{0}-1}\lambda$, and hence $%
k_{0}-1=\lambda e+f$ with $e\geq 1$ and $f \geq 0$. So $%
\lambda +1\leq k_{0}$, whereas $\lambda>k_{0}$ by our assumption. Thus $\Theta = \varnothing$, and hence $\Gamma_{x,y_{i}}$ acts transitively on $\mathcal{B}_{0}(x,y_{i})$ for each $i \in \{1,...,s\}$. Therefore, $\Gamma_{x,y}$ acts transitively on $\mathcal{B}_{0}(x,y)$ for each point $y$ of $\mathcal{D}_{0}$ distinct from $x$. Then (1) follows from \cite[Lemma 7]{Ca}. Moreover, $\Gamma_{x,y}$ acts primitively on $\mathcal{B}_{0}(x,y)$ by \cite[Corollary 3.5B]{DM}. Moreover $\Gamma_{x,y}/K$, where $K$ is the kernel of the action of $\Gamma_{x,y}$ on $\mathcal{B}_{0}(x,y)$, is as (2) by \cite[p. 99]{DM}.
\end{proof}

\bigskip
\begin{corollary}\label{SameRankDes}
Assume that Hypothesis \ref{hyp3/2} holds and $\mathcal{D}_{0}$ is a $2$-$(v_{0},k_{0},\lambda )$ design. Then the conclusions of Lemma \ref{SameRank} hold.  
\end{corollary}

\begin{proof}
Assume that Hypothesis \ref{hyp3/2} holds and $\mathcal{D}_{0}$ is a $2$-$(v_{0},k_{0},\lambda )$ design. Then $\mathcal{D}_{1}$ is a $2$-design, and hence $\lambda>k_{0}$ by Theorem \ref{Teo1}. The assertion now follows from Lemma \ref{SameRank}.    
\end{proof}

\bigskip

\begin{proposition}\label{LambigK=0}
Assume that Hypothesis \ref{hyp2} holds. Then $\mathcal{D}_{1}$ is not of type Ib.
\end{proposition}

\begin{proof}
Assume that $\mathcal{D}_{1}$ is of type Ib. Then $\mathcal{D}_{0}$ is a $2$-$(\lambda
\left( \lambda ^{2}+\lambda -1\right) ,\lambda ,\lambda /\mu )$ design by Theorem \ref{Teo1}.
Since $\lambda \left( \lambda ^{2}+\lambda -1\right) $ is not a power of an
integer, it follows from the O'Nan-Scott theorem (e.g. see \cite[Theorem 4.1A%
]{DM}) that $\Gamma$ is almost simple.

Assume that $\mu=\lambda$. Then $\mathcal{D}_{0}$ is a linear space, and hence one of the following hold \cite[Theorem]{BDDKLS} since both $v_{0}$ and $k_{0}$ are odd:
\begin{enumerate}
    \item $\mathcal{D}_{0}\cong PG_{m-1}(s)$, $m\geq 3$, and $(v_{0},k_{0})=\left(\frac{s^{m}-1}{s-1},s+1\right)$ with $s$ even;
    \item $\mathcal{D}_{0}$ is the Hermitian unital of order $s$, and $(v_{0},k_{0})=\left(s^{3}+1,s+1\right)$ with $s$ even.
\end{enumerate}
Then $\lambda=s-1$ in both cases, and hence $\frac{s^{m}-1}{s-1}=s^{3}-2s^{2}+1$ or $s^{3}+1=s^{3}-2s^{2}+1$ according as case (1) or (2) occurs, respectively. Each of these equations does not provide admissible solutions, and hence cases (1) and (2) are excluded.

Assume that $\mu=1$. Hence, $\mathcal{D}_{0}$ is a $2$-$(\lambda
\left( \lambda ^{2}+\lambda -1\right) ,\lambda ,\lambda )$ design. Let $N$ be the action kernel of $\Gamma_{B}$ on $B$, then either $%
Z_{\lambda }\trianglelefteq \Gamma_{B}/N\leq AGL_{1}(\lambda )$ and $\Gamma_{B}$ acts $%
3/2$-transitively on $B$, or $\Gamma_{B}$ acts $2$-transitively on $B$ by \cite[p.99%
]{DM}. Then $\Gamma$ acts either $3/2$-transitively or $2$%
-transitively on the point set of $\mathcal{D}$, respectively, by Lemma \ref{SameRank}(1). Actually, $\Gamma$ acts point-$2$-transitively on $\mathcal{D}_{0}$ by 
\cite[Theorem 1.2]{BGLPS} since $\Gamma$ is almost simple and $v_{0}$ is odd and distinct from $21$. Further, $Soc(\Gamma)\ncong A_{v_{0}}$ by \cite[Theorem 1]{ZCZ}. Then one of the following
holds by \cite[(A)]{Ka} since $v_{0}=\lambda \left( \lambda ^{2}+\lambda
-1\right) $ is an odd composite integer, and $v_{0}>15$ since $\lambda >3$:

\begin{enumerate}
\item $Soc(\Gamma)\cong PSL_{m}(s)$, $m\geq 2$, $(m,s)\neq (2,2),(2,3)$ and $v_{0}=\frac{%
s^{m}-1}{s-1}$;

\item $Soc(\Gamma)\cong PSU_{3}(s)$, $s=2^{h}$, $h\geq 2$,
and $v_{0}=s^{3}+1$;

\item $Soc(\Gamma)\cong Sz(s)$, $s=2^{h}$, $h\geq 3$ odd, and 
$v_{0}=s^{2}+1$.
\end{enumerate}

Note that $v_{0}-1=\allowbreak \left( \lambda -1\right) \left( \lambda
+1\right) ^{2}$. Then $2^{hl}=\allowbreak \left( \lambda -1\right) \left(
\lambda +1\right) ^{2}$ with $l=3$ or $2$ in cases (2) and (3),
respectively. Then $\lambda =2^{u}+1$ for some $1\leq u<hl$, and hence $%
\lambda +1=2^{u}+2$. Then $(2^{u}+2)^{2}=2^{hl-u}$, and hence $u=1$ and $hl-u=4$.
So $hl=5$, which is not the case since $h,l>1$. Thus, $Soc(\Gamma)\cong PSL_{m}(s)$ with $s=c^i$, $c$ prime, $i \geq 1$ and 
\begin{equation}\label{dzumbuli}
\frac{s^{m}-1}{s-1}=\lambda (\lambda ^{2}+\lambda -1)\text{.}  
\end{equation}

Clearly, the actions on the point sets of $\mathcal{D}_{0}$ and $PG_{m-1}(s)$ are
equivalent. Now, let $x$ and $y$ be any two distinct points of $\mathcal{D}%
_{0}$. Then a quotient group of $G_{x,y}$ is as in Lemma \ref{SameRank}(2). On the
other hand, by \cite[Proposition 4.1.17(II)]{KL}, the quotient groups
of $G_{x,y}$ compatible with the list given in Lemma \ref{SameRank}(2) are those containing either a $c$-group or $%
PSL_{m-2}(s)$ as a normal subgroup, or the solvable ones
of order a divisor of $(m,s-1)\cdot i$. The comparison of the two
information on $G_{x,y}$ leads to the following possible cases: either $\lambda \mid s$ and $m=3$, or $%
(m,s,\lambda )=(4,4,5),(4,5,5)$, or $\lambda =\frac{s^{m-2}-1}{s-1}$ with $m>3$, or $\lambda =11
$ or $23$, or $\lambda \mid (m,s-1)\cdot i$. It is immediate to
see that the previous fourth cases out of five are ruled oud since they do
not fulfill (\ref{dzumbuli}). Thus, $\lambda \mid (m,s-1)\cdot i$,
and hence either $\lambda \mid s-1$, or $\lambda \mid i$ and $%
\lambda \leq s^{1/2}$. The former implies $\frac{s^{m}-1}{s-1}\leq s^{3}-2s^{2}+1$, and hence $m\leq 3$, contrary to $\lambda \mid (m,s-1)$ and $\lambda > 3$. Therefore $\lambda \mid i$ and $%
\lambda \leq s^{1/2}$, and hence $\frac{s^{m}-1}{s-1}\leq s^{1/2}(s+s^{1/2}-1)$. So $m=2$, and hence (\ref{dzumbuli}) implies $c^{i}\leq i(i^2+i-1)$ with $c \geq 5$ and $\lambda \mid (c+1,i)$, a contradiction. This completes the proof.
\end{proof}

\bigskip

\begin{lemma}\label{Fix} 
Assume that Hypothesis \ref{hyp2} holds. The one of the following holds:
\begin{enumerate}
\item $\mu=1$ and $\lambda$ divides the order of $G_{\Delta }^{\Delta }$;
\item $\mu=\lambda$ and either $\lambda \mid v$, or $\lambda \mid v_{1}-1$, or $\lambda$ divides the order of $G_{\Delta }^{\Delta }$.  
\end{enumerate}
\end{lemma}

\begin{proof}
The assertion (1) immediately follows when $\mu=1$. Indeed, in this case, $\mathcal{D}_{0}$ is a $2$-$(v_{0},k_{0},\lambda)$ design, and $r_{0}=\frac{v_{0}-1}{k_{0}-1}\lambda$ divides the order of $G_{\Delta }^{\Delta }$ since $G_{\Delta }^{\Delta }$ acts flag-transitively on $\mathcal{D}_{0}$. 

Assume that $\mu=\lambda$. Therefore, $\mathcal{D}_{0}$ is a (possibly trivial) $2$-$(v_{0},k_{0},1)$ design. Let $L$ be any Sylow $\lambda$-subgroup of $G$. If $\lambda \nmid v(v_{1}-1)$, then $L$ preserves an element $\Delta$ of $\Sigma$ by Theorem \ref{Teo1}. Hence, $L \leq G_{\Delta}$. Suppose that $\lambda $ does not divide the order of $G_{\Delta }^{\Delta }$. Then $L \leq G_{(\Delta)}$, and hence $\Delta \subseteq F$ with $F=Fix(L)$. Moreover, $N_{G}(L)$ acts transitively on $F$ by \cite[Theorem I.3.7]{Wie}, and hence $\F=\bigcup_{\sigma \in N_{G}(L)} \Delta^{\sigma}$ since $\Delta \subseteq F$. Thus, $\left\vert F\right\vert=v_{0}c$ with $c=\left\vert N_{G}(L):N_{G_{\Delta}}(L)\right\vert$.

Let $\mathcal{F}$ be the incidence structure $(F,\mathcal{B}%
_{F})$, where $\mathcal{B}_{F}=\left\{ B\cap F:B\in \mathcal{B},\left\vert B\cap F\right\vert \geq 2\right\} $. Now, let $y,z\in F$, $y\neq z$, and $\mathcal{B}%
(y,z)$ the set of blocks of $\mathcal{D}$ containing $y$ and $z$. Assume that $L$ fixes an element $\mathcal{B}%
(y,z)$, say $C$. Then $L \leq G_{y,C}$, and hence $r=\left\vert C^{G_{y}} \right \vert = \left\vert G_{y}:G_{y,C} \right \vert$ has order coprime to $\lambda$ since $L$ is a Sylow $\lambda$-subgroup of $G$, but this contradicts $r= \frac{v-1}{k-1}\lambda$. Thus, for any $y,z\in F$, $y\neq z$, the group $L$ acts
transitively on $\mathcal{B}(y,z)$, and hence $B_{1}\cap F=B_{2}\cap F$ for
any $B_{1},B_{2}\in \mathcal{B}(y,z)$. Thus, $\mathcal{F}=(F,\mathcal{B}%
_{F})$ is a $2$-$(cv_{0},K_{F} ,1)$ design, where $K_{F}=\left\{ \left\vert B\cap F\right\vert :B\cap F\in 
\mathcal{B}_{F}\right\}$. 



Let $B$ be any
block of $\mathcal{D}$ such that $B\cap \Delta \neq \varnothing $ with $\Delta \subseteq F$, and let $%
\alpha \in G_{(\Delta) }$. Then $B^{\alpha }\cap \Delta =B\cap \Delta\textbf{} $, and
hence there is $\gamma \in L$ such that $B^{\alpha \gamma ^{-1}}=B$ by the above argument since $%
\left\vert B\cap \Delta\right\vert =k_{0}\geq 2$. Therefore, $\alpha \gamma ^{-1} \in G_{(\Delta),B}$ since $\alpha \in G_{(\Delta)}$ and $L \leq G_{(\Delta)}$. Thus $\alpha \in
G_{(\Delta),B}L$, and hence 
\begin{equation}\label{sunnyday}
G_{(\Delta)}=G_{(\Delta),B}L    
\end{equation}
since $G_{(\Delta),B}L\leq G_{(\Delta)}$.

Let $x$ be any point point of $\mathcal{F}$ and let $B_{1}\cap F,B_{2}\cap F\in \mathcal{B}_{F}$ containing $x$. We may assume that $%
x\in \Delta $ since 
$N_{G}(L)$ acts point-transitively on $\mathcal{F}$. Then there is $\varphi \in G_{x}$ such that $B_{1}^{\varphi
}=B_{2}$ since $G$ acts flag-transitively on $\mathcal{D}$. Note that, $$%
G_{x}=G_{(\Delta )}N_{G_{x}}(L)=\left( G_{(\Delta ),B_{1}}L\right)
N_{G_{x}}(L)=G_{(\Delta ),B_{1}}N_{G_{x}}(L)$$ by the Frattini argument since $G_{(\Delta)} \unlhd G_{x}$, and by (\ref{sunnyday}). Hence, $\varphi =\zeta
\vartheta $ with $\zeta \in G_{(\Delta ),B_{1}}$ and $\vartheta \in
N_{G_{x}}(L)$. Thus, $B_{2}=B_{1}^{\varphi }=B_{1}^{\zeta \vartheta
}=B_{1}^{\vartheta }$ with $\vartheta \in N_{G_{x}}(L)$. Therefore $%
B_{1}^{\vartheta }\cap F=B_{2}\cap F$, and hence $\mathcal{F}$ is a $2$-$%
(v_{0}c,k_{F},1)$ design with $k_{F}=\left\vert B_{1}\cap F\right\vert $ admitting $N_{G}(L)$ as a flag-transitive automorphism group. Then $N_{G}(L)$
acts point-primitively on $\mathcal{F}$ by \cite[Proposition 3]{HM}, and hence $c=1$ and $%
\mathcal{D}_{0}=\mathcal{F}$. Indeed, if $\Delta ^{N_{G}(L)}\neq \left\{
\Delta \right\} $, that is $c>1$, then $\Delta ^{N_{G}(L)}$ is a $N_{G}(L)$-invariant
partition of $F$ since $\Sigma $ is a $G$-invariant partition of $\mathcal{D}
$. So $\lambda \mid v-v_{0}$, and hence either $\lambda \mid v_{0}$ or $\lambda \mid v_{1}-1$ since $%
v=v_{0}v_{1}$, contrary to our assumption. This completes the proof
\end{proof}

\bigskip

\bigskip

\begin{proposition}\label{LambdaDividesK1}
 Assume that Hypothesis \ref{hyp2} holds. Then $\mathcal{D}_{1}$ is not of type IIb.
\end{proposition}
\begin{proof}
The $2$-design $\mathcal{D}_{1}$ is not of type Ib by Proposition \ref{LambigK=0}. Hence, $\lambda >k_{0}$ by Theorem \ref{Teo1}. Now, assume that $\mathcal{D}_{1}$ is of IIb. Hence, $\lambda >v_{0}$ again by Theorem \ref{Teo1}. Let $L$ be any Sylow $\lambda$-subgroup of $G$, then $L$ fixes at least a point $x$ of $\mathcal{D}$ since $v=v_{0}v_{1}$, $v_{1}=(v_{0}-2)^2$ and $\lambda >v_{0}$. Then $L$ preserves the element $\Delta$ of $\Sigma$ containing $x$, and hence $L \leq G_{\Delta}$. Actually, $L \leq G_{(\Delta)}$ since $\lambda$ does not divide the order of $G_{\Delta }^{\Delta }$, being $\lambda >v_{0}$ and $\lambda$ prime. Thus $\lambda$ does not divide the order $G_{\Delta}^{\Delta}$ and hence either $\lambda \mid v$ or $\lambda \mid v_{1}-1$ by Lemma \ref{Fix}(2). However, this is impossible since $\lambda >v_{0}$, $v=v_{0}v_{1}$, $v_{1}=(v_{0}-2)^2$, and $v_{1}-1=(v_{0}-2)^2-1=(v_{0}-1)(v_{0}-3)$. Thus $\mathcal{D}_{1}$ is not of type IIb.
\end{proof}

\bigskip

\begin{theorem}\label{TeoFeb23}
Assume that Hypothesis \ref{hyp2} holds. Then one of the
following holds:
\begin{enumerate}
\item $\lambda > \frac{r_{1}}{(r_{1},\lambda_{1})}$ and the following hold:
\begin{enumerate}
    \item $\mathcal{D}_{0}$ is $2$-$(v_{0},k_{0},\lambda)$ design admitting $G_{\Delta}^{\Delta}$ as a flag-transitive automorphism group;
    \item $\mathcal{D}_{1}$ is $2$-$(v_{1},k_{1},\lambda_{1})$ design of type Ia.
\end{enumerate}
\item $\lambda \leq \frac{r_{1}}{(r_{1},\lambda_{1})}$ and the following hold:
\begin{enumerate}
\item $\mathcal{D}_{1}$ is $2$-$(v_{1},k_{1},\lambda_{1})$ design of type Ia, Ic or IIa;
\item $G^{\Sigma}$ is a flag-transitive point-primitive automorphism group of $\mathcal{D}_{1}$ of affine or almost simple type; 
\end{enumerate}
\end{enumerate}    
\end{theorem}
\begin{proof}
Assume that $\lambda > \frac{r_{1}}{(r_{1},\lambda_{1})}$. Then $\mathcal{D}_{1}$ is of type Ia by Propositions \ref{C3}(2) and \ref{LambdaDividesK1}. Suppose that $\mu=\lambda$. Then $\mathcal{D}_{0}$ is a (possibly trivial) $2$-$(v_{0},k_{0},1)$ design. Now, $\lambda \nmid v(v_{1}-1)$ by Theorem \ref{Teo1} since $\mathcal{D}_{1}$ is of type Ia. Then $\lambda$ divides the order of $G_{\Delta}^{\Delta}$ since by Lemma \ref{Fix}(2). Let $\sigma$ be any non-trivial $\lambda$-element of $G^{\Delta}_{\Delta}$. Then $\sigma$ fixes a point $x$ of $\mathcal{D}_{0}$ since $\lambda \nmid v_{0}$. Then $\sigma$ fixes each line of $\mathcal{D}_{0}$ containing $x$ since $\lambda>\frac{r_{1}}{(r_{1},\lambda_{1})}=\frac{v_{0}-1}{k_{0}-1}=r_{0}$. Any such lines is fixed pointwise by $\sigma$ since $\lambda>k_{0}$ by Theorem \ref{Teo1}(2), and hence $\sigma$ fixes $\mathcal{D}_{0}$ pointwise, a contradiction. Thus $\mu=1$ by Lemmas \ref{base}(2) and \ref{L2}, and hence $\mathcal{D}_{0}$ is a (possibly trivial) $2$-$(v_{0},k_{0},\lambda)$ design. This proves (1).

Assume that $\lambda \leq \frac{r_{1}}{(r_{1},\lambda_{1})}$. Then $A^{2}\leq v_{1}-1$ and $v_{1} \neq \left(2A-1\right)^{2}$ by Proposition \ref{C3}(1). Then $G^{\Sigma }$ acts point-primitively on $\mathcal{D}_{1}$ by \cite[Lemma 2.7]{ZZ}, and hence $G^{\Sigma }$ is either of affine type or of almost simple type by \cite[Theorem 1.1]{ZZ}. Moreover, $\mathcal{D}_{1}$ is of type Ia, Ic or IIa by Theorem \ref{Teo1} and Propositions \ref{LambigK=0} and \ref{LambdaDividesK1}. This proves (2).    
\end{proof}

\bigskip

\begin{proposition}\label{NotIIa}
Assume that Hypothesis \ref{hyp2} holds. Then $\mathcal{D}_{1}$ is not of type IIa.    
\end{proposition}
\begin{proof}
Assume that $\mathcal{D}_{1}$ is of type IIa. Then the group $G^{\Sigma }$ is of order divisible by $%
\lambda $ since $k_{1}$ is divisible by $\lambda$ by Theorem \ref{Teo1} and $G^{\Sigma }$ acts flag-transitively on $\mathcal{D}_{1}$. Further, $G^{\Sigma}$ is a flag-transitive point-primitive automorphism group of $\mathcal{D}_{1}$ of affine or almost simple type Theorem \ref{TeoFeb23}.

Let $\zeta $ be any non-trivial $\lambda $-element of $G^{\Sigma }
$. Then $v_{1}=\left( \lambda -1\right)^{2}=m\lambda+f$ with $1\leq m <p$ denoting the number of $p$-cycles of $\zeta$ and $f \equiv 1 \pmod{\lambda}$ the number of fixed points by $\zeta$ on $\mathcal{D}_{1}$. Then either $A_{\left( \lambda -1\right) ^{2}}\trianglelefteq G^{\Sigma
}\leq S_{\left( \lambda -1\right) ^{2}}$ or $\left( G^{\Sigma },\lambda
\right) $ is classified in \cite[Theorem 1.1]{LS}. Suppose that  the former occurs and let $%
\Delta $ be any fixed point of $\mathcal{D}_{1}$. Then $A_{\left( \lambda
-1\right) ^{2}-1}\trianglelefteq G_{\Delta }^{\Sigma }\leq S_{\left( \lambda
-1\right) ^{2}-1}$, and hence $A_{\left( \lambda -1\right) ^{2}-1}$ partitions
the $r_{1}$ blocks of $\mathcal{D}_{1}$ containing $\Delta$ into $a$ orbits of equal
length since $G_{\Delta
}^{\Sigma }$ acts transitively on the $r_{1}$ blocks of $\mathcal{D}_{1}$
containing $\Delta$. Further $a>1$ since $r_{1}>2$, and  $a\mid \lambda \frac{%
\lambda +1}{2}$ and $a\geq \left( \lambda -1\right) ^{2}-1$ since the
smallest non-trivial transitive permutation degree $A_{\left( \lambda
-1\right) ^{2}-1}$ is $\left( \lambda -1\right) ^{2}-1$. Then $\lambda =5$,
and hence $\mathcal{D}_{1}$ is a $2$-$(16,10,9)$ design and $%
A_{16}\trianglelefteq G^{\Sigma }\leq S_{16}$. Then $b_{1}=\binom{16}{10}$ since $A_{16}$ is point-$14$-transitive on $\mathcal{D}_{1}$, whereas $b_{1}=\frac{v_{1}r_{1}}{k_{1}}=24$. Thus this case is ruled out, and hence $\left( G^{\Sigma
},\lambda \right) $ is classified in \cite[Theorem 1.1]{LS}.

Suppose that $Soc(G^{\Sigma })$ is an elementary abelian $u$-group for some
prime $u$. Then $u=2$ and $\lambda =2^{h/2}+1$ with
$h$ a power of $2$ since $v_{1}=(\lambda -1)^{2}=u^{h}$ with $h\geq 1
$, $h$ even, and $\lambda $ is an odd prime. Further, $h>2$ since $\lambda >3$. By \cite[%
Table 1]{LS} and bearing in mind that $\lambda$ is a Fermat prime and $f \equiv 1 \pmod{\lambda}$, the group $G_{\Delta }^{\Sigma }$ contains a normal subgroup isomorphic
to one of the groups $SL_{h/j}(2^{j})$, $Sp_{h/j}(2^{j})$, $\Omega
_{h/j}^{-}(2^{j})$, or $PSL_{2}(17)$ and $\lambda=17$ for $h=8$, or one of the groups $A_{7}$ and $\lambda=5$ for $h=4$.

Assume that $G_{\Delta }^{\Sigma }$ contains a normal subgroup $N$
isomorphic to one of the groups $SL_{h/j}(2^{j})$, $Sp_{h/j}(2^{j})$, $%
\Omega _{h/j}^{-}(2^{j})$. Then%
\begin{equation*}  
r_{1}=\lambda \frac{\lambda +1}{2}=\left( 2^{h/2}+1\right) \left(2^{h/2-1}+1\right)\text{.}
\end{equation*}
Suppose that $h>8$, then each of the integers $\lambda=2^{h/2}+1$ and $\alpha=2^{h/2-1}+1$ admits a primitive prime divisor by Zsigmondy's Theorem (e.g. see \cite[Theorem 5.2.14]{KL}). Since $\left\vert Out(N)\right\vert=2(h/j,2^{j}-1)j$, $2j$, $2j$ according as $N$
isomorphic to $SL_{h/j}(2^{j})$, $Sp_{h/j}(2^{j})$, $%
\Omega _{h/j}^{-}(2^{j})$, respectively, then $\left\vert Out(N)\right\vert \mid 2h$ in each case since $j \mid h$. Thus, $\lambda\alpha \mid \left\vert N \right\vert$ since $\lambda \equiv 1 \pmod{h}$ and $\alpha \equiv 1 \pmod{(h-2)}$ by \cite[Proposition 5.2.15(ii)]{KL}, forcing $j=1$ by \cite[Proposition 5.2.15(i)]{KL}. By a direct inspection, the same conclusion holds for $h=8$ since $r_{1}=17\cdot 7$ divides the order $G_{\Delta}^{\Sigma}$.

Note that, $r_{1}$ is divisible by the index of a maximal parabolic subgroup of $N$ since $N \unlhd G^{\Sigma}$ and by \cite[1.6]{Se}. 
If $N\cong SL_{h}(2)$, then 
$$
2^{h}-1\leq {h\brack t}_{2} \leq  \left( 2^{h/2}+1\right) \left( 2^{h/2-1}+1\right)$$
with $1 \leq t \leq h/2$ by \cite[Proposition 4.1.17(II)]{KL}, contrary to $h \geq 8$.

If $N\cong Sp_{h}(2)$, then 
\begin{equation}\label{torrone}
\prod_{i=0}^{t-1}\frac{2^{h-2i}-1}{2^{i+1}-1}\mid \left( 2^{h/2}+1\right)
\left( 2^{h/2-1}+1\right) 
\end{equation}
with $1 \leq t \leq h/2$ by \cite[Proposition 4.1.19(II)]{KL}. Then $\sum_{i=0}^{t-1}(h-3i-1)<h$ by using \cite[Lemma 4.1]{AB}, and hence $$t(h+2)/4 \leq t(h-1-3(h/2-1)/2)\leq t(h-1)-3t(t-1)/2<h\text{,}$$ which implies $1 \leq t \leq 3$. However, (\ref{torrone}) is not fulfilled for $t=1,2$ or $3$ since $h \geq 8$.

If $N\cong \Omega _{h}^{-}(2)$, then  
\[
{h/2-1\brack t}_{2}\prod_{i=0}^{t-1}\left( 2^{h/2-2-i}+1\right) \mid
\left( 2^{h/2}+1\right) \left( 2^{h/2-1}+1\right) 
\]%
with $1 \leq t \leq h/2-1$ by \cite[Proposition 4.1.19(II)]{KL} (see also \cite[Exercise 11.3]{Tay}). So $2^{h/2-2}+1$ divides $
\left( 2^{h/2}+1\right) \left( 2^{h/2-1}+1\right)$, and hence $2^{h/2-2}+1 \mid 3$, contrary to $h\geq 8$.

Finally, assume that $h=4$. Then $\lambda=5$ and $\mathcal{D}_{1}$ is a $2$-$(16,10,9)$ design. Let $%
T=Soc(G_{\Delta }^{\Sigma })$, then $\left\vert G_{B^{\Sigma }}^{\Sigma
}\cap T\right\vert =2$ since $k_{1}=10$, hence $G_{B^{\Sigma }}^{\Sigma }T/T$
is isomorphic to a subgroup of $G_{\Delta }^{\Sigma }$ of index $%
r_{1}/(k_{1}/2)=3$ since $r_{1}=15$, a contradiction. The previous argument
can still be used to rule out the case $G_{\Delta }^{\Sigma }\cong
A_{7}$ and $h=4$.

Finally, assume that $PSL_{2}(17)\trianglelefteq G_{\Delta }^{\Sigma }$, $\lambda=17$ and $%
h=8$. Then $\mathcal{D}_{1}$ is a $2$-$(256,136,81)$ design, and $G_{\Delta }^{\Sigma }\leq PGL_{2}(17)$ by \cite{AtMod}. Since 
$v_{1}\equiv 1\pmod{3}$ and $v_{1}\equiv 4\pmod{9}$, there is an
element $\Delta ^{\prime }\in \Sigma $, $\Delta ^{\prime }\neq \Delta $,
such that $\left\vert \left( \Delta ^{\prime }\right) ^{G_{\Delta }^{\Sigma
}}\right\vert $ is not divisible by $9$, but this contradicts Lemma \ref{PP}(2) since $r_{1}/(r_{1},\lambda _{1})=9$. Thus, this case is excluded.

Suppose that $Soc(G^{\Sigma })$ is a non-abelian simple group. Then one of
the following holds by \cite[Table 2]{LS} since $v_{1}=(\lambda -1)^{2}$, $\lambda>3$ and $f \equiv 1 \pmod{\lambda}$:

\begin{enumerate}
\item $Soc(G^{\Sigma })\cong A_{m}$, $m\geq 5$, $(\lambda -1)^{2}=m(m-1)/2$
and $2+\lambda \leq m\leq \frac{1}{2}(3\lambda -1)$;

\item $Soc(G^{\Sigma })\cong PSL_{m}(w)$, $m\geq 3$, and $(\lambda -1)^{2}=\frac{w^{m}-1%
}{w-1}$;

\item $Soc(G^{\Sigma })\cong PSL_{2}(w)$, $w>2$ even, $(\lambda -1)^{2}=\frac{%
w(w+1)}{2}$ and $\lambda =w-1$;

\item $Soc(G^{\Sigma })\cong PSL_{2}(w)$, $w>2$ even, $(\lambda -1)^{2}=\frac{%
w(w-1)}{2}$ and $\lambda =w+1$;

\item  $Soc(G^{\Sigma })\cong Sp_{2m}(2)$, $m\geq 2$, $(\lambda -1)^{2}=2^{m-1}(2^{m}+1)$ and $\lambda=2^{m}-1$;

\item $Soc(G^{\Sigma })\cong Sp_{2m}(q)$, $m\geq 2$, $q$ even, $(\lambda -1)^{2}=\frac{1}{2}q^{m}(q^{m}-1)$ and $\lambda=q^{m}+1$;

\item $Soc(G^{\Sigma })\cong HS$, $v_{1}=100$ and $\lambda =11$;

\item $Soc(G^{\Sigma })\cong PSL_{3}(3)$, $v_{1}=144$ and $\lambda =13$;

\item $Soc(G^{\Sigma })\cong PSU_{3}(3)$, $v_{1}=36$ and $\lambda =7$.
\end{enumerate}

Assume that (1) occurs. If $m \leq 10$, then $m=9$ and $\lambda=7$ since $(\lambda -1)^{2}=m(m-1)/2$. Therefore, $\mathcal{D}_{1}$ is a $2$-$(36,21,16)$ design, $b_{1}=48$ and $A_{9}\trianglelefteq G^{\Sigma }\leq S_{9}$. This is impossible since $48$ is not a transitive permutation degree for $G^{\Sigma}$ by \cite{At}. Thus, $m>10$. Let $%
\Delta $ be any fixed point of $\mathcal{D}_{1}$, then $A_{m-2}\trianglelefteq G_{\Delta }^{\Sigma }\leq S_{m-2}\times Z_{2}$. Then $A_{m-2}$ partitions
the $G_{\Delta }^{\Sigma }$-orbit consisting of the $r_{1}$ blocks of $\mathcal{D}_{1}$ containing $\Delta $ into $\theta $%
-orbits each of equal length with $\theta \mid 8$ since $A_{m-2}\trianglelefteq G_{\Delta }^{\Sigma }$. Now, $r_{1}=\lambda 
\frac{\lambda +1}{2}$ and $\lambda =\sqrt{\frac{m(m-1)}{2}}+1$ and imply 
\[
\frac{r_{1}}{\theta}=\left( \sqrt{%
\frac{m(m-1)}{2}}+1\right) \left( \frac{1}{2}\sqrt{\frac{m(m-1)}{2}}+1\right) <\left( \frac{m}{\sqrt{2}}+1\right) ^{2}<\binom{m-2}{3%
}
\]%
for $m>10$. Therefore, either $\lambda \frac{\lambda +1}{2}=\left(
m-2\right) \theta $, or $\lambda \frac{\lambda +1}{2}=\frac{\left(
m-2\right) \left( m-3\right) }{2}\theta $, or $m$ is even $\lambda \frac{%
\lambda +1}{2}=\frac{1}{2}\binom{m}{m/2}\theta $ by \cite[Theorems 5.2A]{DM}%
. In the first case, $\frac{m(m-1)}{4}< \lambda \frac{\lambda +1}{2}\leq 16\left( m-2\right)$ implies $10<m \leq 62$, and actually $m=50$ and $\lambda=36$ since $\lambda =\sqrt{\frac{m(m-1)}{2}}+1$, a contradiction. In the second case, 
$\theta =1$ otherwise $\left( m-2\right) \left( m-3\right) \leq m(m-1)/2$,
contrary to $m>10$. Then $\lambda
^{2}+\lambda =\left( m-2\right) \left( m-3\right) $ and $\lambda
^{2}-2\lambda +1=m(m-1)/2$. Then $\lambda = \frac{1}{6}\left(
m-2\right) \left( m-7\right) $, which substituted in $\lambda ^{2}+\lambda
=\left( m-2\right) \left( m-3\right) $, leads to no admissible solutions for 
$m>8$. Finally, in the remaining case, $2^{m/2-1}\leq \frac{1}{2}%
\binom{m}{m/2}\leq m(m-1)/2$ implies $10<m\leq 15$ but one of these values of $m$ fulfills $(\lambda -1)^{2}=m(m-1)/2$. Thus, (1) cannot occur. 

In case (2), one has $m\neq 3$ by \cite[A7.1]{Rib}. So $\lambda -1=11$ or $20$ by \cite[A8.1]{Rib}, a
contradiction. It is easy easy to see that (4) and (6) cannot occur, whereas $b_{1}=48$ and either $PSL_{2}(8)\trianglelefteq G^{\Sigma }\leq P\Gamma
L_{2}(8)$ or $G^{\Sigma }\cong Sp_{6}(2)$, or $PSU_{3}(3) \unlhd G^{\Sigma} \leq P\Gamma U_{3}(3)$ in (3), (7) or (9) , respectively. However, these groups are excluded since they do not have transitive permutation representations of degree $48$ by \cite{At}. The same argument rules out (7) and (8) since $HS \unlhd G^{\Sigma} \leq HS:Z_{2}$ and $PSL_{3}(3) \unlhd G^{\Sigma} \leq PGL_{3}(3)$ do not have a transitive permutation representation of degree $b_{1}=120$ or $168$, respectively, by \cite{At}. This completes the proof.
\end{proof}

\bigskip
\begin{theorem}\label{OnlyIaIc}
 Assume that Hypothesis \ref{hyp2} holds. Then the following hold:
 \begin{enumerate}
     \item  $\mathcal{D}_{1}$ is of type Ia or Ic;
     \item $\lambda$ divides $k_{1}$ and $\lambda >v_{0}$;
     \item $\eta_{1}=1$, $\eta_{0}=\eta$ and $\eta \mid \frac{v_{0}}{k_{0}}$;
     \item $k_{1}<r_{1}$, and if $\mathcal{D}_{1}$ is of type Ia and $\frac{r_{1}}{(r_{1},\lambda_{1})}\geq \lambda $ then $2k_{1}<r_{1}$.
 \end{enumerate}    
\end{theorem}

\begin{proof}
 Assertion (1) follows from Theorem \ref{Teo1} and Propositions \ref{LambigK=0}, \ref{LambdaDividesK1} and \ref{NotIIa}. Hence $\lambda$ divides $k_{1}$ and $\lambda>k_{0}$, which is (2). Furthermore, (3) follows from Lemma \ref{L3}(4) since $\mathcal{D}_{1}$ is not of type IIa or IIb.
Now, 
\begin{equation}\label{cuta}
\frac{r_{1}}{k_{1}}=\frac{(v_{1}-1)\lambda_{1}}{k_{1}(k_{1}-1)}=\frac{A\frac{v_{0}-1}{k_{0}-1}}{A\frac{v_{0}}{k_{0}}\left(A\frac{v_{0}}{k_{0}}+1\right)}\cdot \frac{v_{0}}{k_{0}} \cdot \frac{v_{0}\lambda}{k_{0}\eta}%
\geq \frac{\frac{v_{0}-1}{k_{0}-1}}{A\left(\frac{v_{0}}{k_{0}}+1\right)} \cdot \frac{v_{0}\lambda}{k_{0}\eta}\text{,}
\end{equation}
and hence
\begin{equation}\label{k1<r1}
\frac{r_{1}}{k_{1}}>\frac{(v_{0}-1)v_{0}}{(k_{0}-1)(v_{0}+k_{0})\eta}\geq \frac{(v_{0}-1)k_{0}}{(k_{0}-1)(v_{0}+k_{0})}\text{.}
\end{equation}
by Proposition \ref{P2}(3)--(4). Note that, $v_{0}>k_{0}^{2}$ by Theorem \ref{Teo1} since $\mathcal{D}_{1}$ is of type Ia or Ic by (1) and $\lambda>3$. Then $\frac{(v_{0}-1)k_{0}}{(k_{0}-1)(v_{0}+k_{0})}>1$, and hence $k_{1}<r_{1}$ by (\ref{k1<r1}). This proves the first part of (4).

Assume that $\mathcal{D}_{1}$ is of type Ia and and $\frac{r_{1}}{(r_{1},\lambda_{1})}\geq \lambda $. Firstly, note that $\frac{r_{1}}{(r_{1},\lambda_{1})}\neq \lambda$ since $\frac{r_{1}}{(r_{1},\lambda_{1})} \mid v_{1}-1$ by Lemma \ref{PP}, whereas $\lambda \nmid v_{1}-1$ by Theorem \ref{Teo1}, being $\mathcal{D}_{1}$  of type Ia. Then $\frac{r_{1}}{(r_{1},\lambda_{1})}> \lambda$ since $\frac{r_{1}}{(r_{1},\lambda_{1})}\geq \lambda $, and hence $\frac{v_{0}-1}{k_{0}-1}> \lambda $ by (\ref{double}). Therefore, $v_{0} \geq \lambda+2$.

If $\eta$ is a proper divisor of $v_{0}/k_{0}$, then the first inequality in (\ref{k1<r1}) and the above argument imply $\frac{r_{1}}{k_{1}}>2\frac{(v_{0}-1)k_{0}}{(k_{0}-1)(v_{0}+k_{0})}>2$, which is the assertion in this case. It remains to consider the case $\eta=v_{0}/k_{0}$. By (\ref{cuta}), $r_{1}>2k_{1}$ is equivalent to 
\begin{equation}\label{smal}
 \frac{v_{0}-1}{k_{0}-1}\lambda > 2\left(A\frac{v_{0}}{k_{0}}+1 \right)\text{,}   
\end{equation}
 which certainly holds when either $k_{0}\geq 3$, or $k_{0}=2$, $\lambda \geq A+1$ and $v_{0}>\lambda+2$ since $\lambda \geq (k_{0}-1)A+1$ by Proposition \ref{P2}(5) and $k_{0} \mid v_{0}$. The remaining case $k_{0}=2$, $\lambda \geq A+1$ and $v_{0}=\lambda+2$ cannot occur. Indeed, $k_{0}\mid v_{0}$, $k_{0}=2$ and $\lambda$ a prime force $\lambda=2$, whereas $\lambda>3$ by our assumption.
\end{proof}


\bigskip

\bigskip

\begin{proof}[Proof of Theorem \ref{Teo2}]
Assertions (1.a) and (1.b) follow from Theorem \ref{TeoFeb23}(1.a) and (1.b), respectively. Further, assertion (2.a) follows from Theorems \ref{TeoFeb23}(2.a)--(2.b) and \ref{OnlyIaIc}(1).

Since $v_{1}=A\frac{v_{0}-1}{k_{0}-1}+1$ and $k_{1}=A\frac{v_{0}}{k_{0}}+1$ by Proposition \ref{P2}(3), it follows that $v_{1}<2k_{1}$. This, together with $v_{1}=\left\vert G^{\Sigma }:G_{\Delta }^{\Sigma
}\right\vert $ and $k_{1}=\left\vert G_{B^{\Sigma }}^{\Sigma }:G_{\Delta
,B^{\Sigma }}^{\Sigma }\right\vert $, implies 
\begin{equation}\label{largedis}
\left\vert G^{\Sigma }:G_{\Delta }^{\Sigma }\right\vert < 2\left\vert
G_{B^{\Sigma }}^{\Sigma }:G_{\Delta ,B^{\Sigma }}^{\Sigma }\right\vert \text{.}
\end{equation}
Assume that $\left\vert G_{\Delta ,B^{\Sigma }}^{\Sigma }\right\vert >1$. Then $%
\left\vert G^{\Sigma }:G_{\Delta }^{\Sigma }\right\vert \leq \left\vert
G_{B^{\Sigma }}^{\Sigma }\right\vert$. Further, $\left\vert
G_{B^{\Sigma }}^{\Sigma }\right\vert< \left\vert G_{\Delta }^{\Sigma
}\right\vert $ since $k_{1}<r_{1}$ by Theorem \ref{OnlyIaIc}(4), and hence $\left\vert G^{\Sigma }\right\vert < \left\vert
G_{\Delta }^{\Sigma }\right\vert ^{2}$ in this case. 

Assume that $\left\vert G_{\Delta ,B^{\Sigma }}^{\Sigma }\right\vert =1$. If $\mathcal{D}_{1}$ is of type Ia, then 
\begin{equation*}
\left\vert G^{\Sigma
}\right\vert< 2\left\vert G_{B^{\Sigma
}}^{\Sigma }\right\vert \cdot\left\vert G_{\Delta }^{\Sigma } \right\vert=2k_{1}\cdot r_{1}<r_{1}^{2}=\left\vert G_{\Delta }^{\Sigma } \right\vert^{2} 
\end{equation*}
by (\ref{largedis}) and Theorem \ref{OnlyIaIc}(4) since $\frac{r_{1}}{(r_{1},\lambda_{1})}>\lambda$, hence $\left\vert G^{\Sigma }\right\vert < \left\vert
G_{\Delta }^{\Sigma }\right\vert ^{2}$. If $\mathcal{D}_{1}$ is of type Ic, then 
$$r_{1}\geq \frac{1}{2}\lambda^{2}(\lambda+1)>\frac{1}{2}(\lambda-1)(\lambda^{2}-2)=v_{1}$$  
by Theorem \ref{Teo1}, and again $\left\vert G^{\Sigma }\right\vert < \left\vert
G_{\Delta }^{\Sigma }\right\vert ^{2}$. Therefore (2.b) holds true.

It follows from Corollary \ref{C3}, Proposition \ref{P2}(4) and (\ref{double}) that $$\frac{r_{1}}{(r_{1},\lambda_{1})}=\frac{v_{0}-1}{k_{0}-1}>\lambda \geq A+1\text{.}$$
Hence, 
\begin{equation}\label{zora}
\left(\frac{r_{1}}{(r_{1},\lambda_{1})}\right)^{2}=\left(\frac{v_{0}-1}{k_{0}-1}\right)^{2}>\left(\frac{v_{0}-1}{k_{0}-1}\right)(A+1)>A\left(\frac{v_{0}-1}{k_{0}-1}\right)+1=v_{1}\text{,}    
\end{equation}
which implies (2c). This completes the proof. 
\end{proof}

\bigskip

\section{Reduction to the case $\frac{r_{1}}{(r_{1},\lambda_{1})}>\lambda $}
In this section, we prove that case (1) of Theorem \ref{Teo2} cannot occur. The existence of $\lambda$-elements in $G_{\Delta}^{\Delta}$ with $\lambda$ a large prime in comparison to $v_{0}$, allows us to use \cite[Theorem 1.1]{LS} to completely determine the possibilities for $G_{\Delta}^{\Delta}$, and from this derive that either $\mathcal{D}_{0} \cong AG_{n}(q)$ with all planes as blocks and $\lambda =\frac{q^{n-1}-1}{q-1}$, or $\mathcal{D}_{0} \cong PG_{n-1}(q)$ with all planes as blocks and $\lambda =\frac{q^{n-2}-1}{q-1}$. However, the parameters of such $2$-designs do not match with those otained for $\mathcal{D}_{0}$ in Theorem \ref{Teo1}. Therefore no cases occur, and hence only (2) of Theorem \ref{Teo2} is admissible. More precisely, we prove the following result:

\bigskip

\begin{theorem}\label{Teo3}
Assume that Hypothesis \ref{hyp2} holds. Then $\lambda \mid k_{1}$, $k_{0}<\lambda  < \frac{r_{1}}{(r_{1},\lambda_{1})}$ and the following hold:
\begin{enumerate}
\item $\mathcal{D}_{1}$ is a $2$-$(v_{1},k_{1},\lambda_{1})$ design of type Ia or Ic;
\item $G^{\Sigma}$ is a flag-transitive point-primitive automorphism group of $\mathcal{D}_{1}$ of affine or almost simple type; 
    \item $\left\vert G^{\Sigma }\right\vert < \left\vert G_{\Delta }^{\Sigma }\right\vert ^{2}$ for any $\Delta \in \Sigma$;
    \item $\frac{r_{1}}{(r_{1},\lambda_{1})}>v_{1}^{1/2}$.
    \item If $e_{1},...,e_{z}$ are the lengths of the point-$G^{\Sigma}_{\Delta}$-orbit on $\mathcal{D}_{1}$ distinct from $\{\Delta\}$, then
\begin{equation}\label{Scorpus}
\frac{r_{1}}{(r_{1},\lambda_{1})} \mid \left(e_{1},...,e_{z},v_{1}-1,\left\vert T_{\Delta} \right\vert \cdot \left\vert Out(T) \right\vert\right)\text{.}
\end{equation}    
\end{enumerate}    
\end{theorem}

\bigskip

We want to focus to case (1) of of Theorem \ref{Teo2}. Hence, we make the following hypothesis and we show that it leads to  no cases.

\bigskip

\begin{hypothesis}\label{hyp3}
The Hypothesis \ref{hyp2} holds and $\frac{r_{1}}{(r_{1},\lambda_{1})}<\lambda$.   
\end{hypothesis}

\bigskip

\begin{lemma}\label{distinction}
Assume that Hypothesis \ref{hyp3} holds. Then $\mathcal{D}_{0}$ is a $2$-$(v_{0},k_{0}, \lambda)$ design with $k_{0} \geq 3$ admitting $G_{\Delta }^{\Delta }$ as a flag-transitive point-primitive automorphism group. In particular, $G_{\Delta }^{\Delta }$ and $v_{0}$ are as in Table \ref{supertable}.
\begin{table}[h!]
\tiny
\caption{Admissible $G_{\Delta }^{\Delta }$ and $v_{0}$.}\label{supertable}
\begin{tabular}{llcc}
\hline
Line & $G_{\Delta }^{\Delta }$ & $v_{0}$ & Conditions on $\lambda $ \\
\hline
1 & \begin{tabular}{l}
    $ASL_{c}(s)\unlhd G_{\Delta }^{\Delta }\leq A\Gamma L_{c}(s)$ \\
    $c\geq 2$, $s=l^{h}$, $l$ prime, $h \geq 1$
  \end{tabular}
& $s^{c}$ & 
\begin{tabular}{l}
(\ref{fundamental}) holds \\ 
$\lambda \mid s^{i}-1$ \\ 
$c>i\geq c/2$ \\ 
$\lambda >s^{c/2}$
\end{tabular}\\
\hline
2& $A_{7}$ & $2^{4}$ & $7$ \\
\hline
3 & \begin{tabular}{l}
$PSL_{c}(s) \unlhd G_{\Delta }^{\Delta }\leq P\Gamma L_{c}(s)$\\
$c\geq 3$, $s=l^{h}$, $l$ prime, $h \geq 1$  
\end{tabular}
& $\frac{s^{c}-1}{s-1}$ &
\begin{tabular}{l}
(\ref{fundamental}) holds \\ 
$\lambda \mid s^{i}-1$ \\ 
$c>i\geq c/2$ \\ 
$\lambda >\left( \frac{s^{c}-1}{s-1}\right) ^{1/2}$
\end{tabular}\\
\hline
\end{tabular}
\end{table}

\end{lemma}

\begin{proof}
Assume that Hypothesis \ref{hyp3} holds. Then $\mathcal{D}_{0}$ is a $2$-$(v_{0},k_{0}, \lambda)$ design admitting $G_{\Delta }^{\Delta }$ as a flag-transitive point-primitive automorphism group, and $\mathcal{D}_{1}$ is of type Ia. Further $\lambda >\frac{v_{0}-1}{k_{0}-1}$ by (\ref{double}), and hence $v_{0} < \lambda(k_{0}-1)+1$.


It follows from Proposition \ref{P2}(4) that $v_{0}< \lambda^{2}+1$ by Proposition \ref{P2}(4) since $k_{0}>2$. Now, let $\zeta $ be any $\lambda $-element. Let $u$ and $f$ be the number of
cycles and fixed points of $\zeta $, respectively. Then $v_{0}=u\lambda +f$
with $1\leq u<\lambda $ since $v_{0} < \lambda ^{2}+1$, and $f\geq 2$\ since $\lambda $ does not divide $%
v_{0}(v_{0}-1)$, being $\mathcal{D}_{1}$ is of type Ia. Therefore, either $A_{v_{0}} \unlhd G_{\Delta}^{\Delta}$, or $( G_{\Delta}^{\Delta},v_{0},\lambda,f)$ is listed in \cite[Corollary 1.3]{LS}. In the former case, $(v_{0},k_{0},\lambda)=(15,7,3),(6,3,2),(10,6,5)$ or $(10,4,2)$ by \cite[Theorem 1]{ZCZ} since $\mathcal{D}_{0}$ is a $2$-$(v_{0},k_{0},\lambda)$ design, then either $k_{0} \mid v_{0}$ or $k_{0}-1 \mid v_{0}-1$, contrary to $\mathcal{D}_{1}$ being of type Ia. Hence, $G_{x}^{\Delta}$, $v_{0}$, $\lambda$  and $f$ are listed in \cite[Corollary 1.3]{LS}.

Suppose that $Soc(G_{\Delta }^{\Delta })$ is an elementary abelian $s$-group for some prime $s$. Then $G_{x}^{\Delta}$, $v_{0}$, $\lambda$  and $f$
are as in Table \ref{AffMinus} by \cite[Table 1]{LS}.
\begin{table}[h!]
\tiny
\caption{Admissible $G_{x}^{\Delta}$, $v_{0}$, $\lambda$  and $f$.} \label{AffMinus}
\begin{tabular}{lllll}
\hline
Line & $G_{x}^{\Delta }$ & $v_{0}=s^{c}$, $c\geq 2$ & $\lambda $ & $f$ \\
\hline
1 & $SL_{c/h}(s^{h})$ & $s^{c}$ & $\lambda \mid s^{i}-1$ for some $c>i>c/2$, $%
\lambda >s^{c/2}$ & $s^{c-i}$ \\ 
2 &  & $s$ & $\lambda =s$, $c/h=2,h=1$   & $\lambda $ \\ 
3 & $A_{7}$ & $16$ & $7$ & $2$\\
\hline
\end{tabular}%
\end{table}

The case as in Line 2 of Table \ref{AffMinus} is excluded since $\lambda
\mid v_{0}$, contrary to $\mathcal{D}_{1}$ being of type Ia. Finally, the cases is in Line 1 and 3 of Table \ref{AffMinus} leads to Lines 1 and 2 of Table \ref{supertable}, resepctively, since $\lambda \nmid v_{0}(v_{0}-1)$.

Assume that $Soc(G_{\Delta }^{\Delta })$ is a non-abelian simple group. Then 
$Soc(G_{\Delta }^{\Delta})$, $v_{0}$, $\lambda$  and $f$ are as in Table \ref{NASMinus} by \cite[%
Table 2]{LS}. 

\begin{table}[h!]
\tiny
\caption{Admissible $Soc(G_{\Delta }^{\Delta})$, $v_{0}$, $\lambda$  and $f$.} \label{NASMinus}
\begin{tabular}{lllll}
\hline
Line & $Soc(G_{\Delta }^{\Delta })$ & $v_{0}$ & $\lambda $ & $f$ \\ 
\hline
1 & $A_{m}$, $m \geq 5$ & $m(m-1)/2$ & $\lambda +2\leq m\leq \frac{1}{2}(3\lambda -1)$ & 
$\frac{1}{2}(\lambda ^{2}+\lambda -m(2\lambda +1-m))$ \\ 
2 & $PSL_{c}(s)$, $s=l^{h}$, & $\frac{s^{c}-1}{s-1}$ & $\lambda \mid s^{i}-1$ for some $%
c>i\geq c/2$, $\lambda >v_{0}^{1/2}$ & $\left\{ 
\begin{tabular}{cc}
$\frac{s^{c-i}-1}{s-1}$ & $c\geq 3$ \\ 
$2$ & $c=2$%
\end{tabular}%
\right. $ \\ 
3 & $M_{24}$ & $24$ & $5,7,11$ & $4,3,2$ \\ 
4 & $M_{23}$ & $23$ & $5,7$ & $3,2$ \\ 
5 & $M_{22}$ & $22$ & $5$ & $2$ \\ 
6 & $M_{12}$ & $12$ & $5$ & $2$ \\ 
7 & $M_{11}$ & $12$ & $5$ & $2$ \\
\hline
\end{tabular}%
\end{table}
We may use the same argument as in case $A_{v_{0}} \unlhd G_{\Delta}^{\Delta}$, involving \cite[Theorem 1]{ZCZ}, to show that either $m=6$, $m(m-1)/2=15$ and $\lambda=3$ for the case as in Line 1 of Table $\ref{NASMinus}$ since $k_{0}>2$ and $m \geq 5$. However, this case is clearly excluded since $\lambda>3$ by our assumption.

 The cases as in Line 3--6 of Table $\ref{NASMinus}$ are excluded since it contradicts $k_{0}\mid v_{0}$, $k_{0}-1 \mid v_{0}-1$ with $2<k_{0}<v_{0}$ (it follows from Lemma \ref{L1} that $k_{0}<v_{0}$).

In the case as in Line 2 of Table $\ref{NASMinus}$ with $c=2$, one has $v_{0}=l^{h}+1$ and $G_{\Delta }^{\Delta}$ acts point $2$-transitively on $\mathcal{D}_{0}$, $\lambda \mid l^{h}-1$ and $\lambda>(l^{h}+1)^{1/2}$. Further $k_{0}=l^{u}+1$ with $0<u<h$ since $k_{0}\mid v_{0}$, $k_{0}-1 \mid v_{0}-1$ and $2<k_{0}<v_{0}$, and hence $r_{0}=l^{h-u}\lambda$. We may identify the point-set of $\mathcal{D}_{0}$ with the projective line $PG_{1}(l^{h})$, and the actions of action of $PSL_{2}(l^{h}) \unlhd G_{\Delta }^{\Delta} \leq P\Gamma L_{2}(l^{h})$ on these two sets are equivalent.

Let $T \cong PSL_{2}(l^{h})$. Then $l^{u}(l^{u}+1)(\frac{l^{h}-1}{(2,l^{h}-1)\lambda}) \mid \left \vert T_{B} \right \vert $, where $B$ is any block of $\mathcal{D}_{0}$, and therefore $PSL_{2}(l^{u})\unlhd T_{B} \leq PGL_{2}(l^{u})$ and $\lambda \mid \frac{l^{h}-1}{(2,l^{h}-1)(l^{u}-1)}$ with $u \mid h$ and $0<u<h$. Now, $PSL_{2}(l^{u})$ has a unique orbit of length $l^{u}+1$ on $PG_{1}(l^{u})$ and this is a subline $PG_{1}(l^{u})$, hence $B \cong PG_{1}(l^{u})$. On the other hand, $PSL_{2}(l^{h})$ has  $1$ conjugacy classes of subgroups isomorphic to $PGL_{2}(l^{u})$, and for $l$ odd either $1$ or $2$ conjugacy classes of subgroups isomorphic to $PSL_{2}(l^{u})$ according as $h-u$ is odd or even, respectively. Moreover, these classes are fused in and these are fused in $PGL_{2}(l^{h})$. Thus the number  $b_{0}= \frac{l^{h}(l^{2h}-1)}{\theta l^{u}(l^{2u}-1)}$ with $\theta =1$ or $2$, and hence $\lambda=\frac{l^{h}-1}{\theta(l^{u}-1)}$. 
Since $k_{1}=A\frac{l^{h}+1}{l^{u}+1}+1$ by Proposition \ref{P2}(3), and $\lambda$ divides $k_{1}$, it follows that 
\begin{equation}\label{PantaRei}
 \frac{l^{h}-1}{\theta \left( l^{u}-1\right) } \mid A\frac{l^{h}+1}{l^{u}+1}+1\text{,}   
\end{equation}
and hence $\frac{l^{h}-1}{\theta \left( l^{u}-1\right) } \mid -A+\left(
l^{u}+1\right)$. 
If $A>l^{u}+1$ then $A=\alpha \frac{l^{h}-1}{\theta \left( l^{u}-1\right) }%
+l^{u}+1=\alpha \lambda +1$ for some $\alpha \geq 1$, contrary to $\lambda
\geq A+1$; if $A=l^{u}+1$ then $\frac{l^{h}-1}{\theta \left( l^{u}-1\right) }%
\mid 3$, whereas $\lambda >3$. Therefore $A<l^{u}+1$, then $%
(l^{h}+1)^{1/2}\leq \lambda = \frac{l^{h}-1}{\theta \left( l^{u}-1\right) }<l^{u}+1$
and hence $u=\frac{h}{2}$ and $\theta =1$ since $l>2$. So $l^{h/2}+1\mid
l^{h}+1$, as $k_{0}\mid v_{0}$, which is not the case.

Finally, the case as in Line 3 with $m\geq 3$ leads to Line 2 of Table \ref{supertable}. 
\end{proof}

\bigskip

Throughout the remainder of this section, unless differently specified, we adopt the same notation for $\mathcal{D}_{0}$, $G_{\Delta}^{\Delta}$: $\mathcal{D}_{0}=(\mathcal{P}_{0},\mathcal{B}_{0})$ is a $2$-$(v_{0},k_{0},\lambda )$ design (here $\lambda_{0}=\lambda$) admitting $\Gamma$ as a flag-transitive automorphism group, where $\Gamma$ denotes $G_{\Delta}^{\Delta}$. Here, $k_{0}\geq 3$ by Lemma \ref{distinction}, and $r_{0}=\frac{v_{0}-1}{k_{0}-1}\lambda$. Finally, we denote $Soc(G_{\Delta}^{\Delta})$ simply by $S$, and the set of blocks of $\mathcal{D}_{0}$ containing any two distinct points $x$ and $y$ by $\mathcal{B}_{0}(x,y)$

\bigskip

\bigskip
In sequel of this section, we will often use the following fact.
\bigskip

\begin{remark}
Since $k=k_{0}k_{1} \leq v \leq 2\lambda ^{2}(\lambda -1)$ by Hypothesis \ref{hyp1}, and $k_{1}=A\frac{v_{0}}{k_{0}}+1$ by Proposition \ref{P2}(3), it follows that
\begin{equation}\label{sem}
 v_{0}+k_{0} \leq  A v_{0}+k_{0}\leq 2\lambda ^{2}(\lambda -1) \text{.} 
\end{equation}
 Recall also that, $\lambda$ does not divide any of the integers $v_{0}(v_{0}-1)$, $v_{1}(v_{1}-1)$ or $v(v-1)$ since $\mathcal{D}_{0}$ is of type Ia.
\end{remark}

\bigskip

\subsection{The affine case}
In this section, we settle the case where $\Gamma$ is an affine group. Hence, $S$ is an elementary abelian $s$-group for some prime $s$. The point set of $\mathcal{D}_{0}$ can be
identified with $c$-dimensional $\mathbb{F}_{s}$-vector space $V$ in a way that $S$ acts as the translation group of $V$ and $\Gamma_{x} \leq GL_{c}(s)$, where $x$ denotes the zero vector of $V$. Further,  $H \unlhd \Gamma$ with either $H=S: SL_{c/h}(s^h)$ or $H=S: A_{7}$ for $(c,h,s)=(4,1,2)$ by Lemma \ref{distinction}. The symbol $H$ will have this fixed meaning throughout this section.  

\bigskip

\begin{lemma}\label{affsubspaces}
Assume that Hypothesis \ref{hyp1} holds. If $\lambda>k_{0}$, $\lambda>3$, $\lambda \nmid v_{0}(v_{0}-1)$ and $\Gamma$ is an affine group, then the following hold:
\begin{enumerate}
    \item the blocks of $\mathcal{D}_{0}$ are subspaces of $AG_{c}(s)$ of size greater
than $s$;
\item $\lambda$ divides the order of $H_{x}$ and does not divide the order of $\Gamma_{x}/H_{x}$.
\end{enumerate}

\end{lemma}

\begin{proof}
Firstly, we note that $k_{0}=s^{u}$ for some $u\mid c$ since $%
k_{0}\mid v_{0}$ and $k_{0}-1\mid v_{0}-1$. Now, let $B$ be any block of $\mathcal{D}_{0}$ incident with $x$. Then $B$ is split
into $T_{B}$-orbits of equal length $s^{j}$, $j\geq 0$, permuted
transitively by $\Gamma_{B}$ since $S_{B}\trianglelefteq \Gamma_{B}$. Thus $\left\vert
\Gamma_{B}:\Gamma_{x,B}S_{B}\right\vert =s^{u-j}$, and hence $\Gamma_{B}/S_{B}$ is
isomorphic to a subgroup $J$ of $\Gamma_{x}$ such that $\left\vert J\right\vert
=s^{u-j}\left\vert \Gamma_{x,B}\right\vert $. Also, $\Gamma_{B}/S_{B}$ contains an
isomorphic copy of $\Gamma_{x,B}$. Then $r_{0}=\left\vert \Gamma_{x}:J\right\vert s^{j-u}$
implies $\left\vert \Gamma_{x}:J\right\vert s^{u-j}=\frac{s^{c}-1}{s^{u}-1}%
\lambda $. If $u-j>0$, then $\lambda =s$ which is not the case since $%
\lambda \nmid v_{0}$. Thus $u=j$, and hence $B$ is a $\mathbb{F}_{s}$-subspace
of $V$, which is (1).

Suppose the contrary that $\lambda$ divides the order of $\Gamma_{x}/H_{x}$. Then $\lambda \mid \left(\frac{c}{h},s^{h}-1 \right )\cdot h$. Actually, $\lambda \mid h$ since $v_{0}=s^{c}$ and $\lambda$ does not divide $v_{0}-1$. It
follows from (\ref{sem}) that 
\begin{equation} \label{nBigger1}
s^{c}< 2\lambda ^{2}(\lambda -1)\leq 2\frac{c^{2}}{h^{2}}\left(\frac{c}{h}-1 \right) 
\end{equation}%
since $%
\lambda >3$. Then $h=1$ and $\left( v_{0},\lambda \right) =\left( 2^{5},5\right),(2^{7},7),(2^{10},5)$, or $(2^{11},11)$. Only $\left( v_{0},\lambda \right)=(2^{10},5)$ with $k_{0}=4$ is admissible since $3 \leq k_{0} < v_{0}$ and $\lambda>k_{0}$. Then $\mathcal{D}_{1}$ is not a symmetric $1$-design by Lemma \ref{L2bis}. Hence, $\mathcal{D}_{1}$ is a $2$-design by Theorem \ref{CamZie}. Then $A=1$ since $\lambda
\geq \left( k_{0}-1\right) A+1$ by Proposition \ref{P2}(4), and hence $k_{1}=2^{8}A+1=257$, contrary to $%
\lambda \mid k_{1}$. Thus $\lambda$ does not divide the order of $\Gamma_{0}/H_{0}$, and hence $\lambda$ divides the order of $H_{x}$ since $\lambda$ divides the order of $\Gamma_{x}$ by Lemma \ref{Fix}(1). This proves (2).
\end{proof}

\bigskip

\begin{proposition}\label{D0AffSpace}
Assume that Hypothesis \ref{hyp1} holds. If $\lambda>k_{0}$, $\lambda>3$, $\lambda \nmid v_{0}(v_{0}-1)$ If $\Gamma$ is an affine group, then the following hold: 
\begin{enumerate}
    \item $\mathcal{D}_{0} \cong AG_{c/h}(s^{h})$, $c/h \geq 3$, with all affine planes as blocks and $\lambda=\frac{%
s^{c-h}-1}{s^{h}-1}$;
\item Either $SL_{c/h}(s^{h})\unlhd \Gamma_{x}$ or $(c,h,s)=(4,1,2)$ and $A_{7}\unlhd \Gamma_{x}$.
\end{enumerate}   
\end{proposition}

\begin{proof}
The assertion is clearly  true if $(c,h,s)=(4,1,2)$ and $H_{x} \cong A_{7}$. Hence, assume that $H_{x} \cong SL_{c/h}(s^{h})$. Let $y$ be any non-zero vector of $V$. Since $\Gamma_{x,y}H_{x}/H_{x}\cong \Gamma_{x,y}/H_{x,y}$, it follows from Lemma \ref{affsubspaces}(2) that $\lambda $
does not divide the order of $\Gamma_{x,y}/H_{x,y}$ and hence the order of $%
\Gamma_{x,y}/H_{x,y}K$, where $K$ is the kernel of the action $\Gamma_{x,y}$ on $\mathcal{B}_{0}(x,y)$. Now, $\Gamma_{x,y}/H_{x,y}K\cong \left( \Gamma_{x,y}/K\right)
/\left( H_{x,y}K/K\right) $ with $H_{x,y}K/K\cong H_{x,y}/(H_{x,y}\cap K)$.
Then $\lambda $ divides the order of $H_{x,y}/(H_{x,y}\cap K)$ since $\Gamma_{x,y}/K$ acts transitively on $\mathcal{B}_{0}(x,y)$ with $\left\vert \mathcal{B}_{0}(x,y) \right\vert=\lambda$ by Lemma \ref{SameRank}(2) and $\lambda$ does not
divide the order of $\Gamma_{x,y}/H_{x,y}K$. In particular,  $H_{x,y}K/K$ acts transitively on $\mathcal{B}_{0}(x,y)$. Then one of the following holds again by Lemma \ref{SameRank}(2):
\begin{enumerate}
\item $Z_{\lambda }\trianglelefteq H_{x,y}/(H_{x,y}\cap K)\leq
AGL_{1}(\lambda )$;

\item $A_{\lambda }\trianglelefteq H_{x,y}/(H_{x,y}\cap K)\leq
S_{\lambda }$;

\item $PSL_{t}(w)\trianglelefteq H_{x,y}/(H_{x,y}\cap K)\leq P\Gamma
L_{w}(t)$, with $w$ prime and $\frac{w^{t}-1}{w-1}=\lambda $.

\item $H_{x,y}/(H_{x,y}\cap K)\cong PSL_{2}(11)$ and $\lambda =11$;

\item $H_{x,y}/(H_{x,y}\cap K)$ is one of the groups $M_{11}$ or $%
M_{23} $ and $\lambda $ is $11$ or $23$, respectively.
\end{enumerate}

On the other hand, $H_{x,y}\cong \lbrack s^{c-h}]:SL_{c/h-1}(s^{h})$ by \cite[Propositions 4.1.17(II)]{KL}. Hence, (1) and (5) cannot occur. In particular, $c/h \ge 3$. Moreover, either (2) holds with $(c/h,s^{h},\lambda)= (3,4,5)$ or $(4,2,5)$, or (3) holds with $w=s^{h}$, $t=c/h-1$ and $\lambda =\frac{s^{c-h}-1}{s^{h}-1}$, or (4) holds and $(c/h,s^{h},\lambda)=(3,11,11)$. The cases $(c/h,s^{h},\lambda)= (4,2,5), (3,11,11)$ are excluded since they contradicts $\lambda\nmid v_{0}(v_{0}-1)$. If $(c/h,s^{h},\lambda)= (3,4,5)$, then $k_{0}=4$ since $3 \leq k_{0}< \lambda$. Therefore $v_{0}=64$, $k_{0}=4$ and hence $\mathcal{D}_{1}$ is not a symmetric $1$-design with $k_{1}=v_{1}-1$ by Lemma \ref{L2bis}. Hence, $\mathcal{D}_{1}$ is a $2$-design. Therefore, $v_{1} \geq 22$ by Proposition \ref{P2}(3), and hence $v \geq 64 \cdot 21$, whereas $v \leq 200$  by \cite[Theorem 1]{DP}.

Finally, assume that $\lambda =\frac{s^{c-h}-1}{s^{h}-1}$. The group $H_{x,B}$ with $B$ is a block of $\mathcal{D}_{0}$ through $x$ contains a Sylow $s$-subgroup $W$ of $H_{x}$ since $r_{0}=\frac{s^{c}-1}{s^{u}-1}\lambda$ is coprime to $s$. Thus, $H_{x,B}$ lies in a maximal parabolic subgroup of $H_{x}$. Therefore, there is a positive integer $e$ such that  
\begin{equation}\label{GaussBin}
e{c/h\brack j}_{s^{h}} = \frac{s^{c}-1}{s^{u}-1} \cdot \frac{s^{c-h}-1}{s^{h}-1}\text{.} 
\end{equation}
Then either $j=1$ and $u \mid \left(c,c-h\right)$, or $j=2$ and $u \mid 2h$. Hence, either $u \leq h$ or $j=2$ and $u=2h$.

Assume that $u \leq h$. Now, the Sylow $s$-subgroup $W$ of $H_{x,B}$ fixes a point of $B\setminus\{x\}$, say $y$, since $k_{0}=s^{u}$ by Lemma \ref{affsubspaces}(1). Then $B$ is contained in the $1$-dimensional $\mathbb{F}_{s^{h}}$-subspace $\left\langle y \right\rangle$ of $V$ fixed pointwise by $W$ since $k_{0} \leq s^{h}$ and each non-trivial $W$-orbit on $V\setminus Fix(W)$ has length divisible by $s^{h}$. Then $B$ is fixed by $H_{x,y}$ since since $H_{x,y}$ fixes $\left\langle y \right\rangle$, but this contradicts the fact that $H_{x,y}$ acts transitively on the $\lambda$ elements of $\mathcal{B}_{0}(x,y)$. Therefore $u>h$, and hence $j=2$, $u=2h$ and $k_{0}=s^{2h}$. Hence, the group $H_{x,B}$ contains a subgroup $R \cong SL_{c/h-2}(s^{h})$ fixing a decomposition of $V=V_{1}\oplus V_{2}$ with $\dim_{\mathbb{F}_{s^{h}}} V_{1}=2$ and $\dim_{\mathbb{F}_{s^{h}}} V_{2}=c/h-2$. Moreover, $R$ fixes $V_{1}$ pointwise and acts transitively on $V_{2}\setminus\{x\}$. This forces $B=V_{1}$ since $x,y \in B \cap V_{1}$ and $k=s^{2h}$. Therefore, $\mathcal{D}_{0} \cong AG_{c/h}(s^{h})$, $c/h \geq 3$ with $\lambda=\frac{%
s^{c-h}-1}{s^{h}-1}$ prime and all affine planes as blocks. This completes the proof.
\end{proof}

\bigskip

\begin{corollary}\label{whatif}
Assume that Hypothesis \ref{hyp2} holds. If $\Gamma$ is an affine group and  $SL_{c/h}(s^{h})\unlhd \Gamma_{x}$ then $\mathcal{D}_{0} \cong AG_{c/h}(s^{h})$, $c/h \geq 3$ even, with all affine planes as blocks and $\lambda=\frac{s^{c-h}-1}{s^{h}-1}$.    
\end{corollary}

\begin{proof}
The proofs of Lemma \ref{affsubspaces} and Theorem \ref{D0AffSpace} are clearly independent from the assumption $\frac{r_{1}}{(r_{1},\lambda_{1})}<\lambda$. Thus they still work under the Hypothesis \ref{hyp2}, and hence the assertion follows.
\end{proof}

\bigskip

\subsection{The projective case}

In this section, we settle the case  $PSL_{c}(s) \unlhd G_{\Delta }^{\Delta }\leq P\Gamma L_{c}(s)$, $c\geq 3$, $s=l^{h}$, $l$ prime, $h \geq 1$, which is the remaining case of Lemma \ref{distinction} to be analyzed. 
\bigskip

\subsection{Elations}Let $\pi$ a projective plane (not necessarily Desarguesian), a collineation $\tau$ of $\pi$ that fixes a line $\ell$ pintwise and a preserves each line through a (unique) point $C$ of $\ell$ is called $(C,\ell)$-\emph{elation} of $\pi$. The point $C$ and the line $\ell$ are called the \emph{center} and the \emph{axis} of $\tau$. If $J$  is a group of collineations of $\pi$, the subgroup of $J$ consisting of the $(C,\ell)$-elations of $\pi$ contained in $J$ is denoted by $J(C,\ell)$. More information on groups of elations a of projective planes can be found in\cite[Section IV.4]{HP}.  
\bigskip

\begin{proposition}\label{D0ProjSpace}
Assume that Hypothesis \ref{hyp3} holds. If $PSL_{c}(s) \unlhd G_{\Delta }^{\Delta }\leq P\Gamma L_{c}(s)$, $c\geq 4$, then  $\mathcal{D}_{0} \cong PG_{c-1}(s)$ with all planes as blocks and $\lambda =\frac{s^{c-2}-1}{s-1}$.
\end{proposition}
\begin{proof}
Recall that $\mathcal{D}_{0}$ is a $2$-$\left(\frac{s^{c}-1}{s-1},k_{0},\lambda\right)$ design with $\lambda \mid s^{i}-1$, $c>i$ admitting $\Gamma$ as a flag-transitive automorphism group by Lemma \ref{distinction}. Moreover, $\lambda$ fulfills (\ref{fundamental}) by Theorem \ref{Teo1} since $\mathcal{D}_{1}$ is of type Ia by Theorem \ref{Teo2}(1.b). The group $\Gamma$ has two $2$-transitive permutation representations of degree $\frac{s^{c}-1}{s-1}$: one on the set of points of $PG_{c-1}(s)$, the other on the set of hyperplanes of $PG_{c-1}(s)$. The two conjugacy classes in $\Gamma$ of the point-stabilizers and hyperplane-stabilizers are fused by a polarity of $PG_{c-1}(s)$. Thus, we may identify the point set of $\mathcal{D}_{0}$ with that of $PG_{c-1}(s)$.

Let $x,y$ be any two distinct points of $\mathcal{D}_{0}$ and let $\ell$ be the line of $PG_{c-1}(s)$ containing them. Since $\left\vert Out(\Gamma)\right\vert=2(c,s-1)\cdot \log_{l}(s)$, $\lambda \nmid v_{0}(v_{0}-1)$, $\lambda>2$ and $(c,s-1)=\left(\frac{s^{c}-1}{s-1},s-1\right)$, it follows from (\ref{sem}) that
\begin{equation} \label{ProjnBigger}
\frac{s^{c}-1}{s-1}< 2\cdot (\log_{l}s-1) \cdot \log_{l}^{2}s 
\end{equation}
However, (\ref{ProjnBigger}) admits no solutions for $c \geq 3$.
Thus $\lambda$ does not divide the order of $\Out(\Gamma)$, and hence it does not divide the order of $\Gamma_{x,y}/S_{x,y}$ since $\left\vert \Gamma_{x,y}/S_{x,y} \right\vert \mid \left\vert \Out(\Gamma)\right\vert$. Consequently, $\lambda$ does not divide the order of $%
\Gamma_{x,y}/S_{x,y}K$, where $K$ is the kernel of the action $\Gamma_{x,y}$ on $\mathcal{B}_{0}(x,y)$. Now, $\Gamma_{x,y}/S_{x,y}K\cong \left( \Gamma_{x,y}/K\right)
/\left( S_{x,y}K/K\right) $ with $S_{x,y}K/K\cong S_{x,y}/(H_{x,y}\cap K)$.
Then $\lambda $ divides the order of $S_{x,y}/(S_{x,y}\cap K)$ since $\Gamma_{x,y}/K$ acts transitively on $\mathcal{B}_{0}(x,y)$ with $\left\vert \mathcal{B}_{0}(x,y) \right\vert=\lambda$ by Lemma \ref{SameRank}(2) and $\lambda$ does not
divide the order of $\Gamma_{x,y}/S_{x,y}K$. Then one of the following holds:
\begin{enumerate}
\item $Z_{\lambda }\trianglelefteq S_{x,y}/(S_{x,y}\cap K)\leq
AGL_{1}(\lambda )$;

\item $A_{\lambda }\trianglelefteq S_{x,y}/(S_{x,y}\cap K)\leq
S_{\lambda }$;

\item $PSL_{t}(w)\trianglelefteq S_{x,y}/(S_{x,y}\cap K)\leq P\Gamma
L_{t}(w)$, with $t$ prime and $\frac{w^{t}-1}{w-1}=\lambda $.

\item $S_{x,y}/(S_{x,y}\cap K)\cong PSL_{2}(11)$ and $\lambda =11$;

\item $S_{x,y}/(S_{x,y}\cap K)$ is one of the groups $M_{11}$ or $%
M_{23} $ and $\lambda $ is $11$ or $23$, respectively.
\end{enumerate}
If (1) holds, then $\lambda \mid s-1$ y \cite[Propositions 4.1.17(II)]{KL}, whereas $\lambda \mid s^{i}-1$ with $i>s/2$. Hence, (1) is ruleed out. Now, the socle of the almost simple quotient groups of $S_{x,y}$ is $PSL_{c-2}(s)$ again by \cite[Propositions 4.1.17(II)]{KL}. Then either (2) holds with $c=3$, $s=4,5$ and $\lambda =5$, or (3) holds with $w=s$, $t=c-2$ and $\lambda=\frac{s^{c-2}-1}{s-1}$, or (5) holds with $s=\lambda=11$. Case (2) with $s=5$ and case (4) are excluded since they contradict $\lambda \nmid v_{0}-1$. Further, (\ref{sem}) implies $4^{c}<601$ in case (4) with $s=4$. Therefore $c=3$ or $4$, but both cases are ruled out since they contradict $\lambda \nmid v_{0}(v_{0}-1)$. 

 Finally, assume that $\lambda=\frac{s^{c-2}-1}{s-1}$ with $c-2$ a prime number. Let $B \in \mathcal{B}_{0}(x,y)$ and set $\theta=\left\vert B \cap \ell \right\vert$, where $\ell$ is the line of $PG_{c-1}(s)$ containing both $x$ and $y$. Let us count the pairs $\left(z,C\right)$ with $C \in \mathcal{B}_{0}(x,y)$ and $z$ a point of $PG_{c-1}(s)\setminus \ell$ such that $z \in C\setminus\{x,y\}$. Since $\Gamma_{x,y}$ acts transitively on $\mathcal{B}_{0}(x,y)$ and on $PG_{c-1}(s)\setminus \ell$, it follows that $k_{0}-\theta$ is the number of points lying in $PG_{c-1}(s)\setminus \ell$ contained in any element of $\mathcal{B}_{0}(x,y)$, hence
 \begin{equation}\label{offL}
\lambda (k_{0}-\theta)=s^{2}\frac{s^{c-2}-1}{s-1}\alpha\text{,}     
 \end{equation}
where $\alpha$ is the constant number of elements of $\mathcal{B}_{0}(x,y)$ containing any fixed point of $PG_{c-1}(s)\setminus \ell$. Then $s^{2} \mid k_{0}-\theta$ since $\lambda=\frac{s^{c-2}-1}{s-1}$.  

If $\theta=2$ then $s^{2}\mid k_{0}-2$; if $\theta >2$, let us count the pairs $\left(E,z^{\prime\prime }\right)$ with $E \in \mathcal{B}_{0}(x,y)$ and $z^{\prime \prime } \in \ell \setminus \{x,y\}$ such that $z^{\prime \prime } \in E$. It follows that $(s-1)\beta=\lambda (\theta-2)$, where $\beta$ is the constant number of elements of $\mathcal{B}_{0}(x,y)$ containing a fixed point of $\ell \setminus\{x,y\}$, since $\Gamma_{x,y}$ acts transitively on $\mathcal{B}_{0}(x,y)$ and on $\ell\setminus\{x,y\}$. Clearly, $\beta \leq \lambda $. Since $\lambda=\frac{s^{c-2}-1}{s-1}$ with $c-2$ a prime number, it follows that $\left(\lambda,s-1\right)=1$. Thus $\beta=\lambda$, and hence $\theta=s+1$. Thus, we obtain $s^{2}\mid k_{0}-s-1$. Therefore, we have proven that either $\theta=2$ and $s^{2}\mid k_{0}-2$, or $\theta=s+1$ and $s^{2} \mid k_{0}-s-1$. 

If $\ell^{\prime}$ is any other line of $PG_{c-1}(s)$ such that $ B \cap\ell^{\prime}=\{x^{\prime},y^{\prime}\}$, we may repeat the above argument with $x^{\prime},y^{\prime}$, $\ell^{\prime}$ and $\theta^{\prime}$ in the role of $x,y$, $\ell$ and $\theta$, thus obtaining $\theta=2$ and $s^{2}\mid k_{0}-2$, or $\theta=s+1$ and $s^{2} \mid k_{0}-s-1$. If $\theta^{\prime} \neq \theta$, then $s^{2}\mid k_{0}-2$ and  $s^{2} \mid k_{0}-s-1$, a contradiction. Thus, one of the following holds:
\begin{enumerate}
    \item[(i)] each line of $PG_{c-1}(s)$ intersects $B$ in at most $2$ points, and hence $B$ is a $k_{0}$-cap of $PG_{c-1}(s)$ (e.g. see \cite[Section 27]{HT} for a definition of cap);
    \item[(ii)] each line of $PG_{c-1}(s)$ intersecting $B$ in at least two distinct points is contained in $B$, and hence B a $(j-1)$-dimensional subspace of $PG_{c-1}(s)$ for some integer $j \geq 2$. 
\end{enumerate}
Assume that (i) occurs. The group $H_{x,y}$ contains a subgroup $X:R$ with $X$ elementary abelian of order $s^{2(c-1)}$ containing $X$ and $R\cong SL_{c-2}(s)$.The group $X:R$
fixes $\ell $ pointwise, $R$ permutes $2$-transitively the pencil $\mathcal{F%
}$ of planes of $PG_{c-1}(s)$ containing $\ell $, and for any $\pi \in 
\mathcal{F}$, and the group $X$ induces the full $(\ell ,\ell )$-elation
group of $\pi $. Note that, $\left\vert X:X_{B}\right\vert =s$ since $\Gamma
_{x}$ acts transitively on the set of blocks of $\mathcal{D}_{0}$ containing 
$x$, 
\[
r_{0}=\frac{s\left( s^{c-1}-1\right) }{\left( s-1\right) (k_{0}-1)}\cdot 
\frac{s^{c-2}-1}{s-1} 
\]
and $s^{2}\mid k_{0}-2$. Therefore $\left\vert X:X_{B}\right\vert \mid s$,
and hence $\left\vert X:X_{B}X_{(\pi )}\right\vert \mid s$. Thus $\left\vert
X^{\pi }:E\right\vert \mid s$ with $E=\left( X_{B}\right) ^{\pi }$ a
subgroup of the full $(\ell ,\ell )$-elation group $X_{\pi }^{\pi }$ of $\pi 
$. In particular, $s\mid \left\vert E\right\vert $. Since $\mathcal{F}$
provides a partition of $PG_{c-1}(s)\setminus \ell $ and $k_{0}\geq 3$, we
may choose $\pi $ such that $\left\vert \pi \cap B\right\vert \geq 3$. Thus $%
E$ is \ group of $(\ell ,\ell )$-elations of $\pi $ of order divisible by $s$%
, and $\pi \cap B$ is an $E$-invariant arc of $\pi $.

Let $w\in \pi \cap B$ be such that $w\neq x,y$, then $w\notin \ell $ since $%
\pi \cap B$ is an arc of $\pi $. If $E(x,\ell )\neq 1$, let $m$ be the line
of $\pi $ containing $x$ and $w$, and let $\tau \in E$ be any nontrivial $%
(x,\ell )$-elation of $\pi $ contained in $\ell $. Such a $\tau $ does exist
since $E$ is the $(\ell ,\ell )$ elation group of $\pi $. Thus $w^{\tau }\in
m\setminus \left\{ x,w\right\} $. Moreover, $w^{\tau }\in B$ since $E=X_{\pi
}^{\pi }$ and $X_{\pi }\leq H_{x,y,B}$. So $\left\{ x,w,w^{\tau }\right\}
\subseteq \pi \cap m$ with $\left\vert \left\{ x,w,w^{\tau }\right\}
\right\vert =3$, whereas $B$ is a $k_{0}$-cap. Thus $E(x,\ell )=1$, and
similarly $E(y,\ell )=1$.

Let $u\in \ell \setminus \left\{ x,y\right\} $ such that $E(u,\ell )\neq 1$.
Clearly, $w^{E(u,\ell )}\subset \pi \cap B$ with $w^{E(u,\ell )}$ is
contained in a line $n$ of $\pi $ and $\left\vert w^{E(u,\ell )}\right\vert
=\left\vert E(u,\ell )\right\vert $. Therefore $\left\vert E(u,\ell
)\right\vert =2$, and hence $s$ is even. Then $k_{0}$ is even since $%
s^{2}\mid k_{0}-2$. However, this is impossible since $k_{0}\mid v_{0}$ and $%
v_{0}=\frac{s^{c}-1}{s-1}$.

Assume that (ii) occurs. Then 
\begin{eqnarray*}
\frac{s^{c-2}-1}{s-1}=\lambda={c-2\brack j-2}_{s}\text{,}     
\end{eqnarray*}
and hence $j=3$ and $c \geq 4$, since $\Gamma$ acts transitively on the set of all $(j-1)$-dimensional subspaces of $PG_{c-1}(s)$ and $\lambda$ is a prime and $j \geq2$.
\end{proof}

\bigskip
\begin{remark}\label{oss}
The proof of Proposition \ref{D0ProjSpace} still works if we replace the assumption $\lambda> \frac{r_{1}}{(r_{1},\lambda_{1})}$ with $\mathcal{D}_{1}$ being of type Ia. Indeed, if $\mathcal{D}_{1}$ is of type Ia, then $\lambda \mid s^{i}-1$, $c>i$ and $\lambda$ fulfills (\ref{fundamental}) by Theorem \ref{Teo1}.      
\end{remark}

\bigskip

Summarizing the previous results, we obtain the following theorem.

\bigskip
\bigskip

\begin{theorem}\label{D0Determined}
Assume that Hypothesis \ref{hyp3} holds. Then one of the following holds:
\begin{enumerate}
    \item $\mathcal{D}_{0} \cong AG_{c/h}(s^{h})$, $c/h \geq 3$, with all affine planes as blocks and $\lambda=\frac{%
s^{c-h}-1}{s^{h}-1}$, and either $SL_{c/h}(s^{h})\unlhd \Gamma_{x}$ or $(c,h,s)=(4,1,2)$ and $A_{7}\unlhd \Gamma_{x}$.
    \item $\mathcal{D}_{0} \cong PG_{c-1}(s)$, $c \geq 4$, with all planes as blocks and $\lambda =\frac{s^{c-2}-1}{s-1}$, and $PSL_{c}(s) \unlhd G_{\Delta }^{\Delta }\leq P\Gamma L_{c}(s)$, $c\geq 3$.
\end{enumerate}
\end{theorem}

\begin{proof}
The assertion immediately follows from Propositions \ref{D0AffSpace} and \ref{D0ProjSpace}.    
\end{proof}

\bigskip

Finally, we are in position to prove Theorem \ref{Teo3}.

\bigskip

\begin{proof}[Proof of Theorem \ref{Teo3}]
Assume that $\lambda>\frac{r_{1}}{(r_{1},\lambda_{1})}$. Then $\mathcal{D}_{1}$ is of type Ia and $\mathcal{D}_{0}$ is a $2$-$(v_{0},k_{0},\lambda)$ design admitting $G_{\Delta}^{\Delta}$ as a flag-transitive automorphism group by Theorem \ref{Teo2}(1). Further Hypothesis \ref{hyp3} holds, and hence $\mathcal{D}_{0}$ is as in Theorem \ref{D0Determined}.

Assume that $\mathcal{D}_{0} \cong AG_{c/h}(s^{h})$, $c/h \geq 3$, with all affine planes as blocks and $\lambda=\frac{s^{c-h}-1}{s^{h}-1}$. Then $\lambda \mid As^{c}+s^{2h}$ since $\lambda \mid k_{1}$ by Theorem \ref{Teo1} and $k_{1}=A\frac{v_{0}}{k_{0}}+1$ by Proposition \ref{P2}(3). Therefore $\lambda \mid As^{h}(s^{c-h}-1)+(A+s^{h})s^{h}$, and hence $\lambda \mid A+s^{h}$ since $\lambda=\frac{s^{c-h}-1}{s^{h}-1}$ is a prime number. However, this is impossible since $\lambda \geq A(s^{2h}-1)+1$ with $A \geq 1$ by Proposition \ref{P2}(4).

Assume that $\mathcal{D}_{0} \cong PG_{c-1}(s)$, $c \geq 4$, with all planes as blocks and $\lambda =\frac{s^{c-2}-1}{s-1}$. Then $\lambda \mid A\frac{s^{c}-1}{s-1}+(s^2+s+1)$ since $\lambda \mid k$, and hence $\lambda \mid As^{2}\frac{s^{c-2}-1}{s-1}+A(s+1)+(s^2+s+1)$. Therefore, $\lambda \mid A(s+1)+(s^2+s+1)$ since $\lambda =\frac{s^{c-2}-1}{s-1}$. On the other hand, $\lambda \geq A(s^{2}+s)+1$ by Proposition \ref{P2}(4). Thus $A(s^{2}+s)+1 \leq A(s+1)+(s^2+s+1)$, which implies either $A=1$ and $\frac{s^{c-2}-1}{s-1} \mid s^2+2s+2$, or $A=s=2$ and $2^{c-2}-1=\lambda=13$. Both cases are clearly impossible since $s \geq 2$ and $c \geq 4$. Thus, we have proven that $\lambda>\frac{r_{1}}{(r_{1},\lambda_{1})}$ leads to a contradiction. Therefore $\lambda \leq \frac{r_{1}}{(r_{1},\lambda_{1})}$, and hence assertions (1)--(4) follow from (2.a)--(2.c) of Theorem \ref{Teo2}. Finally, assertion (5) follows from Lemma \ref{PP}. This completes the proof.
\end{proof}

\bigskip

\bigskip

\section{Further Reductions}
In this section, we collect some additional constraints on the pairs $(\mathcal{D}_{1},G^{\Sigma})$ that will play a central role in the completion of Theorem \ref{main}. More precisely, we show that either the action of $G$ on $\Sigma$ is faithful, or $Soc(G_{\Delta }^{\Delta })$ is an elementary abelian $l$-group for some
prime $l$ acting point-regularly on $\mathcal{D}_{0}$, and $\mathcal{D}_{1}$ is of type Ia with $\eta = \frac{v_{0}}{k_{0}}$. Also, in the second case, we prove that $\mathcal{D}_{0} \cong PG_{c-1}(s)$ with all planes as blocks and $\lambda =\frac{s^{c-2}-1}{s-1}$ when $PSL_{c}(s)\trianglelefteq G_{\Delta }^{\Delta }\leq P\Gamma L_{c}(s)$, $%
c\geq 5$, and $v_{0}\leq s\frac{s^{c}-1}{s-1}\frac{s^{c-1}-1}{s-1}$.

\bigskip

\begin{lemma}\label{quasiprimitivity}
If $G_{(\Sigma)}\neq 1$, then the following hold:
\begin{enumerate}
\item  $Soc(G_{\Delta }^{\Delta })$ is an elementary abelian $l$-group for some
prime $l$ acting point-regularly on $\mathcal{D}_{0}$;
\item $\mathcal{D}_{1}$ is of type Ia with $\eta = \frac{v_{0}}{k_{0}}$, $k_{1}=\lambda$, $r_{1}=\frac{v_{1}-1}{\lambda-1}\lambda
$ and $b_{1}=\frac{v_{1}(v_{1}-1)}{\lambda-1}$.
\end{enumerate}
\end{lemma}

\begin{proof}
Assume that $G_{(\Sigma)}\neq 1$. Clearly, $G_{(\Sigma)}\leq G_{\Delta}$ for each $\Delta \in \Sigma$. If there is $\Delta^{\prime} \in \Sigma$ such that $G_{(\Sigma)}\leq G_{(\Delta^{\prime})}$, then $G_{(\Sigma)}\leq G_{(\Delta)}$ since $G_{(\Sigma)}\unlhd G$ and $G$ acts transitively on $\Sigma$. This forces $G_{(\Sigma)}=1$, contrary to our assumption. Thus, $G_{(\Sigma )}^{\Delta }\neq 1$ for each $\Delta \in \Sigma$, and hence $Soc(G_{\Delta }^{\Delta }) \unlhd G_{(\Sigma )}^{\Delta }$ for each $\Delta \in \Sigma$ by Theorem \cite[Theorem 4.3B]{DM}. Further,
\begin{equation}\label{trenu}
\left\vert (B\cap \Delta )^{G_{(\Sigma )}^{\Delta }}\right\vert =\left\vert
(B\cap \Delta )^{G(\Delta )G(\Sigma )}\right\vert =\left\vert (B\cap \Delta
)^{G(\Sigma )}\right\vert \leq \left\vert B^{G_{(\Sigma )}}\right\vert 
\end{equation}
since $G_{(\Sigma),B}\leq G_{(\Sigma),B\cap \Delta}$. On the other hand, if $C\in B^{G_{(\Sigma )}}$ then $C\cap \Delta ^{\prime
}=B\cap \Delta ^{\prime }$ for each $\Delta ^{\prime }\in \Sigma $. Hence, $%
C^{\Sigma }=B^{\Sigma }$ and so $\left\vert B^{G_{(\Sigma )}}\right\vert
\leq \eta $. Then  $\left\vert (B\cap \Delta )^{G_{(\Sigma )}^{\Delta }}\right\vert \leq \eta$ by (\ref{trenu}).

Let $S=Soc(G_{\Delta }^{\Delta })$, then $S$ is either an
elementary abelian group for some prime $l$ or a non-abelian simple group. Indeed, the previous assertion follows from the $2$-transitivity of $G_{\Delta}^{\Delta}$ for $k_{0}=2$, from \cite[Main Theorem]{BDD} for $k_{0}\geq 3$ and $\mu=\lambda$, and from \cite[Theorem 1]{ZC} for $k_{0}\geq 3$ and $\mu=1$. 

Assume $S$ is non-abelian simple. Then 
\[
\left\vert (B\cap \Delta )^{G_{(\Sigma )}^{\Delta }}\right\vert =\frac{%
\left\vert (B\cap \Delta )^{S}\right\vert }{h}
\]
with $h_{0}=\left\vert \left( G_{(\Sigma )}^{\Delta }\right) _{B\cap \Delta
}:S_{B\cap \Delta }\right\vert $ a divisor of the
order of $\left\vert Out(S)\right\vert $ . Further, 
$\left\vert (B\cap \Delta )^{S}\right\vert =\frac{%
b_{0}}{h_{1}}$ with $h_{1}=\left\vert \left( G_{\Delta }^{\Delta }\right)
_{B\cap \Delta }:S_{B\cap \Delta }\right\vert $ since $S$ is normal in $G_{\Delta }^{\Delta }$, and the latter acts block-transitively on $\mathcal{D}_{0}$.
Therefore,%
\[
\frac{b_{0}}{h_{0}h_{1}}=\left\vert (B\cap \Delta )^{G_{(\Sigma )}^{\Delta
}}\right\vert \leq \eta \leq 
\frac{v_{0}}{k_{0}}
\]%
Thus, $r_{0}\leq h$ with $h=h_{0}h_{1}$ since $b_{0}=v_{0}r_{0}/k_{0}$ and 
\[
v_{0}-1<\frac{v_{0}-1}{k_{0}-1}\lambda = r_{0}\leq \left\vert
Out(S)\right\vert ^{2}
\]%
since $\lambda >k_{0}-1$. Then 
\begin{equation}\label{twotimes}
P(S)\leq v_{0}\leq \left\vert
Out(S)\right\vert ^{2}\text{,}    
\end{equation}
where $P(S)$ denotes the minimal degree of the nontrivial transitive permutation representations of $S$. Then $S$ is neither $A_{m}$ with $m\geq5$ and $m \neq 6$, nor $S$ is sporadic by \cite{At} since $P(S)\geq 5$ and $\left\vert Out(S)\right\vert =2$. Finally, it is a routine exercise showing that $S \cong PSL_{2}(9)$ or $PSL_{3}(4)$  by using \cite{Va1,Va2,Va3} and \cite[Theorem 5.2.2]{KL} according as $S$ is a simple exceptional group of Lie type or a simple classical group, respectively. Then either $S \cong PSL_{2}(9)$, $\lambda=5$ and $v_{0}\leq 16$, or $S \cong PSL_{3}(4)$, $\lambda =5$ or $7$ and $v_{0} \leq 144$. Therefore $v_{0}=6,10,15$ or $21,56,102$, respectively, by \cite{At}. Then $\lambda \mid v_{0}(v_{0}-1)$ in both cases, and hence $\mathcal{D}_{1}$ is of type Ic by Theorem \ref{Teo3}. However, this impossible since $v_{0}=16$ or $29$ according as $\lambda =5$ or $7$, respectively, by Theorem \ref{Teo1}.

Assume that $S$ is an elementary abelian group for some
prime $l$. Then
\begin{equation}
\eta \geq \left\vert B^{G_{(\Sigma )}}\right\vert =\left\vert G_{(\Sigma
)}:G_{(\Sigma ),B}\right\vert \geq \left\vert G_{(\Sigma )}:G_{(\Sigma
),B\cap \Delta }\right\vert   \label{zdob}
\end{equation}%
since $G_{(\Sigma ),B}\leq G_{(\Sigma ),B\cap \Delta }$. Clearly, $%
G_{(\Sigma ),B\cap \Delta }\leq G_{(\Sigma )}$ implies $G_{(\Delta )}\cap
G_{(\Sigma ),B\cap \Delta }\leq G_{(\Delta )}\cap G_{(\Sigma )}$.
Conversely, if $\beta \in G_{(\Delta )}\cap G_{(\Sigma )}$ then $\beta $ 
fixes $\Delta $ pointwise and hence $\beta $ preserves $B\cap \Delta $. Thus 
$\beta \in G_{B\cap \Delta }$, and hence $\beta \in G_{(\Delta )}\cap
G_{(\Sigma ),B\cap \Delta }$ since $G_{(\Sigma ),B\cap \Delta }=G_{(\Sigma
)}\cap G_{B\cap \Delta }$. Thus $G_{(\Delta )}\cap G_{(\Sigma )}=G_{(\Delta )}\cap
G_{(\Sigma ),B\cap \Delta }$, and hence%
\begin{equation}\label{zdub}
\left\vert G_{(\Sigma )}:G_{(\Sigma ),B\cap \Delta }\right\vert =\left\vert
G_{(\Delta )}G_{(\Sigma )}:G_{(\Delta )}G_{(\Sigma ),B\cap \Delta }\right\vert
=\left\vert G_{(\Sigma )}^{\Delta }:G_{(\Sigma ),B\cap \Delta }^{\Delta
}\right\vert \geq \left\vert S:S_{B\cap \Delta }\right\vert
\end{equation}%
since $G_{(\Sigma ),B\cap \Delta }^{\Delta}=\left(G_{(\Sigma )}^{\Delta
}\right)_{B\cap \Delta}$ and $S \unlhd G_{(\Delta )}^{\Sigma}$.
Note that, the blocks of $\mathcal{D}_{0}$ are subspaces of $AG_{d}(p)$. Indeed, the proof of Lemma \ref{affsubspaces}(1) does not make use of the assumption $\lambda> \frac{v_{0}-1}{k_{0}-1}$, and hence it still works under the Hypothesis \ref{hyp2}, but it makes use of the fact $\lambda \nmid v_{0}$, which we may still assume it since $\mathcal{D}$ is of type Ia or Ic by Theorem \ref{Teo3}.

Then $\left\vert S:S_{B\cap \Delta }\right\vert=\frac{v_{0}}{k_{0}}$ since $S$ is the
translation group of $AG_{d}(p)$. Combining (\ref{zdob}) with (\ref%
{zdub}), and bearing in mind that $\eta \mid \frac{v_{0}}{k_{0}}$, we obtain 
$\eta =\frac{v_{0}}{k_{0}}$. Then $r_{1}=\frac{v_{0}-1}{k_{0}-1}\lambda
$ and $b_{1}k_{1}=v_{0}\frac{v_{0}-1}{k_{0}-1}\lambda $, and hence $%
b_{1}\frac{k_{1}}{\lambda }=v_{1}\frac{v_{0}-1}{k_{0}-1}$. Then $(k_{1}-1)\cdot \frac{k_{1}}{\lambda} \cdot b_{1}= v_{1}\cdot (v_{1}-1)$ by Lemma \ref{quasiprimitivity}(2). If $\frac{k_{1}}{\lambda} \geq 2$, then $2(k_{1}-1) \leq v_{1}-1$ since $b_{1} \geq v_{1}$. So $2k_{1}-1<v_{1}<2k_{1}$, which is a contradiction. Thus $k_{1}=\lambda$, and hence $\mathcal{D}_{1}$ is of type Ia by Theorem \ref{Teo1}.
\end{proof}

\bigskip

\begin{proposition}\label{moreover}
Assume that Hypothesis \ref{hyp2} holds. If $PSL_{c}(s)\trianglelefteq G_{\Delta }^{\Delta }\leq P\Gamma L_{c}(s)$, $%
c\geq 5$, and $v_{0}\leq s\frac{s^{c}-1}{s-1}\frac{s^{c-1}-1}{s-1}$, then $\mathcal{D}_{0} \cong PG_{c-1}(s)$ with all planes as blocks and $\lambda =\frac{s^{c-2}-1}{s-1}$.
\end{proposition}

\begin{proof}
Let $S$ be the subgroup of $G_{\Delta }^{\Delta }$ isomorphic to $PSL_{c}(s)$. Then $S$ acts point-transitively on $\mathcal{D}_{0}$ since $G_{\Delta }^{\Delta }$ acts point-primitively on $\mathcal{D}_{0}$, and hence $v_{0}=\left\vert S:S_{x}\right\vert $. It follows that 
\begin{equation*}
\left\vert S_{x}\right\vert \geq \frac{\left\vert S\right\vert }{s\frac{%
s^{c}-1}{s-1}\frac{s^{c-1}-1}{s-1}}=\frac{s^{c(c-1)/2-1}(s-1)^{2}}{(c,s-1)}%
\prod_{j=2}^{c-2}\left( s^{j}-1\right) \geq
s^{c(c-1)/2-1}\prod_{j=1}^{c-2}\left( s^{j}-1\right)\text{,} 
\end{equation*}
and hence $\left\vert S_{x}\right\vert ^{2}>s^{
2c^{2}-4c-2}>s^{c^{2}-1}>\left\vert S\right\vert $ since $c\geq 5$. Let $M$
be a maximal subgroup of $S$ containing $S_{x}$. Then $\left\vert
S\right\vert <\left\vert S_{x}\right\vert ^{2}\leq \left\vert M\right\vert
^{2}$. If $M$ is not a geometric subgroup of $S$,
then $c=5$, $s=3$ and $M\cong M_{11}$ by \cite[Table 7]{AB} since $c\geq 5$.
Note that, $\lambda $ divides $\left\vert S_{x}\right\vert $ since $\lambda $
divides $\left\vert S\right\vert $ but not $v_{0}$. Then $\lambda $ divides $%
\left\vert M\right\vert $ and so $\lambda =5$ or $11$ since $\lambda >3$,
and we reach a contradiction since $v_{0}\leq 2\cdot
11^{2}(11-1)=\allowbreak 2420$, whereas $\left\vert S:M\right\vert =30023136$
divides $v_{0}$. Thus $M$ is a geometric subgroup of $S$, and hence $M$ is
one of the groups listed in \cite[Proposition 4.7]{AB}. Now, it is a routine
exercise checking which groups in the previous list fulfills $S<\left\vert
M\right\vert ^{2}$, and we obtain that either $M$ is a maximal parabolic
subgroup of $S$ of type $P_{m}$ with $m\leq c/2$ or $M$ is a $\mathcal{C}_{3}
$-subgroup of $S$ of type $GL_{c/2}(s^{2})$ or a $\mathcal{C}_{8}$-subgroup
of $S$ of type $Sp_{c}(s)$ or $U_{c}(s^{1/2})$.

If $M$ is either a $\mathcal{C}_{3}$-subgroup of type $GL_{c/2}(s^{2})$, or
a $\mathcal{C}_{8}$-subgroup of $S$ of type $Sp_{c}(s)$ or $U_{c}(s^{1/2})$%
. Then either $\left\vert M\right\vert \leq 2s^{c^{2}/2-2}$ or $\left\vert
M\right\vert \leq 4s^{c(c+1)/2}$, or $\left\vert M\right\vert <2s^{c^{2}/2}$%
, respectively, (see \cite[pp. 34--36]{M5}). Thus either $v_{0}\geq
s^{c^{2}/2}/2$ or $v_{0}\geq s^{(c^{2}-c-4)/2}$ or $v_{0}\geq s^{c^{2}/2-2}$%
, respectively, since $\left\vert S\right\vert \geq s^{c^{2}-2}$. Then $%
v_{0}\geq s^{c^{2}/2-2}$ in each case, and hence $s^{c^{2}/2-2}\leq s\frac{%
s^{c}-1}{s-1}\frac{s^{c-1}-1}{s-1}<2s^{2c-1}$. However, this is impossible
since $c\geq 5$ and $s\geq 2$. Thus, $M$ is a maximal parabolic subgroup of $%
S$ of type $P_{m}$. Then either $S_{x}=M$ is of type $P_{m}$, or $S_{x}$ is
of type $P_{m,c-m}$ and $M$ is of type $P_{m}$ or $P_{c-m}$ by \cite[Table
3.5.H]{KL} for $c\geq 13$ and \cite[Section 8]{BHRD} for $5\leq c\leq 13$.

Assume that $S_{x}$ is
of type $P_{m,c-m}$ and $M$ is of type $P_{m}$ or $P_{c-m}$. By \cite[Proposition 4.1.22(II)]{KL}, one has%
\begin{equation*}
v_{0}={c\brack c-m}_{s}\cdot{c-m\brack m}_{s}=\frac{\prod_{j=1}^{c}\left( s^{j}-1\right) 
}{\left( \prod_{j=1}^{m}\left( s^{j}-1\right) \right) ^{2}\cdot \left(
\prod_{j=1}^{c-2m}\left( s^{j}-1\right) \right)}\text{,}
\end{equation*}
and by using \cite[Lemma 4.1(i)]{AB}, we deduce
\begin{equation*}
 v_{0}>\frac{(1-s-s^{-2})s^{c(c+1)/2}}{(1-s^{-1})^{3}(1-s^{-2})^{3}s^{m(m+1)+(c-2n)(c-2m+1)/2}}>s^{m(2c-3m)}\text{.}    
\end{equation*}
On the other hand, $v_{0}\leq s\frac{%
s^{c}-1}{s-1}\frac{s^{c-1}-1}{s-1}$ implies $v_{0} \leq 2s^{2c-2}$. Hence $%
s^{m(2c-3m)}<2s^{2c-2}$, from which we derive $m=1$ and   
\[
v_{0}=\frac{(s^{c}-1)(s^{c-1}-1)}{(s-1)^{2}}
\]%
since $c \geq 5$. Let $l$ be the prime dividing $s$. Then the highest power of $l$ dividing $v_{0}-1$ is $s$, $2s$ or $2^{c-1}$ according as $s$ is odd, $s$ is even and $s>2$ or $s=l=2$, respectively. Moreover, as pointed out in \cite[p. 339, (a)]{Saxl}, there is a $G_{\Delta}^{\Delta}$-orbit of length a power of $l$. Then either $\frac{v_{0}-1}{%
k_{0}-1}\mid s$ for $s$ odd, or $\frac{v_{0}-1}{k_{0}-1}\mid 2s$ for $s$
even and $s>2$, or $\frac{v_{0}-1}{k_{0}-1}\mid s^{c-1}$ for $s=2$ by Lemma \ref{PP}(2). Then either $%
\left( \frac{v_{0}-1}{k_{0}-1}\right) ^{2}<4s^{2}$ or $\frac{v_{0}-1}{k_{0}-1%
}\mid s^{c-1}$ according as $s>2$ or $s=2$, respectively. Assume that the
former occurs. Then $k_{0}<v_{0}$, $3<v_{0}$, $k_{0}-1 \mid v_{0}-1$ and Proposition \ref{P2}(4) imply $\frac{v_{0}-1}{k_{0}-1}\geq 2>\frac{v_{0}}{v_{0}-1}\cdot \frac{k_{0}-1}{\lambda}$. Therefore $\left( \frac{v_{0}-1}{k_{0}-1}\right) ^{2}>\frac{v_{0}}{\lambda }$, and hence $4s^{2}>\frac{s^{c}-1}{s-1}$ since $\lambda \leq \frac{%
s^{c-2}-1}{s-1}$, being $\lambda \nmid v_{0}$, which is not the case. Hence $s=2$, $\frac{v_{0}-1}{k_{0}-1%
}=2^{x}$ with $1<x\leq c-1$, $v_{0}=(2^{c}-1)(2^{c-1}-1)$ and $%
k_{0}=2^{c-1-x}(2^{c}-3)+1$. We know that $\frac{v_{0}-1}{k_{0}-1}\geq \lambda>k_{0}$ y Theorem \ref{Teo3}. Thus $2^{x}>2^{c-1-x}(2^{c}-3)+1>2^{2(c-1)-x}$ since $c \geq 5$, and hence $x>c-1$, a contradiction.

Assume that $%
S_{x}=M$ is of type $P_{m}$ with $m\leq c/2$. Then 
\[
s^{m(c-m)}\leq {c \brack m}_{s}=v_{0}\leq (s^{c}-1)s\frac{s^{c-2}-1}{s-1}<s^{2c-1}\text{,}
\]%
and hence $m\geq 3$. If $m=3$, then%
\[
\frac{(s^{c}-1)(s^{c-1}-1)(s^{c-2}-1)}{(s-1)(s^{2}-1)(s^{3}-1)}\leq (s^{c}-1)s\frac{s^{c-2}-1}{s-1}
\]%
and hence $m=3$ and $c=6$ since $m\leq c/2$. Then  
\[
v_{0}=(s^{2}+1)(s^{3}+1)\left( s^{4}+s^{3}+s^{2}+s+1\right) 
\]
and hence $\lambda $ divides $s-1$ or $s^{2}+s+1$ since $\lambda \nmid
v_{0}(v_{0}-1)$. However, this is contrary to $v_{0}<k<2\lambda (\lambda
^{2}-1)$.

Suppose that $m=2$. Then 
\[
v_{0}=\frac{(s^{c}-1)(s^{c-1}-1)}{(s^{2}-1)\left( s-1\right) }\text{.}
\]%
Further, since $S$ is a rank $3$ group with subdegrees $s(s+1)\frac{s^{c-2}-1%
}{s-1}$ and $\frac{s^{4}(s^{c-2}-1)(s^{c-3}-1)}{(s^{2}-1)\left( s-1\right) }$, it follows that $\frac{v_{0}-1}{k_{0}-1}$ divides $s\frac{(s^{c-2}-1)}{%
s^{2}-1}$ or $s\frac{(s^{c-2}-1)}{s^{2}-1}(s+1)\left( s+1,\frac{c-3}{2}\right) $
according as $c$ is even or odd respectively. Therefore, 
\[
\frac{\frac{(s^{c}-1)(s^{c-1}-1)}{(s^{2}-1)\left( s-1\right) }-1}{k_{0}-1}%
\mid s\frac{(s^{c-2}-1)}{s^{2}-1}\alpha 
\]%
with $\alpha =1$ or $(s+1)\left( s+1,\frac{c-3}{2}\right) $, respectively, and
hence 
\begin{equation}\label{yesterday}
\frac{s^{c}-1}{s-1}+s\mid \alpha (k_{0}-1)\text{.}    
\end{equation}
Then 
\begin{equation}\label{sadday}
\frac{s^{c}+s^{2}-s-1}{s-1}<
\alpha k_{0}<\alpha\lambda \leq \alpha \frac{s^{c-2}-1}{s-1}\text{,}
\end{equation}
and hence $\alpha>1$. Therefore, $c$ is odd and $\alpha \mid (s+1)\left( s+1,\frac{c-3}{2}\right)$.
If $\alpha \leq \frac{(s+1)^{2}}{2}$, then (\ref{sadday}) leads to $s^{c-2}(s^{2}-2s-1)+(3s^{2}-1) \leq 0$ and hence to $s=2$ and $\alpha \leq 5/2$. Therefore, (\ref{yesterday}) becomes $2^{c}+1 \mid \alpha (k_{0}-1)$. Actually, $\alpha=3$ since $\alpha \mid 3\left(3,\frac{c-3}{2}\right)$ and $\alpha>1$. So $2^{c}+1 \leq 3(2^{c-2}-1)$, a contradiction. Thus, $\alpha=(s+1)^{2}$. Now, if $\frac{s^{c}-1}{s-1}+s\leq \frac{1}{2} (s+1)^{2} (k_{0}-1)$, we obtain $s=2$ and $\alpha=9$ by the previous argument. Moreover, $c \equiv 0 \pmod{3}$. Then $k_{0}=\frac{2^{c}+10}{9}$ since $\frac{2^{c}+1}{9} \mid k_{0}-1$ and $k_{0} <\lambda \leq 2^{c-2}-1$. Then $2^{c}+10 \mid 3(2^{c}-1)(2^{c-1}-1)$ since $k_{0} \mid v_{0}$. Then $2^{c}+10 \mid 8\cdot 9 \cdot 11$, and hence $c= 1$ or $5$, whereas $c \geq 5$. Therefore, $\alpha=(s+1)^{2}$ and $\frac{s^{c}-1}{s-1}+s= (s+1)^{2} (k_{0}-1)$, and hence
\begin{equation*}
 k_{0}=\frac{s^{c}+s^{3}+2s^{2}-2s-2}{(s^{2}-1)(s-1)}\text{.}   
\end{equation*}
Then $s^{c}+s^{3}+2s^{2}-2s-2 \mid (s^{c}-1)(s^{c-1}-1)$ since $k_{0}\mid v_{0}$, and hence 
$$s^{c}+s^{3}+2s^{2}-2s-2 \mid (s^{3}+2s^{2}-2s-1)(s^{3}+2s^{2}-s-1)\text{,}$$
and we reach a contradiction since $c \geq 5$ and $c$ is odd. 

Finally, assume that $m=1$. Hence, $S_{x}=M$ is of type $P_{1}$. As pointed out in Remark \ref{oss}, the conclusion of Proposition \ref{D0ProjSpace} still holds under the assumption that $\mathcal{D}_{1}$ is of type Ia. Thus, $\mathcal{D}_{0} \cong PG_{c-1}(s)$ with all planes as blocks and $\lambda =\frac{s^{c-2}-1}{s-1}$ since $\mathcal{D}_{1}$ is of type Ia.
\end{proof}

\bigskip

\section{Reduction to the case where $G^{\Sigma}$ is almost simple}

In this section we prove the following theorem. More precisely, we show that $G^{\Sigma}$ cannot be of affine type. The remainder of the theorem immediately follows from Theorem \ref{Teo3}. 

\bigskip

\begin{theorem}\label{Teo4}
Assume that Hypothesis \ref{hyp2} holds. Then $\lambda \mid k_{1}$, $k_{0}<\lambda  < \frac{r_{1}}{(r_{1},\lambda_{1})}$ and the following hold:
\begin{enumerate}
\item $\mathcal{D}_{1}$ is $2$-$(v_{1},k_{1},\lambda_{1})$ design of type Ia or Ic admitting $G^{\Sigma}$ as a flag-transitive point-primitive automorphism group of almost simple type; 
    \item $\left\vert G^{\Sigma }\right\vert < \left\vert G_{\Delta }^{\Sigma }\right\vert ^{2}$ for any $\Delta \in \Sigma$;
    \item $\frac{r_{1}}{(r_{1},\lambda_{1})}>v_{1}^{1/2}$.
    \item If $e_{1},...,e_{z}$ are the lengths of the point-$G^{\Sigma}_{\Delta}$-orbits on $\mathcal{D}_{1}$ distinct from $\{\Delta\}$, then
\begin{equation}\label{Scorpus}
\frac{r_{1}}{(r_{1},\lambda_{1})} \mid \left(e_{1},...,e_{z},v_{1}-1,\left\vert T_{\Delta} \right\vert \cdot \left\vert Out(T) \right\vert\right)\text{.}
\end{equation}    
\end{enumerate}    
\end{theorem}

\bigskip

Throughout this section, we assume that $T=Soc(G^{\Sigma})$ is an elementary abelian $p$-group for some prime $p$. The point set of $\mathcal{D}_{1}$ can be
identified with $d$-dimensional $\mathbb{F}_{p}$-vector space $V$ in a way
that $T$ acts as the translation group and $G^{\Sigma
}=T:G_{\Delta}^{\Sigma }$ with $G_{\Delta}^{\Sigma }$ an irreducibly subgroup of $%
GL_{d}(p)$ since $G^{\Sigma }$ acts point primitively on $\mathcal{%
D}_{1}$. Thus $v_{1}=p^{d}$, and hence $\mathcal{D}_{1}$ is of type Ia by Theorems \ref{Teo1}(2) and \ref{Teo3}(1).

\bigskip

We formalize the conditions we have been discussing in the following hypothesis which will be assumed throughout this section.

\bigskip
\begin{hypothesis}\label{hyp4}
Hypothesis \ref{hyp2} holds, $T=Soc(G^{\Sigma})$ is an elementary abelian $p$-group for some prime $p$, $\mathcal{D}_{1}$ is of type Ia with $v_{1}=p^{d}$.     
\end{hypothesis}

\bigskip
Let $n$ be the minimal divisor of $d$ such that $G_{\Delta}^{\Sigma}$ is an irreducible subgroup of $\Gamma L_{n}(q)$, where $q=p^{d/n}$.

\bigskip

\begin{lemma}\label{nBiggerThan1}
Assume that Hypothesis \ref{hyp4} holds. Then $n>1$.    
\end{lemma}
\begin{proof}
Assume that $n=1$. Then $G^{\Sigma}\leq A\Gamma L_{1}(p^{d})$, and hence $\lambda \mid d$ since $\lambda$ divides the order of $G^{\Sigma}$ and $\lambda \nmid p^{d}(p^{d}-1)$ by Theorem \ref{Teo1} since $\mathcal{D}_{1}$ is of type Ia. Then $p^{d}=v_{1}<v\leq 2\lambda^{2}(\lambda-1)\leq 2d^{2}(d-1)$ by \cite[Theorem 1]{DP}, and hence either $p=2$ and $2 \leq d \leq 11 $ or $p=3$ and $d=3,4$. Actually, only $p=2$ and $\lambda=d=5,7,11$ are admissible since $\lambda \mid d$, $\lambda$ is a prime and greater than $3$. Moreover, we have $G^{\Sigma} \cong A\Gamma L_{1}(2^{d})$. Since $v_{1}-1=A\frac{v_{0}-1}{k_{0}-1}$ and $\lambda \geq A(k_{0}-1)+1$ by Proposition \ref{P2}(3)--(4), and $v_{1}-1=31,127$ or $23\cdot 89$, it follows that $A=1$, $k_{1}=\frac{v_{0}}{k_{0}}+1$ and $\frac{v_{0}-1}{k_{0}-1}=2^{d}-1$. Hence, the following admissible cases:
\begin{table}[h!]
\tiny
\caption{Admissible parameters for $\mathcal{D}_{0}$ and $\mathcal{D}_{1}$ when $G^{\Sigma} \cong A\Gamma L_{1}(2^{d})$.} \label{semiL}
\begin{tabular}{lllll}
\hline
$\lambda $ & $v_{0}$ & $k_{0}$ & $v_{1}$ & $k_{1}$ \\
\hline
$5$ & $32$ & $2$ & $2^{5}$ & $17$ \\ 
& $63$ & $3$ & $2^{5}$ & $22$ \\ 
$7$ & $128$ & $2$ & $2^{7}$ & $65$ \\ 
& $255$ & $3$ & $2^{7}$ & $86$ \\ 
& $636$ & $6$ & $2^{7}$ & $107$ \\ 
$11$ & $2048$ & $2$ & $2^{11}$ & $1025$ \\ 
& $4095$ & $3$ & $2^{11}$ & $1366$ \\ 
& $10236$ & $6$ & $2^{11}$ & $1707$\\
\hline
\end{tabular}
\end{table}
Now, it is easy to see that $k_{1}$ does not divide the order of $G^{\Sigma} \cong A\Gamma L_{1}(2^{d})$ in each case as in Table \ref{semiL}, and hence they are all ruled out. This completes the proof. 
\end{proof}

\bigskip
Let $X$ be any of the classical groups $SL_{n}(q)$, $Sp_{n}(q)$, $SU_{n}(q^{1/2})$ or $\Omega^{\varepsilon}(q)$ on $V=V_{n}(q)$, chosen to be minimal such that $G_{\Delta}^{\Sigma} \leq N_{\Gamma L_{n}(q)}(X)$. Firstly, we are going to prove that $G$
\bigskip

\begin{lemma}
\label{No2transitive} Assume that Hypothesis \ref{hyp4} holds. Let $\Delta \in \Sigma$, then the following hold:
\begin{enumerate}
\item $\lambda $ divides $\left\vert G_{\Delta}^{\Sigma }\cap SL_{n}(q)\right\vert $, and $\lambda $ does not divide $\left\vert G_{\Delta}^{\Sigma
}:G_{\Delta}^{\Sigma }\cap SL_{n}(q)\right\vert $;
\item If $G^{\Sigma }$ acts point-$2$-transitively on $\mathcal{D}_{1}$, then one of the following holds:
\begin{enumerate}
    \item $SL_{n}(q)\trianglelefteq G_{\Delta}^{\Sigma } \leq \Gamma L_{n}(q)$, $n \geq 4$;
    \item $Sp_{n}(q)\trianglelefteq G_{\Delta}^{\Sigma }\leq \Gamma Sp_{n}(q)$, $n \geq 10$.
\end{enumerate}
\end{enumerate}
\end{lemma}

\begin{proof}
Assume that $\lambda $ does not divide $\left\vert G_{\Delta}^{\Sigma }\cap
SL_{n}(q)\right\vert $. Then $\lambda $ divides $\left\vert G_{\Delta}^{\Sigma
}:G_{\Delta}^{\Sigma }\cap SL_{n}(q)\right\vert $ since $\lambda $ divides $%
\left\vert G^{\Sigma }\right\vert $, and hence $\lambda $ divides $%
\left\vert \Gamma L_{n}(q):SL_{n}(q)\right\vert $. Thus $\lambda \mid
(q-1)\cdot \log _{p}(q)$, and hence $\lambda \mid \log _{p}(q)$ since $%
v_{1}=q^{n}=p^{d}$ and $\lambda \nmid v_{1}-1$ by Theorem \ref{Teo1} since $\mathcal{D}_{1}$ if of type Ia. Then $p^{d}\leq 2\lambda ^{2}(\lambda -1)$ by \cite[Theorem 1]{DP}, and hence $p=2$ and $\lambda =d=n=5,7$ or $11$ arguing as in
the proof of Lemma \ref{nBiggerThan1}.
Therefore, $G^{\Sigma }\leq SL_{n}(2)$ with $n=5,7$ or $11$, but $%
\lambda $ does not divide $\left\vert G^{\Sigma }:G^{\Sigma }\cap
SL_{n}(q)\right\vert $. This proves (1).

Suppose that $n\leq3$. Then $n=3$ and $\lambda =\frac{q+1}{e}$ for some integer $e\geq 1$
by (1) and since $\lambda \nmid v_{1}(v_{1}-1)$. Actually, $e=1$ and $%
\lambda =q+1$ since $v_{1}\leq 2\lambda ^{2}(\lambda -1)$. Therefore, $%
q=2^{2^{j}}$ with $j\geq 1$ since $\lambda $ is a prime and $\lambda >3$.
Since $A\frac{v_{0}-1}{k_{0}-1}=q^{3}-1$, and $k_{1}=A\frac{v_{0}}{k_{0}}+1$%
, it follows that $k=Av_{0}+k_{0}$ and hence
\[
k=k_{0}k_{1}=\left( q^{3}-1\right) (k_{0}-1)+A+k_{0}>q^{3}(k_{0}-1)\text{%
.}
\]%
Hence $2 \leq k_{0}\leq 4$ or $(k_{0},q)=(5,4)$ since $k\leq 2\lambda ^{2}(\lambda
-1)=2q(q+1)^{2}$ and $k_{0}\geq 2$. The latter is ruled out since $%
k_{0}<\lambda $. Further, since $\lambda =q+1$ divides $k_{1}$, and hence $k$%
, it results $\lambda \mid -2k_{0}+2+A$ and hence $\lambda < A$, contrary to Proposition \ref{P2}(4). Thus, we obtain $n \geq 4$.

Assume that $G^{\Sigma }$ acts point-$2$-transitively on $\mathcal{D}_{1}$.
Then either $SL_{n}(q)\trianglelefteq G_{\Delta}^{\Sigma }$, or $Sp_{n}(q)%
\trianglelefteq G_{\Delta}^{\Sigma }$ ($n$ is even), or $G_{\Delta}^{\Sigma }\cong A_{7}$, $%
v_{1}=2^{4}$ and $\lambda =7$ by \cite[(B)]{Ka} since $n>1$, $\lambda $ divides $%
\left\vert G_{\Delta}^{\Sigma }\cap SL_{n}(q)\right\vert $, $\lambda
\nmid v_{1}(v_{1}-1)$ and $\lambda >3$. Assume that the latter occurs. Then $%
v_{1}-1=A\frac{v_{0}-1}{k_{0}-1}$ and $\frac{v_{0}-1}{k_{0}-1}>\lambda =7 $ \
force $A=1$. Then $zk_{0}+1=\frac{v_{0}-1}{k_{0}-1}=15$ by Proposition \ref{P2}(2). So $(z,k_{0})=(2,7)$
or $(7,2)$ since $\lambda \geq k_{0}\geq 2$. Then $\left(
v_{0},k_{0},v_{1},k_{1}\right) = (91,7,16,14)$ or $(16,2,16,9)$
since $k_{1}=A\frac{v_{0}}{k_{0}}+1$. The former is ruled out by Theorem \ref{Teo1} since it contradicts $\lambda \nmid v_{0}$, the latter since $9 \mid r_{1}$ but $r_{1} \mid 15 \cdot 8 \cdot 7$. Thus, either $SL_{n}(q)\trianglelefteq G_{\Delta}^{\Sigma }$, $n\geq 4$, and we obtain (2.a), or $Sp_{n}(q)%
\trianglelefteq G_{\Delta}^{\Sigma }$ ($n$ is even), $n\geq 4$. In the latter case, either $n=6$ and $\lambda =q^{2}+1$, or $n=8$, $\lambda \mid q^{2}+\varepsilon q+1$, $\varepsilon =\pm$, since $\lambda \mid \left\vert G_{\Delta }^{\Sigma }\cap SL_{n}(q)\right\vert $ by Lemma \ref{No2transitive}(1) and $\lambda \nmid v_{1}(v_{1}-1)$. On the other hand, 
\begin{equation}\label{SMAOL}
4q^{n} \leq v_{0}v_{1} \leq 2\lambda^{2}(\lambda-1)    
\end{equation}
by Lemma \ref{L1} and \cite[Theorem 1]{DP}. Thus, both cases are excluded since they do not fulfill (\ref{SMAOL}). Thus $n \geq 10$, and we obtain (2.b).
\end{proof}

\bigskip

\begin{lemma}\label{SL}
Assume that Hypothesis \ref{hyp4} holds. If $SL_{n}(q)\trianglelefteq G_{\Delta}^{\Sigma } \leq \Gamma L_{n}(q)$, $n\geq 4$, then $G_{(\Sigma )}=1$.
\end{lemma}

\begin{proof}
Assume that $G_{(\Sigma )}\neq 1$. Then $r_{1}=\frac{v_{0}-1}{k_{0}-1}\lambda
$ and $b_{1}=\frac{v_{1}(v_{0}-1)\lambda}{k_{1}(k_{0}-1)}$ by Lemma \ref{quasiprimitivity}(2). Moreover, $\lambda \mid \frac{q^{n-i}-1}{q-1}$ with $%
i\geq 1$ since $\lambda \nmid v_{1}(v_{1}-1)$, and so $r_{1}=\frac{q^{n}-1}{A}\lambda $.

We are going to prove the assertion in two steps.

\begin{claim}
$\lambda =\frac{q^{n-1}-1}{q-1}$ and $r_{1}=\frac{q^{n}-1}{A}\cdot \frac{q^{n-1}-1}{q-1}$.    
\end{claim}

Let $K$ be the stabilizer in $%
G^{\Sigma }$ of a block $B^{\Sigma }$ of $\mathcal{D}_{1}$ containing $\Delta \in \Sigma $. Then $K_{\Delta }$ contains a Sylow $p$-subgroup
of $G_{\Delta }^{\Sigma }$ since $r_{1}$ is coprime to $p$, and hence $%
K_{\Delta }\leq G_{\Delta }^{\Sigma }$ with $[q^{n-1}]:GL_{n-1}(q)%
\trianglelefteq G_{\Delta }^{\Sigma }\leq \lbrack q^{n-1}]:\Gamma L_{n-1}(q)$ by \cite[Proposition 4.1.17(II)]{KL}. Thus,
\begin{equation}
{n-1 \brack j}_{q}\mid \frac{q^{n}-1}{A}\frac{q^{n-i}-1}{q-1}
\label{divisione}
\end{equation}%
for some $1\leq j\leq n/2$. Then $q^{j(n-j)}\leq {n-1 \brack j}%
_{q}<q^{2n-i}$, and hence $n(j-2)<j^{2}-i\leq j^{2}-1$ since $i\geq 1$. If $j>2$ then $n<j+2+\frac{4-i}{%
j-2}\leq n/2+5$, and hence either $j=3$, $n=6,7$ and $i=1,2$. However, it is easy to see that none of these remaining cases fulfills (%
\ref{divisione}), and hence they are all excluded. Thus, $j=1$ or $2$. 

Assume that $j=2$. Then (\ref{divisione}) implies%
\[
\frac{\left( q^{n-1}-1\right) \left( q^{n-2}-1\right) }{\left( q-1\right)
\left( q^{2}-1\right) }\mid \left( q^{n}-1\right) \frac{q^{n-i}-1}{q-1}
\]%
and hence $i=1$ and $q^{n-2}-1\mid \left(q^{n}-1\right)\left( q^{2}-1\right)$. Thus, we obtain $n=4$, $i=1$, $\lambda=q^{2}+q+1$ and $r_{1}=\frac{q^{4}-1}{A}(q^{2}+q+1)$, which is the assertion for $n=4$.

Assume that $j=1$. Hence (\ref{divisione}) implies $i=1$.
Moreover $\frac{q^{n-1}-1}{q-1}\mid \lambda $ since $\frac{q^{n-1}-1}{q-1}$
divides $\frac{q^{n}-1}{A}\lambda $, and hence $\lambda =\frac{q^{n-1}-1}{q-1%
}$ since $\lambda \mid \left\vert G_{\Delta }^{\Sigma }\cap
GL_{n}(q)\right\vert $ and $\lambda \nmid q^{n}(q^{n}-1)$. Therefore, $r_{1}=%
\frac{q^{n}-1}{A}\cdot \frac{q^{n-1}-1}{q-1}$.

\begin{claim}
$G_{(\Sigma )}=1$.     
\end{claim}
Then the group $K/T_{B^{\Sigma}}$ is isomorphic to a subgroup $J$ of $G_{\Delta }^{\Sigma }$ 
containing $K_{\Delta}$ and such that $\left\vert G_{\Delta }^{\Sigma
}:J\right\vert =\frac{r_{1}\left\vert T_{B^{\Sigma}}\right \vert }{k_{1}}$. Then $\left\vert SL_{n}(q):J\cap
SL_{n}(q)\right\vert =\left\vert SL_{n}(q)J:J\right\vert $ divides $\frac{%
r_{1}\left\vert T_{B^{\Sigma}}\right \vert}{k_{1}}$, and hence   
\begin{equation}
\left\vert SL_{n}(q):J\cap SL_{n}(q)\right\vert =\frac{r_{1}\left\vert T_{B^{\Sigma}}\right \vert}{k_{1}}\text{.}
\label{danas}
\end{equation}%

On the other hand, $k_{1}=\alpha \frac{q^{n-1}-1}{q-1}$ with $\alpha \geq 1$
since $\lambda \mid k_{1}$. Moreover 
\begin{equation}\label{alfa}
\frac{q(q-1)}{2}<\frac{q^{n}(q-1)}{2(q^{n-1}-1)}\leq \alpha 
\end{equation}
since $k_{1}\geq \frac{v_{1}}{2}$, and so $k_{1}>\frac{q\left(
q^{n-1}-1\right) }{2}$. Then (\ref{danas}) implies%
\[
\left\vert SL_{n}(q):J\cap SL_{n}(q)\right\vert = \frac{r_{1}\left\vert T_{B^{\Sigma}}\right \vert}{k_{1}}<2%
\frac{(q^{n}-1)\left\vert T_{B^{\Sigma}}\right \vert}{Aq(q-1)}\leq 2%
\frac{(q^{n}-1)(z,A+1)}{Aq(q-1)}
\]%
 since $\left\vert T_{B^{\Sigma
}}\right\vert \mid \left( k_{1},v_{1}\right) $ and $\left( k_{1},v_{1}\right) \mid (z,A+1)$ by Lemma \ref{L3}(2).

If $Aq < 2(z,A+1)$ then either $q=3$ and $A=1$, or $q=2$ and $A+1 \mid z$. The former is ruled out since $3$ divides the order of $T_{B^{\Sigma}}$ and hence $A+1$, which is not the case. Thus $q=2$ and hence $\lambda=2^{n-1}-1$ and $k_{1}=2^{n}-2$ since $\alpha>1$ by (\ref{alfa}) and $k_{1}<v_{1}$. Then $\frac{v_{0}-1}{k_{0}-1}=2^{n}-1$ and $\frac{v_{0}}{k_{0}}=2^{n}-3$, and hence $k_{0}=2^{n-1}-1=\lambda$, which is contrary to Theorem \ref{Teo1}.

If $Aq \geq 2(z,A+1)$ then $\left\vert SL_{n}(q):J\cap SL_{n}(q)\right\vert < \frac{q^{n}-1}{q-1}$, and hence either $SL_{n}(q)\leq J$ or $n=4$ and $q=2$ by \cite[Proposition
5.2.1(i) and Theorem 5.2.2]{KL} since $n\geq 4$. In the latter case, $%
k_{1}<16$ and $\lambda =7$ divides $k_{1}$. Thus $k_{1}=\lambda =7$, and hence $A=3$ and $k_{0}=2$ by Proposition \ref{P2}(4) since $k_{0} \geq 2$. Therefore, $\frac{v_{0}-1}{k_{0}-1}=5$ by Proposition \ref{P2}(3), contrary to $\frac{v_{0}-1}{k_{0}-1}>\lambda=7$. Therefore, $SL_{n}(q)\leq J$ and hence $SL_{n}(q) \leq K/T_{B^{\Sigma}}$.

If $(q,n)= (2,6)$, then $\lambda=31$ and hence $k_{1}=62$ since $\lambda \mid k_{1}$ and $v_{1}/2 \leq k_{1}<v_{1}$. Then $\frac{v_{0}-1}{k_{0}-1}=63$ and $A=1$ since $v_{1}=A\frac{v_{0}-1}{k_{0}-1}+1$ and $\frac{v_{0}}{k_{0}}=61$ since $v_{1}=A\frac{v_{0}-1}{k_{0}-1}+1$. So, $k_{0}=v_{0}=31$, contrary to to Theorem \ref{Teo1}. Thus $(q,n)= (2,6)$, and hence $K$ contains an element $\vartheta$ of order a primitive prime divisor of $q^{n}-1$ by Zsigmondy Theorem \cite[Theorem 5.2.14]{KL}. Moreover, $\vartheta$ normalizes $T_{B^{\Sigma}}$, and hence either $T_{B^{\Sigma}}=T$ or $T_{B^{\Sigma}}=1$ by \cite[Theorem 3.5]{He}. The latter implies $T:SL_{n}(q)$  and hence $k_{1}=v_{1}$, which is not the case. Thus $T_{B^{\Sigma}}=1$ and $SL_{n}(q)\leq J$, and hence $K$ contains a normal subgroup $Q$
isomorphic to $SL_{n}(q)$. Then $Q$ fixes a unique point $\Delta ^{\prime }$
of $\mathcal{D}_{1}$ by \cite[Theorem 2.14(ii)]{KanLib}, and hence $%
Q\trianglelefteq K\leq G_{\Delta ^{\prime }}^{\Sigma }$. Then $%
k_{1}=v_{1}-1=q^{n}-1$ since $Q$ transitively on the point set of $\mathcal{D%
}_{1}$ not containing $\left\{ \Delta ^{\prime }\right\} $ and $1<k_{1}<v_{1}$. However, this is
impossible since $\lambda =\frac{q^{n-1}-1}{q-1}$ divides $k_{1}$ and $n\geq
4$. This completes the proof.
\end{proof}

\bigskip

\begin{lemma}\label{Sp}
Assume that Hypothesis \ref{hyp4} holds. If $Sp_{n}(q)\trianglelefteq G_{\Delta }^{\Sigma }\leq \Gamma Sp_{n}(q)$ with $n\geq 10$, then $G_{(\Sigma )}=1$.
\end{lemma}

\begin{proof}
Assume that $G_{(\Sigma )}\neq 1$. Then $r_{1}=\frac{v_{0}-1}{k_{0}-1}\lambda
$ and $b_{1}=\frac{v_{1}(v_{0}-1)\lambda}{k_{1}(k_{0}-1)}$ by Lemma \ref{quasiprimitivity}(2). Moreover, $\lambda \mid q^{n/2-i}+\varepsilon 1$, with $%
1 \leq i <n$ and $\varepsilon=\pm$ by Lemma \ref{No2transitive}(1) since $\lambda \nmid v_{1}(v_{1}-1)$, and so $r_{1}=\frac{q^{n}-1}{A}\lambda $ is coprime to $p$.

Let $K$ be the stabilizer in $G^{\Sigma }$ of a block $B^{\Sigma }$ of $%
\mathcal{D}_{1}$ containing $\Delta $ for some $\Delta \in \Sigma $. The
group $K_{\Delta }$ is a subgroup of $G_{\Delta }^{\Sigma }$ containing a
Sylow $p$-subgroup of $G_{\Delta }^{\Sigma }$ since $r_{1}$ is coprime to $p$%
. If $K_{\Delta }\cap Sp_{n}(q)<Sp_{n}(q)$, then $K_{\Delta }\cap Sp_{n}(q)$
lies in maximal parabolic subgroup of $Sp_{n}(q)$. Hence, 
\begin{equation}
\prod_{j=0}^{h-1}\frac{q^{n-2j}-1}{q^{j+1}-1}\mid \frac{q^{n}-1}{A}\lambda 
\text{,}  \label{divisionebis}
\end{equation}%
by \cite[Proposition 4.1.19(II)]{KL}. Assume that $h>2$. Since $n\geq 10$ and $h\leq n/2$, it follows that $n-4>h$. For each $s=0,2,4$, the integer $\prod_{j=1}^{h-1}\frac{q^{n-2j}-1}{q^{j+1}-1}$ is divisible by a primitive prime divisor of $q^{n-s}-1$, whereas $\frac{%
q^{n}-1}{A}\left( q^{n/2-i}+\varepsilon 1\right) $ is not. If $h=2$ then $%
\frac{q^{n}-1}{q-1}\cdot \frac{q^{n-2}-1}{q^{2}-1}\mid \frac{q^{n}-1}{A}%
\left( q^{n/2-i}+\varepsilon 1\right) $, forcing $i=1$ since $\frac{q^{n}-1%
}{q-1}\cdot \frac{q^{n-2}-1}{q^{2}-1}$ is divisible by a primitive prime
divisor of $q^{n-s}-1$ for each $s=0,2$. Then $A\left( q^{n/2-1}-\varepsilon
1\right) \mid (q-1)(q^{2}-1)$, and hence $q^{n/2-1}\leq q^{3}-q-q^{2}+2<q^{3}
$ and hence $n/2-1<3$, which implies $n\leq 8$, a contradiction. Thus $h=1$, and hence $A \mid (q-1)\lambda$ by (\ref{divisionebis}). Actually, $A\mid q-1$ since $A<\lambda$ by Proposition \ref{P2}(4). and hence $K_{\Delta }\leq \lbrack q^{n-1}]:\left( (q-1)\times
Sp_{n-2}(q)\right) .(2,q-1)\cdot \log _{p}q$ by \cite[Proposition 4.1.19(II)]{KL}. Recall that $K_{\Delta }$ is a
subgroup of $G_{\Delta }^{\Sigma }$ containing  a Sylow $p$-subgroup of $%
G_{\Delta }^{\Sigma }$, and let $R=\frac{K_{\Delta }}{[q^{n-1}]}$. If $R \cap Sp_{n-2}(q)< Sp_{n-2}(q)$, then $R \cap Sp_{n-2}(q)$ lies in a maximal parabolic subgroup of $Sp_{n-2}(q)$, and hence
\[
\frac{q^{n}-1}{q-1}\cdot \prod_{j=0}^{m-1}\frac{q^{n-2-2j}-1}{q^{j+1}-1}\mid 
\frac{q^{n}-1}{A}\left( q^{n/2-i}+\varepsilon 1\right) 
\]%
Since $n\geq 10$ and $m\leq n/2$ then $n-2>m$ and so $\prod_{j=0}^{m-1}\frac{%
q^{n-2-2j}-1}{q^{j+1}-1}$ is divisible by a primitive prime divisor of $%
q^{n-s}-1$ with $s=0,2,4$ for $m>1$, whereas $\frac{q^{n}-1}{A}\left(
q^{n/2-i}+\varepsilon 1\right) $ is not. Thus, $m=1$ and $\frac{q^{n}-1}{%
q-1}\cdot \frac{q^{n-2}-1}{q-1}\mid \frac{q^{n}-1}{A}\left(
q^{n/2-i}+\varepsilon 1\right) $. Now, arguing as above, we obtain $A\left(
q^{n/2-1}-\varepsilon 1\right) \mid (q-1)(q^{2}-1)$ and we reach a
contradiction since $n\geq 10$. Thus $Sp_{n-2}(q) \leq R$, and hence $K$ contains proper subgroup of order $K_{\Delta}$ such that $Sp_{n-2}(q) \leq \frac{K_{\Delta }}{[q^{n-1}]}$. Then either $K$ lies in a maximal parabolic subgroup of $G^{\Sigma }$ of type $P_{1}$, or $Sp_{n-2}(q)\trianglelefteq K$ since $K$ contains a Sylow $p$-subgroup of $G_{\Delta }^{\Sigma }$. In the former
case, $k_{1}\mid (q-1)(2,q-1)\cdot \log _{p}q$ since $K\leq \lbrack
q^{n-1}]:\left( (q-1)\times Sp_{n-2}(q)\right) .(2,q-1)\cdot \log _{p}q$.
Then $\lambda \mid (q-1)(2,q-1)\cdot \log _{p}q$ since $\lambda \mid k_{1}$,
and hence $\lambda \mid \log _{p}q$ since $\lambda \nmid v_{1}-1$. Now, the same argument used in the proof of Lemma \ref{No2transitive}(1) rules out this case. Thus $Sp_{n}(q)\trianglelefteq K
$, and hence $b_{1} \mid q^{n}(q-1)(2,q-1)\log_{p}q$. Therefore $q^{n}\cdot \frac{k_{1}}{\lambda}\cdot \frac{v_{0}-1}{k_{0}-1} \mid  q^{n}\cdot (q-1)\cdot (2,q-1)\cdot \log_{p}q$ and hence
$$
\frac{q^{n}}{2} \leq k_{1} \leq (q-1) \cdot (2,q-1)\cdot \log _{p}\leq 2q^{1/2}(q-1)
$$
since $\frac{v_{0}-1}{k_{0}-1} \geq \lambda$, which does not admit admissible solutions. This completes the proof. 
\end{proof}

\bigskip 

\begin{theorem}\label{NoAff2tr}
Assume that Hypothesis \ref{hyp4} holds. Then $G^{\Sigma }$ does not act point-$2$-transitively on $\mathcal{D}_{1}$.
\end{theorem}

\begin{proof}
Assume that $G^{\Sigma }$ acts point-$2
$-transitively on $\mathcal{D}_{1}$. Then $G_{(\Sigma )}=1$, and either $%
SL_{n}(q)\trianglelefteq G_{\Delta }\leq \Gamma L_{n}(q)$ with $%
n\geq 4$, or $Sp_{n}(q)\trianglelefteq G_{\Delta }\leq \Gamma
Sp_{n}(q)$ with $n\geq 10$ by Lemmas \ref{No2transitive}(2), \ref{SL} and \ref{Sp}.
Note that, $\lambda \nmid v(v_{1}-1)$ by Theorem \ref{Teo1} since $\mathcal{D}_{1}$ is of type Ia. Hence, $\lambda \mid \left\vert G_{\Delta}^{\Delta }\right\vert $ by Lemma \ref{Fix}. 

The group $Soc(G_{\Delta }^{\Delta })$ is either an
elementary abelian $p$-group or a non-abelian simple group by the point-$2$-transitivity of $G_{\Delta}^{\Delta}$ on $\mathcal{D}_{1}$ for $k_{0}=2$, by \cite[Main Theorem]{BDD} for $k_{0}\geq 3$ and $\mu=\lambda$, and by \cite[Theorem 1]{ZC} for $k_{0}\geq 3$ and $\mu=1$. 

Assume that $SL_{n}(q)\trianglelefteq G_{\Delta }\leq \Gamma L_{n}(q)$ with $%
n\geq 4$. If $SL_{n}(q)\trianglelefteq G_{(\Delta) }$ then $\lambda \mid \left \vert G_{\Delta}:SL_{n}(q)\right\vert $ since $\lambda \mid \left\vert G_{\Delta}^{\Delta }\right\vert $, but this contradicts Lemma \ref{No2transitive}(1). Thus, $G_{(\Delta) }$ is the center of $SL_{n}(q)$ and $PSL_{n}(q)\trianglelefteq G_{\Delta }^{\Delta }\leq \Gamma L_{n}(q)$ with $n\geq 4$ by the previous argument on the strcuture of $Soc(G_{\Delta }^{\Delta })$. Similarly, we have $Sp_{n}(q) \nleq  G_{(\Delta) }$ when $Sp_{n}(q)\trianglelefteq G_{\Delta }\leq \Gamma Sp_{n}(q)$ with $n\geq 10$, and hence $PSp_{n}(q)\trianglelefteq G_{\Delta }^{\Delta}\leq P \Gamma
Sp_{n}(q)$ with $n\geq 10$.  Moreover, in both cases, one has  $\lambda \leq \frac{q^{n-1}-1}{q-1}$ since $\lambda \nmid v_{1}(v_{1}-1)$, being $\mathcal{D}_{1}$ is of type Ia, and hence 
\[
v_{0}=\frac{q^{n}-1}{A}(k_{0}-1)+1\leq \left( q^{n}-1\right) \lambda \leq
\left( q^{n}-1\right) \frac{q^{n-1}-1}{q-1}<q\frac{q^{n}-1}{q-1}\cdot \frac{%
q^{n-1}-1}{q-1}\text{.}
\]%
Then Proposition \ref{moreover} implies the exclusion of $PSp_{n}(q)\trianglelefteq G_{\Delta }^{\Delta }\leq P\Gamma
Sp_{n}(q)$ with $n\geq 10$, and $n$ odd, $v_{0}=\frac{%
q^{n}-1}{q-1}$ and $k_{0}=q^{2}+q+1$ for $PSL_{n}(q)\trianglelefteq
G_{\Delta }^{\Delta }\leq P\Gamma L_{n}(q)$. Thus $\frac{v_{0}-1%
}{k_{0}-1}=\frac{q^{n-1}-1}{q+1}$ since $n$ is odd, and hence $q^{n-1}-1\mid \frac{q^{n}-1%
}{A}\cdot \left( q+1\right) $ since $v_{1}-1=A\frac{v_{0}-1%
}{k_{0}-1}$ and $v_{1}=q^{n}$. So  $q^{n-1}-1\mid 
q^{2}-1$, contrary to $n \geq 4$.
\end{proof}

\bigskip
\begin{lemma}
\label{NoC8}Assume that Hypothesis \ref{hyp4} holds. Then $G_{\Delta }^{\Sigma }$ does not contain $X$ in its natural
action on $V_{n}(q)$.
\end{lemma}

\begin{proof}
Assume that $G_{\Delta }^{\Sigma }$ contains $X$ in its natural action on $%
V=V_{n}(q)$. If $X\cong SL_{n}(q)$ or $Sp_{n}(q)$ then $G^{\Sigma }$ acts
point-$2$-transitively on $\mathcal{D}_{1}$, which is not the case by
Theorem \ref{NoAff2tr}. Hence, either $X\cong SU_{n}(q^{1/2})$ or $X\cong
\Omega ^{\varepsilon }(q)$ with $\varepsilon \in \left\{ \pm ,\circ \right\} $. Moreover, $\lambda \mid \left\vert
X\right\vert $ by Lemma \ref{No2transitive}(1), and $\lambda \nmid v_{1}(v_{1}-1)$ by Theorem \ref{Teo1} since $\mathcal{D}_{1}$ is of type Ia.

Suppose that $X\cong SU_{n}(q^{1/2})$. Then $\frac{r_{1}}{(r_{1},\lambda
_{1})}$ divides the number of nonzero isotropic vectors of $V$, hence 
\[
\frac{r_{1}}{(r_{1},\lambda _{1})}\mid
(q^{n}-1,((q^{1/2})^{n}-(-1)^{n})((q^{1/2})^{n-1}-(-1)^{n-1})) 
\]%
If $n$ is even, then $\frac{r_{1}}{(r_{1},\lambda _{1})}=\frac{v_{0}-1}{%
k_{0}-1}=2(q^{n/2}-1)$ with $q$ odd since $\frac{r_{1}}{(r_{1},\lambda _{1})}%
>q^{n/2}$ by Theorem \ref{Teo4}(4). Then $A=\frac{q^{n/2}+1}{2}$ by
Proposition \ref{P2}(3) since $v_{1}=q^{n}$. Therefore, $\lambda \geq \frac{%
q^{n/2}+1}{2}+1$ by Proposition \ref{P2}(4). On the other hand, $\lambda
\mid q^{i/2}-(-1)^{i}$ with $i<n$ since $\lambda \mid \left\vert
X\right\vert $ and $\lambda \nmid v_{1}(v_{1}-1)$ . So $\frac{%
q^{n/2}+1}{2}+1\leq \lambda \leq q^{(n-1)/2}-1$ with $q$ odd, which is a
contradiction.

If $n$ is odd, then $\frac{r_{1}}{(r_{1},\lambda _{1})}=\frac{%
(q^{n/2}+1)(q-1)}{\theta }$ for some $\theta \geq 1$. Hence $A=\frac{%
\left( q^{n/2}-1\right) \theta }{q^{1/2}-1}$ and $\lambda \geq A+1=\frac{%
\left( q^{n/2}-1\right) \theta }{q^{1/2}-1}+1>\frac{q^{n/2}-1}{q^{1/2}-1}$.
However, $\lambda \mid \frac{q^{i/2}-(-1)^{i}}{q^{1/2}-(-1)^{i}}$ with $i<n$
since $\lambda \mid \left\vert X\right\vert $, $\lambda \nmid v_{1}(v_{1}-1)$
and $\lambda $ is a prime. So $\frac{q^{n/2}-1}{q^{1/2}-1}<\frac{%
q^{i/2}-(-1)^{i}}{q^{1/2}-(-1)^{i}}\leq \frac{q^{(n-1)/2}+1}{q^{1/2}-1}$,
which is impossible for $n\geq 3$.

Suppose that $X\cong \Omega ^{\varepsilon }(q)$, $\varepsilon \in \left\{
\pm ,\circ \right\} $. If $\varepsilon =\circ $ then $\frac{r_{1}}{%
(r_{1},\lambda _{1})}\mid (q^{n}-1,q^{n-1}-1)=q-1$ since $q^{n-1}-1$ is the
number of nonzero singular vectors of $V$, but this is contrary to $\frac{%
r_{1}}{(r_{1},\lambda _{1})}>q^{n/2}$. Thus, $n$ is even. If $\varepsilon =+$%
, then $\frac{r_{1}}{(r_{1},\lambda _{1})}\mid
(q^{n}-1,(q^{n/2}-1)(q^{n/2-1}+1))$ since $%
(q^{n/2}-1)(q^{n/2-1}+1)$ is the number of nonzero singular vectors of $V$. So, $q$ is odd, $\frac{r_{1}}{%
(r_{1},\lambda _{1})}=2(q^{n/2}-1)$, $A=\frac{q^{n/2}+1}{2}$ and $\lambda
\geq \frac{q^{n/2}+1}{2}+1$, whereas $\lambda \mid q^{i/2}-1$ \ since $i<n$, 
$\lambda \mid \left\vert X\right\vert $ and $\lambda \nmid v_{1}(v_{1}-1)$,
and we reach a contradiction. Therefore, $\varepsilon =-$. Therefore $\frac{%
r_{1}}{(r_{1},\lambda _{1})}\mid (q^{n}-1,(q^{n/2}+1)(q^{n/2-1}-1))$ since $%
(q^{n/2}+1)(q^{n/2-1}-1)$ is the number of nonzero singular vectors of $V$,
and hence $\frac{r_{1}}{(r_{1},\lambda _{1})}=\frac{(q^{n/2}+1)(q-1)}{\theta 
}$ for some $\theta \geq 1$. Then $A=\frac{\left( q^{n/2}-1\right) \theta }{%
q-1}$ and $\lambda >\frac{q^{n/2}-1}{q-1}+1$. On the other hand, $\lambda
\mid \frac{q^{i}-1}{q-1}$ with $i<2n$ since $\lambda \mid \left\vert
X\right\vert $ and $\lambda \nmid v_{1}(v_{1}-1)$ and $%
\lambda \nmid v_{1}(v_{1}-1)$, and again we reach a contradiction. This
completes the proof.
\end{proof}
\bigskip

\subsection{Aschbacher's theorem.} \label{Aschbacher}Recall that $G_{\Delta}^{\Sigma} \leq N_{\Gamma L_{n}(q)}(X)$ and $X$ is any of the classical groups $SL_{n}(q)$, $Sp_{n}(q)$, $SU_{n}(q^{1/2})$ or $\Omega^{\varepsilon}(q)$. The case where where $X \leq G_{0}$ has been settled in Lemma \ref{NoC8}, hence, in the sequel, we assume that $G_{\Delta}^{\Sigma}$ does not contain $X$. Now, according to \cite{As}, one of the following holds: 
\begin{enumerate}
    \item[(I)] $G_{\Delta}^{\Sigma}$ lies in a maximal member of one of the geometric classes $\mathcal{C}_{i}$ of $N_{\Gamma L_{n}(q)}(X)$, $i=1,...,7$ (the case  $\mathcal{C}_{8}$ is excluded in Lemma \ref{NoC8});
    \item[(II)] $\left(G_{\Delta}^{\Sigma}\right)^{(\infty)}$ is a quasisimple group, and its action on $V_{n}(q)$ is absolutely irreducible and not realizable over any proper subfield of $\mathbb{F}_{q}$.
\end{enumerate}
Description of each class  $\mathcal{C}_{i}$, $i=1,...,8$, can be found in \cite[Chapter 4]{KL}.   

\bigskip
\bigskip

The proof strategy used by Liebeck in \cite{LiebF} to classify the $2$-$(q^{n},k^{\prime},1)$ designs $\mathcal{D}^{\prime}$ (linear spaces) admitting flag transitive group $H=T:H_{0} \leq A\Gamma L_{n}(q)$, with $T$, $V$, $n$, $q$ having the same meaning as ours, is group-theoretical and is primarily to filter out the possible candidates for the group $H_{0}$ with respect to the properties 
\begin{equation}\label{Liebe}
r^{\prime}\mid (q^{n}-1,c_{1},...,c_{j},\left\vert H_{0}\right\vert)
\text{~and~}
r^{\prime}>q^{n/2}\textit{,}
\end{equation}
where $r^{\prime}=\frac{q^{n}-1}{k^{\prime}-1}$, and $c_{1},...,c_{j}$ are the lengths of the $H_{0}$-orbit on the set of nonzero vectors of $V$, by \cite[Lemma 2.1]{LiebF}. In other words, in Liebeck's paper the admissible groups are those affording a (semi)linear representation with non-trivial orbits divisible by the same factor of size greater than square root of the size of the vector space. In our context, the constraints in (\ref{Liebe}) are replaced by
\begin{equation}\label{pasticciotto}
\frac{r_{1}}{(r_{1},\lambda_{1})}\mid (q^n-1,c_{1},...,c_{j},\left\vert H_{0}\right\vert)
\text{~and~}
\frac{r_{1}}{(r_{1},\lambda_{1})}>\frac{q^{n/2}}{\sqrt{2}} \text{,}
\end{equation}
which hold by Theorem \ref{Teo3}(3)--(4). Hence, we may use Liebeck's argument with $G^{\Sigma}$ and $\frac{r_{1}}{(r_{1},\lambda_{1})}$ in the role of $H$ and $r^{\prime}$ to reduce our investigation to the cases where (\ref{pasticciotto}) is fulfilled. We provide in one exemplary case (namely, when $G_{\Delta}^{\Sigma}$ lies in a maximal $\mathcal{C}_{2}$ subgroup of $N_{\Gamma L_{n}(q)}(X)$) more guidance and proof details  to help the reader control the transfer from the linear space case investigated by Liebeck to our case. For the remaining groups in (I) or (II), the proof strategy, which relies on Liebeck's argument, is similar. The corresponding admissible cases are contained in Tables \ref{AschClasses}, \ref{AltSpor},\ref{LieNat} and \ref{LieCross}, respectively. Hence, the objective of this section is to  prove the following  result by using the above mentioned adaptation of Liebeck's argument. 

\bigskip

\begin{proposition}\label{ReductionNAS}
Assume that Hypothesis \ref{hyp2} holds. Then $\left(G_{\Delta}^{\Sigma}\right)^{(\infty)}$ is a quasisimple group, and its action on $V_{n}(q)$ is absolutely irreducible and not realizable over any proper subfield of $\mathbb{F}_{q}$.    
\end{proposition}

\begin{proof}
Assume that Hypothesis \ref{hyp2} holds. Then the conclusions of Theorem \ref{Teo3} hold. If $G^{\Sigma}$ is of affine type then Hypothesis \ref{hyp4} holds, and hence $G^{\Sigma}$ is as in (I) or (II) above. Assume that $G_{\Delta}^{\Sigma}$ lies in a maximal member of one of the geometric classes $\mathcal{C}_{i}$ of $N_{\Gamma L_{n}(q)}(X)$, $i=1,...,7$. The group $G_{\Delta}^{\Sigma}$ does not lie maximal member of type $\mathcal{C}_{1}$ since $G_{\Delta}^{\Sigma}$ acts irreducibly on $V_{n}(q)$. Moreover, by the definition of $q$, $G_{\Delta}^{\Sigma}$ does not lie in a
member of $\mathcal{C}_{3}$. If $G_{\Delta}^{\Sigma}$ lies in a maximal member of one of the geometric classes $\mathcal{C}_{2}$ of $N_{\Gamma L_{n}(q)}(X)$. The group $G_{\Delta}^{\Sigma}$ does not lie maximal member of type $\mathcal{C}_{1}$ since $G_{\Delta}^{\Sigma}$ acts irreducibly on $V_{n}(q)$. Moreover, by the definition of $q$, $G_{\Delta}^{\Sigma}$ does not lie in a
member of $\mathcal{C}_{3}$ of $N_{\Gamma L_{n}(q)}(X)$. Hence, assume that $G_{\Delta}^{\Sigma}$ lies in a maximal $\mathcal{C}_{2}$-subgroup of of $N_{\Gamma L_{n}(q)}(X)$. Then $G_{\Delta}^{\Sigma}$ preserves a sum decomposition $V=V_{1}\oplus
\cdots \oplus V_{n/a}$, $n/a\geq 2$, with $V_{i} \cong V_{a}(q)$ for $i=1,...,n/a$. Hence, $G_{\Delta}^{\Sigma} \leq N$, where $N \cong \Gamma L_{a}(q) \wr S_{t}$ (see \cite[Section 4.2]{KL}). Now, $N=(L_{1} \times \cdots \times L_{n/a}):K$ with $L_{i} \cong \Gamma L_{a}(q)$ acting transitively on $V_{i}^{\ast}$ for each $i=1,...,n/a$ and $K \cong S_{t}$ acting transitively on $\{V_{1},..., V_{n/a}\}$. Thus $\bigcup_{i=1}^{n/a}V^{\ast}_{i}$ is a $N$-orbit, and hence it is a union of $G_{\Delta}^{\Sigma}$-orbits. Then $\frac{r_{1}}{(r_{1},\lambda_{1})}\mid \frac{n}{a}(q^{a}-1)$ by Theorem \ref{Teo3}(5) since the size of $\bigcup_{i=1}^{n/a}V^{\ast}_{i}$ is $\frac{n}{a}(q^{a}-1)$. On the other hand, $\frac{r_{1}}{(r_{1},\lambda_{1})} > q^{n/2}$ by Theorem \ref{Teo3}(4). Therefore $ q^{n/2}<\frac{n}{a}(q^{a}-1)$, which does not have admissible solutions for $n/a \geq 2$.

Assume that $G_{\Delta}^{\Sigma}$ lies in a maximal member of one of the geometric classes $\mathcal{C}_{i}$ of $N_{\Gamma L_{n}(q)}(X)$, $i=4,...,7$. Then $G_{\Delta}^{\Sigma}$ is as in Table \ref{AschClasses}.
\begin{table}[h!]
\tiny
\caption{Admissible geometric $G_{\Delta}^{\Sigma}$}\label{AschClasses}
\begin{tabular}{lcllll}
\hline
Line & $G_{\Delta}^{\Sigma}$ &  $v_{1}$ & $\frac{r_{1}}{(r_{1},\lambda _{1})}
$ & $A$ & Conditions \\
\hline
1 & $N(GL_{5}(2)\otimes GL_{3}(2))$ &  $2^{15}$ & $7\cdot
31$ & $151$ & \\ 
2 & $N(GL_{4}(p)\otimes GL_{3}(p))$ &  $p^{12}$ & $%
(p^{4}-1)(p^{2}+p+1)$ & $ \frac{p^{3}+1}{p+1} \cdot \frac{p^{6}+1}{p^{2+1}}$ &  \\ 
3 & $N(GL_{a}(p)\otimes GL_{2}(p))$ &  $p^{2a}$ & $%
2(p^{a}-1)$ & $\frac{p^{a}+1}{2}$ &  $a$ is even and $p$ is odd \\ 
4 &   &  & $\frac{(p^{a}-1)(p+1)}{\theta }$, $1\leq
\theta <p+1$ & $\frac{\left( p^{a}+1\right) \theta }{p+1}$ & $a \geq 3$ is odd \\ 
5 & $2_{\pm }^{1+6}\cdot O_{6}^{\pm }(2)$ &  $3^{8}$ & $%
160$ & $ 41$ & $n=8$ and $q=3$ \\ 
6 & $\mathbb{F}_{p^{2}}^{\ast }\circ 2_{\pm }^{1+2}\cdot O_{2}^{\pm
}(2)\cdot 2$ & $p^{4}$ & $2(p^{2}-1)$ & $\frac{p^{2}+1}{2%
}$ & $n=2$ and $q=p^{2}$ \\ 
7 & $\mathbb{F}_{p}^{\ast }\circ 2_{\pm }^{1+2}\cdot O_{2}^{\pm }(2)$ &  $p^{2}$ & $\frac{p^{2}-1}{\theta }$, $1\leq \theta \leq
p-1$ & $\theta $ & $n=2$ and $q=p$ \\ 
8 & $3^{1+2}\cdot Sp_{2}(3)\cdot 2$ & $4^{3}$ & $9$ & $%
7 $ & $n=3$ and $q=4$ \\ 
9 & $\mathbb{F}_{p^{2}}^{\ast }\circ 2_{\pm }^{1+4}\cdot O_{2}^{\pm
}(2)\cdot 2$ &  $p^{8}$ & $2(p^{4}-1)$ & $\frac{p^{4}+1}{2%
}$ & $n=4$, $q=p^{2}$, $p$ odd \\ 
10 & $\mathbb{F}_{p}^{\ast }\circ 2_{\pm }^{1+4}\cdot O_{4}^{\pm }(2)$ &  $p^{4}$ & $10(p^{2}-1)$ & $\frac{p^{2}+1}{10}$ & $n=4$, $%
q=p$, $p^{2} \equiv -1 \pmod{10}$\\
\hline
\end{tabular}
\end{table}
All columns in Table \ref{AschClasses}, except for the fourth one, are
obtained by transferring Liebeck's argument. The fourth one follows from $%
\frac{r_{1}}{(r_{1},\lambda _{1})}=\frac{v_{0}-1}{k_{0}-1}$ $v_{1}=A\frac{%
v_{0}-1}{k_{0}-1}+1$, as these hold by (\ref{double}) and Proposition \ref{P2}(3),
respectively. Note that, $\lambda \mid \left\vert G_{\Delta }^{\Sigma
}\right\vert $ since $r_{1}=\frac{v_{0}-1}{k_{0}-1}\cdot \frac{v_{0}}{\eta }%
\cdot \lambda $, $\lambda \nmid v_{1}(v_{1}-1)$ since $\mathcal{D}_{1}$ of
type Ia, and $\lambda >3$. Then the cases as in lines 2, 6, 7, 8, 9 and 10 are
excluded case do not fulfill all the previous constraints. Moreover,
cases 1 and 5 are excluded since they contradict $\lambda \mid \left\vert
G_{\Delta }^{\Sigma }\right\vert $ and $\lambda >A$ by Proposition \ref{P2}%
(4). 

Assume that the case as in line 3 of Table \ref{AschClasses} occurs. Then $\lambda \geq \frac{%
q^{a/2}+1}{2}+1$ by Proposition \ref{P2}(4). On the other hand, $\lambda
\mid q^{i/2}-(-1)^{i}$ with $i<a$ since $\lambda \mid \left\vert
G^{\Sigma}_{\Delta}\right\vert $ and $\lambda \nmid v_{1}(v_{1}-1)$ . So $\frac{%
q^{a}+1}{2}+1\leq \lambda \leq q^{a-1}-1$ with $q$ odd, which is a
contradiction.

Finally, assume that the case as in line 4 of Table \ref{AschClasses} occurs. Then $\frac{r_{1}}{(r_{1},\lambda _{1})}=\frac{%
(q^{a}+1)(q-1)}{\theta }$ for some $\theta \geq 1$. Hence $A=\frac{%
\left( q^{a}+1\right) \theta }{q+1}$ and $\lambda \geq A+1=\frac{%
\left( q^{a}+1\right) \theta }{q+1}+1>\frac{q^{a}+1}{q+1}$.
However, $\lambda \mid \frac{q^{i}-1}{q-1}$ with $i<a$
since $\lambda \mid \left\vert G^{\Sigma}_{\Delta}\right\vert $, $\lambda \nmid v_{1}(v_{1}-1)$
and $\lambda $ is a prime. So $\frac{q^{a}+1}{q+1}<\frac{%
q^{a-1}-1}{q-1}$, which is impossible for $a\geq 3$. Thus, case (I) is ruled out. 
\end{proof}

\bigskip

\begin{proof}[Proof of Theorem \ref{Teo4}]
The group $\left(G_{\Delta}^{\Sigma}\right)^{(\infty)}$ is a quasisimple group, and its action on $V_{n}(q)$ is absolutely irreducible and not realizable over any proper subfield of $\mathbb{F}_{q}$. Set $W=G^{(\infty)}_{0}/Z(G^{(\infty)}_{0})$ by Proposition \ref{ReductionNAS}. Hence, we may adapt the aforementioned Liebeck's argument to our context since 
\begin{equation}\label{rustico}
\frac{r_{1}}{(r_{1},\lambda _{1})}\mid (q^n-1,c_{1},...,c_{j},(q-1)\cdot|Aut(W)|)
\text{~and~}
\frac{r_{1}}{(r_{1},\lambda _{1})}>q^{n/2}\text{.}
\end{equation}
by Theorem \ref{Teo3}(4)--(5), and we obtain that $\left(G_{\Delta}^{\Sigma}\right)^{(\infty)}$ and the corresponding $v_{1}$ is as in Tables \ref{AltSpor}, \ref{LieNat} \ref{LieCross} according $W$ is alternating, sporadic, or of Lie type in natural characteristic, or of Lie type in cross characteristic, respectively. Moreover, since $G_{\Delta}^{\Sigma}\leq \left(Z_{q-1}\circ \left(G_{\Delta}^{\Sigma}\right)^{(\infty)}\right).Out (L)$, and $\lambda$ does not divide the order of $Out (L)$ by Lemma \ref{No2transitive}(1), it follows that $\lambda$ divides the order of $Z_{q-1}\circ \left(G_{\Delta}^{\Sigma}\right)^{(\infty)}$. Actually, $\lambda$ divides the order of $\left(G_{\Delta}^{\Sigma}\right)^{(\infty)}$, and hence the order of $L$, since $\lambda \nmid v_{1}-1$ being $\mathcal{D}_{1}$ of type Ia. 

Suppose that $W$ is alternating or sporadic. As mentioned above $L$, $v_{1}$ and $\frac{r_{1}}{(r_{1},\lambda _{1})}$ are as in the second, third and fourth column of Table \ref{AltSpor}, respectively, by adapting Liebeck's argument. The entry in the fourth column is the greatest prime divisor of $L$ greater than $3$ and candidate to be $\lambda$. Now, $A=(v_{1}-1)\cdot \left( \frac{r_{1}}{(r_{1},\lambda _{1})} \right)^{-1}$ by Proposition \ref{P2}(3). Except for that as in line 3, line 10 with $1 \leq \theta <7$, line 11 with $(q,\theta,A)=(2,1,1),(2,2,2)$ and $\lambda=7$, lines 13--14 with $\theta <5$,  and those as in lines 15 or 16 with $2 \leq p\leq 7$, all cases in Table \ref{AltSpor} are such that $ \lambda \leq \lambda_{\text{max}} \leq A$, and hence they are excluded by Proposition \ref{P2}(3). 

In case as in line 3, one has $\lambda=5$ since $\lambda_{\text{max}}=5$ and $\lambda=3$. However, this case is excluded since $\frac{r_{1}}{(r_{1},\lambda _{1})}=\lambda=5$, contrary to Theorem \ref{Teo3}. 

In case as in line 10, one has $\lambda =5$ or $7$. Moreover, $\frac{r_{1}}{%
(r_{1},\lambda _{1})}=\frac{p^{2}-1}{\theta }>\lambda $ implies $%
p^{2}>\lambda \theta +1$, and hence $v_{1}>\left( \lambda \theta +1\right)
^{3}$. Also, $\frac{r_{1}}{(r_{1},\lambda _{1})}=\frac{v_{0}-1}{k_{0}-1}$
leads $v_{0}>\lambda (k_{0}-1)+1$. So%
\[
\left( \lambda \theta +1\right) ^{3}\left( \lambda (k_{0}-1)+1\right) <v\leq
2\lambda ^{2}(\lambda -1)\text{,}
\]%
which is contrary to $\lambda =5$ or $7$.

In case as in line 11, the action of $G^{\Sigma }$ on $\mathcal{D}_{1}$ is
either $2$-transitive or (primitive) rank $3$ by \cite[Corolary 4.2]{Ka0}
according as $(q,\theta ,A)=(2,1,1)$ or $(q,\theta ,A)=(2,2,2)$,
respectively. The former is excluded by Theorem \ref{NoAff2tr}, the latter by \cite[%
Table 2]{Lieb2}.

In cases as in lines 13--14, one has $\lambda =5$, $\theta <5$. Moreover, $%
\frac{r_{1}}{(r_{1},\lambda _{1})}>\lambda $ implies $\left( p-1\right)
\cdot 2\cdot 6!\geq \left( p-1\right) \cdot 5!>\theta \lambda >25$ which is
impossible.

Finally, in cases as in lines 15--16, one has $\lambda =5$ and hence $p=2$
or $3$ since $v_{1}\leq v\leq 2\lambda ^{2}(\lambda -1)$ by \cite[Theorem 1]{DP}. However, in both
cases, it results $\lambda \mid v_{1}-1$ which is contrary to Theorem \ref{Teo1} since $\mathcal{D%
}_{1}$ is of type Ia. Thus, $L$ is neither alternating nor sporadic.

\begin{table}[h!]
\tiny
\caption{Admissible cases with $L$ alternating or sporadic}\label{AltSpor}
\begin{tabular}{lllllll}
\hline
Line & $L$ &  & $v_{1}$ & $\frac{r_{1}}{(r_{1},\lambda _{1})}$ & $A$ & 
$\lambda_{\text{max}}$ \\
\hline
1 & $A_{12}$ &  & $2^{10}$ & $33$ & $ 31$ & $11$ \\ 
2 & $A_{14}$ &  & $2^{12}$ & $91$ & $ 45$ &  $13$\\ 
3 & $A_{5}$ &  & $3^{4}$ & $10$ & $ 8$ &  $5$\\ 
4 & $A_{7}$ &  & $4^{6}$ & $105$, $315$ & $ 39$, $
13 $ &  $7$\\ 
5 & $A_{6}$ &  & $4^{3}$ & $9$ & $ 7$ &  $5$\\ 
6 & $A_{6}$ &  & $2^{4}$ & $5$ & $3$ &  $5$\\ 
7 & $A_{m}$, $m=8,9,10,11$ &  & $3^{8}$ & $160$ & $ 41$ & $7,7,7,11$  \\ 
8 & $A_{7}$ &  & $25^{3}$ & $504,252,168$ & $ 31$, $
62$, $ 93$ &  $7$\\ 
9 & &  & $3^{8}$ & $160$ & $41$ &  $7$\\ 
10 &  &  & $p^{4}$ & $\frac{p^{2}-1}{\theta }$, $\theta \leq p-1$ & $\theta $
&  $7$\\ 
11 & &  & $q^{6}$, $q<21$ & $\frac{21(q^{2}-1)}{\theta }$ & $\frac{\left(
q^{4}+q^{2}+1\right) \theta }{21}$ &  $7$\\ 
12 & &  & $3^{8}$ & $160$ & $41$ & $7$ 
\\ 
13 & $A_{6}$ &  & $p^{4}$ & $\frac{\left( p-1\right) \cdot 2\cdot 6!}{%
\theta }$ & $\theta $ & $5$  \\ 
14 & $A_{5}$ &  & $p^{2}$ & $\frac{\left( p-1\right) \cdot 5!}{\theta }$ & 
 $\theta $ &  $5$\\ 
15 &  &  & $p^{4}$ & $\frac{\left( p-1\right) \cdot 5!}{\theta }$ & $\frac{%
(p^{2}+1)(p-1)\theta }{5!}$ &  $5$ \\ 
16 &  &  & $p^{4}$ & $\frac{\left( p^{2}-1\right) \cdot 5!}{\theta }$ & $%
\frac{(p^{2}+1)(p-1)\theta }{5!}$ &  $5$\\ 
17 & $Suz$ &  & $3^{12}$ & $\frac{7280}{\theta }$, $\theta =1,2,4,5,7,8$ & $%
 73\theta $ &  $13$\\ 
18 & $J_{2}$ &  & $19^{6}$ & $7560$ & $ 6223$ &  $7$\\ 
19 &  &  & $17^{6}$ & $6048$ & $ 3991$ &  $7$\\ 
20 &  &  & $11^{6}$ & $2520$ & $ 703$ &  $7$\\ 
21 &  &  & $5^{6}$ & $\frac{504}{\theta }$, $\theta \leq 4$ & $
31\theta $ & $7$ \\ 
22 &  &  & $4^{6}$ & $105,315$ & $ 39$, $ 13$ &  $7$\\ 
23 & $M_{11}$ &  & $3^{5}$ & $22$ & $ 11$ &  $11$\\ 
24 &  &  & $2^{10}$ & $33$ & $ 31$ &  $11$\\ 
25 & $M_{12}$ &  & $2^{10}$ & $33$ & $ 31$ &  $11$\\ 
26 & &  & $2^{12}$ & $315,105$ & $ 39$, $ 13$
&  $11$\\ 
27 & $M_{22}$ &  & $4^{6}$ & $315,105$ & $39$, $ 13$ & $11$ \\ 
28 &  &  & $2^{12}$ & $315,105$ & $39$, $ 13$ & $11$ \\
\hline
\end{tabular}%
\end{table}

Assume that $L$ of Lie type in characteristic $p$. Then $L$, $v_{1}$ and $\frac{r_{1}}{(r_{1},\lambda _{1})}$ are as in the second, third and fourth column of Table \ref{LieNat}, respectively, by adapting Liebeck's argument. At this point $A$ is as in the fifth column of Table \ref{LieNat} by Proposition \ref{P2}(3). If $L$ is as in line 2 of Table \ref{LieNat}, then $\lambda=5,7,$ or $31$ since $\lambda$ divides the order of $L$. Therefore $\lambda <A=151$, and hence this case is excluded by Proposition \ref{P2}(4). Further, $L$ cannot be as in line 10 since in this case $\lambda \mid \left\vert L \right\vert$ implies $\lambda \mid v_{1}$, and this contradicts $\mathcal{D}_{1}$ of type Ia. In the remaining cases, we have $v=q^{2i}$, $\frac{r_{1}}{(r_{1},\lambda _{1})}=2(q^{i}-1)$ and $A=\frac{q^{i}+1}{2}$ with $i=4,5$ or $8$. In particular, $\lambda >A=\frac{q^{i}+1}{2} $ by Proposition \ref{P2}(4). However, this is impossible since $\lambda \leq \frac{q^{i-1}-1}{q-1} \leq \frac{q^{i}}{2}-1$ because $\lambda$ is a prime dividing $\left\vert L \right\vert$ and but not $v_{1}(v_{1}-1)$.

\begin{table}[h!]
\tiny
\caption{Admissible cases with $L$ of Lie type in characteristic $p$}\label{LieNat}
\begin{tabular}{lllll}
\hline
Line & $L$ & $v_{1}$ & $\frac{r_{1}}{(r_{1},\lambda _{1})}$ & $A$ \\
\hline
1 & $PSL_{5}(q)$, $PSU_{5}(q^{1/2})$ & $q^{10}$ & $2(q^{5}-1)$ & $\frac{q^{5}+1%
}{2}$ \\ 
2 & $L_{6}(2)$ & $2^{15}$ & $7\cdot 31$ & $ 151
$ \\ 
3 & $\Omega _{7}(q)$ & $q^{8}$ & $2(q^{4}-1)$ & $\frac{q^{4}+1}{2}$ \\ 
4 & $P\Omega _{8}^{+}(q)$ & $q^{8}$ & $2(q^{4}-1)$ & $\frac{q^{4}+1}{2}$ \\ 
5 & $P\Omega _{8}^{-}(q^{1/2})$ & $q^{8}$ & $2(q^{4}-1)$ & $\frac{q^{4}+1}{2}
$ \\ 
6 & $\Omega _{9}(q)$ & $q^{16}$ & $2(q^{8}-1)$ & $\frac{q^{8}+1}{2}$ \\ 
7 & $P\Omega _{10}^{+}(q)$ & $q^{16}$ & $2(q^{8}-1)$ & $\frac{q^{8}+1}{2}$ \\ 
8 & $P\Omega _{10}^{-}(q)$ & $q^{16}$ & $2(q^{8}-1)$ & $\frac{q^{8}+1}{2}$ \\ 
9 & $^{3}D_{4}(q^{1/3})$ & $q^{8}$ & $2(q^{4}-1)$ & $\frac{q^{4}+1}{2}$ \\ 
10 & $Sz(q)$ & $q^{4}$ & $\frac{(q^{2}+1)(q-1)}{\theta }$, $\theta \leq q-1$
& $(q+1)\theta $\\
\hline
\end{tabular}%
\end{table} 

Finally, assume that $L$ of Lie type in characteristic $p^{\prime}$. Arguing as above, we obtain Table \ref{LieCross}. Except for cases as in lines 7--8 with $\theta <13$, 13 with $q=3$, or 23 with $p=5$ are excluded since $\lambda \leq \lambda_{\text{max}}<A$ thus contradicting Proposition \ref{P2}(4). Further, the case as in Line 7 is ruled out since it contradicts $\lambda \nmid v_{1}-1$. In case as in line 8, we actually have $\theta =1$ or $3$ again by Proposition \ref{P2}(4). The former is excluded by \cite[Corollary 4.2]{Ka0} and Theorem \ref{NoAff2tr}. The latter implies $v_{0}=5461(k_{0}-1)+1$ and hence $5461 \leq v_{0}\leq v$, whereas $v \leq 4056$ by \cite[Theorem 1]{DP}. Finally, the cases as in lines 13 with $q=3$, or 23 with $p=5$ are excluded by \cite{AtMod}.

\begin{table}[h!]
\tiny
\caption{Admissible cases with $L$ of Lie type in characteristic $p^{\prime}$}\label{LieCross}
\begin{tabular}{llllll}
\hline
Line & $L$ & $v_{1}$ & $\frac{r_{1}}{(r_{1},\lambda _{1})}$ & $A$ & $\lambda_{\text{max}}$ \\
\hline
1 & $PSL_{2}(25)$ & $2^{12}$ & $65,195$ & $ 63$, $ 21$
& $13$\\ 
2 & $PSL_{2}(19)$ & $4^{9}$ & $513$ & $ 511$ & $19$\\ 
3 & $PSL_{2}(17)$ & $2^{8}$ & $17,51$ & $ 15$, $ 5$ & $17$ \\ 
4 & $PSL_{2}(13)$ & $3^{12}$ & $1456$ & $ 365$ & $13$\\ 
5 &               & $p^{6}$, $p \geq 3$ & $\frac{91(p^{2}-1)}{\theta }$, $\theta \geq 1$ & $\frac{%
(p^{4}+p^{2}+1)\theta}{91}$ & \\ 
6 &           & $4^{6}$ & $273$ & $ 15$ &  \\ 
7 &           & $8^{12}$ & $\frac{8^{12}-1}{\theta}$, $\theta \geq 1$ & $\theta$ &  \\
8 &           & $2^{14}$ & $\frac{2^{14}-1}{\theta}$, $\theta \geq 1$ & $ \theta$ &  \\ 
9 & $PSL_{2}(11)$ & $3^{5}$ & $22$ & $ 11$ & $11$ \\ 
10 &  & $2^{10}$ & $33$ & $ 31$ & \\ 
11 & $PSL_{2}(7)$ & $5^{6}$ & $168$ & $ 93$ & $7$\\ 
12 &  & $3^{6}$ & $28,56$ & $ 26$, $ 13$ &\\ 
13 &  & $q^{4}$, $(q,2)=(q,7)=1$ & $2(q^{2}-1)$ & $\frac{q^{2}+1}{2}$ &\\ 
14 &  & $9^{3}$ & $28,56$ & $ 26$, $ 13$ &\\ 
15 &  & $11^{3}$ & $70$ & $ 19$ &\\ 
16 &  & $23^{3}$ & $154$ & $ 79$ &\\ 
17 &  & $25^{3}$ & $126,252$ & $ 124$, $ 62$ & \\ 
18 &  & $q^{3}$, $q\geq 37$ & $\frac{21(q-1)}{\theta }$, $\theta =1,2,3$ & $%
\frac{(q^{2}+q+1)\theta }{21}$ &\\ 
19 & $PSL_{3}(3)$ & $2^{12}$ & $117$ & $ 35$ & $13$ \\ 
20 & $PSL_{3}(4)$ & $p^{6},p\geq 5$ & $\frac{21(p^{2}-1)}{\theta }$ & $%
\frac{(p^{4}+p^{2}+1)\theta }{21}$ & $7$\\ 
21 &  & $9^{4}$ & $160$ & $ 41$ & \\ 
22 &  & $3^{6}$ & $28,56$ & $ 26$, $ 13$ &\\ 
23 & $PSp_{4}(3)$ & $p^{4}$, $p \geq 5$ & $\frac{10(p^{2}-1)}{\theta }$, $\theta \geq 1$ & $\frac{%
(p^{2}+1)\theta }{10}$ & $5$\\ 
24 & $Sp_{6}(2)$ & $3^{8}$ & $160$ & $41$ & $7$\\ 
25 & $PSU_{3}(3)$ & $5^{6}$ & $126,252,504$ & $ 124$, $%
 62$, $ 31$ & $7$\\ 
26 & $\Omega _{8}^{+}(2)$ & $3^{8}$ & $160$ & $ 41$ & $17$\\ 
27 & $Sz(8)$ & $5^{8}$ & $1248$ & $ 313$ & $13$\\ 
28 & $G_{2}(4)$ & $3^{12}$ & $\frac{7280}{\theta }$, $1\leq \theta \leq 8$ & 
$ 73\theta $ & $17$\\
\hline
\end{tabular}%
\end{table}   
\end{proof}

\bigskip

\bigskip

\section{Reduction to the case where $\mathcal{D}_{1}$ is a $1$-symmetric design with $k_{1}=v_{1}-1$}.

The aim of this section is to prove the following reduction theorem.

\bigskip

\begin{theorem}\label{Teo5}
Assume that Hypothesis \ref{hyp1} holds. Then $\mathcal{D}_{1}$ is a $1$-symmetric design with $k_{1}=v_{1}-1$, and the parameters for $\mathcal{D}_{0}$, $\mathcal{D}_{1}$, and $\mathcal{D}$ are those recorded in Table \ref{D1sym} by Lemma \ref{L2bis}.
\end{theorem}

\bigskip

To this end, we assume that $T=Soc(G^{\Sigma})$ is non-abelian simple, which is the remaining case of Theorem \ref{Teo4} to be investigated when $\mathcal{D}_{1}$ is a $2$-design. We analyze the cases where $T$ is sporadic, alternating, a Lie type simple classical or exceptional group in sections.

\bigskip

\begin{lemma}\label{Tlarge}Assume that Hypothesis \ref{hyp2} holds. Then the following hold:
\begin{enumerate}
    \item $\left\vert T \right\vert \leq \left \vert T_{\Delta} \right \vert ^{2} \cdot \left \vert Out(T)\right \vert$;
    \item $\lambda \mid \left\vert T_{\Delta}\right\vert$ and $\lambda \mid \left\vert Out(T)\right\vert $;
    \item $T_{\Delta}$ is a large subgroup of $T$.
\end{enumerate}    
\end{lemma}
\begin{proof}
It follows from Theorem \ref{Teo4}(2) that $\left\vert T \right\vert \leq \left \vert T_{\Delta} \right \vert ^{2} \cdot \left \vert Out(T)\right \vert$, which is (1).

Since $r_{1}=\frac{v_{0}-1}{k_{0}-1}\cdot \frac{v_{0}}{k_{0}\eta}\cdot \lambda$ divides $\left\vert G^{\Sigma}_{\Delta}\right\vert$, and $\left\vert G^{\Sigma}_{\Delta}: T_{\Delta}\right\vert \mid \left\vert Out(T)\right\vert $, then either $\lambda \mid \left\vert T_{\Delta}\right\vert$ or $\lambda \mid \left\vert Out(T)\right\vert $. Suppose that the latter occurs. Note that $v_{0} >\lambda +1$ since $\frac{r_{1}}{(r_{1},\lambda_{1})}=\frac{v_{0}-1}{k_{0}-1}$ by (\ref{double}) and $\frac{r_{1}}{(r_{1},\lambda_{1})}> \lambda$ by Theorem \ref{Teo4}. If $v_{1} \geq 2\lambda
^{2}$, then $v=v_{1}v_{0} \geq 2\lambda^{2}(\lambda+1)$ since $v_{0} >\lambda +1$,, but this contradicts \cite[Theorem 1]{DP}. Thus, $P(T) \geq v_{1} < 2\lambda
^{2}$, where $P(T)$ is the minimal degree 
of the non-trivial transitive permutation representations of $T$. Hence, $T$ is
neither alternating or sporadic since $\left\vert Out(T)\right\vert \leq 4$
in these cases, whereas $\lambda $ is a prime and $\lambda >3$. 

Suppose that $T\cong PSL_{n}(q)$, $n\geq 2$, and $(n,q)\neq (2,2),(2,3)$.
Then $\lambda \mid (n,q-1)\cdot \log _{p}q$, and hence $(n,q)\neq
(2,4),(2,5),(2,7),(2,9),(2,11)$ or $(4,2)$ since $\lambda $ is a prime and $%
\lambda >3$. Thus, $P(T)=\frac{q^{n}-1}{q-1}$ by \cite[Theorem 5.2.2]{KL}.
If $\lambda \mid (n,q-1)$, then $\frac{q^{n}-1}{q-1}\leq 2(q-1)^{2}$ and so $%
\lambda \leq n\leq 4$, a contradiction. Therefore $\lambda \mid \log _{p}q$,
and hence $\frac{q^{n}-1}{q-1}\leq 2q$ since $\log _{p}q\leq q^{1/2}$.
Then $n=2$ and $q+1\leq 2 \log _{p}^{2}q$, which is impossible since $q>1$.

Suppose that $T\cong PSp_{n}(q)$, $n\geq 2$. We may assume that $n \geq 4$ since the assertion follows fro $n=2$ since $%
PSp_{2}(q)\cong PSL_{2}(q)$. Moreover, $q>2$
and $(q,n)\neq (4,3)$ otherwise $\left\vert Out(T)\right\vert \leq 2$, whereas $\lambda >3$. Then $P(T)=\frac{q^{n}-1}{q-1}$ by \cite[Theorem 5.2.2]{KL}, and we reach a contradiction as above. An entirely
similar proof the the previous one rules out $\Omega _{n}(q)$ with $n$ odd, $%
P\Omega _{n}^{\varepsilon }(2)$ with $\varepsilon =\pm $, $U_{3}(5)$ and $%
U_{n}(2)$.  

Suppose that either $T\cong P\Omega _{n}^{\varepsilon }(q)$, $\varepsilon
=\pm $ and $q>2$ and $n\geq 8$, or $T\cong PSU_{n}(q)$ with $n,q\geq 3$ and $%
(n,q)\neq (3,5)$. Then either $P(T)\geq \frac{(q^{n/2}+1)(q^{n/2-1}-1)}{q-1}$
or $P(T)\geq q^{3}+1$, respectively. Therefore, either $\frac{%
(q^{n/2}+1)(q^{n/2-1}-1)}{q-1}\leq 2\log _{p}^{2}q\leq 2q$ or $%
q^{3}+1\leq 2q$, respectively, and both inequalities do not have admissible
solutions.

Finally, if $T$ is an exceptional Lie type simple group, then $P(T)$ is provided in \cite%
{Va1,Va2,Va3} and it is easy to see that no cases fulfill $P(T)\leq
v_{1}\leq 2\lambda ^{2}(\lambda -1)$. This proves (2).

Suppose that $\left \vert T_{\Delta} \right \vert \leq \left \vert Out(T)\right \vert$. Then $P(T) \leq \left\vert T:T_{\Delta}\right\vert \leq \left \vert Out(T)\right \vert^{2}$, which is (\ref{twotimes}), and hence $T\cong PSL_{n}(q)$ with $(n,q)=(2,9)$ or $(3,4)$.

Assume that $T \cong PSL_{2}(9)$. Then $v_{1}=6,10$ or $15$ by \cite{At} and $\lambda =5$. Actually, $v_{1}=10$ or $15$ since $\frac{v_{0}-1}{k_{0}-1}>\lambda $. So $\lambda \mid v_{1}$, contrary to Theorem \ref{Teo1} since $\mathcal{D}_{1}$ is of type Ia or Ic.

Finally, assume that $T \cong PSL_{3}(4)$. Then $\lambda=5$ or $7$. Further, $v_{1}=21,56$ or $120$ by \cite{At} since $v_{1}\leq \left \vert Out(T)\right \vert^{2}=144$. Then \ref{Teo1} since $\lambda \mid v_{1}(v_{1}-1)$ and $v_{1} \neq \frac{1}{2}(\lambda-1)(\lambda^{2}-2)$, whereas $\mathcal{D}_{1}$ is of type Ia or Ic.

\end{proof}

\bigskip

\bigskip

\subsection{Novelties.} An important tool in carrying out the following analysis is the notion of novelty: a maximal subgroup $M$ of an almost simple group $A$ is a \emph{novel} maximal subgroup (or, simply, a \emph{novelty}) if $M \cap Soc(A)$ is non-maximal in $Soc(A)$. More information on novelties can be found in \cite{BHRD, KL, Wi1}. 
\bigskip

\subsection{The case where $T$ is sporadic} In this section, we assume that $T$ is a sporadic simple group. 

\bigskip

\begin{lemma}
\label{NoNovelSpor}Assume that Hypothesis \ref{hyp2} holds. If $T$ is a sporadic group, then $T$ acts
point-primitively on $\mathcal{D}_{1}$.
\end{lemma}

\begin{proof}
Suppose the contrary. Hence, $T_{\Delta }$ is non-maximal in $S$. On the
other hand, $G_{\Delta }^{\Sigma }$ is maximal in $G^{\Sigma }$. Thus $%
G_{\Delta }^{\Sigma }$ is a novelty, and hence $T$ and $G_{\Delta }^{\Sigma }
$ are as in the first and second column of Table \ref{Nov1}, respectively, by \cite[%
Table 1]{Wi1} and \cite[Section 4]{Wi2}. Then $v_{1}=\left\vert G^{\Sigma
}:G_{\Delta }^{\Sigma }\right\vert $ as in the third column of Table \ref{Nov1} and
easy computations lead to the fourth column of the same table.   
\begin{table}[h!]
\tiny
\caption{Admissible novelties and sporadic automorphism groups of $\mathcal{D}_{1}$}\label{Nov1}
\begin{tabular}{clllc}
\hline
Line & $T$ & Novelties & $v_{1}$ & $(v_{1}-1,\left\vert G_{\Delta
}^{\Sigma }\right\vert )$ \\ 
\hline
1 & $M_{12}$ & $PGL_{2}(11)$ & $ 2^{4}\cdot 3^{2}$ & $  11$
\\ 
2 & & $3_{+}^{1+2}:D_{8}$ & $  2^{4}\cdot 5\cdot 11$ & $  3$
\\ 
3& & $S_{5}$ & $  2^{4}\cdot 3^{2}\cdot 11$ & $  1$ \\ 
4& $J_{3}$ & $19:18$ & $    2^{7}\cdot 3^{3}\cdot 5\cdot 17$ & $  19$ \\ 
5& $M^{c}L$ & $2^{2+4}:(S_{3}\times S_{3})$ & $  3^{4}\cdot 5^{3}\cdot
7\cdot 11$ & $  8$ \\ 
6 & $Fi_{22}$ & $G_{2}(3):2$ & $  2^{11}\cdot 3^{3}\cdot 5^{2}\cdot 11$
& $  13$ \\ 
7 & & $3^{5}:(U_{4}(2):2\times 2)$ & $  2^{10}\cdot 5\cdot 7\cdot
11\cdot 13$ & $  3$ \\ 
8& $O^{\prime }N$ & $7_{+}^{1+2}:(3\times D_{16})$ & $  2^{6}\cdot
3^{3}\cdot 5\cdot 11\cdot 19\cdot 31$ & $  7$ \\ 
9& & $31:30$ & $  2^{9}\cdot 3^{3}\cdot 7^{3}\cdot 11\cdot 19$ & $%
  31$ \\ 
10& & $PGL_{2}(9)$ & $  2^{6}\cdot 3^{2}\cdot 7^{3}\cdot 11\cdot
19\cdot 31$ & $  1$ \\ 
11& & $PGL_{2}(7)$ & $  2^{6}\cdot 3^{3}\cdot 5\cdot 7^{2}\cdot
11\cdot 19\cdot 31$ & $  1$ \\ 
12& $HS$ & $5_{+}^{1+2}.[2^{5}]$ & $  2^{5}\cdot 3^{2}\cdot 7\cdot 11$
& $  25$ \\ 
13& $He$ & $(S_{5}\times S_{5}):2$ & $    2^{4}\cdot 3\cdot
7^{3}\cdot 17$ & $  1$ \\ 
14& & $2^{4+4}:3^{2}:D_{8}$ & $  2^{2}\cdot 3\cdot 5^{2}\cdot
7^{3}\cdot 17$ & $  1$ \\ 
15& $Fi_{24}^{\prime }$ & $7_{+}^{1+2}:(6\times S_{3}):2$ & $ 
2^{19}\cdot 3^{14}\cdot 5^{2}\cdot 11\cdot 13\cdot 17\cdot 23\cdot 29$ & $7$\\
\hline  
\end{tabular}%
\end{table}
Now, $\frac{%
r_{1}}{(r_{1},\lambda _{1})}\mid (v_{1}-1,\left\vert G_{\Delta }^{\Sigma
}\right\vert )$ implies that\ $\frac{r_{1}}{(r_{1},\lambda _{1})}$ is too
small to satisfy $\left( \frac{r_{1}}{(r_{1},\lambda _{1})}\right) ^{2}>v_{1}
$. Then none of the cases as in Table \ref{Nov1} occur since they contradict Theorem \ref{Teo4}(3).
Thus, $T$ acts point-primitively on $\mathcal{D}_{1}$.
\end{proof}

\begin{lemma}\label{NoSporadic}
Assume that Hypothesis \ref{hyp2} holds. Then $T$ is not a sporadic group.
\end{lemma}

\begin{proof}
We filter the list of maximal subgroups of the sporadic groups given in \cite{At,Wi2} with respect to 
\begin{equation}\label{dadday}
\left\vert T \right \vert \leq \left\vert T_{\Delta}\right\vert ^{2} \cdot \left\vert Out(T)\right\vert \text{ and } \frac{r_{1}}{(r_{1},\lambda_{1})} \mid (v_{1}-1,\left\vert G_{\Delta
}^{\Sigma }\right\vert )  \text{ and } \left(\frac{r_{1}}{(r_{1},\lambda_{1})}\right)^2>v_{1}
\end{equation}
(see Theorem \ref{Teo4}(2)--(4)), thus obtaining Table \ref{spor}.  
\begin{table}[h!]
\tiny
\caption{Sporadic groups fulfilling (\ref{dadday})}\label{spor}
\begin{tabular}{clllccc}
\hline
Line & $T$ & $T_{\Delta }$ & $v_{1}$ & $(v_{1}-1,\left\vert G_{\Delta
}^{\Sigma }\right\vert )$ & Conjugacy $T$-Classes of $T_{\Delta}$ & $\lambda $ \\ 
\hline
1& $M_{11}$ & $M_{10}$ & $11$ & $ 10$ & $1$ & $5,11$ \\ 
2& & $PSL_{2}(11)$ & $12$ & $11$ & $1$ & $5,11$ \\ 
3& & $3^{2}:Q_{8}.2$ & $55$ & $ 18$ & $1$ & $5,11$ \\ 
4& $M_{12}$ & $M_{11}$ & $12$ & $ 11$ & $2$ & $5,11$ \\ 
5& $M_{22}$ & $A_{7}$ & $176$ & $ 35$ & $2$ & $5,7,11$ \\ 
6& & $PSL_{3}(4)$ & $22$ & $ 21$ & $1$ & $5,7,11$ \\ 
7& $M_{23}$ & $M_{11}$ & $1288$ & $ 99$ & $1$ & $5,7,11,23$ \\ 
8& & $2^{4}:A_{7}$ & $253$ & $252$ & $1$ & $5,7,11,23$ \\ 
9& &  $PSL_{3}(4):2_{2}$ & $253$ & $252$ & $1$ & $5,7,11,23$ \\ 
10& & $M_{22}$ & $23$ & $ 22$ & $1$ & $5,7,11,23$ \\ 
11& $M_{24}$ & $M_{12}:2$ & $1288$ & $ 99$ & $1$ & $5,7,11,23$ \\ 
12& & $M_{22}:2$ & $276$ & $ 55$ & $1$ & $5,7,11,23$ \\ 
13& & $M_{23}$ & $24$ & $ 23$ & $1$ & $5,7,11,23$ \\ 
14& $M^{c}L$ & $M_{22}$ & $2025$ & $ 88$ & $2$ & $5,7,11$ \\ 
15& $Ly$ & $G_{2}(5)$ & $8835156$ & $1085$ & $1$ & $5,7,11,31,37,67$ \\ 
16& $HS$ & $P\Sigma U_{3}(5)$ & $176$ & $ 175$ & $2$ & $5,7,11$ \\ 
17& & $M_{22}$ & $100$ & $ 99$ & $1$ & $5,7,11$ \\
\hline
\end{tabular}
\end{table}
Then we use the fact that $\mathcal{D}_{1}$ is of type Ia or Ic (see Theorem \ref{Teo1}) together with $r_{1}/(r_{1},\lambda_{1})>\lambda$, and we obtain Table \ref{spor2}. 
\begin{table}[h!]
\tiny
\caption{Admissible sporadic groups as automorphism groups of $\mathcal{D}_{1}$}\label{spor2}
\begin{tabular}{clllccc}
\hline
Line & $T$ & $T_{\Delta }$ & $v_{1}$ & $r_{1}/(r_{1},\lambda_{1})$ &  $\lambda $ \\ 
\hline
1& $M_{11}$ & $PSL_{2}(11)$ & $12$ & $11$ &  $5$ \\ 
4& $M_{12}$ & $M_{11}$ & $12$ & $ 11$ &  $5$ \\ 
5& $M_{22}$ & $PSL_{3}(4)$ & $22$ & $ 7,21$ &  $5$ \\ 
7& $M_{23}$ & $M_{11}$ & $1288$ & $ 9,33,99$ &  $5$ \\ 
8& & $2^{4}:A_{7}$ & $253$ & $252$ &  $5$ \\ 
9& & $PSL_{3}(4):2_{2}$ & $253$ & $252$ &  $5$ \\ 
10& & $M_{22}$ & $23$ & $ 22$ & $5,7$ \\ 
11& $M_{24}$ & $M_{12}:2$ & $1288$ & $ 99$ &  $5$ \\ 
12& & $M_{22}:2$ & $276$ & $ 55$ &  $7$ \\ 
13& & $M_{23}$ & $24$ & $ 23$ &  $5,7,11$ \\ 
14& $M^{c}L$ & $M_{22}$ & $2025$ & $ 88$ &  $7$ \\ 
15& $Ly$ & $G_{2}(5)$ & $8835156$ & $1085$ &  $7$ \\ 
16& $HS$ & $M_{22}$ & $100$ & $ 99$ & $7,11$ \\
\hline
\end{tabular}
\end{table}
It is worthwhile noting that none of cases as in Table \ref{spor2} satisfies $v_{1}=\frac{1}{2}\left(\lambda-1\right)\left(\lambda ^{2}-2\right)$, and hence in none of them $\mathcal{D}_{1}$ is of type Ic. Thus, $\mathcal{D}_{1}$ is of type Ia. Moreover, $\lambda=5,7$ or $11$.

It is easy to determine all the admissible pairs $(A,k_{0})$ since $\lambda \geq A(k_{0}-1)+1$ by Proposition \ref{P2}(4) for each $\lambda$. Now, for each $v_{1}$ and $(A,k_{0})$ we derive the corresponding value $z$ since $v_{1}=A(zk_{0}+1)+1$ by Theorem \ref{Teo1}. Finally, we substitute the values of the triple $(A,k_{0},z)$ in (\ref{fundamental}) bearing in mind that $\mathcal{D}_{1}$ is a non-trivial $2$-design with $k_{1}<v_{1}<2k_{1}$, and we see that the unique admissible case is that as in Line 13 of Table \ref{spor2}. In this case, one has $(\lambda,A,z,k_{0},v_{0},v_{1},k_{1})=(11,1,2,11,231,24,22)$ and so $\lambda \mid v_{0}$, contrary to Theorem \ref{Teo1}. This completes the proof.

\end{proof}

\bigskip

\subsection{The case where $T$ is alternating} In this section, we assume that $T \cong A_{n}$, $n \geq 5$. 

\bigskip

\begin{lemma}\label{AltCompleteDes}
Assume that Hypothesis \ref{hyp2} holds. If $T \cong A_{n}$, $n \geq 5$, then $v_{1}=n$ and $\mathcal{D}_{1}$ is the complete $2$-$\left(v_{1},k_{1},\binom{v_{1}-2}{k_{1}-2}\right)$ design.
\end{lemma}

\begin{proof}
Assume that $T\cong A_{n}$ for some $n\geq 5$, and let $\Omega_{n}=\{1,...,n\}$. Now, $G_{\Delta }^{\Sigma }$ is maximal subgroup of $G^{\Sigma}\cong A_{n}$ or $S_{n}$, and both $G^{\Sigma}$ and $G_{\Delta }^{\Sigma }$ act on the point-set of $\mathcal{D}_{1}$ as well as on $\Omega_{n}$. We distinguish three cases according as the action of $G_{\Delta }^{\Sigma }$ on $\Omega_{n}$ is primitive, imprimitive or intransitive. In the first case, we use \cite[Theorem 2]{AB} or \cite[Lemma 2.6]{ZTZ} according as $G^{\Sigma}\cong A_{n}$ or $S_{n}$, respectively, since $\left\vert G^{\Sigma}\right\vert \leq \left\vert G_{\Delta}^{\Sigma}\right\vert^{2}$ by Theorem \ref{Teo4}(2), and we obtain Table \ref{alt1}. In each case, one has $\lambda \mid v_{1}(v_{1}-1)$ and $v_{1}\neq \frac{1}{2}\left(\lambda-1\right)\left(\lambda ^{2}-2\right)$, and hence they are excluded by Theorem \ref{Teo1} since $\mathcal{D}_{1}$ is of type Ia or Ic by Theorem \ref{Teo4}(1). Therefore, $G_{\Delta }^{\Sigma }$ acts on $\Omega_{n}$ either imprimitively or intransitively.

\begin{table}[h!]
\tiny
\caption{Admissible alternating groups as automorphism groups of $\mathcal{D}_{1}$}\label{alt1}
\begin{tabular}{c|lllcl||c|lllcl}
\hline
Line & $G^{\Sigma }$ & $G_{\Delta }^{\Sigma }$ & $v_{1}$ & $\left(
v_{1}-1,\left\vert G_{\Delta }^{\Sigma }\right\vert \right) $ & $\lambda $ & 
Line & $G^{\Sigma }$ & $G_{\Delta }^{\Sigma }$ & $v_{1}$ & $\left(
v_{1}-1,\left\vert G_{\Delta }^{\Sigma }\right\vert \right) $ & $\lambda $
\\ 
\hline
1 & $A_{5}$ & $D_{10}$ & $6$ & $5$ & $5$ & 9 & $S_{5}$ & $AGL_{1}(5)$ & $30$
& $1$ & $5$ \\ 
2 & $A_{6}$ & $A_{5}$ & $6$ & $5$ & $5$ & 10 & $S_{6}$ & $S_{5}$ & $6$ & $5$
& $5$ \\ 
3 & $A_{7}$ & $PSL_{2}(7)$ & $15$ & $14$ & $5,7$ & 11 & $S_{7}$ & $AGL_{1}(7)
$ & $120$ & $7$ & $5,7$ \\ 
4 & $A_{8}$ & $AGL_{3}(2)$ & $15$ & $14$ & $5,7$ & 12 & $S_{8}$ & $PGL_{2}(7)
$ & $120$ & $7$ & $5,7$ \\ 
5 & $A_{9}$ & $ASL_{2}(3)$ & $120$ & $7$ & $5,7$ & 13 & $S_{9}$ & $AGL_{2}(3)
$ & $120$ & $7$ & $5,7$ \\ 
6 &  & $P\Gamma L_{2}(8)$ & $120$ & $7$ & $5,7$ &  &  &  &  &  &  \\ 
7 & $A_{11}$ & $M_{11}$ & $2520$ & $11$ & $5,7,11$ &  &  &  &  &  &  \\ 
8 & $A_{12}$ & $M_{12}$ & $2520$ & $11$ & $5,7,11$ &  &  &  &  &  & \\
\hline
\end{tabular}
\end{table}

Assume that $G_{\Delta }^{\Sigma }$ act imprimitively on $\Omega _{n}$. Let $%
\left\{ \Lambda _{0},...,\Lambda _{t-1}\right\} $ be a non-trivial $G_{\Delta
}^{\Sigma }$-invariant partition of $\Omega _{n}$ with $\left\vert \Lambda
_{i}\right\vert =s$, $0\leq i\leq t-1$, $s,t\geq 2$ and $st=n$. Hence, we may identify the point-set of $\mathcal{D}_{1}$ with the set of such partitions of $\Omega_{n}$. We may proceed as in \cite[Proposition 3.3]{WSZ} (or, as in \cite{Dela}). Moreover, the $\Theta_{j}$-cyclic partitions with respect  (a partition of $\Omega_{n}$ into $t$ classes each of size $s$) is an union of point-$G_{\Delta}^{\Sigma }$-orbits of $\mathcal{D}_{1}$ distinct from $\{\Delta\}$ for $j= 2,...,t$ (see \cite{Dela} for definitions and details). Then 
\begin{equation}
v_{1}=\frac{1}{t!}\prod_{j=1}^{t}\binom{js%
}{s}=\prod_{j=1}^{t}\binom{js-1}{s-1}  \label{equat0}
\end{equation}%
If $s=2$, then $t\geq 3$ and hence 
\begin{eqnarray}
v_{1} &=&\prod_{j=1}^{t}(2j-1)
\label{equat1} \\
\theta _{j} &=&\left\vert \Theta _{j}\right\vert =\frac{1}{2}\binom{t}{j}%
\binom{2}{1}^{j}=2^{j-1}\binom{t}{j}\text{.}  \nonumber
\end{eqnarray}%
Then $\frac{r_{1}}{(r_{1},\lambda_{1})}\mid t(t-1)$ since $\frac{r_{1}}{(r_{1},\lambda_{1})}\mid \theta _{2}$ by Theorem \ref{Teo4}(4) and $\theta _{2}=t(t-1)$. Then%
\[
\prod_{j=1}^{t}(2j-1)=v_{1}<\left( \frac{r_{1}}{(r_{1},\lambda_{1})}\right) ^{2}\leq t^{2}(t-1)^{2} 
\]%
which forces $t=3$ or $4$. Then either $v_{1}=15$ and $\frac{r_{1}}{(r_{1},\lambda_{1})}=6$, or $v_{1}=105$ and $\frac{r_{1}}{(r_{1},\lambda_{1})}=140$, respectively, and both cases are ruled out since they contradict $\frac{r_{1}}{(r_{1},\lambda_{1})} \mid v_{1}-1$.

If $s\geq 3$, then $\frac{r_{1}}{(r_{1},\lambda_{1})}$ divides $\theta
_{j}=\left\vert \Theta _{j}\right\vert =\frac{1}{2}\binom{t}{j}\binom{s}{1}%
^{j}$ for each $j$, and hence $\frac{r_{1}}{(r_{1},\lambda_{1})}\mid s^{2}\binom{t%
}{2}^{2}$ \ since $\theta _{2}=s^{2}\binom{t}{2}^{2}$. Then 
\[
2^{(s-1)(t-1)}<v_{1}<\left( \frac{r_{1}}{(r_{1},\lambda_{1})}\right) ^{2}\leq s^{4}\binom{t}{2}^{4}\text{,} 
\]%
which is (2) of \cite[p.5]{WSZ}. Thus $t<6$ or $s<6$, and we get $32$ pairs $(s,t)$ as in \cite[p.6]{WSZ}, which are recorded in Columns 2 and 7 of Table \ref{alt2}. All cases, but those as in Lines 14, 20 and 26 of Table \ref{alt2}, contradict $\frac{r_{1}}{(r_{1},\lambda_{1})}> v^{1/2}$ since $\frac{r_{1}}{(r_{1},\lambda_{1})} \mid (v_{1}-1,\left\vert
G_{\Delta }^{\Sigma }\right\vert )$ by Theorem \ref{Teo4}(4), and so they are ruled out by Theorem \ref{Teo3}(3). Note that, $\lambda$ is $5,7$, or $5,7,11,13$ or $5,7,11,13,17,19$ according as Lines 14, 20 and 26 of Table \ref{alt2} occur, respectively. Then $\lambda \mid v_{1}(v_{1}-1)$ and $v_{1}\neq \frac{1}{2}\left(\lambda-1\right)\left(\lambda ^{2}-2\right)$, and hence they are excluded by Theorem \ref{Teo1} since $\mathcal{D}_{1}$ is of type Ia or Ic by Theorem \ref{Teo4}(1).  
\begin{table}[h!]
\tiny
\caption{Admissible alternating groups as automorphism groups and type of $\mathcal{D}_{1}$ - imprimitive case-}\label{alt2}
\begin{tabular}{c|cclc||c|cclc}
\hline
Line & $(s,t)$ & $n=st$ & $v_{1}$ & $(v_{1}-1,\left\vert G_{\Delta }^{\Sigma
}\right\vert )$ & Line & $(s,t)$ & $n=st$ & $v_{1}$ & $(v_{1}-1,\left\vert
G_{\Delta }^{\Sigma }\right\vert )$ \\
\hline
1 & $(3,2)$ & $6$ & $10$ & $ 9$ & 17 & $(6,2)$ & $12$ & $%
 462$ & $ 1$ \\ 
2 & $(3,3)$ & $9$ & $280$ & $ 9$ & 18 & $(6,3)$ & $18$ & $%
 2858\,856$ & $ 5$ \\ 
3 & $(3,4)$ & $12$ & $ 15\,400$ & $ 9$ & 19 & $(6,4)$
& $24$ & $ 96\, 197\,645\,544$ & $ 1$ \\ 
4 & $(3,5)$ & $15$ & $ 1401\,400$ & $ 9$ & 20 & $(7,2)$
& $14$ & $ 1716$ & $ 245$ \\ 
5 & $(3,6)$ & $18$ & $ 190\,590\,400$ & $ 9$ & 21 & $%
(7,3)$ & $21$ & $66\,512\,160$ & $ 343$ \\ 
6 & $(3,7)$ & $21$ & $36\, 212\,176\,000$ & $ 9$ & 22
& $(8,2)$ & $16$ & $ 6435$ & $ 2$ \\ 
7 & $(3,8)$ & $24$ & $ 9161\, 680\,528\,000$ & $%
 27$ & 23 & $(8,3)$ & $24$ & $1577\,585\,295$ & $ 2$
\\ 
8 & $(3,9)$ & $27$ & $ 2977\,546\, 171\,600\,000$ & $%
 27$ & 24 & $(9,2)$ & $18$ & $ 24\,310$ & $
9$ \\ 
9 & $(4,2)$ & $8$ & $ 35$ & $ 2$ & 25 & $(10,2)$ & $20$
& $ 92\,378$ & $ 1$ \\ 
10 & $(4,3)$ & $12$ & $ 5775$ & $ 2$ & 26 & $(11,2)$ & 
$22$ & $ 352\,716$ & $ 605$ \\ 
11 & $(4,4)$ & $16$ & $ 2627\,625$ & $ 8$ & 27 & $%
(12,2)$ & $24$ & $ 1352\,078$ & $ 1$ \\ 
12 & $(4,5)$ & $20$ & $2546\,168\,625$ & $ 16$ & 28 & $(13,2)$ & $%
26$ & $ 5200\,300$ & $ 1521$ \\ 
13 & $(4,6)$ & $24$ & $ 4509\, 264\,634\,875$ & $%
 2$ & 29 & $(14,2)$ & $28$ & $ 20\,058\,300$ & $%
 1$ \\ 
14 & $(5,2)$ & $10$ & $ 126$ & $ 25$ & 30 & $(15,2)$
& $30$ & $ 77\,558\,760$ & $ 1$ \\ 
15 & $(5,3)$ & $15$ & $ 126\,126$ & $125$ & $ $31 & $%
(16,2)$ & $32$ & $ 300\,540\,195$ & $ 2$ \\ 
16 & $(5,4)$ & $20$ & $ 488\,864\,376$ & $ 625$ & 32 & 
$(17,2)$ & $34$ & $ 1166\,803\,110$ & $ 289$\\
\hline
\end{tabular}
\end{table}

Finally, assume that $G_{\Delta }^{\Sigma }$ acts intransitively on $\Omega _{n}$. Then $G_{\Delta }^{\Sigma }= \left(S_{s} \times S_{n-s} \right) \cap G^{\Sigma}$ with $s <n/2$. Indeed, the case where $G_{\Delta }^{\Sigma }$ has two orbits of equal size $n/2$ on $\Omega _{n}$ is ruled out by the maximality of $G_{\Delta }^{\Sigma }$ in $G^{\Sigma }$, because the stabilizer of such a partition of $\Omega _{n}$ would be between $G_{\Delta }^{\Sigma }$ and $G^{\Sigma }$. The flag-transitivity of $G^{\Sigma }$ implies that $G_{\Delta }^{\Sigma }$ is transitive on the blocks $\mathcal{D}_{1}$ through $\{\Delta\}$, and so $G_{\Delta }^{\Sigma }$ fixes exactly one point in $\mathcal{D}_{1}$. Since $\mathcal{D}_{1}$-stabilizes only one $s$-subset of $\Omega _{n}$, we can identify the point set of $\mathcal{D}_{1}$ with the set of all $s$-subsets of $\Omega _{n}$. So $v_{1}$ and subdegrees of $G^{\Sigma}$ are:
\begin{eqnarray*}
v_{1} &=&\binom{n}{s} \\
\psi _{i} &=&\left\vert \Psi _{i}\right\vert =\binom{s}{i}\binom{n-s}{s-i}
\end{eqnarray*}%
Then $\frac{r_{1}}{(r_{1},\lambda_{1})}\mid s(n-s)$ since $\frac{r_{1}}{(r_{1},\lambda_{1})}\mid \psi _{1}$ and $\psi _{1}=s(n-s)$. Then%
\begin{equation}\label{intrans}
\binom{n}{s}=v_{1}<\left( \frac{r_{1}}{(r_{1},\lambda_{1})}\right) ^{2}\leq s^{2}(n-s)^{2}\text{,} 
\end{equation}
which is (4) of \cite[p.6]{WSZ}, and hence $s\leq 6$. Moreover, if $s \geq 3$ then one of the following holds: $s=3$ and $7\leq n\leq 50$, $s=4$ and $9 \leq n \leq 18$, $s=5$ and $11\leq n \leq 14$, or $s=6$ and $n=13$. Now, exploiting $\frac{r_{1}}{(r_{1},k_{1})} \mid \left(s(n-s), \binom{n}{s}-1\right)$ and (\ref{intrans}) one obtains Table \ref{alt3}.
\begin{table}[h!]
\tiny
\caption{Admissible alternating groups as automorphism groups and type of $\mathcal{D}_{1}$ - intransitive case-}\label{alt3}
\begin{tabular}{cllcccc}
\hline
Line & $s$ & $n$ & $v_{1}$ & $r_{1}/(r_{1},\lambda _{1})$ & $A$ & $\lambda_{\text{max}}$\\ 
\hline
1&$4$ & $14$ & $1001$ & $40$ & $25$ & $13$\\ 
2&& $15$ & $1365$ & $44$ & $31$& $13$\\ 
3&$3$ & $14$ & $364$ & $33$& $11$& $13$\\ 
4&& $22$ & $1540$ & $57$ & $27$& $19$\\ 
5&& $32$ & $4960$ & $87$ & $57$& $31$\\ 
6&& $40$ & $9880$ & $111$ & $89$& $37$\\ 
7&& $50$ & $19600$ & $141$ & $139$& $47$\\
\hline
\end{tabular}%
\end{table}

For each value of $v_{1}$ and $r_{1}/(r_{1},\lambda _{1})$ as in Table \ref{alt3}, we determine $A=(v_{1}-1)/\left(r_{1}/(r_{1},\lambda _{1})\right)$ by Proposition \ref{P2}(3) since $r_{1}/(r_{1},\lambda _{1})=\frac{v_{0}-1}{k_{0}-1}$ by (\ref{double}). Except for the case as in line 3 with $\lambda=13$, we have $\lambda \leq \lambda_{\text{max}} \leq A$, and this is contrary to Proposition \ref{P2}(4). In the remaining case, one has $k_{0}=2$ again Proposition \ref{P2}(4) since $k_{0} \geq 2$. Then $v_{0}=34$, and hence $k_{1}=A\frac{v_{0}}{k_{0}}+1=2^{2}\cdot 47$ by Proposition \ref{P2}(3). However, this is impossible since $k_{1}$ does not divide the order of $G^{\Sigma}$, being $G^{\Sigma}\leq S_{n}$, which is contrary to $G^{\Sigma}$ flag-transitive on $\mathcal{D}_{1}$.  

If $s=2$ then $v_{1}=\frac{n(n-1)}{2}$ and $\frac{r_{1}}{(r_{1},\lambda_{1})}= \frac{2(n-2)}{\alpha}$ for some positive integer $\alpha$, and hence
\begin{equation}\label{EqAlt}
\frac{(n+1)(n-2)}{2}=v_{1}-1< \left(\frac{r_{1}}{(r_{1},\lambda_{1})}\right)^{2}=\frac{4(n-2)^{2}}{\alpha^{2}}\text{.}    
\end{equation}
Thus $\alpha=1,2$. Arguing as above we see that $A=\frac{v_{1}-1}{r_{1}/(r_{1},\lambda _{1})}=\frac{(n+1)\alpha}{4}$. Then $\lambda \geq \frac{(k_{0}-1)\alpha}{4}(n+1)+1$ by Proposition \ref{P2}(4). Since $\lambda$ divides $k_{1}$, and $A_{n} \unlhd G^{\Sigma}\leq S_{n}$ acts flag-transitively on $\mathcal{D}_{1}$, it follows that $\lambda \mid n!$. Thus $\lambda \leq n$, and hence either $\alpha=k_{0}=2$ and $n$ odd, or $\alpha=1$, $n \equiv 3 \pmod{4}$ and $k_{0}=2,3,4$.

Assume that $\alpha=k_{0}=2$ and $n$ is odd. Then $A=v_{0}-1=\frac{n+1}{2}$ and $v_{0}=\frac{n+3}{2}$, and hence $k_{1}=A\frac{v_{0}}{k_{0}}+1=\frac{1}{8}(n^{2}+4n+11)$ with $n \equiv 1 \pmod{4}$ by Proposition \ref{P2}(4). Then $n=5$ since $v_{1} <2k_{1}$ and $n \geq 5$, and hence $k_{1}=\lambda=7$, which is a contradiction.

Assume that $\alpha=1$, $n \equiv 3 \pmod{4}$ and $k_{0}=2,3,4$. Then $A=\frac{n+1}{4}$ and $v_{0}=\frac{n+1}{4}(k_{0}-1)+1$, and hence 
\begin{equation}\label{kappa1}
k_{1}=\frac{n+1}{4}\left(\frac{n+1}{4}(k_{0}-1)+1\right)\frac{1}{k_{0}}+1 \text{.}
\end{equation}
Since $v_{1}<2k_{1}$, it follows that
$$ \frac{n(n-1)}{2}<\frac{n+1}{2}\left(\frac{n+1}{4}(k_{0}-1)+1\right)\frac{1}{k_{0}}+2\text{,}$$
which leads to no cases since $n \geq 5$.


Finally, assume that $s=1$. Hence $v_{1}=n$. Therefore, $G^{\Sigma }$ acts point-$(v_{1}-2)$%
-transitively on $\mathcal{D}_{1}$. Then $G^{\Sigma }$ acts $k_{1}$-transitively since $2<k_{1}<v_{1}-1$, and hence $b_{1}=\binom{v_{1}}{k_{1}}$, where $b_{1}$ is the number of blocks of $\mathcal{D}_{1}$. Now, we deduce from $b_{1}k_{1}=v_{1}r_{1}$ and $\lambda_{1}(v_{1}-1)=r_{1}(k_{1}-1)$ that $r_{1}=\binom{v_{1}-1}{k_{1}-1}$ and $\lambda_{1}=\binom{v_{1}-2}{k_{1}-2}$. This completes the proof.    
\end{proof}

\begin{lemma}\label{noAlternating} 
Assume that Hypothesis \ref{hyp2} holds. Then $T$ is not an alternating group.   
\end{lemma}
\begin{proof}
Suppose that $T \cong A_{n}$, $n \geq 5$. Then $v_{1}=n$ and $\mathcal{D}_{1}$ is the complete $2$-$\left(v_{1},k_{1},\binom{v_{1}-2}{k_{1}-2}\right)$ design by Lemma \ref{AltCompleteDes}. Assume that $G_{(\Sigma )}=1$. Then $G$ acts faithfully on $\Sigma $, and
hence $A_{v_{1}-1}\trianglelefteq G_{\Delta }\leq S_{v_{1}-1}$. If $%
A_{v_{1}-1}\trianglelefteq G_{(\Delta )}$, then $G_{\Delta }^{\Delta }\leq
Z_{2}$ and so $v_{0}\leq 2$ since $G_{\Delta }^{\Delta }$ acts
point-transitively on $\mathcal{D}_{0}$, whereas $v_{0}>3$ by Lemma \ref{L1}%
. Thus $A_{v_{1}-1}\trianglelefteq G_{\Delta }^{\Delta }\leq S_{v_{1}-1}$.
If $k_{0}=2$, then $G_{\Delta }^{\Delta }$ acts point-$2$-transitively on $%
\mathcal{D}_{0}$, and hence either $v_{0}=v_{1}-1$ and $v=v_{1}(v_{1}-1)$,
or $v_{0}=15$ and $v_{1}-1=7$. In the former case, we have $%
A_{8}\trianglelefteq G\leq S_{8}$ and so $\lambda =5$ or $7$ since $\lambda $
is a prime divisor of $\left\vert G\right\vert $ and $\lambda >3$. Then $%
\lambda \mid $ $v_{0}(v_{1}-1)$ and $v_{1}\neq \frac{1}{2}(\lambda
-1)(\lambda ^{2}-2)$, and this contradicts Theorem \ref{Teo1} since $%
\mathcal{D}_{1}$ is of type Ia or Ic by Theorem \ref{Teo4}(1). Thus, $%
v_{0}=v_{1}-1$ and $v=v_{1}(v_{1}-1)$. Then $A=1$ since $k_{1}<v_{1}$ with $%
k_{1}=A\frac{v_{1}-1}{2}+1$ by Proposition \ref{P2}(3). So $k_{1}=\frac{%
v_{1}+1}{2}$, ad hence $k=v_{1}+1$. Then $%
(v-1,k-1)=(v_{1}(v_{1}-1)-1,v_{1})=1$ and hence $G$ acts point-primitively on 
$\mathcal{D}$ by \cite[Corollary 4.2]{Ka0}, but this is contrary to our
assumptions. Thus $k_{0}>2$. If $\mu =1$, then $\mathcal{D}_{0}$ is $2$-$%
(v_{1}-1,k_{0},\lambda )$ design admitting $A_{v_{1}-1}\trianglelefteq
G_{\Delta }^{\Delta }\leq S_{v_{1}-1}$ as a flag-transitive automorphism
group, but this is contrary to Theorem \cite[Theorem 1]{ZCZ}. Therefore, $%
\mu =\lambda $ by Lemma \ref{base}(1), and hence $\mathcal{D}_{0}$ is $2$-$%
(v_{1}-1,k_{0},1)$ design admitting $A_{v_{1}-1}\trianglelefteq G_{\Delta
}^{\Delta }\leq S_{v_{1}-1}$ as a flag-transitive automorphism group, which
is not the case by \cite[Theorem]{BDDKLS}. Therefore $G_{(\Sigma )}\neq 1$,
and hence $\eta =\frac{v_{0}}{k_{0}}$.

Then
\begin{equation}\label{sadday}
\binom{v_{1}-2}{v_{1}-k_{1}} =\binom{v_{1}-2}{k_{1}-2} =\lambda_{1}=\frac{v_{0}}{k_{0}}\cdot \lambda\text{.}
\end{equation}
On the other hand, $k_{1}=A\frac{v_{0}}{k_{0}}+1$ and $\lambda \mid k_{1}$ by Propositions \ref{P2}(3) and \ref{LambdaDividesK1}. Thus, $\frac{v_{0}}{k_{0}}\cdot \lambda \leq k_{1}^{2}$, and hence
\begin{equation}\label{binom}
\sum_{j=0}^{v_{1}-k_{1}}\binom{k_{1}-3+j}{j}= \binom{k_{1}-3+(v_{1}-k_{1})+1}{v_{1}-k_{1}} \leq k_{1}^{2}\text{.}
\end{equation}
If $v_{1}-k_{1} \geq 3$, then (\ref{binom}) leads to $k_{1}^{2}-6 k_{1}-1 \leq 0$, and hence to $k_{1} \leq 6$. Actually, $k_{1}=\lambda=5$ since $\lambda \mid k_{1}$ and $\lambda>3$. Moreover, $v_{1}=8$ or $9$ since $v_{1}-k_{1} \geq 3$ and $v_{1}<2k_{1}$. Then $A\frac{v_{0}-1}{k_{0}-1}=7$ or $8$, respectively. In each case, we have $A=1$ and $\frac{v_{0}-1}{k_{0}-1}=7$ or $8$, respectively, since $\frac{v_{0}-1}{k_{0}-1}> \lambda=5$, and hence $\frac{v_{0}}{k_{0}}=4$ since $k_{1}=5$. So, $\frac{4k_{0}-1}{k_{0}-1}=7$ or $8$, and hence only the former occurs which implies $k_{0}=2$ and $v_{0}=8$. Then $v=56$ and $k=10$, and hence $(v-1,k-1)=1$, and so $G$ acts point-primitively on 
$\mathcal{D}$ by \cite[Corollary 4.2]{Ka0}, but this is contrary to our
assumptions.

If $v_{1}-k_{1} < 3$, then $v_{1}-k_{1} =1$ or $2$ since $k_{1}<v_{1}$. Therefore, $k_{1}-1=\frac{v_{0}}{k_{0}}\lambda$ or $\frac{1}{2}k_{1}(k_{1}-1)=\frac{v_{0}}{k_{0}}\lambda$, respectively. The former is ruled out, whereas the latter yields $\frac{k_{1}}{\lambda}(k_{1}-1)=2\frac{v_{0}}{k_{0}}$, since $\lambda \mid k_{1}$. Then either $A=1$ and $k_{1}=2\lambda$, or $A=2$ and $k_{1}=\lambda$ since $k_{1}=A\frac{v_{0}}{k_{0}}+1$ by Proposition \ref{P2}(3). The latter implies $\lambda=2\frac{v_{0}}{k_{0}}+1>\frac{v_{0}-1}{k_{0}-1}$ since $v_{0}>3$ which is contrary to Theorem \ref{Teo4}. Thus, $A=1$ and $k_{1}=2\lambda$. Then $2\lambda =k_{1}=\frac{v_{0}}{k_{0}}+1$ and $2\lambda +2=k_{1}+2=v_{1}=%
\frac{v_{0}-1}{k_{0}-1}+1$ implies 
\[
\left( \frac{v_{0}}{k_{0}}+1\right) \frac{1}{2}=\left( \frac{v_{0}-1}{k_{0}-1%
}+1\right) \frac{1}{2}-1\text{.}
\]%
Easy computations show that $v_{0}=k_{0}\left( 2k_{0}-1\right) $%
, which substituted in $2\lambda =\frac{v_{0}}{k_{0}}+1$ leads to $\lambda
=k_{0}$. Then $\mathcal{D}_{1}$ is neither of type Ia nor of type Ic by
Theorem \ref{Teo1}, which is not the case by Theorem \ref{Teo4}(1). Thus, $T$ cannot be $A_{n}$ with $n \geq 5$.

\end{proof}

\bigskip

\subsection{The case where $T$ is an exceptional Lie type simple group} In this section, we assume that $T$ is an exceptional group of Lie type 

\bigskip

\begin{lemma}
\label{ReductionSuzRee}Assume that Hypothesis \ref{hyp2} holds. If $T$ is an exceptional group of Lie type, then $T$ acts point-primitively on $\mathcal{D}_{1}$ and one of the following holds:
\begin{enumerate}
    \item $T\cong $ $Sz(q)$, $q=2^{2h+1}\geq 8$, $v_{1}=q^{2}+1$ and $T_{\Delta }\cong
[q^{2}]:Z_{q-1}$;
\item $T\cong $ $^{2}G_{2}(q)^{\prime }$, $q=3^{2h+1}\geq 27$, $v_{1}=q^{3}+1$ and $T_{\Delta }\cong [q^{3}]:Z_{q-1}$.
\end{enumerate}
\end{lemma}
\begin{proof}
Assume that $T$ is an exceptional group of Lie type. It follows from Theorem %
\ref{Teo4}(2) that $G_{\Delta }^{\Sigma }$ is a large maximal subgroup
of $G^{\Sigma }$. Suppose that $G_{\Delta }^{\Sigma }$ is a not a parabolic
subgroup of $G^{\Sigma }$. Then the possibilities for the pairs $%
(T,T_{\Delta })$, where $T_{\Delta }=G_{\Delta }^{\Sigma }\cap T$, are
listed in \cite[Theorem 1.6]{ABD}. For each of these cases it is easy to
compute $v_{1}=\left\vert T:T_{\Delta }\right\vert $ since $G^{\Sigma }$
acts point-primitively on $\mathcal{D}_{1}$ and $T\trianglelefteq G^{\Sigma }
$. Although the $2$-$(v^{\prime },k^{\prime },\lambda ^{\prime })$
designs $\mathcal{D}^{\prime }$ \ with $r^{\prime }=\frac{(v^{\prime
}-1)\lambda ^{\prime }}{k^{\prime }-1}$ admitting a flag-transitive
automorphism group $\Gamma ^{\prime }$ investigated in \cite{Al} are
different from our $\mathcal{D}_{1}$, nevertheless we can apply to our case some purely group-theoretical
results contained in \cite{Al}. Firstly,
note that listed given in \cite[Theorem 1.6]{ABD} is the same as in \cite[%
Tables 2 and 3]{Al}. Now, the lower bound for $\left\vert T:T_{\Delta
}\right\vert $ given in \cite[Table 2]{Al} is a lower bound for our $v_{1}$,
that we denote by $\ell _{v_{1}}$. Hence $\ell _{v_{1}}=\ell _{v^{\prime }}$%
, where $\ell _{v^{\prime }}$ is as in the third column of \cite[Table 2]{Al}%
. As pointed out in \cite[p. 1009]{Al}, the upper bound $u_{r^{\prime }}$
for $r^{\prime }$ recorded in the last column \cite[Table 2]{Al} is obtained
as a divisor of $(v^{\prime }-1,\left\vert \Gamma _{x^{\prime }}^{\prime
}\right\vert )$ with $x^{\prime }$ any fixed point of $\mathcal{D}^{\prime }$
the stabilizer \cite[Table 2]{Al}, in some cases together with the knowledge
of some subdegrees of $\Gamma ^{\prime }$. In our context the roles of $%
\Gamma ^{\prime }$ and $r^{\prime }$ are played by $G^{\Sigma }$ and $\frac{%
r_{1}}{(r_{1},\lambda _{1})}$, respectively, since $\frac{r_{1}}{%
(r_{1},\lambda _{1})}\mid (v_{1}-1,\left\vert G_{\Delta }^{\Sigma
}\right\vert )$ by Theorem \ref{Teo4}(4). Therefore, $u_{r^{\prime
}}=u_{r_{1}/(r_{1},\lambda _{1})}$ for each case listed in \cite[Table 2]{Al}%
. So $\left( \frac{r_{1}}{(r_{1},\lambda _{1})}\right) ^{2}\leq
u_{r_{1}/(r_{1},\lambda _{1})}^{2}<\ell _{v_{1}}\leq v_{1}$ for each case as
in \cite[Table 2]{Al}, and hence the possibilities for $(T,T_{\Delta })$ 
as in \cite[Table 2]{Al} are excluded since they contradict Theorem \ref{Teo4}(3). Thus, $%
(T,T_{\Delta })$ is as one of the possibilities provided in \cite[Table 3]{Al}. By Lemma \ref{Tlarge}(1), we may filter such pairs $%
(T,T_{\Delta })$ with respect $v_{1}=\left\vert T:T_{\Delta }\right\vert
\leq \left\vert T_{\Delta }\right\vert \cdot \left\vert Out(T)\right\vert $, and we see that only the cases as in Table \ref{Exc1}
are admissible (the lines 2--5 of \cite[Table 3]{Al} actually do
not occur since the Ree group exist only for odd powers of $3$):

\begin{table}[h!]
\tiny
\caption{Admissible $T$,  $T_{\Delta }$ and design parameters when $T$ is exceptional of Lie type and $T_{\Delta }$ is non-parabolic}\label{Exc1}
\begin{tabular}{cccc}
$T$ & $T_{\Delta }$ & $v_{1}$ & $\left( v_{1}-1,\left\vert T_{\Delta
}\right\vert \cdot \left\vert Out(T)\right\vert \right) $ \\ 
\hline
$^{3}D_{4}(4)$ & $^{3}D_{4}(2)$ & $320819200$ & $3$ \\ 
$^{3}D_{4}(9)$ & $^{3}D_{4}(3)$ & $25143164583300$ & $7$ \\ 
$G_{2}(4)$ & $J_{2}$ & $416$ & $5$ \\ 
$F_{4}(2)$ & $^{3}D_{4}(2)$ & $15667200$ & $1$ \\ 
$F_{4}(2)$ & $D_{4}(2)$ & $3168256$ & $15$ \\ 
$E_{7}(2)$ & $E_{6}^{-}(2)$ & $2488042946297856$ & $5$ \\ 
$E_{6}^{-}(2)$ & $Fi_{22}$ & $1185415168$ & $3$ \\ 
$E_{6}^{-}(2)$ & $D_{5}^{-}(2)$ & $1019805696$ & $5$\\
\hline
\end{tabular}%
\end{table}

Since $\frac{r_{1}}{(r_{1},\lambda _{1})}\mid \left( v_{1}-1,\left\vert
T_{\Delta }\right\vert \cdot \left\vert Out(T)\right\vert \right) $ by Theorem \ref{Teo4}(3), it follows that $\left( \frac{r_{1}}{(r_{1},\lambda _{1})}\right)
^{2}<v_{1}$, which is contrary to Theorem \ref{Teo4}(3). Thus all cases as in Table \ref{Exc1} are excluded, and hence $G_{\Delta
}^{\Sigma }$ is a maximal parabolic subgroup of $G^{\Sigma }$.

Assume that $T$ is isomorphic to $E_{6}(q)$. If $G_{\Delta }^{\Sigma }$ $%
=P_{3}$ of typer $A_{1}\times A_{4}$, then $%
v_{1}=(q^{3}+1)(q^{4}+1)(q^{9}-1)(q^{12}-1)/(q^{2}-1)(q-1)$ and $\frac{r_{1}%
}{(r_{1},\lambda _{1})}\mid q(q^{5}-1)(q-1)\log q$ since $\frac{r_{1}}{%
(r_{1},\lambda _{1})}\mid \left( v_{1}-1,\left\vert G_{\Delta }^{\Sigma
}\right\vert \right) $ by Theorem \ref{Teo4}(4), and so $\left( \frac{r_{1}}{(r_{1},\lambda _{1})}%
\right) ^{2}<v_{1}$, which is contrary to Theorem \ref{Teo4}(3). If $%
G_{\Delta }^{\Sigma }$ $=P_{1}$, then $\frac{r_{1}}{(r_{1},\lambda _{1})}$
divides the subdegrees $q(q^{3}+1)\frac{q^{8}-1}{q-1}$ and $q^{8}(q^{4}+1)%
\frac{q^{5}-1}{q-1}$, and hence $\left( \frac{r_{1}}{(r_{1},\lambda _{1})}%
\right) ^{2}<v_{1}$, a contradiction (e.g see \cite[p.345]{Saxl}). Then $%
G_{\Delta }^{\Sigma }$ has a unique one point-orbit  on $\mathcal{D}_{1}$
which is a power of the characteristic of $T$ by \cite[Lemma 2.6]{Saxl}, and
hence $\frac{r_{1}}{(r_{1},\lambda _{1})}\mid \left( v_{1}-1\right) _{p}$ \
since $\frac{r_{1}}{(r_{1},\lambda _{1})}\mid v_{1}-1$. At this point, if $T$
is neither $Sz(q)$, $q=2^{2h+1}\geq 8$, nor $^{2}G_{2}(q)^{\prime }$, $%
q=3^{2h+1}\geq 2$, proceeding as in \cite[p. 1013]{Al}, we see that  $\left(
v_{1}-1\right) _{p}^{2}<v_{1}$ and hence no cases occurs. Therefore,
either $T\cong $ $Sz(q)$, $q=2^{2h+1}\geq 8$, and $T_{\Delta }\cong
[q^{2}]:Z_{q-1}$ and $v_{1}=q^{2}+1$, or $T\cong $ $^{2}G_{2}(q)^{\prime }$%
, $q=3^{2h+1}\geq 3$, and $T_{\Delta }\cong [q^{3}]:Z_{q-1}$ and $%
v_{1}=q^{3}+1$. In both cases $T$ acts primitively on $\mathcal{D}_{1}$ by \cite[Tables 8.16 and 8.43]{BHRD}. Moreover, if  $T\cong $ $^{2}G_{2}(3)^{\prime } \cong PSL_{2}(8)$, then $\lambda=7$ and $v_{1}=28$. However, this is impossible by Theorem \ref{Teo1}(1) since $\lambda \mid v_{1}$ and $v_{1} \neq \frac{1}{2}$, because $\mathcal{D}_{1}$ is of type Ia or Ic by Theorem \ref{Teo4}(1). Thus, $T\cong $ $^{2}G_{2}(q)$ with $q\geq 27$. This completes the proof.
\end{proof}

\bigskip

\begin{lemma}
\label{ExcSuz}Assume that Hypothesis \ref{hyp2} holds. If $T$ is not an exceptional group of Lie type, then
\begin{enumerate}
    \item $T\cong $ $Sz(q)$, $q=2^{2h+1}\geq 8$, $T_{\Delta }\cong
[q^{2}]:Z_{q-1}$, $T_{B^{\Sigma} }\cong
[q]:Z_{q-1}$;
\item $T$ acts flag-transitively on $\mathcal{D}_{1}$;
\item $\lambda=q-1$ is a Mersenne prime and the admissible parameters for the $2$-designs $\mathcal{D}_{0}$ and $\mathcal{D}_{1}$ are listed in Table \ref{Suz1}.
\begin{table}[h!]
\tiny
\caption{Admissible parameters for the $2$-designs $\mathcal{D}_{0}$ and $\mathcal{D}_{1}$ when $T$ is the Suzuki group.}\label{Suz1}
\begin{tabular}{ccccc|ccccc}
\hline
$v_{1}$ & $k_{1}$ & $\lambda_{1} $ & $r_{1}$ & $b_{1}$ & $v_{0}$ & $k_{0}$ & $\lambda_{0}$ & $r_{0}$ & $b_{0}$\\
\hline
$q^{2}+1$ & $q(q-1)$ & $q^{3}-2q^{2}+1$ & $q^{2}(q-1)$ &  $q^{2}(q+1)$ & $q^{3}-2q^{2}+1$ & $q-1$ & $\frac{q-1}{\mu}$ & $q^{2}\frac{q-1}{\mu}$ & $q^{2}\frac{q^{3}-2q^{2}+1}{\mu}$ \\ 
\hline
\end{tabular}%
\end{table}
\end{enumerate}
\end{lemma}

\begin{proof}
Since $\frac{r_{1}}{(r_{1},\lambda _{1})}=\frac{v_{0}-1}{k_{0}-1}$, $v_{1}=A\frac{v_{0}-1}{k_{0}-1}+1$ by (\ref{double}) and Proposition \ref{P2}(3), and $v_{1}=q^{i}+1$ with $i=2$ or $3$ according as $T$ is $Sz(q)$ or $^{2}G_{2}(q)$, respectively, it follows that 
\begin{equation}\label{spes}
\frac{r_{1}}{(r_{1},\lambda _{1})}=\frac{v_{0}-1}{k_{0}-1}=\frac{q^{i}}{p^{s}} \text{\quad \text{ and } \quad} A=p^{s}\text{.}    
\end{equation}
Moreover, $k_{1}=p^{s}\frac{v_{0}}{k_{0}}+1$ and
\begin{equation}\label{caj}
b_{1}=\frac{v_{1}r_{1}}{k_{1}}=\frac{(q^{i}+1) \cdot \frac{v_{0}-1}{k_{0}-1}\cdot \frac{v_{0}}{k_{0} \cdot \eta}\cdot \lambda}{p^{s}\frac{v_{0}}{k_{0}}+1}=\frac{(q^{i}+1) \cdot \frac{q^{i}}{p^{s}}\cdot \frac{v_{0}}{k_{0} \cdot \eta}\cdot \lambda}{p^{s}\frac{v_{0}}{k_{0}}+1}\text{.}
\end{equation}

It follows from Theorem \ref{Teo1} that $\mathcal{D}_{1}$ is not of type Ic since $v_{1}=q^{i}+1$ with $i=2,3$. Therefore, $\mathcal{D}_{1}$ is not of type Ia by Theorem \ref{Teo1}, and hence $\lambda \nmid v_{1}(v_{1}-1)$ then $\mathcal{D}_{1}$ by Theorem \ref{Teo1}. Moreover, $\lambda$ divides the order of $T$ since $\lambda$ divides the order of $G^{\Sigma}$, and $\lambda$ divides the order of $Out(T)$ by Lemma \ref{Tlarge}(1). Then $\lambda \mid q-1$ since the order of $T$ is $q^{i}(q^{i}+1)(q-1)$,with $i=2,3$.

Assume that $T\cong $ $^{2}G_{2}(q)$, $q=3^{2h+1}\geq 27$, and $T_{\Delta }\cong [q^{3}]:Z_{q-1}$ and $v_{1}=q^{3}+1$. Then $\lambda \leq (q-1)/2$, and hence \cite[Theorem 1]{DP} implies $q^{3}+1 \leq (q-1)^{2}(q-3)/4$, a contradiction since $q>1$. 

Assume that $T\cong $ $Sz(q)$, $q=2^{2h+1}\geq 8$, and $T_{\Delta }\cong
[q^{2}]:Z_{q-1}$ and $v_{1}=q^{2}+1$. Let $K=T_{B^{\Sigma}}$, where $B^{\Sigma}$ is any block of $\mathcal{D}_{1}$ and let $M$ be any maximal subgroup of $T$ containing $K$. Since $K$ acts transitively on $B^{\Sigma}$, it follows that $k_{1}$ divides the order of $K$, and hence $\lambda$ divides the order of $K$ since $\lambda \mid k_{1}$ by Theorem \ref{Teo4}. Therefore, $\lambda \mid (\left\vert M\right\vert, q-1)$ since $K \leq M$. Then $M$ is either isomorphic to $[q^{2}]:Z_{q-1}$, $D_{2(q-1)}$, or to $Sz(q^{1/m})$ with $m$ a proper (odd) prime divisor of $\log_{p}(q)$ by \cite[Table 8.16]{BHRD}. In the latter case, $\lambda \mid q^{1/m}-1$ since $\lambda \mid (\left\vert M\right\vert, q-1)$ and $(q^{2}+1,q-1)=1$ being $q$ even. By \cite[Theorem 1]{DP}, one has $q^{2}+1 \leq v \leq 2\lambda^{2}(\lambda-1) \leq 2q^{1/m}(q^{2/m}-1)$ with $m$ a proper divisor of $\log_{2}(q)$, and we reach a contradiction. In the first two cases, $q^{2}+1$ divides $\left\vert T:M \right\vert$ and hence $b_{1}$. Thus $2^{s}\frac{v_{0}}{k_{0}}+1 \mid \frac{q^{2}}{2^{s}}\cdot \frac{v_{0}}{k_{0} \cdot \eta}\cdot \lambda$ by (\ref{caj}), and so $2^{s}\frac{v_{0}}{k_{0}}+1 \mid \frac{q^{2}}{2^{s}} \cdot \lambda$. If $s>1$ then $2^{s}\frac{v_{0}}{k_{0}}+1 \mid \lambda$, and hence $k_{1}=\lambda$  since $k_{1}=2^{s}\frac{v_{0}}{k_{0}}+1$ and $\lambda \mid k_{1}$. So $q^{2}+1=v_{1}<2k_{1}=2\lambda\leq 2(q-1)$, a contradiction. Thus $s=1$ and $\frac{v_{0}}{k_{0}}+1 \mid q^{2} \cdot \lambda$. Therefore $A=1$, and hence $k_{1}=\frac{q^{2}}{2^{t}}\lambda$ for some nonnegative integer $t$ such that $\frac{2^{t}}{2}<\lambda <p^{t}$. Thus, $\frac{q^{2}}{2^{t}}\lambda>\frac{q^{2}+1}{2}$ since $k_{1}>v_{1}/2$, and hence 
$\frac{q^{2}}{2^{t}}>\frac{q+1}{2}$ since $\lambda \mid q-1$, forcing $q\mid k_{1}$ since $q$ is a power of $2$. Thus, $M \cong [q^{2}]:Z_{q-1}$.

Note that, any Sylow $p$-subgroup of $T$ partitions the point set of $\mathcal{D}_{1}$ into two orbits of length $1$ and $q^{2}$. Thus the Sylow $p$-subgroup $W$ of $K$ acts semiregularly on $B^{\Sigma}$, and hence it has order $\frac{q^{2}}{2^{t}}$ since $k_{1}=\frac{q^{2}}{2^{t}}\lambda$. Further, if $L$ is any Sylow $\lambda$-subgroup of $K$, then $W:L$ is a Frobenius group by \cite[Theorem IV.24.2.(c)]{Lu}. Then $\lambda \mid 2^{4h+2-t}-1$ by \cite[Proposition 17.3(ii)]{Pass}, and hence $\lambda \mid 2^{(t,2h+1)}-1$ since $\lambda \mid 2^{2h+1}-1$. Actually, $\lambda = 2^{t}-1$ since $2^{t-1}<\lambda <2^{t}$. and so $t$ is a prime being $\lambda$ a prime. If $t<2h+1$, then $3t \leq 2h+1$ since $t\mid 2h+1$. Then \cite[Theorem 1]{DP} implies $$2^{4h+2}+1<2(2^{t}-1)^{2}(2^{t}-2)\text{,}$$ and so $4h+2 < 3t+1 \leq 2h+2$, a contradiction. Thus $t=2h+1$, and hence $\lambda=q-1$ and $k_{1}=q(q-1)$. Moreover, $K_{B^{\Sigma}}=W:L$ acts regularly on $B^{\Sigma}$. Now, $K_{B^{\Sigma}}$ partitions the point set of $\mathcal{D}_{1}$ into three orbits of length $1$, $q$ and $q(q-1)$, the last being the size of $B^{\Sigma}$. Thus $G_{B^{\Sigma}}^{\Sigma}=K:Z_{f}$ where $f=\left\vert G^{\Sigma}:T\right\vert$, and hence $T$ acts block-transitively on $\mathcal{D}_{1}$. Therefore, $T$ acts flag-transitively on $\mathcal{D}_{1}$ since $K$ acts transitively on $B^{\Sigma}$. Then $b_{1}=\left\vert T:K\right\vert=q(q^{2}+1)$, and hence $r_{1}=q^{2}(q-1)$. Then $\eta=\frac{v_{0}}{k_{0}}=q^{2}-q-1$ since $r_{1}=\frac{v_{0}-1}{k_{0}-1}\cdot \frac{v_{0}}{k_{0}\eta} \cdot \lambda$, $\frac{v_{0}-1}{k_{0}-1}=q^{2}$ and $\frac{v_{0}}{k_{0}}=k_{1}-1$, $k_{1}=q(q-1)$ and $\lambda=q-1$. Moreover, $\frac{v_{0}}{k_{0}}=q^{2}-q-1$ and $\frac{v_{0}}{k_{0}}=q^{2}-q-1$ imply
\begin{equation}\label{K0Suz}
q^{2}(k_{0}-1)=(q^{2}-q-1)k_{0}-1    
\end{equation}
From (\ref{K0Suz}) we derive $k_{0}=q-1$ and $v_{0}=(q-1)(q^{2}-q-1)$. The admissible parameters of $\mathcal{D}_{0}$ and $\mathcal{D}_{1}$ are now those recorded in Table \ref{Suz1}. This completes the proof.
\end{proof}

\bigskip

\begin{lemma}\label{NoExc}
 Assume that Hypothesis \ref{hyp2} holds. Then $T$ is not an exceptional group of Lie type.     
\end{lemma}
\begin{proof}
We know that either $\mu=q-1$ or $\mu=1$ by Lemmas \ref{base}(1) and \ref{ExcSuz}. If $\mu=q-1$, then $\mathcal{D}_{0}$ is a $2$-$((q-1)(q^{2}-q-1),q-1,1)$ design admitting $G^{\Delta}_{\Delta}$ as a flag-transitive automorphism group again by Lemma \ref{ExcSuz}. Note that $v_{0}=(q-1)(q^{2}-q-1)$ is an odd composite number, then one of the following holds by \cite[%
Theorem]{BDDKLS}:

\begin{enumerate}
\item $\mathcal{D}_{0}\cong PG_{m-1}(s)$, $m\geq 3$, and either $%
PSL_{m}(q)\trianglelefteq G_{\Delta }^{\Delta }\leq P\Gamma L_{m}(s)$,
or $(m,s)=(3,2)$ and $G\cong A_{7}$;

\item $\mathcal{D}_{0}$ is the Hermitian unital of even order $s$, and $%
PSU_{3}(s)\trianglelefteq G_{\Delta }^{\Delta }\leq P\Gamma U_{3}(s)$.
\end{enumerate}
Then we obtain either $\frac{s^{m}-1}{s-1}=(q-1)(q^{2}-q-1)$ and $s+1=q-1$, or $s^{3}+1=(q-1)(q^{2}-q-1)$ and $s+1=q-1$ according as cases (1) or (2) occurs, respectively. Then either $s^{m}=s(s^{3}+3s^{2}-3)$ with $m \geq 3$, or $s^{3}=s^{3}+4s(s+1)$, and any of these does not have admissible solutions. 

Assume that $\mu=1$. Then $\mathcal{D}_{0}$ is a $2$-$((q-1)(q^{2}-q-1),q-1,q-1)$ design with $q-1$ prime by Lemma \ref{ExcSuz}. Let $N$ be the action kernel of $G_{B}$ on $B$, then either $%
Z_{\lambda }\trianglelefteq G_{B}/N\leq AGL_{1}(\lambda )$ and $G_{B}$ acts $%
3/2$-transitively on $B$, or $G_{B}$ acts $2$-transitively on $B$ by \cite[p. 99%
]{DM}. Then $G_{\Delta }^{\Delta }$ acts either point-$3/2$-transitively or point-$2$%
-transitively on $\mathcal{D}_{0}$, respectively, since $rank(G_{B},B)=rank(G_{\Delta }^{\Delta },\Delta)$ by Lemma \ref{SameRank}(1). Moreover, $G_{\Delta }^{\Delta }$ is almost simple by \cite[Theorem 1.1]{BGLPS} since $v_{0}=(q-1)(q^{2}-q-1)$ is a composite number, and hence $G_{\Delta
}^{\Delta }$ acts point-$2$-transitively on $\mathcal{D}_{0}$ by 
\cite[Theorem 1.2]{BGLPS} since $v_{0} \neq 21$. Thus, $G$ acts point-$2$-transitively on $\mathcal{D}$ in any case. Then one of the following
holds by \cite[(A)]{Ka} since $v_{0}$ is an odd composite integer not equal to $15$ and $Soc(G_{\Delta }^{\Delta
})\ncong A_{v_{0}}$ by \cite[Theorem 1]{ZCZ}:

\begin{enumerate}
\item[(i)] $Soc(G_{\Delta }^{\Delta })\cong PSL_{m}(s)$, $m \geq 2$, $s=s_{0}^{e}$, $s_{0}$ prime, $e \geq 1$ and $v_{0}=\frac{%
s^{m}-1}{s-1}$;

\item[(ii)] $Soc(G_{\Delta }^{\Delta })\cong PSU_{3}(s)$, $s=2^{e}$, $e> 2$,
and $v_{0}=s^{3}+1$;

\item[(iii)] $Soc(G_{\Delta }^{\Delta })\cong Sz(s)$, $s=2^{e}$, $e\geq 3$ odd, and 
$v_{0}=s^{2}+1$.
\end{enumerate}

Note that $v_{0}-1= q^{2}\left( q -2\right)$. Then $2^{et}= q^{2}\left( q -2\right)$ with $t=3$ or $2$ in (ii) and (iii),
respectively, and so $q=4$ which is not the case since $et \geq 6$. Thus (i) holds, and hence 
\begin{equation}\label{vigilia}
2q^{2}\left( \frac{q}{2} -1\right)=s\frac{s^{m-1}-1}{s-1} \text{.}
\end{equation}
Moreover, it is easy to see that $m>2$ since $q \geq 8$ is an odd power of $2$, and hence $s<q$.

Let $x,y$ be any two distinct points of $\mathcal{D}_{0}$. Then $G^{\Delta }_{x,y}$ acts transitively on the $\lambda$ elements of $\mathcal{B}_{0}(x,y)$ by Lemma \ref{SameRank}(2). It follows from \cite[Proposition 4.1.17(II) and Table 3.5.G]{KL}, \cite[p.99]{DM} and by (\ref{vigilia}) that, either $\lambda \mid (m,s-1)\cdot e$ with $s_{0}^{e}=s$, or the group induced by $G^{\Delta }_{x,y}$ on $\mathcal{B}_{0}(x,y)$ contains $PSL_{u}(q)$ as a normal subgroup and $\lambda=\frac{s^{u}-1}{s-1}$ with $u \in \{2,m-2\}$. If $\lambda \mid (m,s-1)\cdot e$, then $\lambda \mid e$ since $s<q$, and hence $\frac{s^{m}-1}{s-1}=v_{0}\leq s^{1/2}(s+s^{1/2}-1)$ since $e\leq s^{1/2}$. So $m=2$, which we saw being impossible. Therefore $q-1=\lambda=\frac{s^{u}-1}{s-1}$ with $u \in \{2,m-2\}$, and hence
\begin{equation}\label{patience}
\left(\frac{s^{u}-1}{s-1}+1\right)^{2}s\frac{s^{u-1}-1}{s-1}=s\frac{s^{m-1}-1}{s-1}    
\end{equation}
by (\ref{vigilia}). It is easy to see that there are no admissible solutions for $u=2$. So $u=m-2$, and hence $(m-4)+2(m-3) \leq m-1$, which leads to $m=3$ since $m>2$. However, there are no admissible solutions for $m=3$. So this cases is excluded, and the proof is thus completed.
\end{proof}

\bigskip 

\bigskip

\subsection{The case where $T$ is a classical group} In this section, we assume that $T$ is a classical group. 

\bigskip

\begin{lemma}
\label{ClassicalAS} Assume that Hypothesis \ref{hyp2} holds. $T_{\Delta }$
lies in a maximal geometric subgroup of $T$.
\end{lemma}

\begin{proof}
Assume that $T_{\Delta }$ is an almost simple irreducible subgroup of the
classical group of $T$. Let $M$ be a maximal subgroup of $T$ containing $%
T_{\Delta }$. Then $M$ is a large subgroup of $T$ by Lemma \ref{Tlarge}(2),
and hence the pair $(T,M)$ is listed in \cite[Theorem 4(ii) and Table 7]{AB}%
. We preliminary filter the pair $(T,M)$ with respect $\left\vert
T\right\vert ^{2}\leq \left\vert M\right\vert \cdot \left\vert
Out(T)\right\vert $ by Lemma \ref{Tlarge}(1) since $T_{\Delta }\leq M$, and we
see that no cases occur for $n>10$ (note that, for $n>10$ the cases are all
numerical). Thus $n\leq 10$, and by a direct inspection of \cite[Section 8]%
{BHRD} we see that $G_{\Delta }^{\Sigma }$ is not a novelty in $G^{\Sigma }$
for any of the groups listed in \cite[Table 7]{AB}, hence $T_{\Delta }=M$.
Thus $n\leq 10$ and $(T,T_{\Delta })$ are listed in \cite[Table 7]{AB}. Now,
filtering such a list with respect  $\left\vert T\right\vert ^{2}\leq
\left\vert T_{\Delta }\right\vert \cdot \left\vert Out(T)\right\vert $, the numerical cases with respect to $\left( \frac{r_{1}}{(r_{1},\lambda _{1})%
}\right) ^{2}>v_{1}$ and $\frac{r_{1}}{(r_{1},\lambda _{1})}\mid \left(
v_{1}-1,\left\vert T_{\Delta }\right\vert \cdot \left\vert Out(T)\right\vert
\right) $, and bearing in mind that $PSp_{4}(2)^{\prime }\cong
PSL_{2}(9)\cong A_{6}$ has been rule out in Lemma \ref{noAlternating}, we obtain the
following table of admissible cases.

\begin{table}[h!]
\tiny
\caption{Admissible pairs $(T,T_{\Delta })$ with $T$ classic and $T_{\Delta }$ almost simple irreducible subgrouop of $T$ and relative parameters for $\mathcal{D}_{1}$}\label{Cl1}
\begin{tabular}{cllllc}
\hline
Line & $T$ & $T_{\Delta }$ & $v_{1}$ &  $\frac{r_{1}}{(r_{1},\lambda _{1})}$
divides & Condition on $q$ \\
\hline
1 & $P\Omega _{8}^{+}(q)$ & $\Omega _{7}(q)$ & $\frac{1}{2}q^{3}(q^{4}-1)$ & 
$13\cdot \log _{p}q$ & $q$ odd \\ 
2 &  & $Sp_{6}(q)$ & $q^{3}(q^{4}-1)$ & $7\cdot \log _{p}q$ & $q$ even \\ 
3 &  & $P\Omega _{8}^{-}(q^{1/2})$ & $ q(q+1)(q^{3}+1)\left(
q^{4}-1\right) $ & $\left( q^{2}+q+1\right) \cdot \log _{p}q$ & - \\ 
4 & $\Omega _{7}(q)$ & $G_{2}(q)$ & $\frac{1}{2}q^{3}(q^{4}-1)$ & $13\cdot
\log _{p}q$ & $q$ odd \\ 
5 & $PSp_{6}(q)$ & $G_{2}(q)$ & $q^{3}\left( q^{4}-1\right) $ & $7\cdot \log
_{p}q$ & $q$ even \\ 
6 & $PSp_{4}(q)$ & $Sz(q)$ & $q^{2}\left( q-1\right) \left( q+1\right) ^{2}$
& $(q^{2}+1)\cdot \log _{p}q$ & $q$ even \\ 
7 & $PSU_{3}(3)$ & $PSL_{2}(7)$ & $36$ & $7$ & $q=3$ \\ 
8 & $PSL_{2}(11)$ & $A_{5}$ & $11$ & $10$ & $q=11$\\
\hline
\end{tabular}%
\end{table}

Since $G^{\Sigma }$ acts point-primitively on $\mathcal{D}_{1}$ and $%
T\trianglelefteq G^{\Sigma }$, it follows that $v_{1}=\left\vert T:T_{\Delta
}\right\vert $. So, for instance, in Line 1 of Table \ref{Cl1}, we have $q$ odd and
$v_{1}=\left\vert P\Omega _{8}^{+}(q):\Omega _{7}(q)\right\vert =\frac{1}{2}%
q^{3}(q^{4}-1)$. Then 
\begin{eqnarray*}
\left( v_{1}-1,\left\vert T_{\Delta }\right\vert \cdot \left\vert
Out(T)\right\vert \right)  &=&\left( \frac{1}{2}q^{3}(q^{4}-1)-1,\left\vert
\Omega _{7}(q)\right\vert \cdot 6\cdot (4,q^{4}-1)\cdot \log _{p}q\right)  \\
&=&\left( \frac{1}{2}q^{3}(q^{4}-1)-1,\left( q^{2}+q+1\right) \cdot \left(
q^{2}-q+1\right) \cdot \log _{p}q\right) \text{.}
\end{eqnarray*}

Now, $\left( \frac{1}{2}q^{3}(q^{4}-1)-1,q^{2}+q+1\right) $ divides 
$13$ and $\left( \frac{1}{2}q^{3}(q^{4}-1)-1,q^{2}-q+1\right) =1$. Then $%
\left( v_{1}-1,\left\vert T_{\Delta }\right\vert \cdot \left\vert
Out(T)\right\vert \right) $ divides $13\cdot \log _{p}q$, and hence $\frac{%
r_{1}}{(r_{1},\lambda _{1})}\mid 13\cdot \log _{p}q$ since  $\frac{r_{1}}{%
(r_{1},\lambda _{1})}\mid \left( v_{1}-1,\left\vert T_{\Delta }\right\vert
\cdot \left\vert Out(T)\right\vert \right) $ by Theorem \ref{Teo4}(4). The remaining entries of Table
\ref{Cl1} are obtained similarly. Then $\left( \frac{r_{1}}{(r_{1},\lambda _{1})}%
\right) ^{2}<v_{1}$ for each case as in Lines 1--6 of Table \ref{Cl1}, and hence
they are excluded by Theorem \ref{Teo4}(3).

Finally, assume that case as in Lines 7--8 of Table \ref{Cl1} occurs. Then $\lambda
=7$ or $5,11$, respectively, since $\lambda $ divides the order of $T$ by Lemma \ref{Tlarge}(2) and $\lambda >3$. Actually, only $T \cong PSL_{2}(11)$ with $\lambda =5$ is
admissible since $\frac{r_{1}}{(r_{1},\lambda _{1})}> \lambda$ by Theorem \ref{Teo4}(3). Moreover $\mathcal{D}_{1}$ is of type Ia by Theorem \ref%
{Teo1} since $v_{1}=11$, and $A=1$ and $\frac{v_{0}-1}{k_{0}-1}=10$ since $\frac{r_{1}}{(r_{1},\lambda _{1})}%
=\frac{v_{0}-1}{k_{0}-1}$ by (\ref{double}) and $v_{1}=A\frac{v_{0}-1}{k_{0}-1}+1$ by
Proposition \ref{P2}(3). Thus $k_{1}=\frac{v_{0}}{k_{0}}+1$ again by by
Proposition \ref{P2}(3), and hence $zk_{0}+1=\frac{v_{0}-1}{k_{0}-1}=10$ by Theorem \ref{Teo1} since $\mathcal{D}_{1}$ is of type Ia. Then $%
(z,k_{0},v_{0})=(1,9,81)$ or $(3,3,21)$ since $k_{0}\geq 2$. Then $k_{1}=10$
or $8$, respectively. Actually, only the latter is admissible since $%
k_{1}<v_{1}$. Thus $k_{1}=8$, $%
G^{\Sigma }\cong PGL_{2}(11)$ and $\left( b_{1},r_{1},\lambda _{1}\right)
=\left( 55,40,28\right) $ or $\left( 165,120,84\right) $, and both cases are
excluded since they contradict $\lambda _{1}=\frac{v_{0}}{k_{0}}\cdot \frac{%
v_{0}}{k_{0}\eta }\cdot \lambda $ with $k_{0}\eta \mid v_{0}$ and $\lambda =5$. This completes the
proof.
\end{proof}

\bigskip

In the next proposition, the proof strategy used to obtain Table \ref{Clfin} is an adaptation of a group-theoretical argument used by Saxl in \cite{Saxl} (which is the analog for the almost simple groups of that of used for affine groups by Liebeck in \cite{LiebF}): the $2$-$(v^{\prime},k^{\prime},1)$ designs $\mathcal{D}^{\prime}$ (linear spaces) admitting a flag-transitive group $T \unlhd H \leq Aut(T)$ with $T$ non abelian simple are such that the group $H_{x}$, are those fulfilling the following properties: 
\begin{equation}\label{Sax}
r^{\prime}\mid (v^{\prime}-1,c_{1},...,c_{j},\left\vert H_{x}\right\vert)
\text{ and }
r^{\prime}>(v^{\prime})^{1/2}\textit{,}
\end{equation}
where $r^{\prime}=\frac{v^{\prime}-1}{k^{\prime}-1}$, $x$ is any point of  $\mathcal{D}^{\prime}$, and $c_{1},...,c_{j}$ are the lengths of the $H_{x}$-orbit on the set of points of  $\mathcal{D}^{\prime}$ distinct from $\{x\}$ by \cite[Lemma 2.1]{Saxl}. In other words, in \cite{Saxl}, an almost simple group $H$ is candidate to be a flag-transitive automorphism group of a linear space with $v^{\prime}$ points if $H$ admits a permutation representation of degree $v^{\prime}$ with nontrivial subdegrees divisible by the same factor of size greater than square root of $v^{\prime}$. In our context, the constraints in (\ref{Sax}) are replaced by
\begin{equation}\label{rustico}
\frac{r_{1}}{(r_{1},\lambda_{1})}\mid (v_{1},c_{1},...,c_{j},\left\vert G_{\Delta}^{\Sigma}\right\vert)
\text{~and~}
\frac{r_{1}}{(r_{1},\lambda_{1})}>(v_{1})^{1/2} \text{,}
\end{equation}
which hold by Theorem \ref{Teo4}(3)--(4) since $G^{\Sigma}$ is allmost simple. Hence, we may use Saxl's argument with $G^{\Sigma}$, $v_{1}$ and $\frac{r_{1}}{(r_{1},\lambda_{1})}$ in the role of $H$, $v^{\prime}$ and $r^{\prime}$, respectively, to reduce our investigation to the cases where (\ref{rustico}) is fulfilled. We did so, and we obtained as admissible cases those recorded in Table \ref{Clfin}.

\bigskip

\begin{proposition}\label{only2}
  Assume that Hypothesis \ref{hyp2} holds. One of the following holds:
 \begin{enumerate}
     \item $T\cong PSL_{n}(q)$, $n \geq 2$ and $(n,q)\neq(2,2),(2,3)$;
     \item $T\cong Sp_{n}(2)$ and $v_{1}=2^{n-1}(2^{n}+\varepsilon1)$, $\varepsilon=\pm$.
 \end{enumerate}
\end{proposition}

\begin{proof}
Assume that $T$ is not isomorphic to $T\cong PSL_{n}(q)$, $n \geq 2$ and $(n,q)\neq(2,2),(2,3)$. Recall that Table \ref{Clfin} consists of the admissible simple classical groups $T$ with $T_{\Delta}$ geometric fulfilling (\ref{rustico}) admissible cases. Further $\lambda$ is a prime, $\lambda \nmid v_{1}$ by Theorems \ref{Teo1} since $\mathcal{D}_{1}$ is of type Ia or Ic by Theorem \ref{Teo4}(1), and $\lambda \mid \left\vert T \right \vert$ by Lemma \ref{Tlarge}(2). Based on the previous constraints on $\lambda$ for each of the cases as in Table \ref{Clfin} we determined an upper bound $u_{\lambda}$ on the possible values of $\lambda$. It is a routine exercise to check that, $v_{1} \geq 2u_{\lambda}^{2}$ for each case as in Table \ref{Clfin}, except for case as in lines 2 or 5. Therefore, $v_{1} \geq 2\lambda^{2}$ for each case as in Table \ref{Clfin}, except for case as in lines 2 or 5, as $u_{\lambda} \geq \lambda$. Moreover, $v_{0}>\lambda +1$ since $\frac{v_{0}-1}{k_{0}-1}>\lambda $ by Theorem \ref{Teo4}(3) and (\ref{double}) and $k_{0}\geq 2$. This leads to $v=v_{0}v_{1}>2\lambda
^{2}(\lambda +1)$ since  $v_{1} \geq \lambda^{2}$, but this is contrary to \cite[Theorem 1]{DP}. The same conclusion hold for the case as in line 2 of Table \ref{Clfin} when $\lambda \leq \frac{q^{2}+1}{2}$. Thus $T \cong PSp_{4}(q)$, $q$ even, $v_{1}=\frac{q^{2}}{2}(q^{2}-1)$ and $\lambda =q^{2}+1$. As pointed out in \cite[p.329]{Saxl}, the subdegrees of $PSp_{4}(q)$, $q$ even, in its primitive action of degree $\frac{q^{2}}{2}(q^{2}-1)$ are one equal to $(q^{2}+1)(q-1)$ and $\frac{q}{2}-1$ ones equal to $q(q^{2}+1)$. Then $\frac{r_{1}}{(r_{1},\lambda_{1})} \mid q^{2}+1$ by Theorem \ref{Teo4}(4), whereas $\frac{r_{1}}{(r_{1},\lambda_{1})} >\lambda= q^{2}+1$ again by Theorem \ref{Teo4}. Thus, also the case as in line 2 of Table \ref{Clfin} is ruled out, and hence it remains the case as in line 5 to be analyzed. In this case $n \neq 2$ since $Sp_{2}(q)\cong PSp_{2}(q)\cong PSL_{2}(q)$ for $q$ even contradicts our assumptions. Hence, $n \geq 4$. Moreover, $v_{1}<\leq 2 u_{\lambda}^{2}$, otherwise we may use the previous argument to rule out this case. Hence, we have
\begin{equation*}
\frac{q^{n/2}}{2}(q^{n/2}+\varepsilon 1) <2 \left( \frac{q^{n/2}-\varepsilon 1}{q-\varepsilon 1} \right)^{2}\text{,}    
\end{equation*}
from which we derive
\begin{equation}\label{papaveri}
(q-\varepsilon 1)^{2} <4 \cdot \frac{q^{n/2}-\varepsilon 1}{q^{n/2}+ \varepsilon 1} \cdot \frac{q^{n/2}-\varepsilon 1}{q^{n/2}}\text{.}    
\end{equation}
Thus either $\varepsilon=+$, $(q-\varepsilon 1)^{2} <4$ and $q=2$, or $\varepsilon=-$, $(q+1)^{2} <8$ since $n \geq 4$, and again $q=2$. Therefore $q=2$ in any case when line 5 of Table \ref{Clfin} holds, and we obtain (2). This completes the proof.    
\begin{sidewaystable}
\small
\caption{Admissible pairs $(T,T_{\Delta })$ with $T$ classical group not isomorphic to $PSL(n,q)$, $T_{\Delta }$ a geometric subgroup of $T$ and some parameters of $\mathcal{D}_{1}$}\label{Clfin}
\begin{tabular}{llcllll}
\hline
Line & $T$ & Aschbacher class & Type of $T_{\Delta }$ & $v_{1}$ & $u_{\lambda}$
 & Condition on $n$ and/or $q$ \\ 
\hline
1 & $PSp_{4}(q)$ & $\mathcal{C}_{2}$ & $Sp_{2}(q)\wr Z_{2}$ & $\frac{%
q^{2}(q^{2}+1)}{2}$ & $q+1$ & $q\equiv 2\pmod{3%
}$ \\ 
2 & $PSp_{6}(q)$ & $\mathcal{C}_{3}$ & $Sp_{2}(q^{3})$ & $\frac{%
q^{6}(q^{4}-1)(q^{2}-1)}{3}$ & $q^{2}+q+1$
& $%
\begin{tabular}{l}
$r_{1}/(r_{1}/\lambda _{1})$ prime to $q+1$ \\ 
$q^{3}+1\mid r_{1}$%
\end{tabular}%
$ \\ 
3 & $PSp_{4}(q)$ & $\mathcal{C}_{3}$ & $Sp_{2}(q^{2})$ & $\frac{%
q^{2}(q^{2}-1)}{2}$ & $q^{2}+1$ & $q$ even \\ 
4 & $PSp_{n}(q)$ & $\mathcal{C}_{3}$ & $GU_{n/2}(q)$ & $\frac{1}{2}q^{\frac{%
n(n+1)}{8}}\prod_{i=1}^{n/2}(q^{i}+(-1)^{i})$ & $\frac{q^{n/2}-(-1)^{n/2}}{q-(-1)^{n/2}}$ & $q$ odd \\ 
5 & $PSp_{n}(q)$ & $\mathcal{C}_{8}$ & $O_{n}^{\varepsilon }(q)$ & $\frac{%
q^{n/2}(q^{n/2}+\varepsilon 1)}{2}$ & $\frac{q^{n/2}-\varepsilon 1}{q-\varepsilon 1}$ & $q$ even \\ 
6 & $\Omega _{n}(q)$ & $\mathcal{C}_{5}$ & $O_{n}(q^{1/2})$ & $\frac{1}{2}%
q^{(n-1)^{2}/8}\prod_{j=1}^{(n-1)/2}(q^{j}+1)$ & $\frac{q^{(n-1)/2}-1}{q-1}$ & $n,q$ odd, $n\geq 5$.
\\ 
7 & $P\Omega _{10}^{+}(q)$ & $\mathcal{C}_{1}$ & $P_{5}$ & $%
(q+1)(q^{2}+1)(q^{3}+1)(q^{4}+1)$ & $\frac{q^{5}-1}{q-1}$ & $q$ odd \\ 
8 & $P\Omega _{n}^{\varepsilon }(q)$ & $\mathcal{C}_{3}$ & $O_{n/2}(q^{2})$
& $\frac{(q^{n/2}-\varepsilon 1)q^{(n^{2}-4)/8}}{(4,q^{n/2}-\varepsilon 1)}%
\prod_{j=1}^{(n-2)/4}(q^{4j-2}-1)$ & $\frac{q^{n-2}-1}{q-1}$ & $n/2$ and $q$ odd \\ 
9 & $P\Omega _{n}^{\varepsilon }(q)$ & $\mathcal{C}_{3}$ & $%
O_{n/2}^{\varepsilon }(q^{2})$ & 
\begin{tabular}{l}
$\frac{q^{n^{2}/8}(q^{n-2}-1)}{\rho }\prod_{j=1}^{n/4-1}(q^{4j-2}-1)$ \\ 
$\left( \rho ,\varepsilon \right) =(4,+),(2,-)$%
\end{tabular}
& $\frac{q^{n}-\varepsilon 1}{q-\varepsilon 1}$ & $n/2$ even and $q$ odd \\ 
10 & $PSU_{3}(q)$ & $\mathcal{C}_{1}$ & $P_{1}$ & $q^{3}+1$ & $q+1$ & - \\ 
11 & $PSU_{n}(q)$ & $\mathcal{C}_{1}$ & $GU_{1}(q)\perp GU_{n-1}(q)$ & $%
q^{n-1}\frac{q^{n}-(-1)^{n}}{q+1}$ & $\frac{q^{n-1}-(-1)^{n-1}}{q-(-1)^{n-1} }$ & - \\ 
12 & $PSU_{n}(q)$ & $\mathcal{C}_{5}$ & $O_{n}(q)$ & $\frac{%
2(q^{n}+1)q^{(n^{2}-1)/4}}{(n,q+1)}\prod_{j=1}^{(n-1)/2}(q^{2j+1}+1)$ & $\frac{q^{n-1}-1}{q-1}$ & $%
n$ and $q$ odd \\ 
13 & $PSU_{n}(q)$ & $\mathcal{C}_{5}$ & $O_{n}^{\varepsilon }(q)$ & $\frac{%
(2,q-1)^{2}(4,q^{n/2}-\varepsilon 1)q^{n^{2}/4}(q^{n/2}+\varepsilon 1)}{%
4(n,q+1)}\prod_{j=1}^{n/2}(q^{2j+1}+1)$ & $\frac{q^{n}-\varepsilon 1}{q-\varepsilon 1}$ & $n$ even \\ 
14 & $PSU_{n}(q)$ & $\mathcal{C}_{5}$ & $Sp_{n}(q)$ & $\frac{q^{n(n-2)/4}}{%
(q+1,n/2)}\prod_{j=1}^{n/2}(q^{2j+1}+1)$ & $\frac{q^{n}-1}{q-1}$ & $n$ even \\
\hline
\end{tabular}%
\end{sidewaystable}
\end{proof}

\normalsize

\bigskip

\begin{proposition}\label{OnlyPSL}
 Assume that Hypothesis \ref{hyp2} holds. $T\cong PSL_{n}(q)$, $n\geq 2$ and $(n,q)\neq (2,2),(2,3)$.
\end{proposition}

\begin{proof}
Assume that $T\cong Sp_{n}(2)$, $T_{\Delta }\cong
O_{n}^{\varepsilon }(2)$ and $v_{1}=2^{n/2-1}(2^{n/2}+\varepsilon 1)$. We
are going to rule out this case in a series of steps, thus obtaining the
assertion by Lemma \ref{only2}. Assume that $\mathcal{D}_{1}$ is of type Ic. Then
\begin{equation*}
2^{n/2-1}\left( 2^{n/2}+\varepsilon 1\right) =\frac{\lambda -1}{2}\cdot
\left( \lambda ^{2}-2\right) 
\end{equation*}
by Theorem \ref{Teo1}, and hence $2^{n/2}\mid \lambda -1$ and $\lambda
^{2}-2\mid 2^{n/2}+\varepsilon 1$. So $\left( 2^{n/2}+1\right) ^{2}-2\leq
2^{n/2}+1$, which is impossible. Thus $\mathcal{D}_{1}$ is of type Ia by
Theorem \ref{Teo4}, and hence $\lambda \nmid v_{1}(v_{1}-1)$.

\begin{claim}
$n\geq 14$.
\end{claim}

Note that, $n>4$ since $Sp_{4}(2)^{\prime }\cong A_{6}$ has been already
excluded in Lemma \ref{noAlternating}. Further, if $n\leq 12$ then $(n,\varepsilon ,v_{1},\lambda
)=(6,-,28,5)$, $(8,+,136,7)$, $(10,-,\allowbreak 496,\lambda )$ with $%
\lambda =7,17$, or $(10,+\allowbreak ,\allowbreak 528,\lambda )$ with $%
\lambda =5,7$, or $(n,\varepsilon ,v_{1},\lambda )$ is either $%
(12,-,2016,\lambda )$ with $\lambda =11,17$, or $(12,+,2080,\lambda )$ with $%
\lambda =17,31$ since $v_{1}=2^{n/2-1}(2^{n/2}+\varepsilon 1)$ and $\lambda
\nmid v_{1}(v_{1}-1)$. Now, for each divisors $\alpha $ of $v_{1}-1$ such that $%
\alpha >\lambda $ and $\alpha ^{2}>v_{1}$ we compute the admissible $A(\alpha )=%
\frac{v-1}{\alpha }$. Further, the admissible $v_{0}$ are given by $%
v_{0}(\alpha )=\alpha (k_{0}-1)+1$ and hence $k_{0}\mid \alpha -1$ since $k_{0} \mid v_{0}$. Hence,
for \ each divisor $\beta $ of $\alpha -1$ such that $1<\beta <\lambda $, we
compute $v_{0}(\alpha ,\beta )=\alpha (\beta -1)+1$ and $k_{1}(\alpha ,\beta
)=A\frac{v_{0}}{\beta }+1$ with $v_{1}/2<k_{1}(\alpha ,\beta
)<v_{1}$. Therefore, the admissible parameters for $\mathcal{D}_{0}$ and $\mathcal{D}_{1}$ are recorded in Table \ref{TavSp}.
\begin{table}[h!]
 \caption{Admissible parameters for $\mathcal{D}_{0}$ and $\mathcal{D}_{1}$ for  $T\cong Sp_{n}(2)$ with $n <14$}
    \label{TavSp}
    \centering
\begin{tabular}{cllllllll}
\hline
Line & $n$ & $\varepsilon $ & $A$ & $v_{0}$ & $k_{0}$ & $v_{1}$ & $k_{1}$ & $%
\lambda $ \\
\hline
1 & $6$ & $-$ & $1$ & $28$ & $2$ & $28$ & $15$ & $5$ \\ 
2 & $8$ & $+$ & $1$ & $136$ & $2$ & $136$ & $69$ & $7$ \\ 
3 & $10$ & $-$ & $1$ & $496$ & $2$ & $496$ & $249$ & $7$ \\ 
4 &  &  & $1$ & $496$ & $2$ & $496$ & $249$ & $17$ \\ 
5 &  &  & $1$ & $5491$ & $13$ & $496$ & $458$ & $17$ \\ 
6 & $10$ & $-$ & $1$ & $528$ & $2$ & $528$ & $265$ & $5$ \\ 
7 &  &  & $1$ & $528$ & $2$ & $528$ & $265$ & $7$ \\ 
8 & $12$ & $-$ & $1$ & $2016$ & $2$ & $2016$ & $1009$ & $11$ \\ 
9 &  &  & $1$ & $2016$ & $2$ & $2016$ & $1009$ & $17$ \\ 
10 & $12$ & $+$ & $1$ & $2080$ & $2$ & $2080$ & $1041$ & $17$ \\ 
11 &  &  & $1$ & $2080$ & $2$ & $2080$ & $1041$ & $31$\\
\hline
\end{tabular}%
\end{table}
However, only case as in line 1 of Table \ref{TavSp} is admissible since $\lambda \mid k_{1}$ by Theorem \ref{Teo4} and 
$v_{1}\leq 2\lambda ^{2}(\lambda -1)$ by \cite[Theorem 1]{DP}. In this case, $T\cong Sp_{6}(2)$ and $%
T_{\Delta }\cong O_{6}^{+}(2)\cong S_{8}$, and hence $T_{x}\cong S_{6}\times
Z_{2}$, which is not divisible $r=135$. Thus $n>12$, and hence $n\geq 14$
since $n$ is even.

\begin{claim}
If $G_{(\Sigma)}\neq 1$, then $%
T_{B^{\Sigma }}$ is a large subgroup of $T$.
\end{claim}

Easy computations show that, $n(n+1)/2-1>2n-2+4\log (2^{n}+1)$ since $n \geq 14$. Moreover, $%
\left\vert T\right\vert >2^{n(n+1)/2-1}$ by \cite[Corollary 4.3]{AB}.
Therefore,%
\[
\left\vert T\right\vert >2^{n(n+1)/2-1}>2^{2n-4}(2^{n}+1)^{4}\geq
2^{2n-4}(2^{n}+\varepsilon 1)^{2}=v_{1}^{4}>b_{1}^{2}. 
\]%
since $b_{1}\mid v_{1}(v_{1}-1)$ by Lemma \ref{quasiprimitivity}(2). Then 
\[
\left\vert T_{B^{\Sigma }}\right\vert ^{2} \geq \frac{\left\vert T\right\vert ^{2}%
}{b_{1}^{2}}>\left\vert T\right\vert \frac{\left\vert T\right\vert}{%
b_{1}^{2}}>\left\vert T\right\vert 
\]%
and hence $T_{B^{\Sigma }}$ is a large subgroup of $T$.

\begin{claim}
 Let $M$ be a maximal
subgroup of $T$ containing $T_{B^{\Sigma }}$. If $G_{(\Sigma)}\neq 1$, then $M\in \mathcal{C}_{1}(T)$.
\end{claim}

Assume that $G_{(\Sigma)} \neq 1$. Let $M$ be a maximal
subgroup of $T$ containing $T_{B^{\Sigma }}$. Then $M$ is a large subgroup of $T$ by Claim 2, and hence $M$ is one of the groups
listed in \cite[Proposition 4.22 and Table 7]{AB}. In particular $M$ is
geometric since $\left\vert M\right\vert ^{2}>\left\vert T\right\vert $ and $%
n\geq 2$. Moreover, since $\left\vert T:M\right\vert $ divides $b_{1}$ and
hence $v_{1}(v_{1}-1)=2^{n/2-1}(2^{n}-1)\left( 2^{n/2-1}+\varepsilon
1\right) $ it follows that $M$ is divisible by a primitive prime divisor of $%
2^{n-2}-1$ or $2^{n-4}-1$ according as $\varepsilon =-$ or $\varepsilon =+$,
respectively, since $n \geq 14$. Therefore, either $M\in \mathcal{C}_{1}(T)\cup 
\mathcal{C}_{8}(T)$ or $M$ is a $\mathcal{C}_{3}$-subgroup of type $%
Sp_{n/2}(4)$ by \cite[Proposition 4.22]{AB} and \cite[Proposition 4.3.7(II)]{KL}. In the latter case, $2^{\frac{1}{%
8}n^{2}-1}$ divides $\left\vert T:M\right\vert $, and hence divides $v_{1}(v_{1}-1)$%
, by \cite[Proposition 4.3.10(II)]{KL}. Therefore $\frac{1}{8}n^{2}-1\leq
n/2-1$, which is contrary to $n\geq 14$.

Assume that $M\in \mathcal{C}_{8}(T)$. Then $M\cong O_{n}^{\varepsilon
^{\prime }}(2)$ by \cite[Proposition 4.8.6(II)]{KL}, and hence $\left\vert
T:M\right\vert =2^{n/2-1}(2^{n/2}+\varepsilon ^{\prime }1)$. Suppose that $%
T_{B^{\Sigma }}<M$. Then $\left\vert M:T_{B^{\Sigma }}\right\vert \geq
P(M^{\prime })$, where $M^{\prime} \cong \Omega_{n}^{\varepsilon
^{\prime }}(2)$ and $P(M^{\prime })$ is the minimum of the non-trivial transitive permutation degree of $M^{\prime}$. Then either $\left\vert M:T_{B^{\Sigma }}\right\vert
\geq 2^{n/2}(2^{n/2-1}-1)$ or $\left\vert M:T_{B^{\Sigma }}\right\vert
>2^{n/2-1}(2^{n/2}-1)$ according as $\varepsilon ^{\prime }=-$ or $%
\varepsilon ^{\prime }=+$ occurs, respectively, by \cite[Proposition 5.2.1(i) and Theorem 5.2.2]{KL}. In each case, we have 
\begin{equation*}
\left\vert T:T_{B^{\Sigma }}\right\vert \geq 2^{n/2-1}(2^{n/2}+\varepsilon
^{\prime }1)\cdot 2^{n/2}(2^{n/2-1}-1) \leq 2^{n-1}(2^{n/2}-1)(2^{n/2-1}-1)\text{.}
\end{equation*}

If $b_{1}\leq v_{1}(v_{1}-1)/2$, then $\left\vert T:T_{B^{\Sigma }}\right\vert \leq v_{1}(v_{1}-1)/2$ and hence 
\begin{equation*}
2^{n-1}(2^{n/2}-1)(2^{n/2-1}-1)\leq \left\vert T:T_{B^{\Sigma }}\right\vert\leq 2^{n/2-2}(2^{n/2}+1)\left(
2^{n/2-1}(2^{n/2}+1)-1\right) 
\end{equation*}
forcing $n\leq 4$, a contradiction. Thus, $b_{1}=v_{1}(v_{1}-1)$, $A=1$ and $k_{1}=\lambda$
since $b_{1} \cdot \frac{k_{1}}{\lambda}= \frac{v_{1}(v_{1}-1)}{A}$ by Lemma \ref{quasiprimitivity}(2). Now, $\lambda \mid \left\vert T_{\Delta }\right\vert $ since $\lambda \nmid \left\vert Out(T)\right\vert $ by Lemma \ref{Tlarge}(2) and $\lambda \nmid v_{1}$. Thus $%
\lambda \mid 2^{n-2i}-1$ with $i>1$ or $\lambda \mid 2^{n/2-1}-\varepsilon 1$ since with $T_{\Delta
}\cong O_{n}^{\varepsilon }(2)$ and $\lambda \nmid v_{1}-1$. In both cases, $\lambda \leq 2^{n-4}-1$ since $n \geq 14$. Then 
\[
2^{n/2-1}(2^{n/2}+\varepsilon1)=v_{1}<2k_{1}=2\lambda <2^{n-3}\text{,} 
\]%
which is clearly impossible. Thus, we have $T_{B^{\Sigma }}=M\cong
O_{n}^{\varepsilon ^{\prime }}(2)$.

If $\varepsilon ^{\prime }=\varepsilon $%
, then $r_{1}=k_{1}$ and hence $\lambda l^{m-w}+1<\frac{l^{m}-1}{l^{w}-1}\lambda
=Al^{m-h}+1$ since $v_{0}=l^{m}$ and $k_{0}=l^{w}$ with $l$ prime and $1\leq w\leq m $ by Lemma \ref{quasiprimitivity}(1).
So $A>\lambda$, which is contrary to Proposition \ref{P2}(4). Therefore, we obtain $\varepsilon ^{\prime
}=-\varepsilon $. Then $b_{1}=2^{n/2-1}(2^{n/2}-\varepsilon 1)$ and $%
v_{1}=2^{n/2-1}(2^{n/2}+\varepsilon 1)$, and hence $\varepsilon =-$ since $%
v_{1}\leq b_{1}$. Thus $b_{1}=2^{n/2-1}(2^{n/2}+1)$ and $%
v_{1}=2^{n/2-1}(2^{n/2}-1)$, and hence $2^{n/2}-1\mid (k_{1},v_{1})$ since $b_{1}k_{1}=v_{1}r_{1}$. On the
other hand, $(k_{1},v_{1})\mid A+1$ by Lemma \ref{L3}(1), hence $A\geq
2^{n/2}-2$. Moreover, $\frac{v_{0}-1}{k_{0}-1}>\left(
2^{n/2-1}(2^{n/2}-1)\right) ^{1/2}$ by Theorem \ref{Teo4}(3). So

\[
(2^{n/2}+1)\left( 2^{n/2-1}-1\right) =A\frac{v_{0}-1}{k_{0}-1}> \left( 2^{n/2}-2\right) \left(
2^{n/2-1}(2^{n/2}-1)\right) ^{1/2}\text{,} 
\]%
which is impossible for $n\geq 14$.

\begin{claim}
$G_{(\Sigma)}= 1$.
\end{claim}
Assume that $M\in \mathcal{C}_{1}(T)$. If $M$ is the stabilizer of a nondegenerate subspace of $PG_{n-1}(2)$, then $2^{n^{2}/4-i^{2}/4-(n-i)^{2}/4}$
divides $\left\vert T:M\right\vert $ for some $1\leq i<n/2$, and hence $%
v_{1}(v_{1}-1)$, by \cite[Proposition 4.1.3(II)]{KL}. Then $%
n^{2}/4-i^{2}/4-(n-i)^{2}/4\leq n/2-1$ since $v_{1}=2^{n/2-1}(2^{n/2}+%
\varepsilon 1)$, hence $3/2+(i-1)n/2\leq 0$, contrary to $i\geq 1$. Thus $M$
is a maximal parabolic subgroup of $T$ of type $P_{h}$ for some $ 1 \leq h \leq n/2$, and hence%
\begin{equation} \label{divSP}
\prod_{j=0}^{h-1}\frac{2^{n-2j}-1}{2^{j+1}-1}\mid 2^{n/2-1}(2^{n}-1)\left(
2^{n/2-1}-\varepsilon 1\right) 
\end{equation}%
by \cite[Proposition 4.1.19(II)]{KL} since $\left\vert T:M \right\vert \mid v_{1}(v_{1}-1)$ and $v_{1}(v_{1}-1)=2^{n/2-1}(2^{n}-1)(2^{n/2-1}-\varepsilon1)$. Now $1 \leq h\leq n/2$ and $n \geq 14$
imply $h<n-4$, and hence the first part of (\ref{divSP}) is
divisible by primitive prime divisor $2^{n-4}-1$ for $h\geq 3$, whereas the second part of (\ref{divSP}) is not. Thus $h=1$ or $2$. Also, if $h=2$, then (\ref{divSP}) implies $\frac{1}{3}\left(
2^{n-2}-1\right) \left( 2^{n}-1\right) \mid (2^{n}-1)\left(
2^{n/2-1}+\varepsilon 1\right) $, which is impossible since $n\geq 14$. Therefore $h=1$, and hence $M \cong Y:Sp_{n-2}(2)$ with $Y \cong [2^{n-1}]$ by \cite[Proposition 4.1.19(II)]{KL}. Therefore  $\left\vert M: T_{B^{\Sigma }} \right\vert =\frac{1}{\delta} 2^{n/2-1}\left(
2^{n/2-1}-\varepsilon 1\right)$ for some $\delta \geq 1$, and hence
\begin{eqnarray*}
(2^{n-2}-1)\cdot \left\vert T_{B^{\Sigma }}Y:T_{B^{\Sigma }}\right\vert &=& P(M/Y) \cdot \left\vert T_{B^{\Sigma }}Y:T_{B^{\Sigma }}\right\vert \\ &\leq& \left\vert M/Y: T_{B^{\Sigma }}Y/Y\right\vert\cdot \left\vert T_{B^{\Sigma }}Y:T_{B^{\Sigma }}\right\vert \\
&=&  \left\vert M:T_{B^{\Sigma }}Y \right\vert\cdot \left\vert T_{B^{\Sigma }}Y:T_{B^{\Sigma }}\right\vert =  \left\vert M:T_{B^{\Sigma }} \right\vert \\ &=& \frac{1}{\delta}2^{n/2-1}\left(
2^{n/2-1}-\varepsilon 1\right)
\end{eqnarray*}

Thus $\varepsilon =-1$, $Y \leq  T_{B^{\Sigma }}$ and $\left\vert M/Y: T_{B^{\Sigma }}/Y \right\vert = 2^{n/2}(2^{n/2}+1)$. Consequently,

\begin{equation}\label{teneruzzolo}
\frac{b_{1}}{\theta}=\left\vert T:T_{B^{\Sigma }}\right\vert= \left\vert T: M \right\vert \cdot \left\vert M/Y: T_{B^{\Sigma }}/Y \right\vert= 2^{n/2-1}\left(
2^{n}-1\right) \left(
2^{n/2-1}+1\right)  
\end{equation}
for some $\theta \geq 1$ since $T \unlhd G^{\Sigma}$ and $G^{\Sigma}$ acts block-transitively on $\mathcal{D}_{1}$. That is, $b_{1}=\theta \cdot v_{1}(v_{1}-1)$, and hence $A\cdot \theta \cdot \frac{k_{1}}{\lambda}=1$ since $A \cdot b_{1} \cdot \frac{k_{1}}{\lambda}=v_{1}(v_{1}-1)$ by Lemma \ref{quasiprimitivity}(2). Thus $A=\theta=1$ and $k_{1}=\lambda$. On the other hand, Now, $\lambda \mid \left\vert T_{\Delta }\right\vert $ since $\lambda \nmid \left\vert Out(T)\right\vert $ by Lemma \ref{Tlarge}(2) and $\lambda \nmid v_{1}$. Thus $%
\lambda \mid 2^{n-2i}-1$ with $i>1$ since with $T_{\Delta
}\cong O_{n}^{- }(2)$ and $\lambda \nmid v_{1}-1$. In particular, $\lambda \leq 2^{n-4}-1$ since $n \geq 14$. Then
\[
2^{n/2-1}(2^{n/2}-1)=v_{1}<2k_{1}=2\lambda <2^{n-3}\text{,} 
\]%
which is impossible for $n \geq 14$.

\begin{claim}
$\mathcal{D}_{0}$ is a $2$-$(v_{0},k_{0},\lambda )$ design with $A\frac{%
v_{0}-1}{k_{0}-1}=(2^{n/2}+1)\left( 2^{n/2-1}-1\right) $, $G_{(\Delta)}=1$ and $\Omega_{n}^{\varepsilon }(2) \unlhd G_{\Delta} \leq GO_{n}^{\varepsilon }(2)$.
\end{claim}

It follows from Claim 4 that $T_{\Delta }\cong O_{n}^{\varepsilon }(2)$ and $G_{\Delta }\cong GO_{n}^{\varepsilon }(2)$. If $G_{(\Delta)} \neq 1$, then $T_{\Delta}^{\prime} \cong  \Omega_{n}^{\varepsilon }(2)$ is contained in $G_{(\Delta)}$, and hence $v_{0} \leq 2$ by \cite[Table 5.1A]{KL} since $n \geq 14$, whereas $v_{0}>3$ by Lemma \ref{L1}. Thus, one has $G_{(\Delta)}=1$ and $\Omega_{n}^{\varepsilon }(2) \unlhd G_{\Delta} \leq GO_{n}^{\varepsilon }(2)$.

If $\mu =\lambda $, then either $k_{0}=2$ or $\mathcal{D}_{0}$ is a linear
space by Lemma \ref{base}(2). In the first case, $G_{\Delta}^{\Delta}$ acts
point-$2$-transitively on $\mathcal{D}_{0}$, and hence either $n=4$, $%
\varepsilon =-$ and hence $A_{5} \unlhd G_{\Delta} \leq S_{5}$, or $n=6$, $\varepsilon
=+ $ and $A_{8} \unlhd G_{\Delta} \leq S_{8}$ (e.g. see \cite[Proposition 2.9.1]{KL}) since $\Omega_{n}^{\varepsilon }(2) \unlhd G_{\Delta} \leq GO_{n}^{\varepsilon }(2)$,
whereas $n\geq 14$. Therefore $\mathcal{D}_{0}$ is a linear space, and hence 
$\mathcal{D}_{0}\cong PG_{3}(2)$, $n=4$, $%
\varepsilon =-$ and $A_{5} \unlhd G_{\Delta}^{\Delta} \leq S_{5}$, or $\mathcal{D}_{0}\cong PG_{5}(2)$, $n=6$, $\varepsilon
=+ $ and $A_{8} \unlhd G_{\Delta}^{\Delta} \leq S_{8}$ by \cite[Theorem]{BDDKLS} and \cite[Proposition 2.9.1]{KL}, which is not the case since $n\geq 14$. Thus $\mu =1$, and hence $\mathcal{D}%
_{0}$ is a $2$-$(v_{0},k_{0},\lambda )$ design by Lemma \ref{base}(2). Moreover $A\frac{v_{0}-1}{k_{0}-1}=v_{1}-1%
=(2^{n/2}+1)\left( 2^{n/2-1}-1\right)$ by Proposition \ref{P2}(3).

\begin{claim}
Set $K=\left( T_{\Delta }\right) ^{\prime }$, then $\left\vert
K_{x}\right\vert ^{2}>\left\vert K\right\vert $. In particular, $K_{x}$ is a
large subgroup of $K$.
\end{claim}

Set $K=\left( T_{\Delta }\right) ^{\prime }$. Then $K \cong \Omega_{n}^{\varepsilon }(2)$ since $K$ acts faithfully on $\mathcal{D}_{0}$ by Claim 5. 
$\left\vert
K:K_{x}\right\vert =v_{0}$ since $K\trianglelefteq G_{\Delta }$ and $G_{\Delta }$
acts point-primitively on $\mathcal{D}_{0}$. If $C$ is any block of $\mathcal{D}_{0}$ containing $x$, then $\left\vert K_{x}:K_{x,C}\right\vert \mid \frac{v_{0}-1}{k_{0}-1}\lambda $ with $r_{0}=\frac{v_{0}-1}{k_{0}-1}\lambda$ since $K_{x} \unlhd G_{x}$. Therefore, 
\[
\left\vert K\right\vert =v_{0}\left\vert K_{x}\right\vert =v_{0}%
\frac{\left(v_{0}-1\right)\lambda}{e\left( k_{0}-1\right) } \left\vert
K_{x,B}\right\vert  
\]%
for some integer $e \geq 1$. If $\left\vert K_{x,B}\right\vert \leq e$, by Claim 5, one has 
\begin{eqnarray*}
\left\vert K\right\vert &\leq &\left( \frac{(2^{n/2}+1)\left(
2^{n/2-1}-1\right) (k_{0}-1)}{A}+1\right) \frac{(2^{n/2}+1)\left(
2^{n/2-1}-1\right) }{A }\lambda \\
&=&\frac{(2^{n/2}+1)^{2}\left( 2^{n/2-1}-1\right) ^{2}(k_{0}-1)}{A^{2}}\lambda +\frac{%
(2^{n/2}+1)\left( 2^{n/2-1}-1\right) }{A }\lambda \\
&=&(2^{n/2}+1)^{2}\left( 2^{n/2-1}-1\right) ^{2}\lambda^{2} +%
(2^{n/2}+1)\left( 2^{n/2-1}-1\right) \lambda \\
&\leq &(2^{n/2}+1)^{2}\left( 2^{n/2-1}-1\right)
^{2}(2^{n/2}+1)^{2}+(2^{n/2}+1)\left( 2^{n/2-1}-1\right) (2^{n/2}+1)\text{.}
\end{eqnarray*}%
Then 
\[
\frac{1}{8}2^{n(n-1)/2}<(2^{n/2}+1)^{2}\left( 2^{n/2-1}-1\right)
^{2}(2^{n/2}+1)^{2}+(2^{n/2}+1)\left( 2^{n/2-1}-1\right) (2^{n/2}+1) 
\]%
by \cite[Corollary 4.3(iv)]{AB} since $n \geq 14$, and we reach a contradiction. Therefore $\left\vert
K_{x,B}\right\vert >e$, and so 
\[
\left\vert K_{x}\right\vert =%
\frac{\left(v_{0}-1\right)\lambda}{e\left( k_{0}-1\right) } \left\vert
K_{x,B}\right\vert >\frac{v_{0}-1}{k_{0}-1}\lambda =%
\frac{v_{0}-1}{k_{0}-1}\lambda \geq v_{0}-1 \text{,}
\]%
and hence $\left\vert K_{x}\right\vert >v_{0}$ since $\lambda \mid \left\vert
K_{x}\right\vert $ and $\lambda \nmid v_{0}$. Then $\left\vert
K_{x}\right\vert ^{2}>v_{0}\left\vert K_{x}\right\vert = \left\vert
K\right\vert $, and hence $K_{x}$ is a large subgroup of $K$.

\begin{claim}
The final contradiction.
\end{claim}

Let $N$ be a maximal subgroup of $K$ containing $K_{x}$. Then $N$ is one of
the groups listed in \cite[Proposition 4.23 and Table 7]{AB} since $N$ is a large subgroup of $K$ by Claim 6. Actually, $%
N\in \mathcal{C}_{1}(K)\cup \mathcal{C}_{2}(K)\cup \mathcal{C}_{3}(K)$ since 
$\left\vert N\right\vert ^{2}>\left\vert K\right\vert $ and $n\geq 14$. If $%
N\notin \mathcal{C}_{1}(K)$, then one of \ the following holds by \cite[%
Sections 4.2--4.3]{KL}:

\begin{enumerate}
\item $N$ is either a $\mathcal{C}_{2}$-subgroup of type $%
O_{n/2}^{\varepsilon ^{\prime }}(2)\wr Z_{2}$ with $\varepsilon ^{\prime
}=\pm $ or a $\mathcal{C}_{2}$-subgroup of type $O_{n/2}^{\varepsilon
^{\prime }}(4)$ with $\varepsilon ^{\prime }=\pm $; and $2^{\frac{1}{8}%
n^{2}-2}\mid \left\vert K:N\right\vert $;

\item $N$ is either a $\mathcal{C}_{2}$-subgroup of type $GL_{n/2}(2)$ or a $%
\mathcal{C}_{3}$-subgroup of type $GU_{n/2}(2)$, and $2^{n(n-2)/8-\log
_{2}(2,n/2)\text{ }}\mid \left\vert K:N\right\vert $.
\end{enumerate}

Therefore, $2^{n(n-2)/8-2\text{ }}$ divides $\left\vert K:N\right\vert $ in
any case. Then $2^{n(n-2)/8-2\text{ }}\mid v_{0}$, and hence 
\begin{eqnarray*}
2^{n(n-2)/8-2-1}&\leq& v_{0}-1=\frac{(2^{n/2}+1)\left( 2^{n/2-1}-1\right)\left( k_{0}-1\right) }{A} \\
&\leq& (2^{n/2}+1)\left( 2^{n/2-1}-1\right) \lambda
\leq (2^{n/2}+1)\left( 2^{n-2}-1\right) 
\end{eqnarray*}
since $\lambda \leq 2^{n/2}+1$, and hence $n=14$ since $n$ is even and $n\geq 14$. This forces $\lambda =127$.
However, this case cannot occur since $\lambda \mid 2^{14}-1$ and $%
2^{14}\mid v_{1}(v_{1}-1)$, whereas $\lambda \nmid v_{1}(v_{1}-1)$.

Assume that $N\in \mathcal{C}_{1}(K)$. Suppose that $N$ is the stabilizer of
a non-singular $(i-1)$-subspace in $PG_{n-1}(2)$. Then $\allowbreak 2^{(in+1)/2}$
divides $\left\vert K:N\right\vert $ and hence $v_{0}$. Then 
\[
2^{(in+1)/2}\leq v_{0}-1 \leq (2^{n/2}+1)\left( 2^{n-2}-1\right) \text{,} 
\]%
and so $i=1$ or $2$. If $i=2$, then $n\leq 6$ since 
\begin{eqnarray*}
\frac{2^{n-2}(2^{n/2}-1)(2^{n/2-1}-1)}{3}&\leq& \frac{2^{n-2}(2^{n/2}-%
\varepsilon 1)(2^{n/2-1}+\varepsilon _{2}1)}{(2-\varepsilon _{1}1)}\\
&=&\left\vert K:N\right\vert \leq v_{0}\leq (2^{n/2}+1)\left( 2^{n-2}-1\right) +1
\end{eqnarray*}%
with $\varepsilon =\varepsilon _{1}\varepsilon _{2}$ by \cite[Proposition
4.1.6(II)]{KL}, contrary to $n\geq 14$. Therefore $i=1$, and $K_{x}=N$ by 
\cite[Tables 3.5.H--I]{KL}, and hence $v_{0}=2^{n/2-1}(2^{n/2}-\varepsilon
1) $. Thus $v_{0}-1\mid (v_{1}-1)(k_{0}-1)$ implies%
\[
\left( 2^{n/2}\allowbreak +\varepsilon 1\right) \left( 2^{n/2-1}\allowbreak
-\varepsilon 1\right) \mid \left( 2^{n/2}\allowbreak -\varepsilon 1\right)
\left( 2^{n/2-1}\allowbreak +\varepsilon 1\right) (k_{0}-1) 
\]%
since $v_{0}-1=\left( 2^{n/2}\allowbreak +\varepsilon 1\right) \left(
2^{n/2-1}\allowbreak -\varepsilon 1\right) $ and $v_{1}-1=\left(
2^{n/2}\allowbreak -\varepsilon 1\right) \left( 2^{n/2-1}\allowbreak
+\varepsilon 1\right) $. So $v_{0}-1\mid k_{0}-1$, contrary to $k_{0}<v_{0}$%
. Therefore, $N$ is a maximal parabolic subgroup of $T$ of type $P_{h}$, $1 \leq h \leq \left( n+\varepsilon 1-1\right) /2$, and hence

\[
{\left( n+\varepsilon 1-1\right) /2 \brack h}_{2} \cdot \prod_{j=0}^{h-1}(2^{%
\left( n-\varepsilon 1-1\right) /2-j}+1)\leq (2^{n/2}+1)\left(
2^{n-2}-1\right)+1 
\]%
by \cite[Proposition 4.1.19(II)]{KL}. Then 
\[
2^{h(\left( n+\varepsilon 1-1\right) /2)-h)}\cdot\prod_{j=0}^{h-1}(2^{\left( n-1-1\right)
/2-j}+1)\leq (2^{n/2}+1)\left( 2^{n-2}-1\right) 
\]
forcing $h=1$ since $n\geq 14$. Further, $T_{x}=N$ by \cite[Tables 3.5.H--I]%
{KL}. Therefore $v_{0}=(2^{n/2}-\varepsilon 1)(2^{n/2-1}+\varepsilon 1)$,
and hence%
\[
(2^{n/2}-\varepsilon 1)(2^{n/2-1}+\varepsilon 1)-1\mid \left(
2^{n/2} -\varepsilon 1\right) \left( 2^{n/2-1}
+\varepsilon 1\right) (k_{0}-1)
\]%
since $v_{0}-1\mid (v_{1}-1)(k_{0}-1)$ and $v_{1}-1=\left(
2^{n/2}\allowbreak -\varepsilon 1\right) \left( 2^{n/2-1}
+\varepsilon 1\right) $. So $v_{0}-1\mid k_{0}-1$, contrary to $%
k_{0}<v_{0}$. This completes the proof.
\end{proof}

\bigskip

\begin{lemma}\label{notPSL(2,q)}
 Assume that Hypothesis \ref{hyp2} holds. Then $T$ is not isomorphic to $PSL_{2}(q)$.
\end{lemma}

\begin{proof}
Assume that $T\cong PSL_{2}(q)$. We are going to prove the assertion in a
series of steps.

\begin{claim}
$T$ acts point-primitively on  $\mathcal{D}_{1}$.
\end{claim}
Assume that $T$ does not act point-primitively on  $\mathcal{D}_{1}$. Then $G^{\Sigma }$, $G_{\Delta }^{\Sigma }$, $v_{1}$ and $\lambda$ are as in Table \ref{Leoniduzzolo} by \cite[Theorem 1.1]{Giu}.  
\begin{table}[h!]
\tiny
 \caption{Admissible $G^{\Sigma }$, $G_{\Delta }^{\Sigma }$, $v_{1}$ and $\lambda$ when $T$ does not act point-primitively on  $\mathcal{D}_{1}$}\label{Leoniduzzolo}
    \centering
\begin{tabular}{llccc}
\hline
Line & $G^{\Sigma }$ & $G_{\Delta }^{\Sigma }$ & $v_{1}$ & $\lambda$  \\
\hline
1& $PGL_{2}(7)$ & $D_{12}$ & $28$ & $7$ \\ 
2& & $D_{16}$ & $21$ &  \\ 
3& $PGL_{2}(9)$ & $D_{20}$ & $36$ & $5$ \\ 
4& & $D_{16}$ & $45$ &  \\ 
5& $P\Sigma L_{2}(9)$ & $Z_{5}:Z{4}$ & $36$ &  \\ 
6& & $SD_{16}$ & $45$ &  \\ 
7& $P\Gamma L_{2}(9)$ & $Z_{10}:Z_{4}$ & $36$ &  \\ 
8& & $N_{G^{\Sigma }}(D_{8})$ & $45$ &  \\ 
9& $PGL_{2}(11)$ & $D_{20}$ & $66$ & $5,11$ \\ 
10& $PGL_{2}(q)$, $q=p\equiv 1\pmod{40}$ & $S_{4}$ & $\frac{q(q^{2}-1)}{24%
}$ & divides $q(q^{2}-1)$\\
\hline
\end{tabular}
\end{table}
\normalsize
Since $\lambda \mid v_{1}(v_{1}-1)$ in each case, then $\mathcal{D}_{1}$ cannot be of type Ia by Theorem \ref{Teo1}. Furthermore, $v_{1} \neq \frac{1}{2}(\lambda-1)(\lambda^{2}-2)$ in each case as in Lines 1--9. Finally, $\lambda \mid v_{1}$ in case as in line 10. So, $\mathcal{D}_{1}$ cannot be of type Ic by Theorem \ref{Teo1}. Therefore, $\mathcal{D}_{1}$ cannot be of type Ia or Ic in any case, but this contradicts Theorem \ref{Teo4}(1). Thus, $T$ acts point-primitively on  $\mathcal{D}_{1}$. 
 
 \begin{claim}
$\mathcal{D}_{1}$ is of type Ia, $v_{1}=q+1$, $\lambda \mid q-1$ and $%
T_{\Delta }\cong \lbrack q]:Z_{q-1}$.
\end{claim}

The group $T_{\Delta }$ is a maximal subgroup of $T$ by Claim 1, and $v_{1}=\left\vert T:T_{\Delta}\right\vert$. One has $\lambda \nmid v_{1}$ by Theorem \ref{Teo1} since $\mathcal{D}_{1}$ is of type Ia or Ic by Theorem \ref{Teo4}(1). Moreover, $\lambda \mid \left\vert T_{\Delta} \right\vert $ and $\lambda \nmid \left\vert Out(T) \right\vert $ by Lemma \ref{Tlarge}(2). Thus    the order of $T_{\Delta }$ divisible by $%
\lambda >3$, it follows from \cite[Hauptsatz II.8.27]{Hup} that $T_{\Delta }$
is isomorphic to one of the groups $D_{q\pm 1}$, $A_{5}$ with $q\equiv \pm 1%
\pmod{10}$, $PSL_{2}(q^{1/m})$ or $PGL_{2}(q^{1/m})$ with $m$ prime, or $[q]:Z_{q-1}$. If $%
T_{\Delta }\cong D_{q\pm 1}$, then $\frac{q(q^{2}-1)}{(2,q-1)}\leq
(q+1)^{2}\cdot (2,q-1)\cdot \log _{p}q$ by Lemma \ref{Tlarge}(1), and hence $q=5$ or $9$.
However, these two cases are excluded by Lemma \ref{noAlternating} since $%
PSL_{2}(5)\cong A_{5}$ and $PSL_{2}(9)\cong A_{6}$. If $T_{\Delta }\cong
A_{5}$ with $q\equiv \pm 1\pmod{10}$, then $\lambda =5$ and $v_{1}=%
\frac{q(q^{2}-1)}{120}$, and hence $q(q^{2}-1)\leq 60^{2}\cdot 2\cdot \log _{p}q$ by Lemma \ref{Tlarge}(1).
So $q=11,19,29$ or $31$ since  $q\equiv \pm 1\pmod{10}$ and we have seen that $q\neq 9$.
In the former case, one has $%
v_{1}=11$ and $\frac{v_{0}-1}{k_{0}-1}=\frac{r_{1}}{(r_{1},\lambda _{1})}=10$
by (\ref{double}) since $\frac{v_{0}-1}{k_{0}-1} \mid v_{1}$ and $\frac{r_{1}}{(r_{1},\lambda _{1})} > \lambda$. Then $25$ divides $r_{1}=\frac{v_{0}-1}{k_{0}-1}\cdot 
\frac{v_{0}}{k_{0}\eta }\cdot \lambda $ and hence divides the order of $G^{\Sigma }$%
, a contradiction since $G^{\Sigma }\leq PGL_{2}(11)$. Then $q=19$, $v_{1}=57
$ and $\lambda=5$. Then $\mathcal{D}_{1}$ is of type Ia by Theorems \ref{Teo1} and \ref{Teo4}. Moreover, either $A=1$ and $k_{0}=2,3,4,5$, or $A=2$ and $k_{0}=2,3$ since $A(k_{0}-1)+1 \leq \lambda=5$ by Proposition \ref{P2}(4). Then $\frac{v_{0}-1}{k_{0}-1}=28$ or $56$ by Proposition \ref{P2}(4), respectively. Actually, only $A=1$ and $\frac{v_{0}-1}{k_{0}-1}=28$ occur with $k_{0}=3$ occur by Proposition \ref{P2}(2). Then $PSL_{2}(19) \unlhd G^{\Sigma}\leq PGL_{2}(19)$ acts $2$-transitively on $\mathcal{D}_{1}$ by \cite[Corollary 4.2]{Ka0}, a contradiction. 

If $T_{\Delta }\cong PSL_{2}(q^{1/m})$ or $PGL_{2}(q^{1/m})$ with $m$ prime, then 
\[
\frac{q(q^{2}-1)}{(2,q-1)}\leq q^{2/m}(q^{2/m}-1)^{2}\cdot (2,q-1)\cdot \log _{p}q
\]%
by Lemma \ref{Tlarge}(1), and hence $m=2$. Thus either $T_{\Delta }\cong PGL_{2}(q^{1/2})$, $q$ is odd and $v_{1}=\frac{q^{1/2}(q+1)}{2}$, or $v_{1}=q^{1/2}(q+1)$. In both cases, $\lambda \mid q-1$ since $\lambda \nmid v{1}$ by Theorem \ref{Teo1}, because $\mathcal{D}_{1}$ is of type Ia or Ic by Theorem \ref{Teo4}(1), and $\lambda \mid \left\vert T_{\Delta} \right\vert$. If $q$ is odd, then $\lambda$ divides either $\frac{q^{1/2}-1}{2}$ or $\frac{q^{1/2}+1}{2}$ since $\lambda$ is an odd prime, and hence $\lambda \leq \frac{q^{1/2}+1}{2}$. Then 
$$3\frac{q^{1/2}(q+1)}{2} \leq 3v_{1}\leq v \leq 2\lambda^{2}(\lambda-1)^{2}\leq \frac{(q^{1/2}+1)^{2}}{2} \cdot \frac{q^{1/2}-1}{2} $$
by Lemma \ref{L1} and \cite[Theorem 1]{DP}, and no admissible cases occurs. Therefore, $q$ is even and $v_{1}=q^{1/2}(q+1)$.
one has
\begin{equation}\label{endofmarch}
\frac{r_{1}}{(r_{1},\lambda _{1})}\mid \left( q^{1/2}(q+1)-1,%
q^{1/2}(q-1) \cdot 2 \cdot \log_{p}q \right)
\end{equation}
by Theorem \ref{Teo4}(4). Now $\left( q^{1/2}(q+1)-1,q^{1/2}(q-1)%
\right) $ divides $q^{1/2}-1$, and
hence $\frac{r_{1}}{(r_{1},\lambda _{1})} \mid (q^{1/2}-1) \cdot \log_{p}q$ by (\ref{endofmarch}). Therefore, 
$$ q^{1/2}(q+1) \leq (q^{1/2}-1)^{2} \cdot \log_{p}^{2}q $$
by Theorem \ref{Teo4}(3), and no admissible cases occur since $q^{1/2}>1$.

Finally, assume that $T_{\Delta }\cong \lbrack q]:Z_{\frac{q-1}{(2,q-1)}}$.
Thus $v_{1}=q+1$, and hence either $\mathcal{D}_{1}$ is of type Ia or Ic by Theorem \ref{Teo4}(1). In the latter case, $q=v_{1}=\frac{1}{2}\lambda(\lambda+1)(\lambda-2)$, which is impossible. Thus $\mathcal{D}_{1}$ is of type Ia, and hence $\lambda \mid q-1$ since $\lambda \nmid v_{1}(v_{1}-1)$ by Theorem \ref{Teo1}.

\begin{claim}
If $K$ be the stabilizer in $T$ of a block $B^{\Sigma }$ of $\mathcal{D}_{1}$%
, then $p$ divides the order of $K$.
\end{claim}

Let $K$ be the stabilizer in $T$ of a block $B^{\Sigma }$ of $\mathcal{D}_{1}
$. We know that $\lambda \mid k_{1}$ by Theorem \ref{Teo4}, hence $\lambda$ divides the order of $G^{\Sigma}_{B^{\Sigma}}$ since $G^{\Sigma}$ acts flag-transitively on $\mathcal{D}_{1}$. Then $\lambda$ divides the order of $K$ by Lemma \ref{Tlarge}(1) since $\left \vert G^{\Sigma}_{B^{\Sigma}}:K \right\vert  \mid \left \vert Out(T) \right\vert$.

Suppose that $p$ does not divide the order of $K$. Then either $K\leq D_{%
\frac{2(q-1)}{(2,q-1)}}$ or $K\cong A_{5}$, $q\equiv 1\pmod{10}$ and $%
\lambda =5$ since $\lambda \mid q-1$ and $\lambda $
is a prime divisor of the order of $K$ and $\lambda >3$. In the latter case, $3(q+1) \leq 3v_{1}\leq 200$ by Lemma \ref{L1} and \cite[Theorem 1]{DP}, and hence $q=11, 21, 31, 41$ or $61$ since $q\equiv 1\pmod{10}$. Therefore, $K=G^{\Sigma}_{B^{\Sigma}}$ and $PSL_{2}(q) \unlhd G^{\Sigma} \leq PGL(2,q)$. Hence, $b_{1}=\frac{q(q^{2}-1)}{60\alpha}$ with $\alpha=1,2$. Also, $\frac{v_{0}-1%
}{k_{0}-1}=q$ and so $A=1$ since $v_{1}=q+1$ and $v_{1}=A\frac{v_{0}-1}{%
k_{0}-1}+1$. Then $k_{1}=\frac{v_{0}}{k_{0}}+1$ with $k_{1}\mid 60$ and $%
2\leq k_{0}\leq 5$, and easy computations show that no $(q,k_{0},v_{0},k_{1}, v_{1})=(11,5,45,10,12),(31,15,435,30,32)$, or $(61,30,1770,60,62)$. Then $r_{1}=\frac{33}{2\alpha}, \frac{899}{2\alpha}$ or $\frac{3599}{\alpha}$, respectively, and hence only $(q,k_{0},v_{0},k_{1},v_{1},r_{1})=(61,30,1770,60,62,3599)$ is admissible. However, this remaining case is excluded since $r_{1}$ is divisible by $59$, whereas the order of $PGL_{2}(61)$, and hence the order of $G^{\Sigma}$, is not.

Finally, assume that $K\leq D_{\frac{2(q-1)}{(2,q-1)}}$. Now, $K$ contains a unique
cyclyc subgroup $C$ of order a multiple of $\lambda$ and a divisor of $q-1$. Furthermore, $C$ is normal in $G_{B^{\Sigma
}}^{\Sigma }$. Clearly, the actions $G^{\Sigma}$ on the point set of $\mathcal{D}_{1}$
and on $PG_{1}(q)$ are equivalent. Hence, $C$ a fixes exactly
two points of $\mathcal{D}_{1}$ and acts semiregularly elsewhere. Then $Fix(C)\cap B^{\Sigma
}=\varnothing $, and hence $k_{1}=\frac{q-1}{y}$ for some $y \geq 1$, since $G_{B^{\Sigma
}}^{\Sigma }$ acts transitively on $B^{\Sigma }$. Actually, $y=1$ since $%
q+1=v_{1}<2k_{1}$. Then $b_{1}=\frac{q(q+1)}{2}$ and hence $r_{1}=\frac{%
q^{2}-1}{2}$, whereas $\frac{r_{1}}{(r_{1},\lambda _{1})}=\frac{v_{0}-1}{%
k_{0}-1}$ divides $q$ and $\frac{v_{0}-1}{k_{0}-1}>\lambda >3$. Thus, $p$
divides the order of $K$.

\begin{claim}
$K\leq [q]:Z_{q-1}$
\end{claim}

Assume that $K$ is not contained in $[q]:Z_{q-1}$. Then either $K$ is $PSL_{2}(q^{1/h})$ or $PGL_{2}(q^{1/h})$ with $%
h>1$, or $K\cong A_{5}$ when $q=3^{4u}$, $u \geq 1$ and $\lambda =5$ by \cite[Hauptsatz II.8.27]{Hup} since $%
\lambda \mid \left(q-1, \left\vert K \right\vert \right)$ with $\lambda>3$ by Claims 2 and 3 and $p$ divides $\left\vert K \right\vert$ by Claim 4. We may use the same argument used to rule out $A_{5}$ in Claim 3, as it is is independent
from the divisibility of the order of $K$ by $p$, to exclude $K\cong A_{5}$ when $q=3^{4u}$, $u \geq 1$ and $\lambda =5$. Thus, $K$ is $PSL_{2}(q^{1/h})$ or $PGL_{2}(q^{1/h})$ with $%
h>1$. The actions of the actions $G^{\Sigma}$ on the point set of $\mathcal{D}_{1}$ on $PG_{1}(q)$ are equivalent, hence the $K^{\prime}$-orbits on the point set of $\mathcal{D}_{1}$ are one of length $q^{h}+1$, which is a copy of $PG_{1}(q^{1/h})$, one of length $q^{h}(q^{h}-1)$ provided that $q^{1/h}+1 \mid q-1$, and the remaining ones each of length $\frac{q^{h}(q^{2h}-1)}{(2,q^{h}-1)}$ for $h>2$ (for instance, see \cite[Lemma 14]{COT} for $q$ odd). Then the $G^{\Sigma}_{B^{\Sigma}}$-orbits on the point set of $\mathcal{D}_{1}$ are one of length $q^{h}+1$, one of length $q^{h}(q^{h}-1)$ provided that $q^{1/h}+1 \mid q-1$, and the remaining ones each of length $\frac{q^{1/h}(q^{2/h}-1)}{(2,q^{1/h}-1)}\theta$ with $\theta \mid (2,q-1)\cdot \log_{p}q$ and $h>2$. Therefore, the possibility for $k_{1}$ is one among $q^{1/h}+1$, $q^{1/h}(q^{1/h}-1)$ when $q^{1/h}+1 \mid q-1$, or $\frac{q^{1/h}(q^{2/h}-1)}{(2,q^{1/h}-1)}\theta$ with $\theta \mid (2,q-1)\cdot \log_{p}q$ since $B^{\Sigma}$ is a point-$G^{\Sigma}_{B^{\Sigma}}$-orbit on $\mathcal{D}_{1}$. On the other hand, $k_{1}>v_{1}/2=(q+1)/2$ which forces $q=4$ for $k_{1}=q^{1/h}+1$, and $h=2$ for $k_{1}=q^{1/h}(q^{1/h}-1)$. However, the former cannot occurs since $T\cong PSL_{2}(4)\cong A_{5}$ is ruled out in Lemma \ref{noAlternating}. The latter implies $k_{1}=q^{1/2}(q^{1/2}-1)$, and hence $v_{0}=(q-q^{1/2}-1)k_{0}$ and $v_{0}=q(k_{0}-1)+1$. Then $k_{0}=q^{1/2}-1$ and $v_{0}=(q-q^{1/2}-1)(q^{1/2}-1)$. Further, $\lambda \mid q^{1/2}-1$ since $\lambda \mid k_{1}$ by Theorem \ref{Teo4} and $\lambda \mid q-1$ by Claim 2. So $\lambda \mid k_{0}$, which is contrary to Theorem \ref{Teo4}. 

Assume that $k_{1}=\frac{q^{1/h}(q^{2/h}-1)}{(2,q^{1/h}-1)}\theta$ with $\theta \mid (2,q-1)\cdot \log_{p}q$ and $h>2$. Then either $\lambda \mid q^{1/h}-1$ or $\lambda \mid q^{1/h}+1$ since $\lambda \mid k_{1}$ and $\lambda \mid q-1$. Then
\begin{equation} \label{cumulative}
3(q+1) \leq v \leq 2\lambda^{2}(\lambda-1)\leq 2q^{1/h}(q^{1/h}+1)^{2}
\end{equation}%
by Lemma \ref{L1} and \cite[Theorem 1]{DP}. Then $h=3$ and $q^{1/3}=2,3$ or $4$. Actually, only $q^{1/3}=4$ and $\lambda=5$ are admissible since $\lambda \mid q^{1/h}-1$ or $\lambda \mid q^{1/h}+1$ and $\lambda$ is a prime greater than $3$. However, this is contrary to $\lambda \mid q-1$. Thus, we obtain $K\leq
[q]:Z_{q-1}$.

\begin{claim}
$T$ is not isomorphic to $PSL_{2}(q)$.
\end{claim}

If follows from Claims 2 and 4 that, $\left\vert K\right\vert =p^{t}\frac{p^{(t,h)}-1%
}{e}$ with $e\mid p^{(t,h)}-1$ and $\lambda =\frac{p^{(t,h)}-1}{\theta }$
with $ \theta \mid e$. Moreover, $Fix(P)\cap B^{\Sigma }=\varnothing $, where $P$ is the Sylow $p$%
-subgroup of $K$, as $G_{B^{\Sigma }}^{\Sigma }$ acts transitively on $%
B^{\Sigma }$. Therefore, $p^{t}\frac{p^{(t,h)}-1}{\theta}\mid k_{1}$ since $%
\lambda =\frac{p^{(t,h)}-1}{\theta }>3$ and $\lambda \mid k_{1}$ by Theorem \ref{Teo4}. Furthermore, $(t,h)\leq t<h$ since $%
k_{1}<v_{1}$. If $(t,h)\leq h/3$, then $\lambda \leq q^{1/3}-1$, and we may use (\ref{cumulative}) to reach a contradiction. Therefore $h/3<(t,h)\leq t<h$, and hence $(t,h)=t=h/2$. Then $q^{1/2}\mid \frac{v_{0}}{k_{0}}+1$, and hence $%
q^{1/2}\mid v_{0}+k_{0}$. On the other hand, $v_{0}+k_{0}=q(k_{0}-1)+1+k_{0}$%
. Thus, we obtain $q^{1/2}\mid 1+k_{0}$. Then $\frac{q^{1/2}-1}{\theta }\geq \lambda \geq
k_{0}\geq q^{1/2}-1$ implies $\theta =1$ and $\lambda =k_{0}=q^{1/2}-1$, contrary to Theorem \ref{Teo4}. This completes the proof.
\end{proof}

\bigskip

\begin{lemma}\label{PSLn2trans}
 Assume that Hypothesis \ref{hyp2} holds. Then $T\cong PSL_{n}(q)$ and $v_{1}=\frac{q^{n}-1}{q-1}$. In particular, $T$ acts point $2$-transitively on $\mathcal{D}_{1}$.   
\end{lemma}

\begin{proof}
We are going to show that $v_{1} >2\lambda^{2}$ in each case as in of Table \ref{TavPSL}, except that as in Line
1. Further $v_{0}>\lambda +1$ since $\frac{v_{0}-1}{k_{0}-1}>\lambda $ by Theorem \ref{Teo4}(3) and (\ref{double}) and $k_{0}\geq 2$. This leads to $v>2\lambda
^{2}(\lambda +1)$, contrary to \cite[Theorem 1]{DP}, thus showing that only case as in line 1 of Table \ref{TavPSL} holds, which is the assertion.
We provide details of computations in one case, namely that as in line 2 of Table \ref{TavPSL}. The computations in the remaining cases are similar, and hence they are omitted. Now, assume that $T_{\Delta}$, $v_{1}$ and $\lambda$ are as in line 2 of Table \ref{TavPSL}. One has
\begin{eqnarray*}
\frac{q^{n}-1}{q-1}\cdot \frac{q^{n-1}-1}{q^{2}-1}-2\left( \frac{q^{n-2}-1}{%
q-1}\right) ^{2} &=&\frac{\left( q^{n}-1\right) \left( q^{n-1}-1\right)
-2\left( q^{n-2}-1\right) ^{2}\left( q+1\right) }{(q-1)^{2}(q+1)} \\
&>&\frac{q^{n}(q^{n-4}\left( q^{3}-2q-2\right) -1)}{%
(q-1)^{2}(q+1)}>0
\end{eqnarray*}%
for $q>1$. Therefore, 
\[
\frac{q^{n}-1}{q-1}\cdot \frac{q^{n-1}-1}{q^{2}-1}>2\left( \frac{q^{n-2}-1}{%
q-1}\right) ^{2}\geq 2\lambda ^{2}\text{.}
\]%
Further $v_{0}>\lambda +1$ since $\frac{v_{0}-1}{k_{0}-1}>\lambda $ by
Theorem \ref{Teo4}(3) and (\ref{double}) and $k_{0}\geq 2$. So $2\lambda ^{2}(\lambda
+1)<v_{1}v_{0}=v$, whereas $v\leq 2\lambda ^{2}(\lambda -1)$ by \cite[%
Theorem 1]{DP}.

All cases of Table \ref{TavPSL}, except that as in Line
1, are ruled out similarly. Thus, only case as in
line 1 occurs, which is the assertion.   
\end{proof}

\bigskip

\begin{sidewaystable}
    \caption{Admissible $T_{\Delta}$, $v_{1}$ and $\lambda$ for $T\cong PSL_{n}(q)$}\label{TavPSL}
    \begin{tabular}{lcccc}
    \hline
$T$ & Aschbacher class & type of $T_{\Delta }$ & $v_{1}$ & $\lambda $ divides
\\
\hline
$PSL_{n}(q)$ & $\mathcal{C}_{1}$ & $P_{1}$ & $\frac{q^{n}-1}{q-1}$
& $\frac{q^{j}-1}{q-1}$, $1\leq j\leq n-2$ \\ 
$PSL_{n}(q)$, $n$ odd & $\mathcal{C}_{1}$ & $P_{2}$ & $\frac{(q^{n}-1)\cdot
\left( q^{n-1}-1\right) }{\left( q-1\right) \left( q^{2}-1\right) }$ & $%
\frac{q^{j}-1}{q-1}$, $1\leq j\leq n-2$ \\ 
$PSL_{3}(q)$ & $\mathcal{C}_{5}$ & $GL_{3}(q^{1/2})$ & $\frac{%
q^{3/2}(q^{3/2}+1)(q+1)}{\theta }$, $\theta =1,3$ & $q^{1/2}-1$ or $%
q+q^{1/2}+1$ \\ 
$PSL_{4}(q)$ & $\mathcal{C}_{5}$ & $GL_{4}(q^{1/2})$ & $\frac{%
q^{3}(q^{4}-1)(q^{3/2}+1)}{\theta (q-1)}$, $\theta =1,4,8$ & $q+q^{1/2}+1$
\\ 
$PSL_{n}(q)$ & $\mathcal{C}_{8}$ & $Sp_{n}(q)$ & $\frac{1}{(n/2,q-1)}%
q^{n(n-2)/4}\prod_{i=1}^{n/2-1}(q^{2i+1}-1)$ & $\frac{q^{2j}-1}{q-1}$, $%
1\leq j\leq n/2$ \\ 
$PSL_{n}(q)$ & $\mathcal{C}_{8}$ & $O_{n}^{\circ }(q)$ & $2q^{(n^{2}-1)/4}%
\frac{q^{n}-1}{(n,q-1)}\prod_{j=1}^{(n-3)/2}(q^{2j+1}-1)$ & $\frac{q^{2j}-1}{%
q-1}$, $1\leq j\leq (n-1)/2$ \\ 
$PSL_{n}(q)$ & $\mathcal{C}_{8}$ & $O_{n}^{\varepsilon }(q)$, $\varepsilon
=\pm $ & $\frac{q^{n^{2}}(4,q^{n/2}+\varepsilon 1)(q^{n/2}+\varepsilon 1)}{%
2\theta (q-1,n)}\prod_{j=1}^{n/2-1}(q^{2j+1}-1)$, $\theta =1,2$ & $\frac{%
q^{n/2}-\varepsilon 1}{q-\varepsilon 1}$ or $\frac{q^{n/2}-\varepsilon 1}{q-1%
}$, $1\leq j\leq n/2-1$ \\ 
$PSL_{3}(q)$ & $\mathcal{C}_{8}$ & $GU_{3}(q^{1/2})$ & $\frac{%
q^{3/2}(q^{3/2}-1)(q+1)}{\theta }$, $\theta =1,3$ & $q^{1/2}+1$ or $%
q-q^{1/2}+1$\\
\hline
\end{tabular}
\end{sidewaystable}

\normalsize

\bigskip

\begin{lemma}\label{toward}
Assume that Hypothesis \ref{hyp2} holds. Then $G_{(\Sigma)} = 1$ and $T \cong PSL_{n}(q)$ with $n\geq5$.
\end{lemma}

\begin{proof}

Assume that $G_{(\Sigma)} \neq 1$. Then $(k_{1}-1)\cdot \frac{k_{1}}{\lambda} \cdot b_{1}= v_{1}\cdot (v_{1}-1)$ by Lemma \ref{quasiprimitivity}(2). If $\frac{k_{1}}{\lambda} \geq 2$, then $2(k_{1}-1) \leq v_{1}-1$ since $b_{1} \geq v_{1}$. So $2k_{1}-1<v_{1}<2k_{1}$, which is a contradiction. Thus $k_{1}=\lambda$. On the other hand, $\lambda \mid \left\vert T_{\Delta} \right\vert$ by Lemma \ref{Tlarge}(2), hence $\lambda \leq \frac{q^{n-1}-1}{q-1}$ since $T \cong PSL_{n}(q)$ and $\lambda$ is  a prime. Therefore, $\frac{q^{n}-1}{q-1}=v_{1}< 2k_{1}=2\lambda\leq 2\frac{q^{n-1}-1}{q-1}$, which is not the case.

Assume that $n<5$. Then either $PSL_{3}(q)\trianglelefteq G\leq P\Gamma L_{3}(q)$ and $%
v_{1}=q^{2}+q+1$, or $PSL_{4}(q)\trianglelefteq G\leq P\Gamma L_{4}(q)$ and $%
v_{1}=q^{3}+q^{2}+q+1$ by Lemmas \ref{PSLn2trans} and \ref{notPSL(2,q)}.

Assume that $PSL_{3}(q)\trianglelefteq G\leq P\Gamma L_{3}(q)$ and $%
v_{1}=q^{2}+q+1$. Then $\mathcal{D}_{1}$ is of type Ic by Theorem \ref{Teo1} since $v_{1}$ is odd, and hence $\mathcal{D}_{1}$ is of type Ia by Theorem \ref{Teo4}. Further, $\lambda \nmid v_{1}(v_{1}-1)$ by Theorem \ref{Teo1}. Then $\lambda \mid q-1$ since $\lambda $ divides the order of $\left\vert T_{\Delta }\right\vert$ by Lemma \ref{Tlarge}(2), and $A\frac{v_{0}-1}{k_{0}-1}=q(q+1)$ by Proposition \ref{P2}(3).

Suppose that $p\nmid \frac{v_{0}-1}{k_{0}-1}$. Then $\frac{v_{0}-1}{k_{0}-1}%
\mid q+1$. If $\frac{v_{0}-1}{k_{0}-1}\leq \frac{q+1}{2}$ then $2q\leq A\leq
(v_{1}-1)^{1/2}\leq (q+1)^{1/2}$, a contradiction since $q\geq 2$. Thus, $\frac{v_{0}-1}{k_{0}-1}=q+1$ and $A=q$ since $\frac{v_{0}-1}{k_{0}-1}>\lambda>A$ by Proposition \ref{P2}(4) and Theorem \ref{Teo4}. However, this impossible since $\lambda \mid q-1$.

Suppose that $p\mid \frac{v_{0}-1}{k_{0}-1}$. Then $\left(\frac{v_{0}-1}{k_{0}-1}\right)_{p} \mid q$ since $A\frac{v_{0}-1}{k_{0}-1}=q(q+1)$, and hence $(r_{1})_{p}\mid q$ since $r_{1}=\frac{v_{0}-1}{k_{0}-1} \cdot \frac{v_{0}}{k_{0}\eta} \cdot \lambda$ with $\lambda \mid q-1$. If $\Delta$ is any point of $\mathcal{D}_{1}$ the $r_{1}$ blocks of $\mathcal{D}_{1}$ containing $\Delta$ are partitioned in $T$-orbits of equal length since $T \unlhd G^{\Sigma}$ and $G^{\Sigma}$ acts flag-transitively on $\mathcal{D}_{1}$. Then  $q^{2}\mid \left\vert T_{\Delta ,B^{\Sigma }}\right\vert $, where $b^{\Sigma}$ is any block of $\mathcal{D}_{1}$ containing $\Delta$, since $%
q^{3}\mid \left\vert T_{\Delta }\right\vert $ and $(r_{1})_{p}\mid q$. Consequently, $q^{2}\lambda \mid
\left\vert T_{B^{\Sigma }}\right\vert $ since $\lambda \mid (k_{1},q-1)$, $k_{1} \mid \left\vert G^{\Sigma}_{B^{\Sigma}}\right\vert$, $\left\vert G^{\Sigma}_{B^{\Sigma}}:T_{B^{\Sigma}}\right\vert \mid \left\vert Out(T)\right\vert$, and $\lambda \nmid \left\vert Out(T)\right\vert$ by Lemma \ref{Tlarge}(2). Further, one has $\left\vert T_{B^{\Sigma }}\right\vert \geq \frac{\left\vert G^{\Sigma}_{B^{\Sigma}}\right\vert}{3 \cdot log_{p}q} \geq q^{2}\frac{%
q^{2}+q+1}{3 \cdot log_{p}q}$, and hence $T_{B^{\Sigma }}$ lies in the stabilizer in $T$ of
a point or a line of of $PG_{2}(q)$ by \cite{Ha,Mi}. Therefore $b_{1}=e(q^{2}+q+1)$ for some $%
e\geq 1$, and hence $k_{1}=\frac{r_{1}}{e}$ since $b_{1}k_{1}=v_{1}r_{1}$ and $v_{1}=q^{2}+q+1$. Thus, $k_{1}=\frac{A+k_{0}}{\alpha }\cdot \lambda$ by (\ref{fundamental}) in Theorem \ref{Teo1} since $\mathcal{D}_{1}$ is of type Ia and Lemma \ref{L3}(2). Moreover,  $\lambda = \frac{q-1}{\beta}$ for some $\beta \geq 1$. Thus, we obtain
\begin{equation*}
 k_{1}= \frac{A+k_{0}}{\alpha }\cdot \frac{q-1}{\beta}\text{.}  
\end{equation*}

If $A+k_{0} > \lambda$, then $A+k_{0} > \lambda \geq A(k_{0}-1)+1$ by Proposition \ref{P2}(4), and hence either $k_{0}=2$ and $\lambda =A+1$, or $k_{0}>2$ and $A=1$. The latter implies $k_{0}+1> \lambda$, whereas $\lambda \geq k_{0}+1$ by Theorem \ref{Teo1} since $\mathcal{D}_{1}$ is of type Ia. Thus $k_{0}=2$ and $\lambda =A+1$. Then $A=\frac{q-1}{\beta}-1$ since $\lambda = \frac{q-1}{\beta}$ for some $\beta \geq 1$, and hence
\[
\frac{q^{2}+q+1}{2}<k_{1}\leq \frac{1}{\alpha} \cdot \left(\frac{q-1}{\beta}+1\right)\cdot \left(\frac{q-1}{\beta}\right)
\]
since $v_{1}<2k_{1}$. Thus $\alpha=\beta=1$, and hence $k_{1}=q(q-1)$, and so $(q-2)\frac{v_{0}}{2}+1=q(q-1)$ since $k_{1}=A\frac{v_{0}}{k_{0}}+1$ by Proposition \ref{P2}(3), and we reach a contradiction.

If $A+k_{0} \leq \lambda$, then  
\[
\frac{q^{2}+q+1}{2}<k_{1}\leq \frac{A+k_{0}}{\alpha }\cdot \lambda \leq \frac{%
\lambda }{\alpha }\cdot \frac{q-1}{\beta }\leq \frac{(q-1)^{2}}{\alpha
\beta ^{2}}\text{,}
\]%
forcing $\alpha =\beta =1$. Thus~$\left( k_{1},\frac{v_{0}-1}{k_{0}-1}%
\right) =A+k_{0}$ and $\lambda =q-1$ is a Mersenne prime. Hence, $q=2^{t}$
with $t$ prime. Since $\frac{v_{0}-1}{k_{0}-1}>\lambda$ by Theorem \ref{Teo4}, it follows that $v_{0}>(q-1)(k_{0}-1)+1$ and hence
$$((q-1)(k_{0}-1)+1)q(q+1)<v_{0}v_{1}=v \leq 2\lambda^{2}(\lambda-1)=2(q-1)^{2}(q-2)$$
by \cite[Theorem 1]{DP}. Thus $k_{0}=2$, and hence $v_{0}$ is even since $k_{0}\mid v_{0}$. Then $v_{0}=q+2$ and $A=q$ since $A\frac{v_{0}-1}{k_{0}-1}=q(q+1)$ and $q$ is a power of $2$. Then $k_{1}=(A+k_{0})(q-1)=(q+2)(q-1)$ and $k_{1}=A\frac{v_{0}}{k_{0}}+1=q\left(\frac{q}{2}+1\right)+1$, a contradiction.

Assume that $PSL_{4}(q)\trianglelefteq G\leq P\Gamma L_{4}(q)$ and $%
v_{1}=q^{3}+q^{2}+q+1$. Suppose that $\mathcal{D}_{1}$ is of type Ic. Then
\begin{equation}\label{30marz0}
q(q^{2}+q+1)=\frac{1}{2}\lambda(\lambda+1)(\lambda-2)    
\end{equation}
and hence there is a positive integer $\gamma$ such that either $\lambda=q \gamma$, or $\lambda=q \gamma -1$, or $\lambda=q \gamma +2$. It is not difficult to see that, if we substitute any of these three possible values of $\lambda$ in (\ref{30marz0}), no admissible solutions arise since $\lambda$ is a prime number. Thus, $\mathcal{D}_{1}$ is of type Ic by Theorem \ref{Teo4}(1). Then $\lambda \nmid v_{1}(v_{1}-1)$ by Theorem \ref{Teo1}, and hence $\lambda = \frac{q-1}{e}$, $e \geq 1$
since $\lambda $ divides the order of $\left\vert T_{\Delta }\right\vert $
by Lemma \ref{Tlarge}(2). Then 
$$3(q^{3}+q^{2}+q+1)< v\leq 2\lambda ^{2}(\lambda -1) \leq 2\left(\frac{q-1}{e}\right)^{2} \cdot \left(\frac{q-1}{e}-1\right)$$
by Lemma \ref{L1} and  \cite[Theorem 1]{DP}, and we reach a contradiction. This completes the proof.
\end{proof}

\bigskip

\begin{theorem}\label{TnotAS}
 Assume that Hypothesis \ref{hyp2} holds. Then $T$ cannot be non-abelian simple.
\end{theorem}

\begin{proof}
We know that $PSL_{n}(q) \unlhd G\leq P\Gamma L_{n}(q) $, $n \geq 5$, acts in its $2$-transitive action of degree $\frac{q^{n}-1}{q-1}$ on the point set of $\mathcal{D}_{1}$ by Lemmas \ref{PSLn2trans} and \ref{toward}, and $\mathcal{D}_{1}$ is of tyoe Ia or Ic by Theorem \ref{Teo4}. Further, $Soc(G_{\Delta }^{\Delta })$ is either an elementary
abelian $u$-group for some prime $u$, or a non abelian simple group since $G_{\Delta }^{\Delta }$ acts point-$2$%
-transitively on $\mathcal{D}_{0}$ for $k_{0}=2$, by \cite[Theorem]{BDDKLS} for $%
k_{0}>2$ and $\mu =\lambda $, and by \cite[Theorem 1]{ZC} for $k_{0}>2$ and $\mu =1$.
It follows that $Soc(G_{\Delta }^{\Delta })$ is either an elementary
abelian $u$-group for some prime $u$, or a non abelian simple group in any case. Actually, either $Soc(G_{\Delta }^{\Delta })$ an elementary abelian $p$-group of order $q^{n-1}$ and $SL_{n-1}(q) \unlhd G_{x }^{\Delta }$, or $Soc(G_{\Delta }^{\Delta })\cong PSL_{n-1}(q)$ in any case by \cite[Proposition 4.1.17(II)]{KL} since $\lambda \nmid \left\vert Out(T) \right\vert$ by Lemmas \ref{Tlarge}(2) and $\lambda \nmid q-1$ by \cite[Theorem 1]{DP} for $n\geq 5$.

Assume that $Soc(G_{\Delta }^{\Delta })$ is an elementary abelian $p$-group of order $q^{n-1}$ and $SL_{n-1}(q) \unlhd G_{x }^{\Delta }$. Then $v_{0}=q^{n-1}$ since $G_{\Delta}^{\Delta}$ acts point-primitively  on $\mathcal{D}_{0}$.  If $k_{0}=2$, then $G_{\Delta }^{\Delta }$ acts point-$2$%
-transitively on $\mathcal{D}_{0}$. Also, $q$ is even since $k_{0} \mid v_{0}$. Then $$\frac{q^{n}-1}{q-1}=A(q^{n-1}-1)+1$$ by Proposition \ref{P2}(3) since $v_{1}=\frac{q^{n}-1}{q-1}$, and hence $A=q$ and $k_{1}=\frac{q^{n}}{2}+1$. Then $q=2$ since $k_{1}<v_{1}$ and $q$ is even. Therefore, $v_{1}=2^{n}-1$, $k_{1}=2^{n-1}+1$, $r_{1}=(2^{n-1}-1)\cdot \frac{2^{n-2}}{\eta}\cdot \lambda$. Then $(k_{1},v_{1}) \mid 3$ by Lemma \ref{L3}(1), and hence $\lambda= \frac{2^{n-1}+1}{\theta}$ with $\theta \mid 3$. However, this is impossible since $\lambda \mid 2^{n-i}-1$ with $i\geq 1$ since $\lambda \mid \left\vert T \right\vert$ by Lemmas \ref{Tlarge}(2) and $\lambda$ is odd. Thus $k_{0}>2$, and hence $\mathcal{D}_{1}$ is of type Ia by Theorem \ref{Teo4}(1). 

Suppose that $\mu =\lambda $. Then $\mathcal{D}_{0}$ is a linear space, and hence $\mathcal{D}_{0}\cong AG_{n-1}(q)$ with all the affine lines as blocks by \cite[Theorem]{BDDKLS} since $SL_{n-1}(q) \unlhd G_{x }^{\Delta }$. In particular, $k_{0}=q$. Then $$\frac{q^{n}-1}{q-1}=A\frac{q^{n-1}-1}{q-1}+1$$ by Proposition \ref{P2}(3). Thus, $A=1$, and hence $k_{1}=q^{n-1}+1$. Then $q=2$ since $k_{1}<v_{1}$ and $q$ is even. Therefore, $v_{1}=2^{n}-1$, $k_{1}=2^{n-1}+1$, $r_{1}=(2^{n-1}-1)\cdot \frac{2^{n-2}}{\eta}\cdot \lambda$, and we reach a contradiction as above.

Suppose that $\mu =1$. Then $\mathcal{D}_{0} \cong AG_{n-1}(q)$, $n-1 \geq 3$ is even, with all affine planes as blocks and $\lambda=\frac{q^{n-2}-1}{q-1}$ by Corollary \ref{whatif} since $Soc(G_{\Delta }^{\Delta })$ is an elementary abelian $p$-group of order $q^{n-1}$ and $SL_{n-1}(q) \unlhd G_{x }^{\Delta }$. In particular, $k_{0}=q^{2}$. Then $$\frac{q^{n}-1}{q-1}=A\frac{q^{n-1}-1}{q^{2}-1}+1\text{,}$$ and hence $A=q(q+1)$ and $k_{1}=(q+1)q^{n-2}+1$. Then $\frac{q^{n-2}-1}{q-1} \mid q+2$ since $\lambda=\frac{q^{n-2}-1}{q-1}$ divides $k_{1}$ by Theorem \ref{Teo4}, contrary to $n \geq 5$. Thus, $Soc(G_{\Delta }^{\Delta })$ cannot be elementary abelian.

Finally, assume that $Soc(G_{\Delta }^{\Delta })\cong PSL_{n-1}(q)$, $n \geq 5$. If $k_{0}=2$, then $G_{\Delta }^{\Delta }$ acts point-$2$%
-transitively on $\mathcal{D}_{0}$, and hence $v_{0}=\frac{q^{n-1}-1}{q-1}$ since 
$n\geq 5$. Then $\frac{q^{n}-1}{q-1}=v_{1}=A(v_{0}-1)+1$ implies $A\frac{%
q^{n-2}-1}{q-1}q\mid \frac{q^{n-1}-1}{q-1}q$, contrary to $n\geq 5$.

Assume that $k_{0}>2$ and $\mu =\lambda $. Then $\mathcal{D}_{0}$ is a
linear space, and hence $\mathcal{D}_{0}\cong PG_{n-2}(q)$ with al lines as
blocks by \cite[Theorem]{BDDKLS} since  $Soc(G_{\Delta }^{\Delta })\cong PSL_{n-1}(q)$, $n \geq 5$. Thus, $v_{0}=\frac{q^{n-1}-1}{q-1}$ and $k_{0}=q+1$. Then 
\[
\frac{q^{n}-1}{q-1}=A\frac{q^{n-2}-1}{q-1}+1
\]%
and again a contradiction since $n \geq 5$.

Finally, assume that $k_{0}>2$ and $\mu =1$. Note that, $v_{1}=A\frac{v_{0}-1}{k_{0}-1}+1$ implies $$v_{0}-1<v_{1}\frac{k_{0}-1}{A}<v_{1}\cdot \lambda \leq v_{1} \cdot k_{1}\leq v_{1}(v_{1}-1)\text{,}$$ and hence
$$v_{0} \leq q\cdot \frac{q^{n}-1}{q-1}\cdot \frac{q^{n-1}-1}{q-1}\text{.}$$

Then $\mathcal{D}%
_{0}\cong PG_{n-2}(q)$ with $n$ even and all planes as blocks and $\lambda =%
\frac{q^{n-3}-1}{q-1}$ for $n-1\geq 5$ by Proposition \ref{moreover}. Thus $v_{0}=\frac{q^{n-1}-1}{q-1}$ and $k_{0}=\frac{q^{3}-1}{q-1}$ for $n\geq 6$, and hence
$$
\frac{q^{n}-1}{q-1}=A\frac{q^{n-2}-1}{q^{2}-1}+1\text{,}
$$
contrary to $n\geq 6$. Thus $n < 6$, and hence $n=5$ since $n \geq 5$. Then $\lambda \leq q^{2}+q+1$ since $\lambda$ is a prime such that $\lambda \nmid v_{1}$ by Theorem \ref{Teo4}(1) and $\lambda \mid \left\vert T \right\vert$ by Lemma \ref{Tlarge}(2). If $q=2$, then $G_{\Delta}^{\Delta}$ acts flag-transitively on a $\mathcal{D}_{0}$ which is a $2$-design with $\lambda$ prime, but this is contrary to \cite[Theorem 1]{ZCZ} since $Soc(G_{\Delta }^{\Delta })\cong PSL_{4}(2)\cong A_{8}$. Therefore $q>2$, and hence $v_{0}\geq \frac{q^{4}-1}{q-1}$ by \cite[Theorem 5.2.2]{KL}. So
$$\frac{q^{4}-1}{q-1} \cdot \frac{q^{5}-1}{q-1} \leq v_{0}v_{1}=v\leq 2\lambda^{2}(\lambda-1)\leq 2(q^{2}+q+1)^{2}(q^{2}+q)$$
by \cite[Theorem 1]{DP}, which is not the case. This completes the proof.
\end{proof}

\bigskip

\begin{proof}[Proof of Theorem \ref{Teo5}]
Let $\mathcal{D}=(\mathcal{P}, \mathcal{B})$ is a $2$-$(v,k,\lambda)$ design with $\lambda$ a prime number, and $G$ is a
flag-transitive, point-imprimitive automorphism group of $\mathcal{D}$ preserving a nontrivial 
partition $\Sigma $ of $\mathcal{P}$ with $v_1$ classes of size $v_0$. Possibly by substituting $\Sigma $ with a $G$-invariant partition of $\mathcal{P}$ finer that $\Sigma$, we may assume that $\Sigma$ is minimal with respect to the ordering $\preceq$ defined in the Subsection \ref{min}. Then $\lambda \mid r$ by Theorem \ref{PetereDemb}. Further, either $\mathcal{D}_{1}$ is either symmetric $1$-design, or $\mathcal{D}_{1}$ is a (possibly trivial) $2$-design by Theorem \ref{CamZie}(2). Assume that the latter occurs. Then $G^{\Sigma}$ is an almost simple acting flag-transitively and point-primitively on $\mathcal{D}_{1}$ by Theorem \ref{Teo4}(1). However, this is impossible Theorem \ref{TnotAS}. Therefore $\mathcal{D}_{1}$ is symmetric $1$-design with $k_{1}=v_{1}-1$ by Theorem \ref{CamZie}, and hence the parameters for $\mathcal{D}_{0}$, $\mathcal{D}_{1}$, and $\mathcal{D}$ are those recorded in Table \ref{D1sym} of Lemma \ref{L2bis}. 
\end{proof}
\bigskip

\section{Completion of the proof of Theorem \ref{main}}

This final small section is devoted to the completion of the proof of Theorem \ref{main}. From Theorem \ref{Teo5}, we know  that $\mathcal{D}_{1}$ is a symmetric $1$-design with $k_{1}=v_{1}-1$. Thus $G^{\Sigma}$ acts point-$2$-transitively on $\mathcal{D}_{1}$. We are going show that $\mathcal{D}_{0}$ is a translation plane and hence $G_{\Delta}^{\Delta}$ is well known group of affine type. Finally, we get the conclusion of Theorem \ref{main} by matching the parameters of $\mathcal{D}$ with the possibilities for $G_{\Delta}^{\Delta}$ and $G^{\Sigma}$, being $G$ permutationally isomorphic to a subgroup of $G_{\Delta}^{\Delta} \wr G^{\Sigma}$ by \cite[Theorem 5.5]{PS}.

\bigskip

\begin{proposition}\label{P1}
Assume that Hypothesis \ref{hyp1} holds. Then $\mathcal{D}_{1}$ is a symmetric $1$-design with $k_{1}=v_{1}-1$, then $\mathcal{D}_{0}$ is a translation plane of order $k_{0}$ and $G_{\Delta}^{\Delta}$ is of affine type.
\end{proposition}

\begin{proof}
It follows from Lemma \ref{base}(2) and Lemma \ref{L2bis} that either $\mathcal{D}_{0}$ is a $2$-$(k_{0}^{2},k_{0},1)$ design, or $\mu=1$, $\lambda_{0}=\lambda$, $\mathcal{D}_{0}$ is a $2$-$(k_{0}^{2},k_{0}, \lambda)$ design with $k_{0} \geq 3$ and $r=r_{0}=(k_{0}+1)\lambda$. In the former case, $\mathcal{D}_{0}$ is a translation
plane of order $k_{0}$, and $G_{\Delta}^{\Delta}$ is of affine type by \cite[Theorems 2 and
4]{Wa} since $\mathcal{D}_{0}$ is an affine plane admitting $G_{\Delta}^{\Delta}$ as a flag-transitive automorphism group, which is the assertion in this case.

Assume that $\mu=1$, $\lambda_{0}=\lambda$, $\mathcal{D}_{0}$ is a $2$-$%
(k_{0}^{2},k_{0},\lambda )$ design with $k_{0} \geq 3$ and $r=r_{0}=(k_{0}+1)\lambda$ admitting $G_{\Delta }^{\Delta }$ as a
flag-transitive automorphism group. If $\lambda=k_{0}$, then $\mathcal{D}$ is a symmetric $2$-design by Lemma \ref{L2bis}. Then $\mathcal{D}_{0} \cong AG_{2}(2)$ or $AG_{2}(3)$ as a consequence of \cite[Theorem 1.1]{Mo}, whereas $\mathcal{D}_{0}$ is a $2$-$%
(k_{0}^{2},k_{0},\lambda )$ design. Thus $\lambda>k_{0}$

The group $G_{\Delta }^{\Delta }$ contains a nontrivial 
$\lambda $-element $\beta $ since $r_{0}=(k_{0}+1)\lambda $ and $G_{\Delta
}^{\Delta }$-acts flag-transitively on $\mathcal{D}_{0}$. Further, $%
A_{v_{0}}\trianglelefteq G_{\Delta }^{\Delta }$ is ruled out by \cite[%
Theorem 1]{ZCZ}. Hence $G_{\Delta }^{\Delta }$ is either affine or an almost
simple group classified in \cite[Theorem 1.1 and Corollary 1.3]{LS} since $\lambda >k_{0}$. If $G_{\Delta }^{\Delta }$ is
an almost simple group. Then $G_{\Delta }^{\Delta }$ is one of the groups as
in \cite[Table 2]{LS}. Now, $A_{v_{0}}\nleq G_{\Delta }^{\Delta }$, $%
v_{0}=k_{0}^{2}$ and $\lambda >k_{0}$, and $r_{0}=\lambda ^{2}$ for $\lambda
=k_{0}+1$ lead to one of the following cases:

\begin{enumerate}
\item $Soc(G_{\Delta }^{\Delta })\cong PSL_{w}(s)$ and $k_{0}^{2}=\frac{%
s^{w}-1}{s-1}$;

\item $Soc(G_{\Delta }^{\Delta })\cong PSL_{2}(s)$, $s$ even, and $k_{0}^{2}=%
\frac{s}{2}(s\pm 1)$;

\item $Soc(G_{\Delta }^{\Delta })\cong Sp_{w}(2)$ and $%
k_{0}^{2}=2^{w-1}(2^{w}+1)$;

\item $Soc(G_{\Delta }^{\Delta })\cong Sp_{w}(s)$, $s$ even, and $k_{0}^{2}=%
\frac{s^{w}}{2}(s^{w}-1)$.
\end{enumerate}

Assume that Case (1) holds. If $w\geq 3$, then either $(w,s,k_{0})=(4,7,20)$
or $(5,3,11)$ by \cite[A7.1(a) and A8.1]{Rib}. The former is ruled out since 
$PSL_{4}(7)$ $\trianglelefteq G_{\Delta }^{\Delta }\leq PGL_{4}(7)$ has no
prime divisors greater than $20$, whereas $\lambda>k_{0}=20$. Thus $(w,s,k_{0})=(5,3,11)$, $PSL_{5}(3)$ $%
\trianglelefteq G_{\Delta }^{\Delta }\leq PGL_{5}(3)$ and $\lambda =13$, and
hence $b_{0}=11\cdot 12\cdot 13$. Then $G_{\Delta }^{\Delta }$ must have
transitive permutation representation of degree $b_{0}$ since $G_{\Delta
}^{\Delta }$ acts block-transitively on $\mathcal{D}_{0}$, which is not the
case by \cite[Tables 8.18 and 8.19]{BHRD}. Thus $w=2$ and hence $%
k_{0}^{2}=s+1=l^{e}+1$. Then $e>1$, and hence $(k_{0},l,e)=(3,2,3)$ by \cite[%
A5.1]{Rib}. Therefore, $\lambda =7$. So $\mathcal{D}_{0}$ is a $2$-$(9,3,7)$
design and $PSL_{2}(8)$ $\trianglelefteq G_{\Delta }^{\Delta }\leq P\Gamma
L_{2}(8)$. If $x$ is any point of $\mathcal{D}_{0}$, then $G_{x}^{\Delta }$
must have transitive permutation representation of degree $28$ since $%
r_{0}=28$ and $G_{\Delta }^{\Delta }$ acts flag-transitively on $\mathcal{D}%
_{0}$. However, this is impossible since $G_{x}^{\Delta }$ contains a
subgroup of index at most $3$ isomorphic to a Frobenius group of order $%
8\cdot 7$.

Assume that Case (2) holds. Then $k_{0}^{2}=2^{e-1}(2^{e}\pm 1)$ since $%
s=2^{e}$, and hence there is a positive integer $X$ such that $%
X^{2}=2^{e}\pm 1$ with $e\geq 3$ is odd. Then there are no $X$ such that $%
X^{2}=2^{e}-1$ by \cite[A3.1]{Rib}, whereas $X=e=3$ is the unique solution
for $X^{2}=2^{e}+1$ \cite[A5.1]{Rib}. This case leads to $s=8$ and $k_{0}=6$%
, and hence $\lambda =k_{0}+1=7$. Further, $PSL_{2}(8)$ $\trianglelefteq
G_{\Delta }^{\Delta }\leq P\Gamma L_{2}(8)$. However, this case is excluded
since $r_{0}=7^{2}$ does not divide the order of $G_{\Delta }^{\Delta }$.
Finally, Cases (3) and (4) are ruled out since $w$ is even.

Assume that $G_{\Delta }^{\Delta }$ is of affine type either $%
SL_{c/h}(s^{h})\trianglelefteq G_{\Delta }^{\Delta }\leq \Gamma L_{w}(s)$
and $k_{0}^{2}=s^{c}$, or $G_{\Delta }^{\Delta }\cong A_{7}$, $k_{0}^{2}=16$%
, $s=2$, $h=1$, $c=4$ and $\lambda =7$ by \cite[Table 1]{LS} since $%
v_{0}=k_{0}^{2}$ and $\lambda >k_{0}$, and $r_{0}=\lambda ^{2}$ for $\lambda
=k_{0}+1$. Also, in the two remaining cases, $\lambda \neq k_{0}+1$. Thus, $%
\lambda \nmid v_{0}(v_{0}-1)$ and $\lambda >k_{0}=\sqrt{v}$. Then $\mathcal{D}_{0}\cong
AG_{c/h}(s^{h})$, $c/h\geq 3$, with all affine planes as blocks and $\lambda
=\frac{s^{c-h}-1}{s^{h}-1}$ by Proposition \ref{D0AffSpace}. In particular, $v_{0}=s^{c}$ and $k_{0}=s^{2h}$.
Therefore, $v_{0}=k_{0}^{2}$ implies $c/h=4$ and so $\lambda =s^{2h}+s^{h}+1$%
. Thus $\lambda =k_{0}+\sqrt{k_{0}}+1>v_{1}=k_{0}+2$ and hence $\lambda
\nmid \left\vert G^{\Sigma }\right\vert $ since $\lambda $ is prime and $%
G^{\Sigma }$ acts transitively on $\Sigma $. Therefore, $\lambda \mid
\left\vert G_{(\Sigma )}\right\vert $.

If $G^{\Sigma }$ is of affine type, then $v_{1}=p^{d}$ with $p$ prime and $%
d\geq 1$. Then $s^{2h}+2=p^{d}$ and hence $(s,q,n)=(5,3,3)$ for $d>1$ by 
\cite[Lemma 4]{SS}. If $d=1$, then $s^{h}$ is not a power of $3$, then $%
s^{2h}\equiv 1\pmod{3}$ and hence $p=3$. So $s^{h}=1$, a contradiction.
Thus $s^{h}$ is a power of $3$. Therefore, either $SL_{3}(3)\trianglelefteq
G_{\Delta }^{\Sigma }\leq GL_{3}(3)$ and $s^{h}=5$ or $G^{\Sigma }\cong
Z_{p-1}$ with $p=s^{2h}+2$ prime and $s^{h}$ a power of $3$. If $G^{\Sigma }$
is almost simple, since $Soc(G^{\Sigma })$ is one of the group $A_{v_{1}}$, $%
PSL_{n}(q)$ ($n\geq 2$),\thinspace\ or $PSL_{2}(11)$ or $M_{11}$ with $s=3$
by \cite[(A)]{Ka} since $s^{2h}+2\neq q^{j}+1$ with $j=1,2$ by \cite[B1.1]%
{Rib}. Thus, one of the following holds:

\begin{enumerate}
\item[(i)] $G^{\Sigma }\cong AGL_{1}(p)$ with $p=s^{2h}+2$ prime and $s^{h}$ a
power of $3$;

\item[(ii)] $ASL_{3}(3)\trianglelefteq G^{\Sigma }\leq AGL_{3}(3)$ and $%
(v_{0},k_{0},\lambda ,v_{1},k_{1})=(5^{4},5^{2},31,3^{3},26)$;

\item[(iii)] $Soc(G^{\Sigma })\cong A_{s^{2h}+2}$;

\item[(iv)] $Soc(G^{\Sigma })\cong PSL_{n}(q)$, $n\geq 2$ and $(n,q)\neq
(2,2),(2,3)$, and $s^{2h}+2=\frac{q^{n}-1}{q-1}$;

\item[(v)] $Soc(G^{\Sigma })\cong PSL_{2}(11)$ or $M_{11}$, and and $%
(v_{0},k_{0},\lambda ,v_{1},k_{1})=(3^{4},3^{2},13,11,10)$.
\end{enumerate}

In (ii)--(v), $G_{\Delta }^{\Sigma }$ contains a non-solvable perfect
subgroup $N$ isomorphic to $[3^{2}]:SL_{2}(3)$, $A_{s^{2h}+1}$, $%
[q^{n-1}]:SL_{n-1}(q)$, $A_{5}$ or $PSL_{2}(11)$, respectively, acting transitively on $%
\Sigma \setminus \left\{ \Delta \right\} $

Assume that $G_{(\Sigma )}^{\Delta }\neq 1$. Then $Soc(G_{\Delta }^{\Delta
})\trianglelefteq G_{(\Sigma )}^{\Delta }$ since $G_{\Delta }^{\Delta }$
acts point-primitively on $\mathcal{D}_{0}$ by the minimality of $\Sigma$. Therefore, $SL_{4}(s^{h})%
\trianglelefteq G_{\Delta }^{\Delta }/G_{(\Sigma )}^{\Delta }$ or $G_{\Delta
}^{\Delta }/G_{(\Sigma )}^{\Delta }\leq \Gamma L_{1}(s^{h})$ since $%
ASL_{4}(s^{h})\trianglelefteq G_{\Delta }^{\Delta }\leq A\Gamma L_{4}(s^{h})$%
. The former implies $SL_{4}(s^{h})\trianglelefteq G_{\Delta }^{\Sigma
}/G_{(\Delta )}^{\Sigma }$. Then $\lambda =s^{2h}+s^{h}+1$ divides the order
of $G_{\Delta }^{\Sigma }$, and hence the order of $G^{\Sigma }$, a
contradiction. Thus $G_{\Delta }^{\Delta }/G_{(\Sigma )}^{\Delta }\leq
\Gamma L_{1}(s^{h})$, and hence $G_{\Delta }^{\Sigma }/G_{(\Delta )}^{\Sigma
}\leq \Gamma L_{1}(s^{h})$ since $G_{\Delta }^{\Delta }/G_{(\Sigma )}^{\Delta } \cong G_{\Delta }^{\Sigma }/G_{(\Delta )}^{\Sigma
}$. Then $N\leq $ $G_{(\Delta )}^{\Sigma }$ since $N$
is a normal perfect group $N$. Therefore $G_{(\Delta )}^{\Sigma }$, and
hence $G_{(\Delta )}$ acts transitively on $\Sigma \setminus \left\{ \Delta
\right\} $.

Let $B$ be any block of $\mathcal{D}$ such that $B\cap \Delta \neq
\varnothing $. Thus $B\cap \Delta $ is a block of $\mathcal{D}_{0}$
preserved by $G_{(\Delta )}$ Moreover, $B$ is the unique block of $\mathcal{D%
}$ intersecting $\Delta $ in $B\cap \Delta $ since $\mu =1$, hence $%
G_{(\Delta )}$ preserves $B$. Then $B$ intersects each element in $\Sigma
\setminus \left\{ \Delta \right\} $ in a non-empty set since $k_{1}\geq 2$, $%
G_{(\Delta )}$ preserves $B$ and acts transitively on $\Sigma \setminus
\left\{ \Delta \right\} $. So $B$ intersects each element of $B$ in a
non-empty, whereas $k_{1}=v_{1}-1$.

Finally, assume that (i) holds. Hence, $G_{\Delta }^{\Sigma }\cong Z_{s^{2h}+1}$ with $s^{h}$
a power of $3$. Then $Z_{\frac{s^{2h}+1}{u}}\leq \Gamma L_{1}(s^{h})$ since $%
G_{\Delta }^{\Sigma }/G_{(\Delta )}^{\Sigma }\leq \Gamma L_{1}(s^{h})$, and
hence $\frac{s^{2h}+1}{2u}\mid h$. Then $2u$ is divisible by a primitive
prime divisor $t$ of $s^{2h}+1$ \cite[Proposition 5.2.15(ii)]{KL}. Let $%
\vartheta \in G_{(\Delta )}^{\Sigma }$ of order $t$. Let $B$ be any block of 
$\mathcal{D}$ such that $B\cap \Delta \neq \varnothing $. Then $\left\langle
\vartheta \right\rangle \leq G_{B\cap \Delta }$, and hence $\left\langle
\vartheta \right\rangle \leq G_{B}$ since $B$ is the unique block of $%
\mathcal{D}$ intersecting $\Delta $ in $B\cap \Delta $ being $\mu =1$. Then $%
\left\langle \vartheta \right\rangle $ fixes a further element $\Delta
^{\prime }\in \Sigma \setminus \left\{ \Delta \right\} $ such that $B\cap
\Delta ^{\prime }\neq \varnothing $ since $\left\langle \vartheta
\right\rangle \leq G_{B\cap \Delta }$ and $k_{1}-1=s^{2h}$, and hence $%
\left\vert \left( \Delta ^{\prime }\right) ^{G_{\Delta }^{\Sigma
}}\right\vert \leq \frac{s^{2h}+1}{t}=\frac{v_{1}-1}{t}$. So, $\left( \Delta
^{\prime }\right) ^{G_{\Delta }^{\Sigma }}\neq \Sigma \setminus \left\{
\Delta \right\} $, whereas $G_{\Delta }^{\Sigma }\cong Z_{s^{2h}+1}$ acts
regularly on $\Sigma \setminus \left\{ \Delta \right\} $.

Assume that $G_{(\Sigma )}^{\Delta }=1$. Then $G_{(\Sigma )}\leq G_{(\Delta
)}$, and hence $G_{(\Sigma )}\leq G_{(\Delta ^{\prime \prime })}$ for each $%
\Delta ^{\prime \prime }\in \Sigma \setminus \left\{ \Delta \right\} $ since 
$G_{(\Sigma )}\trianglelefteq G$ and $G$ acts transitively on $\Sigma $.
Thus $G_{(\Sigma )}=1$, and hence $G=G^{\Sigma }$ is as in one of the cases
(1)--(5). Moreover, $ASL_{4}(s^{h})\trianglelefteq G_{\Delta }^{\Delta }\leq
A\Gamma L_{4}(s^{h})$. Then the unique admissible case is (4) with $s^{h}=q$%
. So $q^{2}+2=\frac{q^{n}-1}{q-1}$, and hence $q^{2}+1=q\frac{q^{n-1}-1}{q-1}
$ which is a contradiction.

\end{proof}

\bigskip

\begin{theorem}\label{Teo6}
Assume that Hypothesis \ref{hyp1} holds. Then one of the following holds:
\begin{enumerate}
\item $\mathcal{D}$ is one of the two $2$-$(16,6,2)$ designs as in \cite[Section 1.2]{ORR};
\item $\mathcal{D}$ is the $2$-$(45,12,3)$ design as in \cite[Construction 4.2]{P}.
\item $\mathcal{D}$ is a $2$-$(2^{2^{j+1}}(2^{2^{j}}+2),2^{2^{j}}(2^{2^{j}}+1),2^{2^{j}}+1)$ design when $2^{2^{j}}+1>3$ is a Fermat prime, and the following holds:
\begin{enumerate}
    \item either $\mathcal{D}_{0} \cong AG_{2}(2^{2^{j}})$ and $ASL_{2}(2^{2^{j}}) \unlhd G_{\Delta}^{\Delta} \leq A\Gamma L_{2}(2^{2^{j}})$, or $\mathcal{D}_{0}$ is a translation plane of order $2^{2^{j}}$ and  $G_{\Delta}^{\Delta} \leq A\Gamma L_{1}(2^{2^{j+1}})$;
    \item $\mathcal{D}_{1}$ is the trivial symmetric $2$-$(2^{2^{j}}+2,2^{2^{j}}+1,2^{2^{j}})$ design and either $A_{2^{2^{j}}+1} \unlhd G^{\Sigma} \leq A_{2^{2^{j}}+1}$ or $PSL_{2}({2^{2^{j}}+1}) \unlhd G^{\Sigma} \leq PGL_{2}({2^{2^{j}}+1})$
\end{enumerate}
\end{enumerate}
\end{theorem}

\begin{proof}
Recall that $\mathcal{D}_{0}$ is a translation plane of order $k_{0}$ and $G_{\Delta}^{\Delta}$ is of affine type by Proposition \ref{P1}. If $k_{0}=2$, then $v_{1}=v_{0}=4$ and hence $k_{1}=3$ and $\lambda \geq 2$.
Actually, $\lambda =2$ or $3$ since $G$ is permutationally isomorphic to $G_{\Delta }^{\Delta }\wr
G^{\Sigma }$ with both $G_{\Delta }^{\Delta }$ and $G^{\Sigma }$ subgroups
of $S_{4}$. Then $\mathcal{D}$ is one of the
two symmetric $2$-$(16,6,2)$ designs constructed \cite[Section 1.2]{ORR} by by Corollary \ref{C2} since $v=64$.

If $k_{0}\geq 3$, then $\mathcal{D}_{0}$ and $G_{\Delta }^{\Delta }$ are as in Table \ref{transplanes} by \cite[Main Theorem and Examples 1.1 and 1.2]{LiebF}.

\begin{table}[h!]
\tiny
 \caption{Admissible $\mathcal{D}_{0}$ translation plane and $G_{\Delta }^{\Delta }$}\label{transplanes}
    \centering
\begin{tabular}{llll}
\hline
Line & $\mathcal{D}_{0}$ & $k_{0}$ & Conditions on $G_{\Delta }^{\Delta }$
\\ 
1 & translation plane & $s=l^{e}$ & $G_{\Delta }^{\Delta }\leq A\Gamma
L_{1}(s^{2})$ \\ 
2 & $AG_{2}(s)$ & $s=l^{e}$ & $SL_{2}(s)\trianglelefteq G_{\Delta }^{\Delta
}\leq \Gamma L_{2}(s)$ \\ 
3 & $AG_{2}(s)$ & $s=11,23$ & $G_{\Delta }^{\Delta }$ is one of three
solvable flag-transitive groups in \cite[Table II]{Fo1} \\ 
4 & $AG_{2}(s)$ & $s=9,11,29,59$ & $\left( G_{\Delta }^{\Delta }\right)
_{x}^{(\infty )}\cong SL_{2}(5)$ \\ 
7 & Nearfield plane & $s=9$ & $G_{\Delta }^{\Delta }$ is one of seven
solvable flag-transitive groups in \cite[5.3]{Fo2} \\ 
6 & L\"{u}neburg plane & $s^{2}$ with $s=2^{2e+1}\geq 8$ & $%
Sz(s)\trianglelefteq G_{x}^{\Delta }\leq (Z_{s-1}\times Sz(s)).Z_{2e+1}$ \\ 
7 & Hering plane & $27$ & $G_{x}^{\Delta }\cong SL_{2}(13)$\\
\hline
\end{tabular}%
\end{table}

Assume that $\lambda \geq k_{0}+2$. Then, it is easy to verify
that, the order of $G_{\Delta }^{\Delta }$ is not divisible by $\lambda $
when $G_{\Delta }^{\Delta }$ is as in Table \ref{transplanes}. Therefore, $\lambda \mid
\left\vert G^{\Sigma }\right\vert $ since $\lambda \mid \left\vert
G\right\vert $ and $G$ is isomorphic to a subgroup of $G_{\Delta }^{\Delta
}\wr G^{\Sigma }$. Then $\lambda \leq k_{0}+2$ since $G^{\Sigma }$ acts
faifthfully on $\mathcal{D}_{1}$ and $v_{1}=k_{0}+2$. Actually, one has $%
\lambda =k_{0}+2$ since $\lambda \geq k_{0}+2$. Then $G^{\Sigma }$ is one of
the groups listed in \cite[p. 99]{DM}, and hence any Sylow $\lambda $%
-subgroup of $G^{\Sigma }$ is cyclic of order $\lambda $ and acts
point-regularly on $\mathcal{D}_{1}$. Then any Sylow $\lambda $-subgroup of $%
G$ is cyclic of order $\lambda $ and acts point-regularly on $\mathcal{D}%
_{1} $ since $G$ is isomorphic to a subgroup of $G_{\Delta }^{\Delta }\wr
G^{\Sigma }$ and the order of $G_{\Delta }^{\Delta }$ is not divisible by $%
\lambda $. However, this is impossible since the order of $G_{x}$, where $x$
is any point of $\mathcal{D}$, is divisible by $r=(k_{0}+1)\lambda $, because $G$ acts flag-transitively on 
$\mathcal{D}$. Thus $\lambda <k_{0}+2$, and hence either $\lambda =k_{0}$%
\bigskip , or $\lambda =k_{0}+1$ since $\lambda \geq k_{0}\geq 3$. If $\lambda =k_{0}$ then $\mathcal{D}$ is a $2$-design by Lemma \ref{L2bis}, and hence the assertions (1) and (2) follow from \cite[Theorem 1.1]{Mo}.

Finally, assume that $\lambda =k_{0}+1$. Then $k_{0}=2^{2^{j}}$ is an even power of $2$ and $\lambda $ is a Fermat prime since $k_{0}$ is a power of a prime. Thus, $\mathcal{D}$ is a $2$-$(2^{2^{j+1}}(2^{2^{j}}+2),2^{2^{j}}(2^{2^{j}}+1),2^{2^{j}}+1)$ design.
Now, checking Table \ref{transplanes}, one has $k_{0}=s=2^{2^{j}}$, and either $G_{\Delta
}^{\Delta }\leq A\Gamma L_{1}(q^{2})$, or $\mathcal{D}_{0}\cong AG_{2}(q)$
and $SL_{2}(q)\trianglelefteq G_{\Delta }^{\Delta }\leq \Gamma L_{2}(q)$, which is (3a).
Then $v_{1}=\lambda +1=2^{2^{j}}+2$, which is not a power of a prime. Thus $%
G^{\Sigma }$ is almost simple since $G^{\Sigma }$ acts point-$2$%
-transitively on $\mathcal{D}_{1}$, and hence either $A_{\lambda
+1}\trianglelefteq G^{\Sigma }\leq S_{\lambda +1}$, or $PSL_{n}(s)%
\trianglelefteq G^{\Sigma }\leq P\Gamma L_{n}(s)$ with $\lambda +1=\frac{%
s^{n}-1}{s-1}$ by \cite[(A) and (B)]{Ka} since $v_{1}=\lambda +1$ with $%
\lambda $ a Fermat prime. The latter implies $\lambda =s\frac{s^{n-1}-1}{s-1}
$ and hence $\lambda =s=2^{2^{j}}+1$ and $n=2$. Therefore, $A_{\lambda
+1}\trianglelefteq G^{\Sigma }\leq S_{\lambda +1}$, or $PSL_{2}(\lambda
)\trianglelefteq G^{\Sigma }\leq P\Gamma L_{2}(\lambda )$, which is (3b). This completes the proof.
\end{proof}

\bigskip

\begin{proof}[Proof of Theorem \ref{main}]
Let $\mathcal{D}=(\mathcal{P}, \mathcal{B})$ be any $2$-$(v,k,\lambda)$ design with $\lambda$ a prime number, and $G$ be a
flag-transitive, point-imprimitive automorphism group of $\mathcal{D}$ preserving a nontrivial 
partition $\Sigma $ of $\mathcal{P}$ with $v_1$ classes of size $v_0$. Possibly by substituting $\Sigma $ with a $G$-invariant partition of $\mathcal{P}$ finer that $\Sigma$, we may assume that $\Sigma$ is minimal with respect to the ordering $\preceq$ defined in the Subsection \ref{min}. Then $\lambda \mid r$ by Theorem \ref{PetereDemb}. Further, either $\mathcal{D}_{1}$ is either symmetric $1$-design with $k_{1}=v_{1}-1$, or $\mathcal{D}_{1}$ is a (possibly trivial) $2$-design by Theorem \ref{CamZie}(2). Actually, the latter cannot occur by Theorem \ref{Teo5}. Thus $\mathcal{D}_{1}$ is symmetric $1$-design with $k_{1}=v_{1}-1$, and hence the assertion now follows from Theorem \ref{Teo6}. This completes the proof.     
\end{proof}

\bigskip

\begin{proof}[Proof of Corollary \ref{mainCor}]
Let $\mathcal{D}=(\mathcal{P}, \mathcal{B})$ is a $2$-$(v,k,\lambda)$ design with $\lambda$ a prime number admitting a
flag-transitive automorphism group $G$, not isomorphic to the three possibilities of Theorem \ref{main}. Then $G^{\Sigma}$ acts $G$ acts point-primitively on $\mathcal{D}$ by Theorem Theorem \ref{main}. Moreover, either $(r,\lambda)=1$ or $\lambda \mid r$  since $\lambda$ is a prime number. Therefore, either $G$ is of affine type, or $G$ is an almost simple group by \cite[Theorem]{Zie} or \cite[Theorem 1]{ZC}, respectively.
\end{proof}

\bigskip


\end{document}